\documentclass[12pt]{amsart}     %% Input Standard CPAM class.
\usepackage{tkz-fct}

\usepackage[margin=1in]{geometry}

\usepackage[pagebackref=true,
]{hyperref} 
\hypersetup{colorlinks=true,linkcolor=blue,filecolor=magenta,citecolor=green,urlcolor=blue,final} 

 \usepackage{amssymb}

\usepackage{amsthm,amssymb,amsmath,amscd}
\usepackage{amsfonts,color}
\usepackage{latexsym}
\usepackage[all]{xy}
\usepackage{ulem}

\usepackage{mathptmx}
\newtheorem{thm}{Theorem}[section]
\newtheorem{lem}[thm]{Lemma}
\newtheorem{cor}[thm]{Corollary}
\newtheorem{prop}[thm]{Proposition}
\theoremstyle{definition}
\newtheorem{df}[thm]{Definition}
\theoremstyle{remark}
\newtheorem{rem}[thm]{Remark}

\newtheorem{example}[thm]{Example}

\numberwithin{equation}{section}

\newtheorem*{ack}{Acknowledgements}

\newcommand{\bC}{{\mathbb C}}

\newcommand{\bK}{{\mathbb K}}

\newcommand{\bN}{{\mathbb N}}

\newcommand{\bR}{{\mathbb R}}
\newcommand{\bQ}{{\mathbb Q}}
\newcommand{\bZ}{{\mathbb Z}}
\newcommand{\bT}{{\mathbb T}}

\newcommand{\cA}{{\mathcal A}}

\newcommand{\cE}{{\mathcal E}}
\newcommand{\cF}{{\mathcal F}}

\newcommand{\cH}{{\mathcal H}}

\newcommand{\cO}{{\mathcal O}}

\newcommand{\Hom}{{\rm{Hom}}}
\newcommand{\Sp}{{\rm{Spec}}}

\newcommand{\wti}{\widetilde}

\newcommand{\Gr}{\text{\rm Gr}}

\newcommand{\codim}{\hbox{\rm codim}\,}

\newcommand{\lra}{\longrightarrow}

\newcommand{\con}{{\it Cone}}
\newcommand{\Int}{{\rm Int}}
\newcommand{\Relint}{{\rm Relint}}

\newcommand{\ch}{{\rm ch}}

\newcommand{\td}{{\rm td}}
\newcommand{\mult}{{\rm mult}}

\newcommand{\os}{$O_{\sigma}$}
\newcommand{\vs}{$V_{\sigma}$}

\newcommand{\bb}[1]{\mbox{$\mathbb{#1}$}}

\def\sig{\sigma}
\def\Sig{\Sigma}

\def\be{\begin{equation}}
\def\ee{\end{equation}}
\def\bt{\begin{thm}}
\def\et{\end{thm}}
\def\bc{\begin{cor}}
\def\ec{\end{cor}}
\def\br{\begin{rem}}
\def\er{\end{rem}}
\def\bp{\begin{prop}}
\def\ep{\end{prop}}
\def\bl{\begin{lem}}
\def\el{\end{lem}}
\def\bn{\begin{enumerate}}
\def\en{\end{enumerate}}
\def\bex{\begin{example}}
\def\eex{\end{example}}
\def\bd{\begin{df}}
\def\ed{\end{df}}

\begin{document}                        %% Standard LaTeX command

%%      ---------------------------------------------------------------------
%%      -------------------------------- TITLE -----------------------------
%%      ---------------------------------------------------------------------

\title[Euler-Maclaurin formulae]{Equivariant toric geometry \\ and Euler-Maclaurin formulae }

%%      ---------------------------------------------------------------------
%%      ------------------------------- AUTHORS -----------------------------
%%      ---------------------------------------------------------------------
\author[S. E. Cappell ]{Sylvain E. Cappell}
\address{S. E. Cappell: Courant Institute, New York University, 251 Mercer Street, New York, NY 10012, USA}
\email {cappell@cims.nyu.edu}

\author[L. Maxim ]{Lauren\c{t}iu Maxim}
\address{L. Maxim : Department of Mathematics, University of Wisconsin-Madison, 480 Lincoln Drive, Madison WI 53706-1388, USA, \newline
{\text and} \newline Institute of Mathematics of the Romanian Academy, P.O. Box 1-764, 70700 Bucharest, ROMANIA.}
\email {maxim@math.wisc.edu}

\author[J. Sch\"urmann ]{J\"org Sch\"urmann}
\address{J.  Sch\"urmann : Mathematische Institut,
          Universit\"at M\"unster,
          Einsteinstr. 62, 48149 M\"unster,
          Germany.}
\email {jschuerm@math.uni-muenster.de}

\author[J. L. Shaneson ]{Julius L. Shaneson}
\address{J. L. Shaneson: Department of Mathematics, University of Pennsylvania, 209 S 33rd St., Philadelphia, PA 19104, USA}
\email {shaneson@sas.upenn.edu}



\date{\today}

\begin{abstract}
We first investigate torus-equivariant motivic characteristic classes of  toric varieties, and then apply them via the equivariant Riemann-Roch formalism to prove very general Euler-Maclaurin type formulae for full-dimensional simple lattice polytopes.

We consider $\bT$-equivariant versions $mC_y^\bT$ and $T_{y*}^\bT:=\td_*^\bT \circ mC_y^\bT$ of the {\it motivic Chern} and, resp.,  {\it Hirzebruch characteristic classes} of a  toric variety $X$ (with  corresponding torus $\bT$), and extend many known results from the non-equivariant context to the equivariant setting. For example, the equivariant motivic Chern class $mC_y^\bT(X)$ is computed as the sum of the equivariant Grothendieck classes of the $\bT$-equivariant sheaves of {\it Zariski $p$-forms} $\widehat{\Omega}^{p}_X$ weighted by $y^p$. Using the motivic, as well as the characteristic class nature of $T_{y*}^\bT(X)$, the corresponding generalized {\it equivariant Hirzebruch $\chi_y$-genus} $\chi^\bT_y(X,\cO_{X}(D))$ of a $\bT$-invariant Cartier divisor $D$  on $X$ is also calculated.

Further global formulae for $T_{y*}^\bT(X)$  are obtained in the simplicial context  based on the Cox construction and the {\it equivariant Lefschetz-Riemann-Roch theorem} of Edidin-Graham. Alternative proofs of all these results are given via  
 {\it localization techniques} at the torus fixed points in $\bT$-equivariant $K$- and, resp., homology theories of toric varieties, due to Brion-Vergne and, resp., Brylinski-Zhang. These localization results apply to any  toric variety with a torus fixed point. In localized $\bT$-equivariant $K$-theory, we extend a classical {\it formula of Brion} for a full-dimensional lattice polytope $P$ to a weighted version. We also generalize  the {\it Molien formula} of Brion-Vergne for the localized class of the structure sheaf of a  simplicial toric variety  $X$ to the context  of  $mC_y^\bT(X)$. Similarly,  we calculate the {\it localized Hirzebruch class}  in localized  $\bT$-equivariant  homology, extending the corresponding results of  Brylinski-Zhang for the {\it localized Todd class} (fitting with the equivariant Hirzebruch class for $y=0$).

As main applications of our equivariant characteristic class formulae, we provide a geometric perspective on several  {\it weighted Euler-Maclaurin type formulae for full-dimensional simple lattice polytopes} (corresponding to simplicial toric varieties), coming from the 
 {\it equivariant toric  geometry via the equivariant Hirzebruch-Riemann-Roch} (for an ample torus invariant Cartier divisor).  Our main results even provide generalizations to {\it arbitrary equivariant coherent sheaf coefficients}, including algebraic geometric proofs of  (weighted versions of)  the Euler-Maclaurin formulae of Cappell-Shaneson, Brion-Vergne, Guillemin, etc. (all of which correspond to the choice of the structure sheaf), via the equivariant Hirzebruch-Riemann-Roch formalism. 
 In particular, we give a first complete proof of the Euler-Maclaurin formula of Cappell-Shaneson. Our approach, based on {\it motivic characteristic classes}, allows us to obtain such Euler-Maclaurin formulae also for (the interior of) a face, as well as for the polytope with several facets (i.e., codimension one faces) removed, e.g., for the interior of the polytope (as well as for equivariant characteristic class formulae for locally closed $\bT$-invariant subsets of a toric variety).  Moreover, we prove such results also in the weighted context, as well as for $\bN$-Minkowski summands of the given full-dimensional lattice polytope (corresponding to globally generated torus invariant Cartier divisors in the toric context). Some of these results are extended to local Euler-Maclaurin formulae for the tangent cones at the vertices of the given  full-dimensional lattice polytope (fitting with the localization at the torus fixed points in equivariant $K$-theory and equivariant (co)homology). Finally, we also give an application of our abstract Euler-Maclaurin formula to generalized reciprocity for Dedekind sums.
\end{abstract}

% With AMS-LaTeX, \maketitle follows the abstract
\maketitle

 \tableofcontents

%%      ---------------------------------------------------------------------
%%      ---------------------------- BODY OF PAPER --------------------------
%%      ---------------------------------------------------------------------

%%      Please input or insert the body of your paper here.

\section{Introduction}\label{intro}

Many problems in mathematics, computer science or engineering can be cast in terms of counting the number of elements of a finite set. For instance,  in view of its connection with integer programming, counting lattice points in (full-dimensional) convex polytopes is an important problem in computational and discrete geometry, as well as in operation research and computer science.
Furthermore, since in many applications such polytopes come endowed with some continuous function $f$ (e.g., a probability distribution, or color intensity in a picture), it is often the case that one is interested in computing the sum of the values of $f$ at the lattice points contained in the polytope. The resulting expression is called the \index{Euler-Maclaurin formula} {\it Euler-Maclaurin sum} for $f$ and $P$. The
lattice point counting problem then corresponds to the choice of the constant function $f=1$. 
An {\it Euler-Maclaurin formula} computes the Euler-Maclaurin sum for a suitable function $f$ on $P$ in terms of integrals over the polytope and/or its faces, 
thus establishing a powerful connection between sums (which encode discrete information) and integrals (which are continuous quantifiers); e.g., see \cite[Section 2]{KSW} for an overview.
A closely related problem is the computation of the \index{volume} {\it volume} of a lattice polytope $P$ (i.e., vertices of $P$ are lattice points)  in terms of discrete data. For example, Pick's formula computes the area of a planar triangle in terms of the number of lattice points in the triangle and those on its boundary. More generally, it is well-known that for a full-dimensional  lattice polytope $P$ the number of lattice points in the dilated polytope $\ell P$ (for $\ell \in \bZ_{>0}$) is a polynomial $E_P(\ell)$ in
the dilation factor $\ell$, called the Ehrhart polynomial of $P$, whose highest coefficient is the volume of $P$.

\medskip

As observed by Danilov \cite{D}, the lattice point counting problem is closely related to the celebrated \index{Riemann-Roch theorem} Riemann-Roch theorem \cite{BFM} from algebraic geometry, via the correspondence between lattice polytopes and \index{toric variety} {\it toric varieties}. An $n$-dimensional {toric variety} $X=X_\Sigma$ is an irreducible normal variety on which the complex affine $n$-torus $\bT\simeq (\bC^*)^n$ acts with an open orbit, e.g., see \cite{CLS, D,F1}. Toric varieties arise from combinatorial objects $\Sigma \subset N \otimes \bR \simeq \bR^n$ called \index{fan} {\it fans}, which are collections of cones in a lattice $N \simeq \bZ^n$. Here $N$ corresponds to one-parameter subgroups of $\bT$. 
Let $M \simeq \bZ^n$ be the character lattice of $\bT$. 
From a full-dimensional lattice polytope $P\subset M \otimes \bR\simeq \bR^n$ one constructs (via the associated \index{inner normal fan} {\it inner normal fan} $\Sigma_P$ of $P$) a {toric variety} $X_P:=X_{\Sig_P}$, together with an ample Cartier divisor $D_P$, so that  the number of lattice points of $P$, i.e., points of $M \cap P$, is computed by the holomorphic Euler characteristic $\chi(X_P,\cO(D_P))$.  The Riemann-Roch theorem of \cite{BFM} expresses the latter in terms of  the Chern character of $\cO(D_P)$ and the Baum-Fulton-MacPherson homology \index{Todd class} Todd class $\td_*(X_P):=\td_*([\cO_{X_P}])$ of the toric variety $X_P$, thus reducing  the lattice point counting problem to a characteristic class computation (see formulae \eqref{4a} and \eqref{4b}). Similarly, a more intricate relation between lattice point counting and the theory of Goresky-MacPherson homology $L$-classes was noted by the first and fourth author in \cite{CS1}. Hence, the problem of finding explicit formulae for characteristic classes of toric varieties is of interest not only to topologists and algebraic geometers, but also to combinatorists, programmers, etc.

\medskip

In \cite{MS1, MS2}, the second and third author computed the \index{motivic Chern class} {\it motivic Chern class} $mC_y$ and, resp., homology \index{Hirzebruch class} {\it Hirzebruch classes} $T_{y*}$, $\widehat{T}_{y*}$ \cite{BSY}  of (possibly singular) toric varieties.  
As important special cases, they obtained new (or recovered well-known) formulae for the Baum-Fulton-MacPherson Todd  classes $\td_*(X)=T_{0*}(X)=\widehat{T}_{0*}(X)$ (or MacPherson-Chern classes  $c_*(X)=\widehat{T}_{-1*}(X)$) of a toric variety $X$, as well as for the Thom-Milnor L-classes $L_*(X)=\widehat{T}_{1*}(X)$ of a simplicial projective toric variety $X$.
For instance, by taking advantage of the torus-orbit decomposition and  the motivic properties of the homology  
Hirzebruch classes, one can express the latter in terms of the (dual) 
Baum-Fulton-MacPherson Todd classes of closures of orbits. (The same  method also applies to torus-invariant subspaces of a given toric variety.) As a consequence, one gets \index{Pick formula}
{\it generalized Pick-type formulae}, namely, generalizing Danilov's observation one shows that the weighted lattice point counting, where each point in a face $E$ of the polytope 
$P$ carries the weight $(1+y)^{\dim(E)}$, amounts to the computation of the Hirzebruch class $T_{y*}(X_P)$ of the associated toric variety (see Theorem \ref{wco} and Remark \ref{wcor} for a precise statement). In the case of 
 simplicial toric varieties (e.g., associated to a simple lattice polytope), these characteristic classes were computed in \cite{MS1, MS2} by using the \index{Lefschetz-Riemann-Roch theorem} Lefschetz-Riemann-Roch theorem of Edidin-Graham \cite{EG} in the context of the geometric quotient description of such varieties \cite{Cox}; see Subsection \ref{necl} for an overview of such formulae.

\medskip

A natural generalization of the counting lattice point problem is the Euler-Maclaurin formula, which relates the sum $\sum_{m\in P \cap M} f(m)$ of the values of a suitable function $f$ at the lattice points in a lattice polytope $P \subset M_\bR:=M\otimes\bb{R}$ to integrals over the polytope and/or its faces. In this paper, we consider $f$ to be either a polynomial on $M_\bR$ or an exponential function $f(m)=e^{\langle m, z \rangle}$, or products of these two types of functions, where $\langle \cdot , \cdot \rangle: M\times N\to \bb{Z}$ is the canonical pairing and $z \in N_\bC:=N \otimes_\bZ \bC=\Hom_\bR(M_\bR, \bC)$. 

Khovanskii and Pukhlikov \cite{KP} obtained an Euler-Maclaurin formula for the sum of the values of a polynomial over the lattice points in a {\it regular} lattice polytope (corresponding to a smooth projective toric variety). The first substantial advance for non-regular polytopes was a different type of Euler-Maclaurin formula for {\it simple} polytopes achieved by the first and fourth authors in \cite{CS2,S}, by using their theory of topological {\it characteristic $L$-classes} of singular spaces (agreeing with $\widehat{T}_{1*}$ in this context). 
A few years later, the Khovanskii-Pukhlikov formula was extended to {simple} lattice polytopes by Brion-Vergne \cite{BrV1} \cite{BrV2}, the latter using the {\it equivariant Hirzebruch-Riemann-Roch} theorem for the corresponding complete simplicial toric varieties, together with localization techniques in equivariant cohomology. Other Euler-Maclaurin type formulae 
were obtained by Guillemin \cite{G} by using methods from symplectic geometry and geometric quantization, by Karshon et al. \cite{KSW} by combinatorial means, etc. 
While most of above-mentioned Euler-Maclaurin formulae are obtained by integrating the function $f$ over a dilation of the polytope, the Cappell-Shaneson formula involves a summation over the faces of the polytope of integrals (over such faces) of linear differential operators with constant coefficients, applied to the function; see also \cite{BrV3} for some formulae of this type. Later on, these Euler-Maclaurin formulae have been extended (even for non-simplicial toric varieties, resp.,  non-simple lattice polytopes)
by Berline--Vergne \cite{BeV1, BeV2} and Garoufalidis--Pommersheim \cite{GP}, together with Fischer--Pommersheim \cite{FP}, to local formulae satisfying a {\it Danilov condition} (see also the work of Pommershein--Thomas \cite{PT} for such formulae for the Todd class in the non-equivariant context). In the geometric context of toric varieties, all of these Danilov type formulae depend in a crucial way on the birational invariance of the (equivariant) Todd class of the structure sheaf. However, this birational invariance property does not apply to the more general context studied in this paper.

\medskip

In the present work we consider $\bT$-equivariant versions $mC_y^\bT$ and $T_{y*}^\bT:=\td_*^\bT \circ mC_y^\bT$
of the motivic Chern and, resp.,  Hirzebruch characteristic classes of \cite{BSY}, and we extend the formulae from \cite{MS1, MS2} to the equivariant setting. Such constructions and results are first discussed in Section \ref{sec:3}. The proofs given in Section \ref{sec:3}  are based on the Cox construction and the equivariant Lefschetz-Riemann-Roch theorem of Edidin-Graham \cite{EG}, and they resemble those of \cite{MS1,MS2}, up to some technical modifications dictated by the equivariant context. Alternative proofs of all these results are given in Section \ref{reprloc} via localization at the torus fixed points.
Here we restrict ourselves for simplicity to the toric context, although many results hold in greater generality and have a long history.
We refer to Baum-Fulton-Quart \cite{BFQ} for the equivariant Lefschetz-Riemann-Roch theorem for finite group actions,
and to Goresky-MacPherson-Kottwitz \cite{GKM} for a review of the  localization at torus fixed points.
General versions of these localization results can be found in Anderson-Fulton \cite{AF}[Sections 5, 7 and 17] for equivariant (co)homology, Brion 
\cite{Br2}[Cor.2, Sec.2.3] for equivariant Chow groups, and Thomason \cite{Th}[Thm.2.1] for equivariant algebraic K-theory.
 
Furthermore, we explain the relation (cf. also \cite{BrV2}) between the  equivariant toric  geometry via the (equivariant) Hirzebruch-Riemann-Roch (abbreviated HRR for short) and Euler-Maclaurin type formulae for simple lattice polytopes (corresponding to  projective simplicial toric varieties), fitting into the following diagram:
\begin{equation*}\begin{CD}
\text{Euler-Maclaurin  formulae for $f$}@<<< \text{equivariant HRR for $D_P$}  \\
@V f= 1 VV @VVV \\
\text{counting lattice points} @<<< \text{ HRR  for $D_P$}\:,
\end{CD}\end{equation*}
with $D_P$ the $\bT$-invariant ample Cartier divisor associated to the simple lattice polytope $P$.

Our main results provide generalizations of the top row in the above diagram  to {\it arbitrary equivariant coherent sheaf coefficients}, which for natural choices related to the toric variety (or the polytope) give uniform geometric proofs of (weighted) Euler-Maclaurin formulae via the equivariant Hirzebruch-Riemann-Roch formalism. This includes, in particular, {\it both} versions of the Euler-Maclaurin formulae, i.e., of Brion-Vergne and, resp., Cappell-Shaneson (both of which correspond to the choice of the structure sheaf). 
Moreover, this provides a first complete proof of the Euler-Maclaurin formula of Cappell-Shaneson, which is different from the one envisioned in  \cite{CS2,S}.
Our approach, based on motivic characteristic classes, allows us to obtain such Euler-Maclaurin formulae also for (the interior of) a face, as well as for the polytope with several facets (i.e., codimension one faces) removed, e.g., for the interior of the polytope. 
Moreover, we prove such results also in the weighted context, as well as for $\bN$-Minkowski summands of the given full-dimensional lattice polytope (corresponding to globally generated torus invariant Cartier divisors in the toric context). Finally, some of these results are extended to local Euler-Maclaurin formulae for the tangent cones at the vertices of the given  full-dimensional lattice polytope (fitting with localization at the torus fixed points in equivariant $K$-theory and equivariant (co)homology).

\medskip

In what follows we give a brief overview of our main results. 
Later on  each section has its own  introduction with a more detailed description of the main results contained therein.
We first introduce some notations. For simplicity, in this introduction we formulate  most of the results for a  complete  simplicial toric variety (e.g., associated to a simple full-dimensional lattice polytope). 

\subsection{Rational equivariant cohomology} 
Let $X=X_\Sig$ be an $n$-dimensional complete  simplicial  toric variety with fan $\Sig \subset N_\bR=N \otimes \bR$ and torus $\bT=T_N$.
Denote by $H^*_\bT(X;\bb{Q})$ the (Borel-type) rational \index{equivariant cohomology} equivariant cohomology of $X$, and note that for a point space one has  
\[ H^*_\bT(pt;\bQ)\simeq \bQ[t_1,\ldots,t_n] =:(\Lambda_\bT)_{\bQ}.\]
Let $M$ be the dual lattice of $N\simeq \bZ^n$.
Viewing characters $m\in M$   (resp., $\chi^m\in \bZ[M]\simeq K_0^\bT(pt)$) of $\bT$ as $\bT$-equivariant line bundles $\bC_{\chi^m}$ over a point space $pt$ gives an isomorphism $M\simeq Pic_\bT(pt)$. Taking the first equivariant Chern class 
$c^1_\bT$ (or the dual $-c^1_\bT$) gives an  isomorphism
\[
c=c^1_\bT,  \:\text{resp.}, \:  s=-c^1_\bT :\:
M \simeq H_\bT^2(pt;\bb{Z}).
\]
Hence, upon choosing a basis $m_i$ ($i=1,\dots,n$) of $M\simeq \bb{Z}^n$, one has that $H_\bT^*(pt;\bb{Q})=(\Lambda_\bT)_{\bb{Q}}\simeq \bb{Q}[t_1,\dots,t_n],$  with
$t_i=\pm c^1_\bT(\bC_{\chi^{m_i}})$ for $i=1,\dots,n$. As explained in Subsection \ref{eqcoh}, $H^*_\bT(X;\bb{Q})$ can be described as a  $H_\bT^*(pt;\bb{Q})=(\Lambda_\bT)_{\bb{Q}}$-algebra. This fact plays an important role for proving Euler-Maclaurin type formulae. In fact, we will be working with the completions $\widehat{H}_\bT^*(X;\bb{Q})$ and $(\widehat{\Lambda}_\bT)_{\bb{Q}}\simeq \bQ[[t_1,\ldots,t_n]]$ of these rings. Moreover, the equivariant Chern character $\ch^\bT$ and Todd homology classes $\td_*^\bT$ take values (upon using the equivariant Poincar\'e duality for the latter) in an {\it analytic subring} 
(defined as the image of the {\it analytic Stanley-Reisner ring,} cf.  \eqref{ans})
$$(H^*_\bT(X;\bb{Q}))^{an}\subset  \widehat{H}^*_\bT(X;\bb{Q}) \:,$$
with $  (H^*_\bT(pt;\bb{Q}))^{an} \simeq \bb{Q}\{t_1,\dots,t_n\}=:(\Lambda^{an}_\bT)_{\bb{Q}}\subset (\widehat{\Lambda}_\bT)_{\bb{Q}}$ the subring of {\it convergent power series} (around zero) with rational coefficients, i.e., after pairing with $z\in N_\bC$ one gets a convergent power series {\it function} in $z$ around zero, whose corresponding Taylor polynomials have rational coefficients.

\subsection{Generalized equivariant HRR formula and applications}
Let $X=X_\Sig$ be a  complete  simplicial toric variety, with a $\bT$-equivariant coherent sheaf $\cF$. The cohomology spaces $H^i(X;\cF)$ are finite dimensional $\bT$-representations, vanishing for $i$ large enough. Using the corresponding $\bT$-eigenspaces $H^i(X;\cF)_{\chi^m}$ as in \eqref{eigen},  on which $t \in \bT$ acts by multiplication by ${\chi^m}(t)$, the (cohomological)  {\it Euler characteristic} of $\cF$ is defined by
\be\label{ie1}
\chi^\bT(X,\cF)=\sum_{m \in M} \sum_{i=0}^n (-1)^i \dim_\bC H^i(X;\cF)_{\chi^m} \cdot e^{c(m)} \in (\Lambda^{an}_\bT)_{\bb{Q}}
\subset  (\widehat{\Lambda}_\bT)_{\bb{Q}}\:.
\ee
For $\cE$, resp., $\cF$, a $\bT$-equivariant vector bundle, resp., coherent sheaf on $X$, one then has the following \index{equivariant Hirzebruch-Riemann-Roch formula} {\it equivariant Hirzebruch-Riemann-Roch formula} (cf. \cite{EG0}[Cor.3.1]): 
\be\label{eHRRi}
\chi^\bT(X,\cE \otimes \cF)=\int_X \ch^\bT(\cE) \cap \td^\bT_*([\cF]),
\ee
where $\int_X:\widehat{H}^*_\bT(X;\bb{Q}) \to \widehat{H}^*_\bT(pt;\bb{Q})=(\widehat\Lambda_\bT)_\bQ$ is the equivariant pushforward for the constant map $X \to pt$. 
Here, 
$$\td^\bT_*:K^\bT_0(X)  \lra \widehat{CH}_*^\bT(X)\otimes \bQ \lra \widehat{H}^\bT_{2*}(X;\bQ),$$
is the equivariant Riemann-Roch map of Edidin-Graham from \cite{EG0}[Thm.3.1]
composed with the equivariant cycle map \cite{EG-1}[Sec.2.8],
and we identify its target with the completed equivariant cohomology $\widehat{H}_\bT^{2*}(X;\bQ)$ via equivariant Poincar\'e duality.
The class $\td^\bT_*(X):=\td^\bT_*([\cO_X]_\bT])$ is the {\it equivariant Todd class} of $X$.
See Subsection \ref{seRR} for examples and applications of formula \eqref{eHRRi}.

In Section \ref{gehrr}, formula \eqref{eHRRi} is extended to a \index{generalized equivariant Hirzebruch-Riemann-Roch formula} {\it generalized equivariant Hirzebruch-Riemann-Roch formula}, cf. Theorem \ref{gHRR} (which is proved more generally for closed algebraic subsets of $X$ defined by $\bT$-invariant closed subsets, using the corresponding Ishida sheaves $\widetilde{\Omega}^{p}_X$ instead of the Zariski sheaves):
\be\label{gHRRi}
\begin{split}
\chi^\bT_y(X,\cO_{X}(D))&:=
\sum _{p=0}^{n} \chi^\bT(X,\widehat{\Omega}^{p}_X \otimes  \cO_{X}(D)) \cdot y^p\\
&=\int_{X} \ch^\bT(\cO_{X}(D)) \cap T^\bT_{y*}(X),
\end{split}
\ee
with $D$ a $\bT$-invariant Cartier divisor on $X$, $\ch^\bT$ the equivariant Chern character,  and $\widehat{\Omega}^{p}_X$ the sheaf of Zariski 
$p$-forms on $X$. Here  we use the following explicit description of the equivariant motivic Chern class $mC_y^\bT(X)$ and, resp.,  equivariant  Hirzebruch class $T^\bT_{y*}(X)$ of $X$ (which hold for {\it any} toric variety, see Proposition \ref{p32}):
\be\label{14i} mC^\bT_y(X)=\sum_{p=0}^{\dim(X)} [\widehat{\Omega}^p_X]_\bT \cdot y^p \in K_0^\bT(X)[y]
\quad \text{and} \quad 
{T}^\bT_{y*}(X)=\sum_{p=0}^{\dim(X)} \td_*^\bT([\widehat{\Omega}^p_X]_\bT) \cdot y^p.
\ee
(The reader unfamiliar with these characteristic class notions can take the formulae in \eqref{14i} as their definitions.)

Let now $P$ be a full-dimensional \index{simple lattice polytope} simple lattice polytope in $M_\bR\simeq \bR^n$ with associated toric variety $X=X_P$ with torus $\bT$, inner normal fan $\Sigma=\Sig_P$ and ample Cartier divisor $D=D_P$. 
As a consequence of \eqref{gHRRi}, we obtain the following weighted formula (see Corollary \ref{c38}):
\be\label{ic1}
\chi^\bT_y(X,\cO_{X}(D))= \sum_{E \preceq P} (1+y)^{\dim(E)} \cdot \sum_{m \in \Relint(E) \cap M} e^{s(m)},
\ee
where the first sum is over the faces $E$ of $P$ and $\Relint(E)$ is the relative interior of the face $E$.

Let us next consider a globally generated $\bT$-invariant Cartier divisor $D'$ on $X=X_\Sig$, with associated (not necessarily full-dimensional) lattice polytope $P_{D'} \subset M_\bR$. Let $X_{D'}$ be the toric variety of the lattice polytope $P_{D'}$, defined via the corresponding {generalized fan} $\Sig'$ as in \cite{CLS}[Prop.6.2.3]. There is a proper toric morphism $f:X \to X_{D'}$, induced by the corresponding surjective lattice projection $N \to  N_{D'}$ given by dividing out by the minimal cone of the generalized fan of $P_{D'}$. In particular, $f\colon X \to X_{D'}$ is a \index{toric fibration} {\it toric fibration}, with $ker(\bT\to \bT')$ connected. For $\sig'$ a cone in the generalized fan $\Sig'$  of $P_{D'}$, let
\be\label{bl1}
d_\ell(X/\sig'):=\vert \Sigma_\ell(X/\sig') \vert,\ee with  
\be\label{bl2} \Sigma_\ell(X/\sig'):=\{\sig \in \Sig \mid O_\sig \subset X, \ f(O_\sig)=O_{\sig'}, \ \ell=\dim(O_\sig) - \dim(O_{\sig'})\},\ee
where $O_\sig$ ($\sig \in \Sig$) and $O_{\sig'}$ ($\sig' \in \Sig'$) are $\bT$-, and resp. $\bT'$-orbits, 
and $\vert - \vert$ denoting the cardinality of a finite set.
If $E$ is the face of $P_{D'}$ corresponding to $\sig' \in \Sig'$, we denote these multiplicities by $d_\ell(X/E)$.
Then we have the following generalization of formula \eqref{ic1} (see Corollary \ref{gHRRnef} for a more general statement)
\be\label{ic2}
\chi^\bT_y(X,\cO_{X}(D'))= \sum_{E \preceq P_{D'}} \left( \sum_{\ell \geq 0} 
(-1)^\ell \cdot d_\ell(X/E) \cdot (1+y)^{\ell + \dim(E)}\right)  \cdot  \sum_{m \in \Relint(E) \cap M} e^{s(m)}.
\ee
By forgetting the $\bT$-action (i.e., setting $s(m)=0$ for all $m \in M$), 
 we get the following weighted lattice point counting for lattice polytopes associated to globally generated $\bT$-invariant Cartier divisors:
 \be\label{fort}
 \chi_y(X,\cO_{X}(D'))= \sum_{E \preceq P_{D'}} \left( \sum_{\ell \geq 0} 
(-1)^\ell \cdot d_\ell(X/E) \cdot (1+y)^{\ell +\dim(E)}  \right) \cdot  \vert \Relint(E) \cap M \vert .
\ee

\subsection{Equivariant motivic Chern and Hirzebruch classes of toric varieties}
In Section \ref{sec:3}, we extend the characteristic class formulae
 from \cite{MS1, MS2} for the motivic Chern and Hirzebruch classes to the equivariant setting. Here we use the global \index{Cox construction} Cox construction (see Subsection \ref{Cox}) and the equivariant Lefschetz-Riemann-Roch theorem of Edidin-Graham \cite{EG}. 
Let $X:=X_\Sig$ be an $n$-dimensional complete  simplicial  toric variety with fan $\Sig$. 
Then the \index{equivariant Hirzebruch class} equivariant Hirzebruch class ${T}^\bT_{y*}(X)$ is computed by (cf. Theorem \ref{eqHirz}):
\be\label{ieHirz0}
{T}^\bT_{y*}(X)= (1+y)^{n-r} \cdot \sum_{g \in G_{\Sig}}  \prod_{\rho \in \Sig(1)} \frac{ F_{\rho} \cdot 
\big( 1+y  \cdot a_{\rho}(g)  \cdot e^{-F_{\rho}}\big)}{1-a_{\rho}(g) \cdot e^{-F_{\rho}}}  \in \widehat{H}^*_\bT(X;\bQ)[y] \:,
\ee
with $r=\vert \Sigma(1)\vert$ the number of rays of $\Sig$, and $F_\rho=[D_\rho]_\bT$ denoting the equivariant fundamental class of the $\bT$-invariant divisor $D_\rho$ corresponding to the ray $\rho \in \Sig(1)$.  See Subsection \ref{Cox} for the meaning of $G_\Sig$ and $a_\rho(g)$ in the context of the Cox construction. Note that if $X$ is smooth, then $G_\Sig$ is just the identity element, and all $a_\rho(g)=1$. For $y=0$, with $T_{0*}^\bT(X)=\td_*^\bT(X)$ the \index{equivariant Todd class} equivariant Todd class of $X$, formula \eqref{ieHirz0} specializes to the classical counterpart of Brion-Vergne \cite{BrV2} for the equivariant Todd class of $X$. A more general statement is obtained in Theorem \ref{eqHirzcomp}, for the equivariant Hirzebruch classes of complements of $\bT$-invariant divisors in $X =X_\Sig$ a simplicial toric variety associated to a fan $\Sig$ of full-dimensional convex support.
Alternative proofs of all these characteristic class formulae are given in Subsection \ref{reprloc} via localization at the torus fixed points, as discussed below.

\subsection{Localized equivariant motivic Chern and Hirzebruch classes}
In Section \ref{sec:4}, we apply localization techniques in $\bT$-equivariant $K$- and (Borel-Moore) homology theories of toric varieties, due to Brion-Vergne \cite{BrV2} and, resp., Brylinski-Zhang \cite{BZ}, for the calculation of the $\bT$-equivariant motivic Chern and Hirzebruch classes in the toric context.  In this section the fan $\Sig$ need not to be simplicial.

Let $X=X_\Sig$ be an $n$-dimensional toric variety with torus $\bT=T_N$ such that the fixed-point set $X^\bT \neq \emptyset$, e.g., $X$ is projective. Let $x_\sig \in X^\bT$ be a fixed point corresponding to $\sig \in \Sig(n)$, with $U_\sig$ the corresponding $\bT$-invariant open affine variety. Consider the multiplicative subset $S\subset \bZ[M]=K_0^\bT(pt)$ generated by the elements $1-\chi^m$, for $0\neq m \in M$. Then we have the Brion-Vergne localization isomorphism $K_0^\bT(X^\bT)_S \simeq K_0^\bT(X)_S$ induced by the inclusion of fixed points. The projection map 
$$pr_{x_\sig} : K_0^\bT(X)_S \simeq K_0^\bT(X^\bT)_S \to  K_0^\bT(x_\sig)_S=\bZ[M]_S$$
can be calculated, after restriction to $U_\sig$, as $pr_{x_\sig}=\mathbb{S} \circ \chi^\bT_\sig$, with $\chi^\bT_\sig:K_0^\bT(U_\sig) \to \bZ[M]_{\rm sum} \subset \bZ[[M]]$  the local counterpart of the equivariant Euler characteristic \eqref{ie1} (with $\chi^m\in \bZ[M]$ instead of $e^{c(m)}$),  and $\mathbb{S}$ the corresponding summation map (defined on the subset $\bZ[M]_{\rm sum} \subset \bZ[[M]]$ of ``summable'' elements in $\bZ[[M]]$) as introduced by Brion-Vergne \cite{BrV2}. 
We then have (see formula \eqref{mclocal} for a more general version):
\be\label{imclocal}
\chi_\sig^\bT(mC^\bT_y(X)\vert_{U_\sig})=
\sum_{\tau \preceq \sigma} (1+y)^{\dim(O_\tau)} 
\sum_{m \in {\rm Relint}(\sigma^\vee \cap \tau^\perp) \cap M} \chi^{-m} \in  \bb{Z}[M]_{\rm sum} \otimes_\bZ \bZ[y],\ee
where the first sum is over the faces of $\sig$,  with $\sigma^\vee$  the dual cone and $\tau^\perp \subset M_\bR$ the orthogonal of $\tau \subset N_\bR$ with respect to the canonical pairing.  
For $y=0$, this specializes to $\chi_\sig^\bT(\cO_X\vert_{U_\sig})$, since $X$ has rational singularities, so that $mC_0^\bT(X)=[\cO_X]_\bT\in K_0^\bT(X)$. Similar results hold for a closed $\bT$-invariant algebraic subset of the toric variety $X$, see Remark \ref{Ishida-loc}.
As a consequence, we get the following weighted version of Brion's formula (see Corollary \ref{wBr}).
\bc\label{iwBr}
Let $P$ be a full-dimensional lattice polytope with associated projective toric variety $X=X_P$ and ample Cartier divisor $D=D_P$. For each vertex $v$ of $P$, consider the cone $C_v={\rm Cone}(P \cap M -v)=\sigma_v^\vee$, with faces $E_v={\rm Cone}(E \cap M -v)$ for $v\in E$. 
Then the following identity holds in $\bZ[M]_S \otimes_\bZ \bZ[y]$:
\be\label{if94} \chi^\bT(X, mC^\bT_y(X) \otimes \cO_X(D))= \sum_{v \ \text{\rm vertex}} \chi^{-v} \cdot \mathbb{S} \left( \sum_{v \in E \preceq P} (1+y)^{\dim(E)} \cdot \sum_{m \in \Relint(E_v) \cap M} \chi^{-m} \right).
\ee
\ec
Brion's formula \cite{Br1} is obtained from \eqref{if94} by specializing to $y=0$.

In the case of a \index{simplicial cone} simplicial cone, we get more explicit formulae by using a Lefschetz type variant 
$tr^{\bT'}_\sig$ of the Euler characteristic $\chi_\sig^\bT$, and a corresponding summation map. Let $\sig \in \Sig(n)$ be a simplicial cone with $u_1,\ldots, u_n \in N=N_\sig$ the generators of the the rays $\rho_j \in \sig(1)$, $j=1,\ldots,n$. Let $N'=N'_\sig$ be the finite index sublattice of $N$ generated by $u_1,\ldots, u_n$, and consider $\sig \in N'_{\bR}=N_\bR$ so that it is smooth with respect to the lattice $N'$. With $\bT$, $\bT'$ the corresponding $n$-dimensional tori of the lattices $N$, resp., $N'$, the inclusion $N' \hookrightarrow N$ induces a toric morphism $\pi: U'_\sig \to U_\sig$ of the associated affine toric varieties. Let $G_\sig$ be the finite kernel of the epimorphism $\pi:\bT'\to \bT$, so that $U'_\sig/G_\sig \simeq U_\sig$. 
Let $m'_{\sig,1},\ldots, m'_{\sig,n}$ be the dual basis in the dual lattice $M'=M_{\sig}'$ of $N'$, with corresponding characters $a_{\rho_j}:G_\sig \to \bC^*$ of $G_\sig$ as in the global context of the Cox construction mentioned before (see \eqref{a-intr}).  With these notations, we have (see Example \ref{extw} for details):
\be\label{isumsmusd}
\mathbb{S}\left((\chi_\sig^\bT(mC^\bT_y(X)\vert_{U_\sig})\right)=\frac{1}{\vert G_\sig \vert} \sum_{g \in G_\sig}  \prod_{i=1}^n \frac{1+y \cdot a_{\rho_i}(g^{-1}) \cdot \chi^{-m'_{\sig,i}}}{1-a_{\rho_i}(g^{-1}) \cdot \chi^{-m'_{\sig,i}}},
\ee
which can be regarded as a local $K$-theoretical version of the global formula \eqref{ieHirz0}. For $y=0$, this specializes to the \index{Molien formula} {\it Molien formula} of Brion-Vergne \cite{BrV2}, see \eqref{sumsmu}.
Similar results also hold for $\mathbb{S}\left((\chi_\sig^\bT(mC^\bT_y(V\hookrightarrow  X)\vert_{U_\sig})\right)$, with $V\subset X$ the open complement of some of the $\bT$-invariant divisors $D_\rho$, see \eqref{sumsmusf}.

We next discuss some (co)homological counterparts of the above localization formulae, again
for $X=X_\Sig$ an $n$-dimensional toric variety with torus $\bT=T_N$ such that the fixed-point set $X^\bT \neq \emptyset$.
Let $L\subset (\Lambda_\bT)_\bQ=H^{*}_\bT(pt;\bQ)$ be the multiplicative subset  generated by the elements $\pm c(m)$, for $0\neq m \in M$. With $x_\sig \in X^\bT$ a fixed point corresponding to $\sig \in \Sig(n)$, there is an associated \index{homological localization map}  {\it homological localization map} (for equivariant Borel-Moore homology) at $x_\sig $, 
$$pr_{x_\sig}:\widehat{H}_{*}^\bT(X;\bQ)_L \simeq \widehat{H}_{*}^\bT( X^\bT;\bQ)_L \lra \widehat{H}_{*}^\bT(x_\sig;\bQ)_L = L^{-1}(\widehat{\Lambda}_\bT)_\bQ.$$
These $K$-theoretic and (co)homological localization maps are compatible with the equivariant Todd class transformation of Edidin-Graham \cite{EG0} (and Brylinski-Zhang \cite{BZ}), in the following sense:
\bp\label{ipr47}
Let $\cF$ be a $\bT$-equivariant coherent sheaf on $X=X_\Sig$, and let $x_\sig \in X^\bT$ be a given fixed point of the $\bT$-action. Then:
 \be\label{if98}
\td_*^\bT([\cF])_{x_\sig}:= pr_{x_\sig} (\td_*^\bT([\cF]) = \ch^\bT \left(( \mathbb{S} \circ \chi_\sig^\bT)(\cF)\right)
\in  L^{-1}(\Lambda^{an}_\bT)_\bQ \subset L^{-1}(\widehat{\Lambda}_\bT)_\bQ,
 \ee
with $\ch^\bT:\bZ[M]_S \to L^{-1}(\Lambda^{an}_\bT)_\bQ$ induced by the $\bT$-equivariant Chern character on a point space.
\ep

As an example, we get by \eqref{imclocal} the following localized equivariant Hirzebruch class \index{localized equivariant Hirzebruch class} formula for the toric variety $X=X_\Sig$ (due to Weber \cite{W}[Thm.11.3], see also \cite{R}[Thm.6.1]):
\be\label{ieq147}\begin{split}
 T^\bT_{y*}(X)_{x_\sig}&:= \td_*^\bT\left([mC_y^\bT(X)]\right)_{x_\sig} =\sum_{\tau \preceq \sigma} (1+y)^{\dim(O_\tau)} 
\cdot (\ch^\bT \circ \mathbb{S}) \left( \sum_{m \in {\rm Relint}(\sigma^\vee \cap \tau^\perp) \cap M} \chi^{-m} \right) \\
&=\sum_{\tau \preceq \sigma} (1+y)^{\dim(O_\tau)} 
\cdot \mathbb{S} \left( \sum_{m \in {\rm Relint}(\sigma^\vee \cap \tau^\perp) \cap M} e^{s(m)} \right) \:.
\end{split}
 \ee
Specializing the above formula to $y=0$ reduces it to a corresponding localized Todd class formula of Brylinski-Zhang \cite{BZ} for $\td_*^\bT(X)_{x_\sig}=T^\bT_{0*}(X)_{x_\sig}$. In the case of a simplicial cone $\sig \in \Sig(n)$ corresponding to $x_\sig \in X^\bT$, formula \eqref{ieq147} implies the following:
\be\label{isum}
T^\bT_{y*}(X)_{x_\sig}=\frac{1}{\vert G_\sig \vert} \sum_{g \in G_\sig}  \prod_{i=1}^n \frac{1+y \cdot a_{\rho_i}(g^{-1}) \cdot e^{-c(m'_{\sig,i})}}{1-a_{\rho_i}(g^{-1}) \cdot e^{-c(m'_{\sig,i})}}.
\ee
This is exactly the localization of the global formula \eqref{ieHirz0}, as explained in Subection \ref{reprloc}.
Similar results also hold for $T^\bT_{y*}(V\hookrightarrow X)_{x_\sig}$, with $V\subset X$ the open complement of some of the $\bT$-invariant divisors $D_\rho$, see \eqref{sumsmusg}.

\subsection{Euler-Maclaurin formulae for simple lattice polytopes}

We now describe applications of the above characteristic class results to Euler-Maclaurin formulae for full-dimensional simple lattice polytopes. 

\subsubsection{Euler-Maclaurin formulae via polytope dilation}
Let $P$ be a full-dimensional  simple lattice polytope in $M_\bR\simeq \bR^n$, with toric variety $X=X_P$, inner normal fan $\Sigma=\Sigma_P$ and ample Cartier divisor $D=D_P$. Let $\Sigma(1)$ be the set of rays of $\Sigma$, corresponding to the facets $F$ of $P$. For each ray $\rho \in \Sig(1)$, let $u_\rho \in N$ be the corresponding ray generator. We also let $F_\rho:=[D_\rho]_\bT$ be the equivariant fundamental class of the $\bT$-equivariant divisor $D_\rho$ on $X$ corresponding to the ray $\rho \in \Sig(1)$.
Let $P(h)$ be the \index{dilated polytope} dilation of $P$ with respect to the vector $h=(h_\rho)_{\rho \in \Sigma(1)}$ with real entries indexed by the rays of $\Sigma$. So, if $P$ is defined by inequalities of the form
$$\langle m, u_\rho\rangle +c_\rho \geq 0,$$
with $u_\rho$ the ray generators and $c_\rho \in \bZ$, for each $\rho \in \Sig(1)$, then $P(h)$  is defined by inequalities 
$$\langle m, u_\rho \rangle +c_\rho +h_\rho \geq 0,$$
for each $\rho \in \Sig(1)$. 
Then $P(h)$ is also a simple polytope for $h_\rho$ small enough ($\rho\in \Sig(1)$).
In these notations, we have that $D=D_P=\sum_{\rho \in \Sig(1)} c_\rho \cdot D_\rho$.

Using localization techniques and adapting arguments of Brion-Vergne to our setup, we get for small $h_\rho$ ($\rho\in \Sig(1)$ and  $z\in N_\bC:=Hom_\bR(M_\bR,\bC)$) with all ${\langle i_\sig^*F_\rho,  z \rangle}\neq 0$)
 the following key relation (see Theorem \ref{them1}) between the exponentional integral and 
localized cohomology classes at the torus fixed points,
 based on the equivariant Hirzebruch-Riemann-Roch theorem: 
\be \label{eq:Int1}
\int_{P(h)} e^{\langle  m, z\rangle} \ dm =
 \sum_{\sig \in \Sig(n)} \frac{ e^{\langle i_\sig^* c^\bT_1(\cO_X(D_{P})),z\rangle}}{\langle Eu^\bT_X(x_\sig),  z \rangle} \cdot 
e^{ \sum_\rho h_\rho \langle i_\sig^* F_\rho,z\rangle} \:.
\ee
Here, we use a more explicit description of the homological localization map $pr_{x_\sig}$ given as 
$pr_{x_\sig}=\frac{ i_\sig^*}{Eu^\bT_X(x_\sig)}$, with $i_\sig: \{x_\sig\}\to X$ the fixed point inclusion and
$Eu^\bT_X(x_\sig):=i_\sig^* \left( \mult(\sig) \cdot \prod_{\rho \in \sig(1)} F_\rho\right) $
its generalized Euler class  (see \eqref{f106} for more details). Compare also with Theorem \ref{them1loc} for a counterpart of exponentional
integrals over the dilated tangent cones of the vertices of $P$. Using this we get our  our first abstract Euler-Maclaurin formula \index{Euler-Maclaurin formula}
 (see Theorem \ref{abstrEM}, Corollary \ref{abstrEMc} and Proposition \ref{pr520}):

\bt\label{abstrEMi}
Let $[\cF]\in K^\bT_0(X)$ be fixed, and choose a convergent power series $p(x_\rho) \in \bQ\{x_\rho \mid \rho \in \Sig(1) \}$ so that $p(F_\rho)=\td^\bT_*([\cF]) \in \left( H^*_\bT(X;\bQ) \right)^{an}$. Then, with $p(\frac{\partial}{\partial h})$ the corresponding infinite order differential operator obtained from $p(x_\rho)$ by substituting $x_\rho \mapsto \frac{\partial}{\partial h_\rho}$, for all $\rho \in \Sig(1)$, we have for any polynomial  function $f$ on $M_\bR$:
\begin{multline}\label{f1134i}
p\left(\frac{\partial}{\partial h}\right) \left( \int_{P(h)} f(m) \cdot e^{\langle  m, z\rangle} \ dm \right)_{\vert_{h=0}} =\\
= \sum_{m\in M} \left( \sum_{i=0}^n (-1)^i \cdot \dim_\bC H^i(X;\cO_X(D) \otimes \cF)_{\chi^{-m}}\right) \cdot f(m) \cdot e^{\langle m, z \rangle}
\:,
\end{multline}
as analytic functions in $z\in N_\bC$ with $z$ small enough.

If $v \in P$  is a vertex with  tangent cone $Tan(P,v)$ and corresponding cone
$\sig \in \Sig(n)$,
then for $z$ small enough with $-z \in Int(\sig)$ we get:
\be\label{if1133b}
p\left(\frac{\partial}{\partial h}\right) \left( \int_{Tan(P,v)(h)} e^{\langle  m, z\rangle} \ dm \right)_{\vert_{h=0}} =
  \left\langle 
\frac{ e^{v} \cdot \left( i_\sig^* p(F_\rho)\right)}{Eu^\bT_X(x_\sig)}  
, z  \right\rangle 
= \langle  e^{v} \cdot \td^\bT_*([\cF])_{x_\sig}, z  \rangle 
\:.
\ee
\et

Here $\left( H^*_\bT(X;\bQ) \right)^{an}$ is the image of the convergent power series ring
$$\bQ\{x_\rho \mid \rho \in \Sig(1) \} \subset \bQ[[x_\rho \mid \rho \in \Sig(1) ]] \to \widehat{H}^*_\bT(X;\bQ) $$
under the surjective evaluation homomorphismus $x_{\rho}\mapsto F_{\rho}$ ($ \rho \in \Sig(1) $).
Then $[\cF]\in K^\bT_0(X)$ can be represented by such a convergent power series $p(x_\rho) \in \bQ\{x_\rho \mid \rho \in \Sig(1) \}$
thanks to Proposition \ref{toddan}.

It is important to note that the operator $p(\frac{\partial}{\partial h}) \left( \int_{P(h)} f(m) \cdot e^{\langle  m, z\rangle} \ dm \right)_{\vert_{h=0}}$ (with $f=1$ in the local case) depends only on
the class $\td^\bT_*([\cF]) \in \left( H^*_\bT(X;\bQ) \right)^{an}$ and not on the chosen convergent power series representative. Such a convergent power series representative exists since 
$\td^\bT_*([\cF]) \in \left( H^*_\bT(X;\bQ) \right)^{an}$.
If one uses only polynomials in formula \eqref{f1134i} (i.e., setting $z=0$), this formula even holds for $p(x_\rho) \in \bQ[[x_\rho \mid \rho \in \Sig(1) ]]$ a formal power series with $p(F_\rho)=\td^\bT_*([\cF]) \in \widehat{H}^*_\bT(X;\bQ)$.

\begin{rem}
The formulae of Theorem \ref{abstrEMi} hold with the same operator $p(\frac{\partial}{\partial h})$ (which is fixed by the choice of a convergent power series $p(x_\rho) \in \bQ\{x_\rho \mid \rho \in \Sig(1) \}$ so that $p(F_\rho)=\td^\bT_*([\cF]) \in \left( H^*_\bT(X;\bQ) \right)^{an}$), but applied to different full-dimensional simple lattice polytopes $P$ corresponding to different choices of ample divisors $D=D_P$ on the toric variety $X_\Sig$ (e.g., for dilations of a given polytope).
\end{rem}

Note that by evaluating 
formula \eqref{f1134i} at $z=0$ and for $f=1$ (i.e., forgetting the $\bT$-action), we get a generalized \index{volume formula} {\it volume formula}, namely, 
\[
p\left(\frac{\partial}{\partial h}\right) \left( {vol} \ {P(h)}  \right)_{\vert_{h=0}} = \chi(X,\cO_X(D) \otimes \cF),
\]
with ${vol} \ {P(h)} =\int_{P(h)} dm$ the volume of $P(h)$ and the Lebesgue measure normalized so that the unit cube in $M \subset M_\bR$ has volume $1$. See \cite{BrV1}[Thm.2.15] for the case when $\cF=\cO_X$ (corresponding to counting points in $P \cap M$) and $\cF=\omega_X$ (corresponding to counting points in $\Int(P) \cap M$).

\medskip

In Subsection \ref{eEM} we explain how formula \eqref{f1134i} can be specialized to yield old and new Euler-Maclaurin type formulae. In particular, we obtain an Euler-Maclaurin formula for a simple lattice polytope with some facets removed (generalizing the classical case of (the interior of) a polytope), see formula \eqref{EMgen}. We also obtain  Euler-Maclaurin formulae for a face $E$ of $P$ (see Theorem \ref{EMintfaceb}) and for the interior of $E$ (see Theorem \ref{EMintface}).

Another way to obtain examples of explicit weighted Euler-Maclaurin formulae is by 
considering the classes $[\cF]:=[mC_y^\bT(X)] \in K_0^\bT(X)[y]$, or twisting these 
by $\cO_X(D'-D)$, for $D=D_P$ the original ample divisor associated to the full-dimensional simple lattice polytope $P$, and $D'$ any $\bT$-invariant Cartier divisor on $X$ (see Theorem \ref{bgm1}, Corollary \ref{bgm11} and Example \ref{ex526b}).

Let $D'$ be a globally generated $\bT$-invariant Cartier divisor on $X$, with associated (not necessarily full-dimensional) lattice polytope $P_{D'} \subset M_\bR$. Let $D'-D=\sum_{\rho \in \Sig(1)} d_\rho D_\rho$ as a $\bT$-invariant Cartier divisor. 
Let $X_{D'}$ be the toric variety of the lattice polytope $P_{D'}$, defined via the corresponding {generalized fan}. 
Consider the infinite order differential operator
\begin{equation}\label{itodd222b}
T'_y\left(\frac{\partial}{\partial h}\right):= e^{\sum_{\rho \in \Sig(1)} d_\rho \cdot \frac{\partial}{\partial h_{\rho} }} \cdot T_y\left(\frac{\partial}{\partial h}\right)  \in \bQ\{ \frac{\partial}{\partial h_\rho} \mid \rho \in \Sig(1)\}[y],
\end{equation}
with $T_y(\frac{\partial}{\partial h})$ obtained by substituting $F_\rho \mapsto \frac{\partial}{\partial h_\rho}$ (for each $\rho \in \Sig(1)$) into the right-hand side of formula \eqref{ieHirz0} for the equivariant Hirzebruch class $T_{y*}^\bT(X)$.
For any polynomial $f$ on $M_\bR$, one then has by Theorem \ref{abstrEMi} the following new \index{weighted Euler-Maclaurin formula} weighted Euler-Maclaurin formula:
\begin{multline}\label{iwem01b}
T'_{y}\left(\frac{\partial}{\partial h}\right) \left( \int_{P(h)} f(m) \cdot e^{\langle  m, z\rangle} \ dm \right)_{|_{h=0}} = \\ =\sum_{E \preceq P_{D'}} \left( \sum_{\ell \geq 0} 
(-1)^\ell \cdot d_\ell(X/E) \cdot (1+y)^{\ell + \dim(E)}   \right) \cdot \sum_{m \in \Relint(E) \cap M} f(m) \cdot e^{\langle  m, z\rangle},
\end{multline}
with multiplicities $d_\ell(X/E)=d_\ell(X/\sig')$ as in \eqref{bl1}, and the face $E$ of $P_{D'}$ corresponding to the cone $\sig' \in \Sig'$.
Note that in this context $P_{D'}$ is an \index{$\bN$-Minkowski summand} $\bN$-Minkowski summand of the original polytope $P$.
Forgetting the $\bT$-action (i.e., for $f=1$ and $z=0$), one gets the following volume formula (fitting with \eqref{fort}):
\be
T'_{y}\left(\frac{\partial}{\partial h}\right) \left( vol \ P(h) \right)_{\vert_{h=0}}=
\sum_{E \preceq P_{D'}} \left( \sum_{\ell \geq 0} 
(-1)^\ell \cdot d_\ell(X/E) \cdot (1+y)^{\ell + \dim(E)}  \right) \cdot  \vert \Relint(E) \cap M \vert .
\ee

Formula \eqref{iwem01b} also provides a generalization of the classical Euler-Maclaurin formula of Brion-Vergne \cite{BrV2}. Indeed, 
specializing \eqref{iwem01b} to the case $D=D'$, with $P=P_{D'}$, one first gets the following weighted Euler-Maclaurin formula:
\be\label{iwem01bb}
T_{y}\left(\frac{\partial}{\partial h}\right) \left( \int_{P(h)} f(m) \cdot e^{\langle  m, z\rangle} \ dm \right)_{|_{h=0}} =\sum_{E \preceq P} (1+y)^{\dim(E)} \cdot \sum_{m \in \Relint(E) \cap M} f(m) \cdot e^{\langle  m, z\rangle}.
\ee
For $y=0$ this further reduces to the classical Euler-Maclaurin formula of Brion-Vergne \cite{BrV2} for simple lattice polytopes, and 
Khovanskii-Pukhlikov \cite{KP} for Delzant lattice polytopes 
(corresponding to smooth projective toric varieties):
\be\label{itodd}
Todd\left(\frac{\partial}{\partial h}\right) \left( \int_{P(h)} f(m) \cdot e^{\langle  m, z\rangle} \ dm \right)_{|_{h=0}} = \sum_{m \in P \cap M} f(m) \cdot e^{\langle  m, z\rangle},
\ee
with $Todd(\frac{\partial}{\partial h}) :=T_{0}(\frac{\partial}{\partial h}) $.

\medskip

For simplicity, we fix a basis $m_1,\dots,m_n$ of $M\simeq \bb{Z}^n$, with $t_i$ the corresponding coordinates on $M_\bK$ (for $\bK=\bQ,\bR$) with  respect to this basis, so that $\frac{\partial}{\partial t_i }=: \partial_i$ for $i=1,\dots,n$. 
Following \cite{KSW}, in Section \ref{emcs} we introduce 
power series of the form $$p(t_i,x_\rho):=\sum_{\alpha=(\alpha_i)\in \bb{N}_0^{n}}\: p_{\alpha}(x_\rho)\prod_{i=1}^n t_i^{\alpha_i}$$ in 
$\bQ\{t_1,\dots,t_n, x_\rho \mid \text{$\rho\in \Sigma(1)$}\},$ resp., $\bQ[[t_1,\dots,t_n, x_\rho \mid \text{$\rho\in \Sigma(1)$}]]$,
as convergent, resp., formal power series in the $t_i, x_\rho$, with corresponding differential operator
\begin{equation*}
\begin{split}
p\left(\partial_i, \frac{\partial}{\partial h}\right) &\left( \int_{P(h)} f(m) \cdot e^{\langle  m, z\rangle} \ dm \right)_{\vert_{h=0}}\\ & :=
\sum_{\alpha=(\alpha_i)\in \bb{N}_0^{n}}\:p_{\alpha}\left( \frac{\partial}{\partial h}\right) \left( \int_{P(h)} 
\left(\prod_{i=1}^n \partial_i^{\alpha_i} f(m) \cdot e^{\langle  m, z\rangle}\right)\ dm \right)_{\vert_{h=0}}
\end{split}
\end{equation*}
for $f$ a polynomial function on $M_\bR$, obtained by substituting $t_i\mapsto \partial_i, x_\rho \mapsto \frac{\partial}{\partial h_\rho}$ into the power series $p(t_i,x_\rho)$. We can now state our second abstract \index{Euler-Maclaurin formula} Euler-Maclaurin formula based on  equivariant Hirzebruch-Riemann-Roch, see Theorem \ref{abstrEM2} and Corollary  \ref{abstrEMc2}.
\bt\label{abstrEM2i}
Let $X=X_P$ be the projective simplicial toric variety associated to a full-dimensional simple lattice polytope $P\subset M_\bR$. Let $\Sig:=\Sig_P$ be the inner normal fan of $P$, and $D:=D_P$ the ample Cartier divisor associated to $P$. Let $[\cF]\in K^\bT_0(X)$ be fixed, and choose a convergent power series $p(t_i,x_\rho) \in \bQ\{t_1,\dots,t_n,x_\rho \mid \rho \in \Sig(1) \}$ so that $p(s(m_i),F_\rho)=\td^\bT_*([\cF]) \in \left( H^*_\bT(X;\bQ) \right)^{an}$. Then, for a polynomial  function $f$ on $M_\bR$,
\begin{multline}\label{cs1134i}
p\left(\partial_i,\frac{ \partial}{\partial h}\right) \left( \int_{P(h)} f(m) \cdot e^{\langle  m, z\rangle} \ dm \right)_{\vert_{h=0}} = \\
= \sum_{m\in M} \left( \sum_{i=0}^n (-1)^i \cdot \dim_\bC H^i(X;\cO_X(D) \otimes \cF)_{\chi^{-m}}\right) \cdot f(m) \cdot e^{\langle m, z \rangle}
\:,
\end{multline}
as analytic functions in $z$ with $z$ small enough.
\et

This time we take the $\widehat{H}^*_\bT(pt;\bQ)= (\widehat{\Lambda}_\bT)_\bQ$-algebra structure into account and present
$\left( H^*_\bT(X;\bQ) \right)^{an}$ as the image of the convergent power series ring (see Proposition \ref{SRnew})
$$\bQ\{t_1,\dots,t_n, x_\rho \mid \rho \in \Sig(1) \} \subset \bQ[[t_1,\dots,t_n, x_\rho \mid \rho \in \Sig(1) ]] \to \widehat{H}^*_\bT(X;\bQ) $$
under the surjective evaluation homomorphismus $t_i\mapsto s(m_i)=-c^1_{\bT}(\bC_{\chi^{m_i}})$ $(i=1,\dots,n$) and $x_{\rho}\mapsto F_{\rho}$ ($ \rho \in \Sig(1) $).\\

Of course, Theorem \ref{abstrEM2i} reduces to the first part of Theorem \ref{abstrEMi} in case $$p(t_i,x_\rho) \in \bQ\{x_\rho \mid \rho \in \Sig(1) \} \subset 
\bQ\{t_1,\dots,t_n,x_\rho \mid \rho \in \Sig(1) \}$$ does not depend on the variables $t_i$. As before, the operator $$p(\partial_i,\frac{ \partial}{\partial h}) \left( \int_{P(h)} f(m) \cdot e^{\langle  m, z\rangle} \ dm \right)_{\vert_{h=0}}$$ depends only on
the class $\td^\bT_*([\cF]) \in \left( H^*_\bT(X;\bQ) \right)^{an}$ and not on the chosen convergent power series representative. If one uses only polynomials in formula \eqref{cs1134i} (i.e., setting $z=0$), this formula even holds for $p(t_i,x_\rho) \in \bQ[[t_1,\dots,t_n,x_\rho \mid \rho \in \Sig(1)]]$ a formal power series with $p(s(m_i),F_\rho)=\td^\bT_*([\cF]) \in \widehat{H}^*_\bT(X;\bQ)$.

\subsubsection{Euler-Maclaurin formulae of Cappell-Shaneson type} 

Let $P\subset M_\bR\simeq \bR^n$ be as before a full-dimensional simple lattice polytope, with associated projective simplicial toric variety $X=X_P$,  inner normal fan $\Sigma=\Sigma_P$ and ample Cartier divisor $D=D_P$.
Instead of using a dilation $P(h)$ of the polytope $P$, the Euler-Maclaurin formula of Cappell-Shaneson \cite{CS2,S} uses a summation of integrals over the faces $E$ of $P$. In more detail, for any polynomial function $f$ on $M_\bR$ one has:
\begin{equation}\label{EM-CSi}
\sum_{m\in P\cap M}\; f(m) = \sum_{E\preceq P}\:\int_E \left(p_{E}(\partial_i) f(m)\right) dm \:,
\end{equation}
with $p_{E}(\partial_i)\in \bb{Q}[[\partial_1,\dots,\partial_n]]$ suitable infinite order differential operators with constant rational coefficients
in the partial derivatives with respect to the coordinates of the vector space $M_\bR$.
Here the Lebesgue measure $dm$ on $E$ is normalized so that the unit cube in the lattice $Span(E_0)\cap M$ has volume $1$, with $E_0:=E-m_0$ a translation of $E$ by a vertex $m_0 \in E$. 

The infinite order differential operators $p_{E}(\partial_i)$ appearing in \eqref{EM-CSi} are defined through some relations 
in what is called in \cite{KSW}[Sect.6] the {\it Cappell-Shaneson algebra} of $P$. In this paper, we give a geometric proof of a generalized Cappell-Shaneson type Euler-Maclaurin formula with arbitrary coherent sheaf coefficients.

For the use of integration over the faces $E$ of $P$, we prove in Theorem \ref{them2} an analogue of \eqref{eq:Int1} for faces of a polytope. Let $P(h)$ be the dilation of $P$ with respect to the vector $h=(h_\rho)_{\rho \in \Sigma(1)}$ with real entries indexed by the rays of $\Sigma$.  Similarly, we consider the dilation $E(h)$ of a fixed face $E$ of $P$, and let $\sig_E \in \Sig$ be the cone corresponding to $E$. Then we have for small $h_\rho$ ($\rho\in \Sig(1)$) and  $z\in N_\bC:=Hom_\bR(M_\bR,\bC)$  with all ${\langle i_\sig^*F_\rho,  z \rangle}\neq 0$:

\be\label{eq:Int2} 
\int_{E(h)} e^{\langle  m, z\rangle} \ dm =
 {\rm mult}(\sig_E) \cdot \sum_{\sig \in \Sig(n)} \frac{ e^{\langle (i_\sig^* c^\bT_1(\cO_X(D_{P})),z\rangle}}{\langle Eu^\bT_X(x_\sig),  z \rangle} \cdot  e^{ \sum_\rho h_\rho \langle i_\sig^* F_\rho,z\rangle} \cdot  \prod_{\rho \in \sig_E(1)} \langle i_\sig^* F_\rho, z\rangle \:.
\ee

 Let $V_E:=V_{\sig_E}$ be the closure of the $\bT$-orbit in $X$ corresponding to $\sig_E$. As explained in Remark \ref{lastrem}, the equivariant fundamental classes $[V_{E}]_\bT$ generate $\widehat{H}^*_\bT(X;\bQ)$ as a  
$\widehat{H}^*_\bT(pt;\bQ)= (\widehat{\Lambda}_\bT)_\bQ$-algebra.
Let $[\cF]\in K^\bT_0(X)$ be fixed, and choose 
elements $p_{E}(t_i)\in \widehat{H}^*_\bT(pt;\bQ)= (\widehat{\Lambda}_\bT)_\bQ\simeq \bQ[[t_1,\dots,t_n]]$ with
\begin{equation}\label{main-toddi}
\td_*^\bT([\cF])= \sum_{E\preceq P} \: p_{E}(t_i)[V_{E}]_\bT \in \widehat{H}^*_\bT(X;\bQ)\:.
\end{equation}
Then $\td_*^\bT([\cF])=p(s(m_i),F_\rho) \in \widehat{ H}^*_\bT(X;\bQ) $ for
\be\label{main-todd2i}
p(t_i,x_\rho):=\sum_{E\preceq P} \: \left( {\rm mult}(\sig_E) \cdot \prod_{\rho \in \sig_E(1)} x_\rho\right) \cdot p_{E}(t_i) \in \bQ[[t_1,\dots,t_n,x_\rho \mid \rho \in \Sig(1) ]]\:.
\ee
With these notations, we can now state our third and final abstract \index{Euler-Maclaurin formula} Euler-Maclaurin formula based on the equivariant Hirzebruch-Riemann-Roch theorem, which also provides a generalization, as well as a first complete proof, of the Cappell-Shaneson Euler-Maclaurin formula, see Theorem \ref{abstrEM3} and Remark \ref{r716}.
\bt\label{abstrEM3i}
Let $X=X_P$ be the projective simplicial toric variety associated to a full-dimensional simple lattice polytope $P\subset M_\bR$. Let $\Sig:=\Sig_P$ be the inner normal fan of $P$, and $D:=D_P$ the ample Cartier divisor associated to $P$. Let $[\cF]\in K^\bT_0(X)$ be fixed, and choose the formal power series $p(t_i, x_\rho) \in \bQ[[t_1,\dots,t_n,x_\rho \mid \rho \in \Sig(1) ]]$ as in \eqref{main-todd2i}. Then for a polynomial  function $f$ on $M_\bR$, we have:
\be\label{cs2137i}
\sum_{E\preceq P} \:  \int_{E} \left(p_{E}(\partial_i) f(m)\right) \ dm 
= \sum_{m\in M} \left( \sum_{i=0}^n (-1)^i \cdot \dim_\bC H^i(X;\cO_X(D) \otimes \cF)_{\chi^{-m}}\right) \cdot f(m)\:.
\ee
\et

\begin{rem} On the left hand side the (restriction of the) polynomial function $p_{E}(\partial_i) f(m)$ on $M_\bR$ is integrated over $E$.
Formula \eqref{cs2137i} of Theorem \ref{abstrEM3i} holds with the same operators $p_{\sig_E}(\partial_i)$ (which are fixed by formula \eqref{main-toddi}), but applied to different full-dimensional simple lattice polytopes $P$ corresponding to different choices of ample divisors $D=D_P$ on the toric variety $X_\Sig$ (e.g., for dilations of a given polytope).
\end{rem}

In the classical case $\cF:=\cO_X$, 
this is exactly Cappell-Shaneson's recipe for the definition of the differential operators $p_{E}(\partial_i)$, described here geometrically in terms of 
the equivariant Todd class $\td_*^\bT(X):=\td_*^\bT([\cO_X])\in  \widehat{H}^*_\bT(X;\bb{Q})$  (see \cite{CS2}[Thm.2] or \cite[Sect.6.2]{S}). 
In this case, \eqref{cs2137i} reduces to the Cappell-Shaneson formula \eqref{EM-CSi} (see Remark \ref{r716} for more details).

\subsubsection{Generalized reciprocity for Dedekind sums via Euler-Maclaurin formulae}
In Section \ref{sec7.3}, 
we give an application of formula \eqref{cs2137i} to {\it generalized reciprocity for Dedekind sums} (see Corollary \ref{corec} for a more general statement): 

\bc\label{coro8i} In the context of Theorem \ref{abstrEM3i}, one gets the following identity:
\be\label{274i}
\sum_{v \in P} \left( p_{{v}}(\partial_i) f \right) (0) = \sum_{m\in M} \left( \sum_{i=0}^n (-1)^i \cdot \dim_\bC H^i(X;\cF)_{\chi^{-m}}\right) \cdot f(m).
\ee
where the left hand sum is over the vertices of $P$.
\ec

If $\cF=\cO_X$ in \eqref{274i}, one gets for a polynomial  function $f$ on $M_\bR$ the following identity:
$$
\sum_{v \in P} \left( p_{{v}}(\partial_i) f \right) (0) = f(0).
$$
For instance, in the case of lattice polygons, this formula yields generalizations of \index{reciprocity law} reciprocity laws for classical \index{Dedekind sum} {\it Dedekind sums} (using, e.g., the explicit description of the operators $p_{{v}}(\partial_i)$ from \cite{CS2}).

\medskip

The interested reader may also want to consult our paper \cite{CMSSEM}, where a summary of results and explanation of key ideas of this paper are given. Further calculations of (equivariant) motivic Chern and Hirzebruch classes of (not necessarily simplicial) toric varieties are performed in \cite{MS3,MS4} via the theory of (equivariant) mixed Hodge modules on toric varieties. This allows us to also consider characteristic classes associated to the intersection cohomology complex.
These calculations are then applied to giving geometric proofs of several combinatorial results \cite{BGM} from weighted Ehrhart theory of (not necessarily simple) lattice polytopes, fitting also with the weights given by Stanley's g-polynomials.

\begin{ack} We thank the referee for some helpful comments and suggestions.
L. Maxim was partially supported by the Simons Foundation (Grant MPS-TSM-00007095), the MPIM-Bonn, and by the project ``Singularities and Applications'' - CF 132/31.07.2023 funded by the European Union - NextGenerationEU - through Romania's National Recovery and Resilience Plan.
J. Sch\"urmann was funded by the Deutsche Forschungsgemeinschaft (DFG, German Research Foundation) Project-ID 427320536 -- SFB 1442, as well as under Germany's Excellence Strategy EXC 2044 390685587, Mathematics M\"unster: Dynamics -- Geometry -- Structure. The authors also thank the University of M\"unster and the University of Wisconsin-Madison for funding our collaboration and for providing ideal working conditions.
\end{ack}

%%%%%%%%%%%%%%%%%%%%%%%%%%%

\section{Preliminaries}\label{prelim}

In this section, we collect preparatory material needed throughout the rest of the paper. In Subsection \ref{introtoric} we review basic terminology about toric varieties and lattice polytopes. In Subsection \ref{cTodd} we recall the relation, via the Riemann-Roch theorem, between counting lattice points in the (interior of a) lattice polytope and the (dual) Todd classes of the associated toric variety. This relation is generalized in Subsection \ref{wcHRR} to a weighted lattice point counting (with certain weights reflecting the face decomposition of a lattice polytope) by using the homology Hirzebruch classes and a generalized version of the Hirzebruch-Rieman-Roch theorem. In Subsection \ref{necl}, we recall from \cite{MS1} formulae for the Todd and Hirzebruch classes of a simplicial toric variety, which were deduced in loc.cit. via the Lefschetz-Riemann-Roch theorem. Subsection \ref{eqcoh} is devoted to introducing and describing the rational equivariant cohomology of a complete simplicial toric variety. The equivariant Euler characteristic is defined in Subsection \ref{esHRR}. Together with the equivariant Chern character; this is then used in Subsection \ref{seRR} to formulate the equivariant Riemann-Roch theorem of Edidin-Graham. 

%%%%%%%%%%%%

\subsection{Background on toric varieties and lattice polytopes}\label{introtoric}
In this subsection we review some basic facts and terminology from the theory of toric varieties and lattice polytopes. Specifically, we recall the definition of toric varieties via fans (Subsection \ref{211}), we give a brief account of the Cox construction of simplicial toric varieties as geometric quotients (Subsection \ref{Cox}), and we describe the relation between lattice polytopes and toric varieties (Subsection \ref{213}).
For complete details, the interested reader is referred to  \cite{CLS,D,F1,O}.

\subsubsection{Toric varieties}\label{211}
Let $M \simeq \bZ^n$ be an $n$-dimensional lattice in $\bR^n$, with its dual lattice $N=\Hom(M,\bZ)$. Denote the natural pairing by $$\langle \cdot , \cdot \rangle: M \times N \to \bZ.$$
A \index{rational polyhedral cone} {\it rational polyhedral cone} $\sig \subset N_{\bR}:=N \otimes \bR$ is a cone  
$\sig=\con(S)$ 
on a finite set $S \subset N$. Such a cone $\sig$ is \index{strongly convex cone} {\it strongly convex} if $\sig \cap (-\sig)=\{0\}$. The dimension of a cone $\sig$ is the dimension of the subspace of $N_{\bR}$ spanned by $\sig$. 
A \index{face} {\it face} $\tau$ of a cone $\sig$ (we write $\tau \preceq \sig$) is a subset $$\tau :=\{ u \in \sig \ | \  \langle m,u \rangle=0 \} \subset \sig $$
for some $m \in M \cap \check{\sig}$, where 
$$\check{\sig}=\{ m \in M_{\bR} \ | \ \langle m,u \rangle \geq 0,  \forall u \in \sig \}$$
is the \index{dual cone} {\it dual cone} of $\sig$. 
 A one-dimensional face $\rho$ of a cone $\sig$ is called a \index{ray} {\it ray}. The collection of rays of a cone $\sig$ is denoted by $\sig(1)$. For each ray $\rho \in \sig(1)$, let $u_{\rho}$ be the unique \index{ray generator} generator of the semigroup $\rho \cap N$. The $\{ u_{\rho} \}_{\rho \in \sig(1)}$ are called the {\it generators} of $\sig$. A cone $\sigma$ is called \index{smooth cone} {\it smooth} if it is generated by a part of the $\bZ$-basis of $N$. A cone $\sigma$ is \index{simplicial cone} {\it simplicial} if its generators are linearly independent over $\bR$. The \index{multiplicity} {\it multiplicity} $\mult(\sig)$ of a simplicial cone $\sig$ with generators $u_{\rho_1}, \dots, u_{\rho_k}$ is $| N_{\sig}/{(u_{\rho_1}, \dots, u_{\rho_k})} |$, where $N_{\sig}$ is the sublattice of $N$ spanned by points in $\sig \cap N$, and $|-|$ denotes the cardinality of a (finite) set.

A \index{fan} {\it fan} $\Sig$ in $N_{\bR}$ is a finite collection of strongly convex rational polyhedral cones in $N_{\bR}$ so that any face of $\sig \in \Sig$ is also in $\Sig$, and if $\sig_1, \sig_2 \in \Sig$ then $\sig_1 \cap \sig_2$ is a face of each.
The \index{support of a fan} {\it support} of a fan $\Sig$ is the set $|\Sig |:=\bigcup_{\sig \in \Sig} \sig \subset N_{\bR}$. The fan is called \index{complete fan} {\it complete} if $|\Sig |=N_{\bR}$.  
 $\Sig$ is called \index{simplicial fan} {\it simplicial} if every cone in $\sigma \in \Sig$ is simplicial. 
We denote by $\Sigma(i)$ the set of $i$-dimensional cones of $\Sigma$, and similarly, by $\sig(i)$ the collection of $i$-dimensional faces of a cone $\sig$.

\medskip

To any fan $\Sig$ one associates a \index{toric variety} {\it toric variety} $X_{\Sigma}$ as follows. Each cone $\sigma \in \Sig$ gives rise to an affine toric variety $$U_{\sig}=\Sp\big(\bC[M \cap \check \sig]\big)$$
where $\bC[M \cap \check \sig]$ is the $\bC$-algebra with generators $\{ \chi^m \ | \ m \in M \cap \check \sig \}$. 
The toric variety $X_{\Sigma}$ is obtained by gluing these affine pieces together.
The affine toric variety corresponding to the trivial cone $\{0\}$ is the torus $T_N=N \otimes \bC^*=\Sp(\bC[M])$. It follows that $M$ can be identified with the character lattice of $T_N$, i.e., $M \simeq \Hom_{\bZ}(T_N,\bC^*)$, while $N\simeq \Hom_{\bZ}(\bC^*,T_N)$ is the lattice of one-parameter subgroups of $T_N$. 

The action of the torus $T_N$ on itself extends to an algebraic action of $T_N$ on $X_{\Sigma}$ 
with finitely many \index{torus orbit} orbits $O_\sig$, one for each cone $\sig \in \Sig$. In fact, by the orbit-cone correspondence, to a $k$-dimensional cone $\sigma \in \Sig$ there corresponds a $(n-k)$-dimensional torus-orbit \os $ \ \cong (\bC^*)^{n-k}$, and the \index{orbit closure} closure \vs \ of \os \   is itself a toric variety and a $T_N$-invariant subvariety of $X_\Sig$.
In particular, each ray $\rho \in \Sig(1)$ corresponds to an irreducible $T_N$-invariant divisor $D_{\rho}:=V_{\rho}$ on $X_{\Sigma}$.  
Moreover, if $\sig$ is a face of $\nu$, i.e., $\sig \preceq \nu$, then $O_{\nu} \subseteq \overline{O}_{\sig}=V_{\sig}$ and 
$$
V_{\sig}=\bigsqcup_{\sig \preceq \nu} O_{\nu}.
$$
So the toric variety $X_{\Sigma}$ is stratified by the orbits of the $T_N$-action.

Let us also describe here, for future reference, the fan of the toric variety $V_{\sigma}$, i.e., the closure of the orbit corresponding to the cone $\sigma \in \Sigma$. Let $$N(\sigma)=N/N_\sigma,$$
with $N_\sigma$ denoting as before the sublattice of $N$ spanned by the points in $\sigma \cap N$.  Let $T_{N(\sigma)}=N(\sigma) \otimes_\bZ \bC^*$ be the torus associated to  $N(\sigma)$. For each cone $\nu \in \Sigma$ containing $\sigma$, let $\overline{\nu}$ be the image cone in $N(\sigma)_\bR$ under the quotient map $N_\bR \to N(\sigma)_\bR$. Then 
\be \label{star}
Star(\sigma)=\{ \overline{\nu} \subseteq N(\sigma)_\bR \mid \sig \preceq \nu \}
\ee 
is a fan in $N(\sigma)_\bR$, with associated toric variety isomorphic to $V_{\sigma}$ (see \cite{CLS}[Prop.3.2.7]). Note that $T_N$ acts on $V_{\sigma}$ via the morphism \be\label{quot}T_N \to T_{N(\sigma)}\ee induced by the quotient map $N \to N(\sigma)$.

If $\Sig$ is a fan in $N_{\bR} \simeq \bR^n$, then the toric variety $X_{\Sigma}$ is a complex algebraic variety of dimension $n$, 
which is  complete if and only if $X_{\Sigma}$ is  complete.
The toric variety $X_{\Sigma}$ associated to a simplicial fan is called a \index{simplicial toric variety} {\it simplicial toric variety}. Such a variety is an orbifold and therefore also a rational homology manifold, so it satisfies Poincar\'e duality over $\bQ$.  The singular locus of a 
 toric variety $X_{\Sigma}$ is $\bigcup_{\sig \in \Sig_{\rm sing}} V_{\sigma}$, the union being taken over the collection $\Sig_{\rm sing}$ of all singular (non-smooth) cones in the fan $\Sigma$.

%%%%%%%%%%%%

\subsubsection{Cox quotient construction of simplicial toric varieties}\label{Cox}
In this subsection, we recall the \index{Cox construction} Cox construction of simplicial toric varieties as geometric quotients, see \cite{Cox} and \cite{CLS}[Sect.5.1].

Let $\Sig$ be a fan in $N_{\bR} \simeq \bR^n$, with associated toric variety $X:=X_{\Sig}$ and torus $\bT:=T_N =(\bC^*)^n$. For each ray $\rho \in \Sig(1)$, denote by $u_{\rho}$ the corresponding ray generator. Let $r=| \Sig(1) |$ be the number of rays in the fan $\Sig$. For simplicity, we also assume that $X=X_{\Sig}$ contains {no torus factor}.

Using the fact that $N \simeq \Hom_{\bZ}(\bC^*,\bT)$ is identified with the one-parameter subgroups of $\bT$, define the map of tori $$\gamma:\widetilde{\bT}:=(\bC^*)^r \lra \bT \quad \text{by} \quad (t_{\rho})_{\rho} \mapsto \prod_{\rho \in \Sig(1)} u_{\rho}(t_{\rho}),$$ 
and let $G:=\ker (\gamma)$. Then $G$ is a product of a torus and a finite abelian group, so $G$ is reductive. Let $Z(\Sig) \subset \bC^r$ be the variety defined by the monoidal ideal $$B(\Sig):=\langle  \hat{x}_{\sig} : \sig \in \Sig \rangle \quad
\text{where} \quad \hat{x}_{\sig}:=\prod_{\rho \notin \sig(1)} x_{\rho}$$ 
and $(x_{\rho})_{\rho \in \Sig(1)}$ are coordinates on $\bC^r$. Then the variety 
$W:=\bC^r \setminus Z(\Sig)$
is a toric manifold, and there is a toric morphism $$\pi:W \to X.$$
The group $G$ acts on $W$ by the restriction of the diagonal action of $(\bC^*)^{r}$, and the toric morphism $\pi$ is constant on $G$-orbits. Moreover, Cox \cite{Cox} (see also \cite{CLS}[Thm.5.1.11]) proved that if $X=X_{\Sig}$ is a simplicial toric variety  containing no torus factor,
then $X$ is the \index{geometric quotient} geometric quotient $W/G$.

Under the quotient map $\pi:W \to X$, the coordinate hyperplane $\{ x_{\rho}=0 \}$ maps to the invariant Weil divisor $D_{\rho}$. More generally, if $\sig$ is a $k$-dimensional cone generated by rays $\rho_1,\dots,\rho_k$, then the orbit closure $V_{\sig}$ is the image under $\pi$ of the linear subspace $W_{\sig}$ defined by the ideal $(x_{\rho_1},\dots,x_{\rho_k})$ of the total coordinate ring $\bC\big[ x_{\rho} \mid \rho \in \Sig(1) \big]$. Let $G_{\sig}$ be the \index{stabilizer} stabilizer of $W_{\sig}$. Then, if  
\be \label{char-arho}
a_{\rho}:(\bC^*)^{r} \to \bC^*
\ee
 is the projection onto the $\rho$-th factor (and similarly for any restriction of this projection),  we have by definition that
\be\label{stab}
\begin{split}
G_{\sig}&=\{g \in G \ | \ a_{\rho}(g)=1, \forall \rho \notin \sig(1) \}\\
&\simeq \{(t_{\rho})_{\rho \in \sig(1)}  \ | \ t_{\rho}\in \bC^*, \prod_{\rho \in \sig(1)} u_{\rho}(t_{\rho})=1 \},
\end{split}
\ee
so $G_{\sig}$ 
depends only on $\sig$
(and not on the fan $\Sig$ nor the group $G$). Here $u_\rho \in N$ is the corresponding ray generator for $\rho \in \sig(1)$.

For a $k$-dimensional rational simplicial cone $\sig$ generated by the rays $\rho_1,\dots,\rho_k$ one can moreover show that 
\be\label{Gsig-int}
G_{\sig}\simeq N_{\sig}/( u_1,\dots,u_k),
\ee
so $|G_{\sig}|=\mult(\sig)$ is just the \index{multiplicity} multiplicity of $\sig$, with $\mult(\sig)=1$ exactly in the case of a smooth cone.
 Let $m_i\in M_{\sig}$ for $1\leq i\leq k$ be the unique primitive elements in the dual lattice $M_{\sig}$ of $N_{\sig}$ satisfying $ \langle m_i,u_j \rangle = 0$ for $i\neq j$ and
 $ \langle m_i,u_i \rangle > 0$, 
 so that the dual lattice $M_{\sig}'$ of $(u_1,\dots,u_k)$ is generated by the elements $\frac{m_j}{\langle m_j,u_j \rangle}.$
  Then, for $g=n+(u_1,\dots,u_k)\in G_{\sig}=N_{\sig}/(u_1,\dots,u_k)$, the character $a_{\rho_j}(g)$ is also given by (see \cite{F1}[page 34]):
 \be\label{a-intr}
 a_{\rho_j}(g)=\exp\left(2\pi i\cdot \gamma_{\rho_j}(g)
 \right), \quad \text{with} \quad \gamma_{\rho_j}(g):= 
 \frac{\langle m_j,n \rangle}{\langle m_j,u_j \rangle}\:.  
 \ee
 
Consider next the finite set $$G_{\Sig}:=\bigcup_{\sig \in \Sig} G_{\sig}.$$
In what follows we write as usual $\tau \preceq \sig$ for a face $\tau$ of $\sig$, and we use the notation $\tau \prec \sig$ for a proper face $\tau$ of $\sig$. Note that
$$\tau \preceq \sig  \Longrightarrow G_{\tau} \subseteq G_{\sig}.$$
Set \be\label{subs} G_{\sig}^{\circ}:=G_{\sig} \setminus \bigcup_{\tau \prec \sig} G_{\tau}=\{ g \in G_{\sig} \mid a_{\rho}(g) \neq 1, \forall \rho \in \sig(1) \}.\ee
Since $G_{\sig}=\{id_G\}$ for a smooth cone $\sig$, it follows that $G_{\sig}^{\circ}=\emptyset$ if $\sig$ is a smooth cone of positive dimension, while $G_{\{0\}}^{\circ}=G_{\{0\}}=\{id_G\}$.
Moreover, $${G_{\sig}=\bigsqcup_{\tau \preceq \sig}G_{\tau}^{\circ},}$$ (with $\bigsqcup$ denoting a disjoint union), 
so it follows that \be\label{last} G_{\Sig}=\bigcup_{\sig \in \Sig} G_{\sig}=\bigsqcup_{\sig \in \Sig} G_{\sig}^{\circ}=\{id_G\} \sqcup \bigsqcup_{\sig \in \Sig_{\rm sing}} G_{\sig}^{\circ},
\ee
with $\Sig_{\rm sing}$ the collection of singular cones in $\Sig$.

%%%%%%%%%%%%%%%%%%%%%%%%%%%%%

\subsubsection{Lattice polytopes and their associated toric varieties}\label{213}
A \index{polytope} {\it polytope} $P \subset M_{\bR}:=M \otimes \bR$ is the convex hull of a finite set $S \subset M_{\bR}$. $P$ is called a \index{lattice polytope} {\it lattice polytope} if its vertices lie in $M$. The dimension $\dim({P})$  of a polytope $P$ is the dimension of the smallest affine subspace of $M_{\bR}$ containing $P$. A polytope $P$ whose dimension equals $n=\dim(M_{\bR})$ is called {\it full-dimensional}. Faces of codimension one of $P$ are called \index{facet} {\it facets}.  A polytope $P$ is called \index{simple polytope} {\it simple} if every vertex of $P$ is the intersection of precisely $\dim({P})$ facets. 

Let $P$ be a full-dimensional lattice polytope. For each facet $F$ of $P$ there is a unique pair $(u_F,c_F) \in N \times \bZ$ so that $P$ is uniquely described by its \index{facet presentation} {\it facet presentation}:
\be \label{facet-presentation}
P=\{ m \in M_{\bR} \mid \langle m,u_F \rangle  +c_F \geq 0, \ {\rm for \ all \ facets} \ F \ {\rm of } \ P \}.
\ee
Here $u_F$ is the unique ray generator of the inward-pointing facet normal of $F$.

To a full-dimensional lattice polytope $P$ one associates a fan $\Sig_P$ in $N_\bR$, called the \index{inner normal fan}  {\it inner normal fan} of $P$, which is defined as follows: to each face $E$ of $P$ associate the cone $\sig_E$ by 
$$
\sig_E:=\con(u_F \mid F \ {\rm contains } \ E).
$$
In particular, $\rho_F:=\sig_F=\con(u_F)$ is the ray generated by $u_F$. Let 
$$ X_P:=X_{\Sig_P} $$
 be the corresponding toric variety, which is commonly referred to as the toric variety associated to the polytope $P$.
As $\Sig_P$ is a complete fan, it follows that $X_P$ is proper.  If $P$ is simple, then $X_P$ is simplicial. Moreover, if $P$ is a \index{Delzant polytope} {\it Delzant} (or regular) polytope, then $X_P$ is smooth.

For a cone $\sig_E \in \Sig_P$ associated to a face $E$ of $P$, the corresponding orbit closure $V_{\sig_E}$ can be identified with the toric variety $X_E$, with corresponding lattice polytope defined as follows (cf. \cite{CLS}[Prop.3.2.9]): translate $P$ by a vertex $m_0$ of $E$ so that the origin is a vertex of $E_0:=E-m_0$; while this translation by $m_0$ does not change $\Sigma_P$ or $X_P$, $E_0$ is now a full-dimensional polytope in $Span(E_0)$. So $E_0$ is a full-dimensional lattice polytope relative to $Span(E_0) \cap M$, and  $X_E$ is the associated toric variety.

With the above notations, the \index{torus-invariant Cartier divisor} {\it divisor $D_P$ of the polytope $P$} is defined as:
\be\label{dp}
D_P:=\sum_{F} c_F D_F,
\ee
where $F$ runs over the collection of facets of $P$ and $D_F$ is the torus-invariant divisor corresponding to the ray $\rho_F$. Then $D_P$ is a torus-invariant Cartier divisor on $X_P$, which  is ample and basepoint free (e.g., see \cite{CLS}[Prop.6.1.10]). Hence $X_P$ is a {projective} variety. 

\begin{rem} \label{polyhedra}
The above construction also works for $P\subset M_\bR$ a full-dimensional lattice \index{polyhedron} {\it polyhedron} as in \cite{CLS}[Def.7.1.3]. Here a polyhedron  $P$ is the intersection of finitely many closed half-spaces. $P$ is called a lattice polyhedron, if its recession cone is a strongly convex rational polyhedral cone and the non-empty set of vertices of $P$ belongs to $M$.
If $P$ is full-dimensional then it has a unique facet presentation as in \eqref{facet-presentation}. In this case, the toric variety $X_P$ has a fan $\Sig_P$ with full-dimensional convex support in the sense of  \cite{CLS}[page 265] (see \cite{CLS}[Thm.7.1.6]). Then the toric divisor $D_P$ of the polyhedron $P$ as in \eqref{dp} is a basepoint free (hence, by \cite{CLS}[Thm.6.3.12], a \index{nef Cartier divisor} {\it nef}) Cartier divisor (see \cite{CLS}[page 322 and Thm.7.2.2]). Moreover, the toric variety $X_P$ is quasi-projective, and therefore semi-projective in the sense of \cite{CLS}[Prop.7.2.9].
\end{rem}

If $P$ is a full-dimensional lattice polytope (or even a full-dimensional lattice polyhedron as introduced before), one has (e.g., see  \cite{CLS}[Prop.4.3.3])
\be\label{am1}
\Gamma(X_P;\cO_{X_P}(D_P)) = \bigoplus_{m \in P \cap M} \bC \cdot \chi^m \subset \bC[M],
\ee
where $\chi^m$ denotes the \index{character} character defined by $m \in M$, and $ \bC[M]$ the coordinate ring of the torus $T_N$.
The nefness of $D_P$ also yields that (e.g., see  \cite{CLS}[Prop.9.2.3])
\be\label{am2}
H^i(X_P;\cO_{X_P}(D_P))=0, \ {\rm for \ all } \ i>0.
\ee
Moreover, if $P$ is a full-dimesional lattice polytope, one has (e.g., see  \cite{CLS}[Prop.9.2.7])
\be\label{am3}
H^i(X_P;\cO_{X_P}(-D_P))=0, \ {\rm for \ all } \ i\neq n,
\ee
and 
\be\label{am4}
H^n(X_P;\cO_{X_P}(-D_P)) = \bigoplus_{m \in \Int(P) \cap M} \bC \cdot \chi^{-m} ,
\ee
with $\Int({P})$ denoting the \index{interior} interior of $P$. Also, by \index{toric Serre duality} toric Serre duality \cite{CLS}[Thm.9.2.10], 
\be\label{am4d}
\bigoplus_{m \in \Int(P) \cap M} \bC \cdot \chi^{m} = \left( H^n(X_P;\cO_{X_P}(-D_P)) \right)^\vee \simeq H^0(X_P;\cO_{X_P}(D_P) \otimes \omega_{X_P}),
\ee
with $\omega_{X_P}$ the \index{dualizing sheaf} {\it dualizing sheaf} of $X_P$, and $H^i(X_P;\cO_{X_P}(D_P) \otimes \omega_{X_P})=0$ for all $i \neq 0$.

\br
If one wishes to apply the above formulae to the toric variety $X_E$ associated to a face $E$ of a full-dimensional lattice polytope $P$ as previously described, one needs to work with the corresponding divisor \be\label{tran} D_{E_0}={D_{P-m_0}}\vert_{X_E}=D_P\vert_{X_E} + div(\chi^{m_0}).\ee
\er

%%%%%%%%%%%%%%%%

\subsection{Counting lattice points via Todd classes}\label{cTodd}
Danilov used the Hirzebruch-Riemann-Roch theorem in \cite{D} to establish a direct connection between the problem of counting the number of lattice points in a lattice polytope and the Todd classes of the associated toric variety (see also \cite{Br1, BrV1}).

Let $P \subset M_{\bR} \simeq \bR^n$ be a full-dimensional lattice polytope with associated projective toric variety $X:=X_P$ and ample Cartier divisor $D:=D_P$. Then one has by (\ref{am1}) and (\ref{am2}) the following:
\be\label{4a} 
\begin{split}
| P \cap M | =
\chi(X,\cO_{X}(D)) & =\int_{X} \ch(\cO_{X}(D)) \cap \td_*(X),\\
&=\sum_{k\geq 0} \frac{1}{k!} \int_{X} [D]^k \cap \td_k(X),
\end{split}
\ee
with $\td_*:K_0(X) \to H_*(X)\otimes \bQ$ the Baum-Fulton-MacPherson \index{Todd class} Todd class transformation \cite{BFM} and $\td_*(X):=\td_*([\cO_X])$. Here, $K_0(X)$ is the Grothendieck group of coherent sheaves, 
and for $H_*(-)$ one can use either the Chow group $CH_*$ 
or the even degree Borel-Moore homology group $H^{BM}_{2*}$. The second equality of \eqref{4a} follows from the \index{module property} module property for $\td_*$, i.e., if 
$\beta \in K^0(X)$ and $\alpha \in K_0(X)$, then (cf. \cite{F0}[Thm.18.3]):
\be\label{modp}
\td_* (\beta \otimes \alpha)=\ch(\beta) \cap \td_*(\alpha).
\ee 
Similarly, (\ref{am4d}) and the module property \eqref{modp} yield:
\be\label{4b} 
\begin{split}
| \Int({P}) \cap M | =
\chi(X,\cO_{X}(D)\otimes \omega_X)
 =\int_{X} \ch(\cO_{X}(D)) \cap \td_*\left([\omega_{X}]\right).
\end{split}
\ee

It thus follows from (\ref{4a}) and (\ref{4b}) that counting lattice points in (the interior of) a  full-dimensional lattice polytope $P$ amounts to computing the Todd class $\td_*(X)$ (resp., the \index{dual Todd class} dual Todd class $\td_*\left([\omega_{X}]\right)$) of the associated projective toric variety $X=X_P$. 

%%%%%%%%%%%%%%%%%%%%%%%

\subsection{Weighted lattice point counting via the generalized Hirzebruch-Riemann-Roch}\label{wcHRR}
In this subsection, we recall how the homology Hirzebruch classes of \cite{BSY} can be used for counting lattice points in a full-dimensional lattice polytope $P$,  with certain weights reflecting the face decomposition
\be\label{fd} P = \bigcup_{E \preceq P}\: \Relint(E), \ee
with $\Relint(E)$ denoting the relative interior of a face $E$ of $P$. Recall that Brasselet-Sch\"urmann-Yokura defined in \cite{BSY} un-normalized homology Hirzebruch classes $T_{y*}(X)$ of a complex algebraic variety $X$ as the image of the distinguished element $[id_X]$ under a natural transformation:
$$T_{y*}:K_0(var/X) \lra H_*(X) \otimes \bQ[y]$$
defined on the relative Grothendieck group $K_0(var/X)$ of complex algebraic varieties over $X$. A normalized version $\widehat{T}_{y*}(X)$ is obtained by multiplying each degree $k$-piece of $T_{y*}(X)$ by $(1+y)^{-k}$. With these notations, one has the following \index{generalized Pick formula} {\it generalized Pick-type formula} (see \cite{MS1}[Thm.1.3]):
\bt\label{wco} Let $M$ be a lattice of rank $n$ and
let $P \subset M_{\bR}$ be a full-dimensional lattice polytope with associated projective toric variety $X:=X_P$ and ample Cartier divisor $D:=D_P$.
Then the following formula holds:
\be\label{i5}
\sum_{E \preceq P} (1+y)^{\dim(E)} \cdot | \Relint(E) \cap M |=\int_{X} \ch(\cO_{X}(D)) \cap T_{y*}(X),
\ee
where the summation on the left is over the faces $E$ of $P$, and the number of points inside a face $E$ is counted with respect to the lattice ${\rm Span}(E) \cap M$.
\et

\br\label{wcor} In the notations of Theorem \ref{wco}, we make the following observations:
\item[(a)] If $y=0$, then formula (\ref{i5}) reduces to (\ref{4a}); indeed, in the context of toric varieties, one has the specialization: $T_{y*}(X)\vert_{y=0}=\td_*(X)$ (see \cite{BSY}).
\item[({b})] The proof of (\ref{i5}) in \cite{MS1} uses motivic properties of the homology Hirzebruch classes (see \cite{MS1}[Thm.3.3]).
\item[({c})] The right-hand side of formula (\ref{i5}) is computed by the \index{generalized Hirzebruch-Riemann-Roch theorem} {\it generalized Hirzebruch-Riemann-Roch theorem} (see \cite{MS1}[Thm.2.4]), namely:
\be\label{ghrr}
\begin{split}
\chi_y(X,\cO_X(D)):=\sum_{p \geq 0} \chi(X,\widehat{\Omega}^p_X \otimes \cO_{X}(D)) \cdot y^p
&=\int_{X} \ch(\cO_{X}(D)) \cap T_{y*}(X)\\
&=\int_{X} \ch_{(1+y)}(\cO_{X}(D)) \cap \widehat{T}_{y*}(X),
\end{split}
\ee
with $\widehat{\Omega}^p_X:=j_*\Omega^p_{X_{\rm reg}}$ the \index{sheaf of Zariski $p$-forms} sheaf of Zariski $p$-forms on $X$ and $j:X_{\rm reg} \hookrightarrow X$ the inclusion of the non-singular locus. This follows from the following calculation of the un-normalized homology Hirzebruch class (see \cite{MS1}[(1.2)]): 
\be\label{19e}
 T_{y*}(X)=\sum_{p \geq 0} \td_*([\widehat{\Omega}^p_X]) \cdot y^p.
\ee
The last equality in \eqref{ghrr} is just a renormalization of the Chern character by powers of $(1+y)$, namely $\ch_{(1+y)}(\cO_{X}(D)):=e^{(1+y)c_1(\cO_{X}(D))}$, which makes up for the switch from ${T}_{y*}(X)$ to $\widehat{T}_{y*}(X)$.
\item[(d)] For future reference, let us elaborate here on the meaning of formula \eqref{19e}. First recall that the Hirzebruch class transformation $T_{y*}$ for a variety $X$ is defined by the  composition
$$T_{y*}:=\td_* \circ mC_y,$$
with $mC_y:K_0(var/X) \to K_0(X)[y]$ the \index{motivic Chern class transformation} {\it motivic Chern class transformation} from \cite{BSY}. Setting $mC_y(X):=mC_y([id_X])$, it was shown in \cite{BSY} that $mC_y(X)$ can be computed from the filtered \index{Du Bois complex} {\it Du Bois complex} $(\underline{\Omega}^{\bullet}_X,F)$ of $X$ \cite{DB}.  More precisely, the following identification holds:
\be\label{dbm} 
mC_y(X)=\sum_{p \geq 0} \ [ \ \uuline{\Omega}^{p}_X] \cdot y^p:=
\sum_{i,p \geq 0} (-1)^i \ [\cH^i(\uuline{\Omega}^{p}_X) ] \cdot y^p \in K_0(X)[y],\ee
where $$\uuline{\Omega}^{p}_X:=\Gr^p_F(\underline{\Omega}^{\bullet}_X)[p] \in D^b_{\rm coh}(X)$$ is a bounded complex of sheaves with coherent cohomology, which coincides with the sheaf of $p$-forms $\Omega^p_X$ on a smooth variety $X$.
If $X$ is a toric variety, there is a natural quasi-isomorphism 
(see \cite{Is}[p.119 and Prop.4.2])
$$ \uuline{\Omega}^{p}_X \simeq \widehat{\Omega}^p_X,$$
with $\widehat{\Omega}^p_X$ denoting as before the sheaf of Zariski $p$-forms on $X$. In particular, for a toric variety $X$ this further yields:
\be\label{mct}
mC_y(X)=\sum_{p \geq 0} \ [ \ \widehat{\Omega}^p_X] \cdot y^p
\ee
and \eqref{19e} is obtained by applying the Todd class transformation $\td_*$ to \eqref{mct}.
\er
%%%%%%%%%%%%%%%%%%%%%%%

\subsection{Todd and Hirzebruch classes of a simplicial toric variety}\label{necl}
In the context of a simplicial toric variety $X_{\Sig}$ (e.g., $X_P$ for $P$ a simple full-dimensional lattice polytope), a formula for the Todd class $\td_*(X_{\Sig})$ was obtained in \cite{BrV2} via equivariant cohomology, in \cite{EG} by using the Lefschetz-Riemann-Roch theorem, and in \cite{P} by using a resolution of singularities and the birational invariance of Todd classes. 

A formula for the \index{Todd class} Todd class $\td_*(X_{\Sig})$ of a simplicial toric variety $X_{\Sig}$ can also be deduced from a more general result of \cite{MS1} on the computation of the (homology) \index{Hirzebruch class} Hirzebruch classes of \cite{BSY} via the \index{Lefschetz-Riemann-Roc \} Lefschetz-Riemann-Roch theorem \cite{EG}. More precisely, with the notations from Subsection \ref{Cox}, the following result holds with the projection character $a_\rho: \widetilde{\bT}=(\bC^*)^r\to \bC^*$ as in \eqref{char-arho} from the Cox quotient construction, restricted here to $G_{\Sig}\subset\widetilde{\bT}$
(see \cite{MS1}[Thm.5.4]):
\bt\label{MSt1}
Let $X=X_{\Sig}$ be a $n$-dimensional simplicial toric variety, with $r=\vert \Sigma(1) \vert$.
The un-normalized and, resp., normalized Hirzebruch classes of  $X$ are computed by:
\be\label{unHirzc}
{T}_{y*}(X)= (1+y)^{n-r} \cdot \left( \sum_{g \in G_{\Sig}}  \prod_{\rho \in \Sig(1)} \frac{[D_{\rho}] \cdot 
\big( 1+y  \cdot a_{\rho}(g)  \cdot e^{-[D_{\rho}]}\big)}{1-a_{\rho}(g) \cdot e^{-[D_{\rho}]}} \right) \cap [X]\:.
\ee
\be\label{nHirzc}
\widehat{T}_{y*}(X)= \left( \sum_{g \in G_{\Sig}}  \prod_{\rho \in \Sig(1)} \frac{[D_{\rho}] \cdot 
\big( 1+y  \cdot a_{\rho}(g)  \cdot e^{-[D_{\rho}](1+y)}\big)}{1-a_{\rho}(g) \cdot e^{-[D_{\rho}](1+y)}} \right) \cap [X]\:.
\ee
\et
In particular, for $y=0$, either of \eqref{unHirzc} or (\ref{nHirzc}) yields the following formula for the homology Todd class $\td_*(X)$ of a simplicial toric variety $X=X_{\Sig}$ (compare with \cite{BrV2, EG}):
\be\label{toddc}
\td_*(X):=\left( \sum_{g \in G_{\Sig}}  \prod_{\rho \in \Sig(1)} \frac{[D_{\rho}]}{1-a_{\rho}(g)e^{-[D_{\rho}]}} \right)
\cap [X]\:.
\ee

A different formula for the Todd class $\td_*(X)$ of a simplicial toric variety $X:=X_{\Sig}$ was obtained in \cite{MS1}[Cor.1.16], by also using the motivic properties of Hirzebruch classes in the context of the orbit stratification of a toric variety. More precisely, with the notations from Subsection \ref{Cox}, we have the following:
\bt\label{24} Let $X=X_{\Sig}$ be a simplicial toric variety. Then:
\be\label{todd-EM-intro}
\td_*(X)= \sum_{\sig\in \Sig}\: \alpha(\sig)\cdot \left(
\sum_{\{\tau | \sig \preceq \tau\}}\: \mult(\tau)\prod_{\rho \in \tau(1)} \frac{1}{2}[D_{\rho}]
\prod_{\rho \notin \tau(1)} \frac{\frac{1}{2}[D_{\rho}]}{\tanh(\frac{1}{2}[D_{\rho}])}
\right) \cap [X],
\ee
where, for a singular cone $\sig \in \Sig_{\rm sing}$,
\be\label{alpha-intro} \begin{split} \alpha(\sig)&:= \frac{1}{\mult(\sig)} \cdot \sum_{g \in G_{\sig}^{\circ}}\prod_{\rho \in \sig(1)} 
\frac{ 1+ a_{\rho}(g)  \cdot e^{-[D_{\rho}]}}{1-a_{\rho}(g) \cdot e^{-[D_{\rho}]}}\\
&= \frac{1}{\mult(\sig)} \cdot \sum_{g \in G_{\sig}^{\circ}}\prod_{\rho \in \sig(1)} 
\coth \left( \pi i \cdot \gamma_{\rho}(g)+\frac{1}{2}[D_{\rho}] \right),
\end{split}
\ee
with 
$\alpha(\{0\}):=1$ and $\alpha(\sig):=0$ for any other smooth cone $\sig\in \Sig$. 
\et
Similar formulae were obtained by Cappell and Shaneson in the early 1990s in the case of complete simplicial toric varieties, see \cite{CS1,S}.
Note that Theorem \ref{24} is deduced in \cite{MS1}[Sec. 6] from the specialization for $y=1$ of \eqref{nHirzc}
(fitting for $X$ a  projective simplicial toric variety with the Thom-Milnor L-class $L_*(X)$ of $X$),
together with the use of the so-called Mock Hirzebruch and $T$-classes for switching between Todd and L-classes.

%%%%%%%%%%%%%%%%%%%%%%%%%%

\subsection{Rational equivariant (co)homology of a toric variety}\label{eqcoh}

Let $X$ be a 
complex algebraic variety with an algebraic action of the torus $\bT$ (e.g., $X$ is a toric variety
 with torus $\bT=T_N=\Hom_{\bZ}(M,\bC^*)$). Recall that the (Borel-type) rational \index{equivariant cohomology}  {\it equivariant cohomology} of $X$ is defined as \be H^*_\bT(X;\bb{Q}):=H^*(E\bT\times_\bT X;\bb{Q}),\ee
where $\bT\hookrightarrow E\bT\to B\bT$ is the universal principal $\bT$-bundle, i.e., $B\bT=(B\bb{C}^*)^n=(\bC P^{\infty})^n$ and $E\bT$ is contractible. In particular, if $X=pt$ is a point space, 
\be (\Lambda_\bT)_{\bQ} := H^*_\bT(pt;\bQ)=H^*(B\bT;\bQ)\simeq \bQ[t_1,\ldots,t_n].\ee

For $m\in M$, with $M$ the character lattice of the torus $\bT$, one can view the corresponding character $\chi^m:\bT \to \bC^*$ as a $\bT$-equivariant line bundle $\bC_{\chi^m}$ over a point space $pt$, where the $\bT$-action on $\bC$ is induced via $\chi^m$.
This gives rise to an isomorphism
$$M\simeq Pic_\bT(pt) \:.
$$
Note that $m\mapsto -m$ corresponds to the duality involution $(-)^{\vee}$. Taking the first equivariant Chern class 
$c^1_\bT$ (or the dual $-c^1_\bT=c^1_\bT\circ (-)^{\vee}$) of $\bC_{\chi^m}$ gives an  isomorphism
\begin{equation}\label{1}
c=c^1_\bT,  \:\text{resp.}, \:  s=-c^1_\bT :\:
M \simeq H_\bT^2(pt;\bb{Z})
\end{equation}
and
\begin{equation}\label{2}
 c,  \:\text{resp.}, \: s: Sym_{\bb{Q}}(M)  \simeq  H_\bT^*(pt;\bb{Q})=(\Lambda_\bT)_{\bb{Q}}\:, 
\end{equation}
with $Sym_{\bb{Q}}(M)=\bigoplus_{k=0}^{\infty} Sym^k_{\bb{Q}}(M)$ the (rational) symmetric algebra of $M$.
So if $m_i$ ($i=1,\dots,n$) are a basis of $M\simeq \bb{Z}^n$, then  $H_\bT^*(pt;\bb{Q})=(\Lambda_\bT)_{\bb{Q}}\simeq \bb{Q}[t_1,\dots,t_n],$  with
$t_i=\pm c^1_\bT(\bC_{\chi^{m_i}})$ for $i=1,\dots,n$. 

\medskip

Note that $E\bT\times_\bT X$ is a fiber bundle over $B\bT$ with fiber $X$, so $H^*_\bT(X;\bb{Q})$ is a $H^*(B\bT;\bQ)=(\Lambda_\bT)_{\bQ}$-algebra. Furthermore, restriction to fibers defines a graded ring homomorphism $$i_X^*:H^*_\bT(X;\bb{Q}) \to H^*(X;\bb{Q})$$ called ``forgetting the $\bT$-action''.
For the definition of the $\bT$-equivariant Chow groups $CH^\bT_k(X)$ or  $\bT$-equivariant Borel-Moore homology groups
$H^{\bT, BM}_k(X)$ (or its versions with rational coefficients) one uses finite dimensional  approximations of $E\bT$ given by open subsets
$U\subset V$ in a linear $\bT$-representation $V$ of high dimension so that $\bT$ acts freely on $U$, with $U\subset V$ of sufficiently high codimension (compared to the given degree $k$), and defines (see e.g. \cite{EG-1, BZ} and \cite{AF}[Section 17]):
\be
CH^\bT_k(X):=CH_{k+\dim(V)-\dim(\bT)}(U\times_\bT X) \:\: \:\text{and}\:\:\: H^{\bT, BM}_k(X):=H^{\bT, BM}_{k+2(\dim(V)-\dim(\bT))}(U\times_\bT X)\:.
\ee
Here $U\times_\bT X$ will be in general only a complex algebraic space. These groups are (up to isomorphism) independent of all choices,
and satisfy all the usual properties, e.g., one has a cycle class map $CH^\bT_k(X)\to  H^{\bT}_k(X):= H^{\bT, BM}_{2k}(X)$, and for
$Z\hookrightarrow X$ a closed $\bT$-invariant algebraic subset with open complement $U$ one has short exact sequences
\be
CH^\bT_k(Z)\to CH^\bT_k(X)\to CH^\bT_k(U)\to 0
\ee
and long exact sequences
\be \label{long-BM}
\cdots \to H^{\bT, BM}_k(Z) \to H^{\bT, BM}_k(X) \to H^{\bT, BM}_k(U) \to \cdots \:,
\ee
and similarly also with rational coefficients. If $\bT$ acts on $X$ with only finitely many orbits (like for toric varieties), 
then $H^{\bT, BM}_k(X; \bQ)$
vanishes in odd degrees, and the last sequences (with $\bQ$-coefficients)  become short exact sequences in even degrees 
(see, e.g., \cite{BZ, W2}). In particular the cycle class map induces for $\bT$ acting on $X$ with finitely many orbits  a canonical isomorphism 
$$CH^\bT_*(X)\otimes \bQ \simeq H^\bT_*(X;\bQ) \:.$$
Similar results and definitions also hold for a general complex linear group $G$. Specific to the torus context is a result of Brion 
\cite{Br2}[Thm.2.1], giving a presentation of $CH^\bT_k(X)$ only in terms of $\bT$-invariant subsets of $X$.
Now we come back to $X$ a toric variety, working for this reason only with even degree equivariant Borel Moore homology
$H^\bT_*(X;\bQ):=H^{\bT,BM}_{2*}(X;\bQ) $.

\medskip

If $X$ is moreover a simplicial toric variety with torus $\bT$, then $X$ is a rational homology manifold so that equivariant Poincar\'e duality holds rationally. For any cone $\sig \in \Sig$, the orbit closure  $V_\sig=\overline{O}_\sig$ then defines via Poincar\'e duality an equivariant cohomology class $$[V_\sig]_\bT \in H^{2 \dim(\sig)}_\bT(X;\bQ).$$ We reserve the notation $[D_\rho]_\bT\in H^2_\bT(X;\bQ)$ for the equivariant cohomology class corresponding to a ray $\rho \in \Sig(1)$. Then one has the relation:
\be\label{23}
[V_\sig]_\bT=\mult(\sig) \cdot \prod_{\rho \in \sig(1)} [D_\rho]_\bT.
\ee

\medskip

We assume now that $X$ is a simplicial toric variety which has a  fan $\Sig$  of full-dimensional convex  support  (see \cite{CLS}[page 265]), e.g., $X=X_P$ for $P$ a full-dimensional lattice polyhedron. By \cite{CMM}[Thm.3.6], the cohomology ring $H^*(X;\bQ)$ is even (i.e., it vanishes in odd degrees). Then the proof of \cite{CLS}[Prop.12.4.7] yields an isomorphism of $(\Lambda_\bT)_{\bQ}$-modules
\be\label{mm}
H^*_\bT(X;\bb{Q}) \simeq (\Lambda_\bT)_{\bQ} \otimes_\bQ H^*(X;\bQ).
\ee
In particular, the equivariant cohomology ring $H^*_\bT(X;\bb{Q})$ is a free $(\Lambda_\bT)_{\bQ}$-module.
Moreover, in this case the ordinary cohomology ring $H^*(X;\bQ)$ is determined by the $(\Lambda_\bT)_{\bQ}$-algebra structure of $H^*_\bT(X;\bb{Q})$ via the isomorphism
\be\label{ki} H^*(X;\bQ) \simeq H^*_\bT(X;\bb{Q})/{I_\bT H^*_\bT(X;\bb{Q})},\ee
where $I_\bT \subset (\Lambda_\bT)_{\bQ}$ is the ideal generated by positive degree elements, as in \cite{BrV2}[Prop.3.2] or \cite{CLS}[Cor.12.4.8]. (In particular, the quotient $(\Lambda_\bT)_{\bQ}/{I_\bT} \simeq \bQ$ gives $\bQ$ the structure of a $(\Lambda_\bT)_{\bQ}$-module.)
See also \cite{GKM}[Thm.14.1] for more general results of this type for complex varieties with a torus action,
e.g., \eqref{mm} holds by \cite{GKM}[Thm.14.1 (8)] for a complete complex algebraic variety $X$ which is a rational homology manifold.

\medskip

For the remaining of this Subsection \ref{eqcoh}, we assume moreover that $X$ is a {\it complete simplicial} toric variety. 

There are several different descriptions of the rational equivariant cohomology $H^*_\bT(X;\bb{Q})$ of a complete simplicial toric variety $X=X_\Sigma$ associated to the fan $\Sigma$ in $N_\bR=N\otimes \bb{R}$, e.g., see \cite{Br2}, \cite{CLS}[Sect.12.4]
or \cite{AF}[Chapter 8]. Such descriptions, which are recalled below, are related to various aspects of the Euler-Maclaurin formulae for simple lattice polytopes (see Sections \ref{emf} and \ref{emcs} below):

\smallskip

\noindent({\rm a})
$H^*_\bT(X;\bb{Q})$ can be realized as the \index{Stanley-Reisner ring} {\it Stanley-Reisner ring} $SR_{\bb{Q}}(\Sigma)$ of the fan $\Sigma$, i.e., (see, e.g., \cite{CLS}[Sect.12.4]),  
$$SR_{\bb{Q}}(\Sigma):=
\bb{Q}[x_1,\dots,x_r]/\sim_{SR}\:\: \simeq H^*_\bT(X;\bb{Q})$$ 
$$x_i\longmapsto [D_i]_\bT=c^1_\bT(\mathcal{O}(D_i)) \:,$$
with $x_i$ of degree $2$, and $D_i$ the $\bT$-invariant divisor corresponding to the ray $\rho_i= \langle u_i \rangle \in \Sigma(1)$ with ray generator $u_i$
($i=1,\dots, r$). The \index{Stanley-Reisner relation} Stanley-Reisner relation $\sim_{SR}$ 
 is generated by 
$$x_{i_1}\cdots x_{i_s}=0,$$ for distinct $i_j$ with $\{\rho_{i_j}| j=1,\dots, s\}$ not spanning  a cone of $\Sigma$.
 
\smallskip
\noindent({\rm b}) $H^*_\bT(X;\bb{Q})$ can be realized as the ring $R_{\Sigma}$ of continuous  piecewise polynomial functions (aka, polynomial splines) with rational coefficients on the fan $\Sigma$ (see, e.g., \cite{BrV2}[Prop.3.2] or \cite{CLS}[page 606]). This can be seen by starting with  the following  exact sequence in equivariant cohomology (as appearing in the proof of \cite{CLS}[Lem.12.4.17]):
\begin{equation}\label{loc}
0\to H^*_\bT(X;\bb{Q}) \to \bigoplus_{\sigma\in \Sigma(n)} H^*_\bT(U_{\sigma};\bb{Q}) \to \bigoplus_{\tau \in \Sigma(n-1)} H^*_\bT(U_{\tau};\bb{Q}) \:,
\end{equation}
where $U_{\sigma}\subset X_\Sigma$ denotes as usual the unique $\bT$-invariant open affine subset containing the corresponding  $\bT$-orbit
$O_\sigma$, and $O_\sigma=x_{\sigma}\in X$ are the \index{torus-fixed point} $\bT$-fixed points in the case $\sigma\in \Sig(n)$ (i.e., $\sigma$ is of top dimension
$n$). Note also that $x_{\sigma}$ corresponds to a vertex of the lattice polytope $P$ in the case when $\Sigma=\Sigma_P$ is the corresponding inner normal fan. There are equivariant deformation retracts $U_\sig \to O_\sig$ for any $\sig \in \Sig$ (e.g., see \cite{CLS}[Prop.12.1.9(a)]), hence 
$$H^*_\bT(U_{\sigma};\bb{Q})  \simeq  H^*_\bT(O_{\sigma};\bb{Q}).$$
Moreover, for $\sigma\in \Sigma(n)$ one gets by (\ref{2}): $$H^*_\bT(O_\sig;\bQ) \simeq H^*_\bT(x_{\sig};\bQ) \simeq Sym_{\bQ}(M),$$
and for $\tau \in \Sig(n-1)$ one has an isomorphism $$H^*_\bT(O_\tau;\bQ) \simeq Sym_\bQ(M_\tau),$$ for $M_\tau=M/{\tau^{\perp} \cap M}$. 
So the localization sequence (\ref{loc}) translates into the exact sequence
\begin{equation}\label{pp}
0\to H^*_\bT(X;\bb{Q}) \to \bigoplus_{\sigma\in \Sigma(n)} Sym_{\bb{Q}}(M) \to \bigoplus_{\tau \in \Sigma(n-1)} 
Sym_{\bb{Q}}(M_{\tau})\:,
\end{equation}
with $Sym_{\bb{Q}}(M)$, resp., $Sym_{\bb{Q}}(M_{\tau})$ the ring of polynomial functions with rational coefficients on $\sigma\in \Sigma(n)$, resp.,
 $\tau \in \Sigma(n-1)$. In other words, an element of $H^*_\bT(X;\bb{Q})$ can be thought of as a collection $\{f_\sig\}_{\sig \in \Sig(n)}$, where $f_\sig$ is a polynomial function on $\sig$, such that $f_\sig|_\tau=f_{\sig'}|_\tau$ whenever $\sig \cap \sig'=\tau \in \Sig(n-1)$.
 
 \smallskip

\noindent({\rm c}) $H^*_\bT(X;\bb{Q})$ can also be described as the $H_\bT^*(pt;\bb{Q})=(\Lambda_\bT)_{\bb{Q}}$-algebra
$$(\Lambda_\bT)_{\bb{Q}}[x_1,\dots,x_r]/\sim\:\: \simeq H^*_\bT(X;\bb{Q})$$ $$x_i\longmapsto [D_i]_\bT=c^1_\bT(\mathcal{O}(D_i)), $$
and $\sim$  generated by  the Stanley-Reisner  relation $\sim_{SR}$ as well as
\begin{equation}\label{alg}
\pm m:= c(m)= -s(m)= c^1_\bT(\bC_{\chi^m})=
\sum_{i=1}^r \; \langle m,u_i \rangle x_i  \quad \text{for all $m\in M$} \:,
\end{equation}
with $u_i$ the corresponding ray generators.
Here we consider the $+$-sign in the case when we use the isomorphism 
$$c: Sym_{\bb{Q}}(M)  \simeq  H_\bT^*(pt;\bb{Q})=:(\Lambda_\bT)_{\bb{Q}}$$ coming from $c^1_\bT$ as usual in algebraic geometry (see, e.g., \cite{AF}[Chapter 8]).
But the use of the minus-sign and the isomorphism  
$$s: Sym_{\bb{Q}}(M)  \simeq  H_\bT^*(pt;\bb{Q})=:(\Lambda_\bT)_{\bb{Q}}$$
 induced from $c^1_\bT\circ (-)^{\vee}$ (e.g., as in \cite{CLS}[Sect.12.4]) fits better with the corresponding Euler-Maclaurin formulae and the equivariant Riemann-Roch theorem (e.g., see \cite{CLS}[Sect.13.3]).

\begin{rem}
The inclusion $\bb{Q}[x_1,\dots,x_r] \subset (\Lambda_\bT)_{\bb{Q}}[x_1,\dots,x_r]$ induces the natural map
$$\bb{Q}[x_1,\dots,x_r]/\sim_{SR} \ \ \longrightarrow (\Lambda_\bT)_{\bb{Q}}[x_1,\dots,x_r]/\sim \:\:,$$
which is an isomorphism relating parts (a) and (c) above,  see Lemma \ref{SRnew} and compare also with 
\cite{AF}[Chapter 8]. Note that in (a) and (b) the $(\Lambda_\bT)_{\bb{Q}}$-algebra structure of
$H^*_\bT(X;\bb{Q})$ is not explicitly mentioned, but follows then from relation (\ref{alg})
(compare also with \cite{CLS}[Sect.12.4]). 

In the context of a complete simplicial toric variety $X=X_\Sigma$ there is also no difference between the rational equivariant (co)homology ring and the corresponding rational equivariant Chow ring of $X$ (see \cite{Br2}[Sect.5]).
Finally note that by \eqref{ki} one has that $c^1_\bT(\bC_{\chi^m})\mapsto 0$ under the natural map $H^*_\bT(X;\bb{Q}) \to H^*(X;\bb{Q})$ obtained by forgetting the $\bT$-action. 
\end{rem}

As it shall be indicated below, the equivariant Chern character $\ch^\bT(\mathcal{E})$ of a $\bT$-equivariant line or vector bundle $\mathcal{E}$ on $X$,
as well as the image of the equivariant Todd homology class $\td_*^\bT(X) \in \widehat{H}^\bT_*(X;\bQ) \simeq {\widehat{CH}_*^\bT}(X)\otimes \bQ$ under equivariant Poincar\'{e} duality,
live in the completion \index{completion}
$$\widehat{H}^*_\bT(X;\bb{Q}) := \prod_{i\geq 0} \: H^i_\bT(X;\bb{Q}),$$
with $$\widehat{H}^*_\bT(pt;\bb{Q}) \simeq  \bb{Q}[[t_1,\dots,t_n]]=:(\widehat{\Lambda}_\bT)_{\bb{Q}}\:.$$
Under the isomorphisms described in this subsection, this completion corresponds to:
\begin{enumerate}
\item[(a)] $\bb{Q}[[x_1,\dots,x_r]]/\sim_{SR}\:\: \simeq \widehat{H}^*_\bT(X;\bb{Q})$.
\item[(b)] the corresponding completion  $\widehat{R}_{\Sigma}\simeq \widehat{H}^*_\bT(X;\bb{Q})$ of the ring of continuous 
 piecewise polynomial functions with rational coefficients on the fan $\Sigma$.
\item[(c)] $(\widehat{\Lambda}_\bT)_{\bb{Q}}[[x_1,\dots,x_r]]/\sim\:\: \simeq \widehat{H}^*_\bT(X;\bb{Q})$.
\end{enumerate}
Abstractly, these are completions of connected integer graded commutative rings $R^*$ with respect to the maximal ideal $R^{>0}$ given by positive degree elements.

In fact, as we will see
 and use later on, the equivariant Chern character and Todd homology classes live in an {\it analytic subring} \index{analytic subring}
(defined as the image of the {\it analytic Stanley-Reisner ring}, see \eqref{ans}):
$$(H^*_\bT(X;\bb{Q}))^{an}\subset  \widehat{H}^*_\bT(X;\bb{Q}) \:,$$
with $ \bb{Q}\{t_1,\dots,t_n\}\simeq (H^*_\bT(pt;\bb{Q}))^{an}=:(\Lambda^{an}_\bT)_{\bb{Q}}\subset (\widehat{\Lambda}_\bT)_{\bb{Q}}$ the subring of {\it convergent power series} (around zero) with rational coefficients, i.e., after pairing with $z\in N_{\bb{K}}=N\otimes_{\bb{Z}}\bb{K}$ (for $\bb{K}=\bb{R},\bb{C}$) one gets a convergent power series {\it function} in $z$ around zero, whose corresponding Taylor polynomials have rational coefficients.

\begin{rem}\label{reminj}
The injectivity part of the sequence \eqref{loc} still holds (as in \cite{CLS}[Cor.12.4.9]) even if $X$ is only a simplicial toric variety with a fan of full-dimensional convex support, since the equivariant cohomology of $X$ is a free $(\Lambda_\bT)_{\bQ}$-module.
\end{rem}

%%%%%%%%%%%%%%%%%%%

\subsection{Equivariant Euler characteristic}\label{esHRR}

Assume that a torus $\bT$ with character lattice $M$ acts linearly on a finite dimensional complex vector space $W$, with eigenspace 
\be\label{eigen}
W_{\chi^m}:=\{w\in W \mid t\cdot w=\chi^m(t) w \ \text{ for all} \ t \in \bT\},
\ee
for $\chi^m$ a character of $\bT$. Then, as in \cite{CLS}[Prop.1.1.2], one has an {\it eigenspace} decomposition \index{eigenspace decomposition}
$$ W= \bigoplus_{m \in M} W_{\chi^m}.$$
This induces an isomorphism 
\be\label{bun}
 K^0_\bT(pt) \overset{\simeq}\lra \bb{Z}[M], \ \ [W] \mapsto \sum_{m \in M} \dim_\bC W_{\chi^m} \cdot \chi^m.
\ee

Let $X=X_{\Sig}$ be a complete toric variety of dimension $n$, with a $\bT$-equivariant coherent sheaf $\cF$. The cohomology spaces $H^i(X;\cF)$ are finite dimensional $\bT$-representations, vanishing for $i$ large enough. We define the \index{Euler characteristic, $K$-theoretic}  {\it $K$-theoretic Euler characteristic} of $\cF$ as:
\be\label{Kchi}\begin{split}
\chi^\bT(X,\cF)&:=[H^*(X;\cF)]:=\sum_i (-1)^i [H^i(X;\cF)] \in K^0_\bT(pt)\\
&=\sum_{m \in M} \sum_{i=0}^n (-1)^i \dim_\bC H^i(X;\cF)_{\chi^m} \cdot \chi^m \in \bZ[M].
\end{split}
\ee
By applying the equivariant Chern character $\ch^\bT:\bZ[M] \simeq K^0_\bT(pt) \hookrightarrow  (\Lambda^{an}_\bT)_{\bb{Q}}\subset 
 (\widehat{\Lambda}_\bT)_{\bb{Q}}$, we further get the  \index{Euler characteristic, cohomological} {\it cohomological Euler characteristic} of $\cF$:
\be\label{f41}
\chi^\bT(X,\cF)=\sum_{m \in M} \sum_{i=0}^n (-1)^i \dim_\bC H^i(X;\cF)_{\chi^m} \cdot e^{c(m)} \in (\Lambda^{an}_\bT)_{\bb{Q}}
\subset  (\widehat{\Lambda}_\bT)_{\bb{Q}}\:.
\ee

Let $D$ be a torus-invariant Cartier divisor on $X$. For each $i \geq 0$, consider the \index{weight decomposition} {\it weight} decomposition of $H^i(X;\cO_X(D))$ given by (e.g., see \cite{CLS}[Sect.9.1]):
\be\label{b30}
H^i(X;\cO_X(D))=\bigoplus_{m \in M} H^i(X;\cO_X(D))_m,
\ee
induced via a \u{C}ech resolution from 
\be\label{cek} H^0(U_\sig; \cO_X(D)) = \bigoplus_m H^0(U_\sig;\cO_X(D))_m  \subset \bC[M],\ee
where each $H^0(U_\sig;\cO_X(D))_m$ is either $0$ or $\bC \cdot \chi^m$.

For comparing this weight decomposition with the eigenspace decomposition of \eqref{eigen}, we recall here that in this paper  the torus $\bT$ acts  from the left on $\bC[M]$ (with $M\simeq \Hom(\bT,\bC^*)$) as follows: if $t \in \bT$ and $f \in \bC[M]$, then $t \cdot f \in \bC[M]$ is given by $p \mapsto f(t^{-1} \cdot p)$, for $p \in \bT$ (see \cite{CLS}[pag.18]. But some other references like \cite{BrV2} use a different convention).
 In particular, $t \cdot \chi^m = \chi^m(t^{-1}) \chi^m$, so that
$$H^i(X;\cO_X(D))_m = H^i(X;\cO_X(D))_{\chi^{-m}}.$$
So, via $s(m)=c(-m)=-c_1^\bT(\bC_{\chi^m})$,  formula \eqref{f41} translates into 
 the \index{equivariant Euler characteristic} {\it equivariant Euler characteristic} of $D$, as defined in \cite{CLS}[Def.13.3.2]:
\be\label{30}
\chi^\bT(X,\cO_X(D))=\sum_{m \in M} \sum_{i=0}^n (-1)^i \dim H^i(X;\cO_X(D))_m \cdot e^{s(m)} \in (\Lambda^{an}_\bT)_{\bb{Q}}
\subset  (\widehat{\Lambda}_\bT)_{\bb{Q}}\:.
\ee
Note that $i^*_{pt} \chi^\bT(X,\cO_X(D))=\chi(X,\cO_X(D))$, where $i^*_{pt}: (\widehat{\Lambda}_\bT)_{\bb{Q}} \to \bQ$ sends elements of positive degree to zero (i.e., one forgets the $\bT$-action). 

\begin{example}\label{lp}
In the special case when $X_P$ is the projective toric variety associated to a full-dimensional lattice polytope $P \subset M_{\bR} \simeq \bR^n$, and $D_P$ is the corresponding ample Cartier divisor, then the formulae (\ref{am1})--(\ref{am4}) yield
\be\label{am5}
\chi^\bT(X_P,\cO_{X_P}(D_P))=\sum_{m \in P \cap M} e^{s(m)}
\ee
and
\be\label{am6}
\chi^\bT(X_P,\cO_{X_P}(-D_P))=(-1)^n\sum_{m \in \Int(P) \cap M} e^{c(m)}.
\ee
\qed
\end{example}

\begin{example}\label{lpp} For later applications, let us also consider the following more general context.
Let $X$ be a toric variety and $D$ a torus-invariant Cartier divisor. Then 
\be\label{am11}
\Gamma(X;\cO_{X}(D)) = \bigoplus_{m \in P_D \cap M} \bC \cdot \chi^m \subset \bC[M],
\ee
where $\chi^m$ denotes the character defined by $m \in M$, and $P_D$ is the polyhedron associated to the divisor $D$ (see \cite{CLS}[Prop.4.3.3]). If, moreover, the fan of $X$ has full-dimensional convex support and $\cO_{X}(D)$ is globally generated, then $P_D$ is a {\it lattice} polyhedron (see \cite{CLS}[Thm.6.1.7]), and 
\be\label{am2m}
H^i(X;\cO_{X}(D))=0, \ {\rm for \ all } \ i>0
\ee
(e.g., see  \cite{CLS}[Prop.9.2.3]). Hence, for $X$ complete and $D$ as above (with $P_D$ the corresponding lattice polytope), we get
\be\label{am51}
\chi^\bT(X,\cO_{X}(D))=\sum_{m \in P_D \cap M} e^{s(m)} 
\in (\Lambda^{an}_\bT)_{\bb{Q}} \subset  (\widehat{\Lambda}_\bT)_{\bb{Q}}.
\ee
\qed
\end{example}

%%%%%%%%%%%%%%%%

\subsection{Equivariant Chern character and equivariant Riemann-Roch map.}\label{seRR} 
Let $X$ be a complex algebraic variety with an algebraic action of the torus $\bT$ with character lattice $M$ (e.g., $X$ is a toric variety). For any $\bT$-equivariant vector bundle $\cE\to X$, one can define \index{equivariant Chern classes} equivariant Chern classes $c_i^\bT(\cE)$, \index{equivariant Todd classes} equivariant Todd classes ${\rm Td}^\bT(\cE)$ and an \index{equivariant Chern character} equivariant Chern character $\ch^\bT(\cE)$ in $\widehat{H}_\bT^*(X;\bQ),$ by means of the corresponding Chern roots, see \cite{EG0}[Sect.3.1] for complete details. Forgetting the $\bT$-action, these notions reduce under $i_X^*:\widehat{H}_\bT^*(X;\bQ) \to {H}^*(X;\bQ)$ to the classical non-equivariant counterparts.  
For example, if $D$ is a $\bT$-invariant Cartier divisor on $X$, then the {equivariant Chern character} of the line bundle $\cO_X(D)$ is given by:
\be
\ch^\bT(\cO_X(D))=e^{[D]_\bT}=1+[D]_\bT+\frac{1}{2!}[D]_\bT^2+\cdots \in \widehat{H}_\bT^*(X;\bQ),
\ee
where $[D]_\bT=c_1^\bT(\cO_X(D)) \in H^2_\bT(X;\bQ)$ is the equivariant cohomology class of the torus-invariant Cartier divisor $D$. Forgetting the $\bT$-action yields: $i_X^*\ch^\bT(\cO_X(D))=\ch(\cO_X(D)).$

 We denote as usual by $K_\bT^0(X)$ the Grothendieck group of $\bT$-equivariant vector bundles on $X$, and we let  $K_0^\bT(X)$ denote the Grothendieck group of $\bT$-equivariant coherent sheaves on $X$. If $X$ is smooth, then $K_\bT^0(X) \simeq K^\bT_0(X)$, since in this case any $\bT$-equivariant coherent sheaf has a finite resolution by $\bT$-equivariant locally free sheaves  (see, e.g., \cite{Ed}[page 28] and the references therein).
In general, the tensor product gives a ring structure on $K^0_\bT(X)$, and $K^\bT_0(X)$ is a module for this ring. 
Note that the Chern character $\ch^\bT$ induces a contravariant, functorial, ring homomorphism
$$\ch^\bT: K_\bT^0(X) \lra \widehat{H}_\bT^{2*}(X;\bQ).$$

Edidin-Graham defined in \cite{EG0}[Thm.3.1]  an \index{equivariant Riemann-Roch map} {\it equivariant Riemann-Roch map} which in our setup can be given as 
$$\td^\bT_*:K^\bT_0(X) \lra {\widehat{CH}_*^\bT}(X)\otimes \bQ \lra  \widehat{H}^\bT_{2*}(X;\bQ),$$
with the same functoriality as in the nonequivariant case of Baum-Fulton-MacPherson \cite{BFM}. 
Here, we use the completions
$${\widehat{CH}_*^\bT}(X)\otimes \bQ :=\prod_{i \leq \dim(X)} {CH}_i^\bT(X) \otimes \bQ, \ \ \text{and} \ \ \widehat{H}_{*}^\bT(X; \bQ) :=\prod_{i \leq \dim(X)} \widehat{H}_{i}^\bT(X;\bQ).$$ 
(These equivariant Chow and Borel-Moore homology groups can be non-zero also in negative degrees, but they vanish in degrees bigger than $\dim(X)$).
Compare also with \cite{BZ} for the homological version of this transformation.

The transformation $\td^\bT_*$ has the following properties:
\begin{enumerate}
\item[(a)] (functoriality) \ $\td^\bT_*$ is covariant for $\bT$-equivariant proper morphisms;
\item[(b)] (module property) \ if $\beta \in K_\bT^0(X)$ and $\alpha \in K_0^\bT(X)$, \index{module property} then:
\be\label{modpe}
\td^\bT_* (\beta \otimes \alpha)=\ch^\bT(\beta) \cap \td_*^\bT(\alpha).
\ee 
\end{enumerate}
Forgetting the $\bT$-action, one recovers under the forgetful map $\widehat{CH}_*^\bT(X) \to CH_*(X)$ the classical nonequivariant Todd class transformation of Baum-Fulton-MacPherson \cite{BFM}. 

The {\it equivariant Todd class} of $X$ is defined as: $$\td^\bT_*(X):=\td^\bT_*([\cO_X]_\bT]).$$
If $X$ is smooth, one also has the {\it normalization property} (see \cite{AMSS2}[Sect.6.1] or \cite{Ed}[Remark 3.2]):
\be \td^\bT_*(X) = {\rm Td}^\bT(T_X) \cap [X]_\bT,\ee
with ${\rm Td}^\bT(T_X)$ the equivariant cohomological Todd class of the tangent bundle $T_X$ of $X$.

\bex
Assume $X=pt$ is a point space. By the normalization property, $\td^\bT_*(pt) =  [pt].$ So by the module property, one gets via the equivariant Poincar\'e duality that $\td^\bT_* : K^\bT_0(pt) \to \widehat{H}_*^\bT(pt;\bQ)$ reduces to the equivariant Chern character, resp., Euler characteristic  map 
$$\ch^\bT=\chi^\bT(pt,-):\bZ[M] \simeq K^0_\bT(pt) \hookrightarrow   (\Lambda^{an}_\bT)_{\bb{Q}} \subset (\widehat{\Lambda}_\bT)_{\bb{Q}}.
$$ \qed
\eex

For a complete variety $X$, using the functoriality and module properties of $\td^\bT_*$ (cf. \cite{EG0}[Thm.3.1]) for the constant map $f:X \to pt$, one gets the following equivariant Hirzebruch-Riemann-Roch formula:
\bt[Equivariant Hirzebruch-Riemann-Roch]\label{tehrr} \ \index{equivariant Hirzebruch-Riemann-Roch theorem}
Let $X$ be a complete complex algebraic variety with an algebraic action of the torus $\bT$, and let $\cE$, resp., $\cF$, be a $\bT$-equivariant vector bundle, resp., coherent sheaf on $X$. Then
\be\label{eHRR}
\chi^\bT(X,\cE \otimes \cF)=\int_X \ch^\bT(\cE) \cap \td^\bT_*([\cF]),
\ee
where $\int_X:\widehat{H}_*^\bT(X;\bb{Q}) \to \widehat{H}_*^\bT(pt;\bb{Q})=(\widehat\Lambda_\bT)_\bQ$ is the equivariant pushforward for the constant map $X \to pt$.
\et

Example \ref{lp} and Theorem \ref{tehrr} yield the following:
\bc\label{lpc}
Let $P \subset M_{\bR} \simeq \bR^n$ be a full-dimensional lattice polytope, with corresponding projective toric variety $X=X_P$ and ample Cartier divisor $D=D_P$. Then:
\be\label{lpc1}
\sum_{m \in P \cap M} e^{s(m)}=\int_X \ch^\bT(\cO_X(D)) \cap \td^\bT_*(X), 
\ee
\be\label{lpc2}
\sum_{m \in \Int(P) \cap M} e^{s(m)}=\int_X \ch^\bT(\cO_X(D)) \cap \td^\bT_*([\omega_X]_\bT).
\ee
\ec
\begin{proof}
Formula \eqref{lpc1} follows directly from \eqref{am5} and \eqref{eHRR}. For \eqref{lpc2} we proceed as in \eqref{am4d} and \eqref{4b}, using the fact that toric Serre duality holds equivariantly.
\end{proof}

Similarly, Example \ref{lpp} and Theorem \ref{tehrr} yield:
\bc\label{lpc1c}
Let $X$ be a complete toric variety, and $D$ a globally generated torus-invariant Cartier divisor with associated lattice polytope $P_D \subset M_{\bR}$. Then 
\be\label{lpc11}
\sum_{m \in P_D \cap M} e^{s(m)}=\int_X \ch^\bT(\cO_X(D)) \cap \td^\bT_*(X).
\ee
\ec

%%%%%%%%%%%%%%%%%%%%%%%%

%%%%%%%%%%%%%%%%%%%%%%%%

\section{Equivariant characteristic classes of toric varieties.}\label{sec:3}

In this section, we first extend in Subection \ref{sec3.1} the characteristic class formulae from \cite{MS2, MS1} for the motivic Chern and Hirzebruch classes of a toric variety to the equivariant setting. Furthermore, we calculate the pushforward of these classes under toric fibrations, using the motivic nature of these characteristic classes. This will be used in Subsection \ref{gehrr} for applying the generalized equivariant Hirzerbruch-Riemann-Roch formula  in the context of a globally generated $\bT$-invariant Cartier divisor.
In Subsection \ref{calch}, 
 we use the global Cox construction (as reviewed in Subsection \ref{Cox}) and the equivariant Lefschetz-Riemann-Roch theorem of Edidin-Graham \cite{EG} for computing various equivariant characteristic classes in a suitable context of simplicial toric varieties, including the classical projective case, as well as the case for the complement of a $\bT$-invariant divisor.  
Alternative proofs of all these characteristic class formulae will be given in the next section via localization at the torus fixed points.

\subsection{Definition. Properties}\label{sec3.1} In this subsection we recall the definitions of equivariant motivic Chern and Hirzebruch classes of complex  varieties with an algebraic $\bT$-action.

\subsubsection{Equivariant motivic Chern class transformation.}
For $X$ a complex  algebraic variety with an algebraic $\bT$-action,
the (relative) \index{equivariant motivic Grothendieck group} equivariant motivic Grothendieck group $K_0^\bT(var/X)$ of varieties over $X$ is the free abelian group generated by symbols $[Z \to X]$, where $Z$ is a  complex algebraic $\bT$-variety and $Z \to X$ is a $\bT$-equivariant morphism, modulo the additivity relation:
$$[Z \to X]=[U \to X]+[Z \setminus U \to X],$$
for $U \subset Z$ an open $\bT$-invariant subvariety. If $X=pt$ is a point space, then $K_0^\bT(var/pt)$ is a ring with product given by the external product of morphisms, and the group $K_0^\bT(var/X)$ is a module over $K_0^\bT(var/pt)$  with respect to the external product via the identification $pt \times X \simeq X$. 

For any equivariant morphism $f:X \to Y$ of  complex algebraic $\bT$-varieties, there is a well-defined push-forward $f_!:K_0^\bT(var/X) \to K_0^\bT(var/Y)$ defined by composition. One can also define an exterior product: $K_0^\bT(var/X) \boxtimes K_0^\bT(var/X') \to K_0^\bT(var/X \times X')$ via the cross-product: $[Z \to X] \times [Z' \to X']= [Z \times Z' \to X \times X']$.

The following result provides an equivariant analogue of the motivic Chern class transformation of \cite{BSY}; see \cite{AMSS}[Thm.4.2]
for the quasi-projective context and \cite{FRW}[Sec.2.3 and 2.4] for the general context. (Note that these results use
(implicitly) the {\it equivariant weak factorization theorem} as in \cite{Ber}[Thm.3.4].)
\bt\label{emc}
Let $X$ be a complex  algebraic variety with an action of the torus $\bT$. There exists a unique natural transformation
$$mC_y^\bT:K_0^\bT(var/X) \lra K_0^\bT(X)[y]$$
satisfying the following properties:
\begin{enumerate}
\item[(a)] (functoriality) \ $mC_y^\bT$ is covariant for $\bT$-equivariant proper morphisms;
\item[(b)] (normalization) \ if $X$ is smooth, then 
\be\label{e52}
mC_y^\bT([id_X])=\Lambda_y(T^*_X):=\sum_{p=0}^{\dim(X)} [\Lambda^p T^*_X]_\bT \cdot y^p \in K^0_\bT(X)[y] \simeq K_0^\bT(X)[y].
\ee 
\end{enumerate}
Moreover,  $mC_y^\bT$ commutes with exterior products. 
\et

As a consequence of these properties, 
the transformation $mC_y^\bT$ is determined by its image on classes $[f:Z \to X]=f_*[id_Z]$, where $Z$ is a {\it smooth} 
 variety and $f:Z \to X$ is a $\bT$-equivariant {\it proper} morphism. For such $[f]$, one has: $$mC_y^\bT([f:Z \to X])=f_* mC_y^\bT([id_Z])=f_*(\Lambda_y(T^*_Z) \otimes [\mathcal{O}_Z]).$$ So the uniqueness of $mC_y^\bT$ follows from equivariant resolution of singularities.

\bd \index{equivariant motivic Chern class}
The {\it equivariant motivic Chern class} of a complex  algebraic variety $X$ with an algebraic $\bT$-action is defined as:
$$mC^\bT_y(X):=mC^\bT_y([id_X]).$$
\ed

We next discuss a calculation of motivic Chern classes in terms of cubical hyper-resolutions, which will be used for proving a generalized equivariant Hirzebruch-Riemann-Roch theorem in the toric context.
Choose an equivariant simplicial resolution $f_{\centerdot}:X_{\centerdot} \to X$ of $X$ derived from a cubical hyper-resolution in the sense of \cite{GNPP}[Th\'eor\`eme 2.15], with each $X_i$ a smooth $\bT$-variety and $f_i:X_i\to X$ a proper $\bT$-equivariant morphism (with $\dim(X_i) \leq \dim(X) - i$). Then 
\be\label{add} [id_X]=\sum_i (-1)^i [f_i:X_i\to X] \in K^\bT_0(var/X).\ee 
Indeed, \eqref{add} follows from the inductive construction of an equivariant cubical hyper-resolution, by the following abstract equivariant blowup relation. Let $f:\widetilde{X} \to X$ be a proper $\bT$-equivariant map of complex quasi-projective varieties. Let $D\subset X$ be a $\bT$-invariant subvariety, with $E:=f^{-1}(D)\subset \widetilde{X}$. Assume $f : \widetilde{X} \setminus E \to X \setminus D$ is an isomorphism. Then 
\be\label{bl}
[\widetilde{X} \to X]-[E \to X]=[id_X]-[D\to X] \in K^\bT_0(var/X).
\ee
Applying the transformation $mC_y^\bT$  to \eqref{add} yields:
\be\label{sr}
\begin{split}
mC^\bT_y(X)&=\sum_i (-1)^i mC^\bT_y([f_i:X_i\to X])\\
&=\sum_i (-1)^i {f_i}_*(\Lambda_y(T^*_{X_i}) \otimes [\mathcal{O}_{X_i}])\\
&=\sum_{p \geq 0} \left(   \sum_i (-1)^i {f_i}_*([\Omega^p_{X_i}]_\bT)   \right) \cdot y^p.
\end{split}
\ee
In particular, for any $p \geq 0$, the Grothendieck class 
$$ \sum_i (-1)^i {f_i}_*([\Omega^p_{X_i}]_\bT)=:[Rf_{\centerdot *}\Omega^p_{X_{\centerdot}}]_\bT \in K_0^\bT(X)$$
is independent of the choice of an equivariant cubical hyper-resolution, and it provides an equivariant analogue of the Grothendieck class 
$[\ \uuline{\Omega}^{p}_X] \in K_0(X)$ appearing in \eqref{dbm}, corresponding to the graded pieces of the filtered Du Bois complex of $X$. So, while an equivariant version of the Du Bois complex is not available in the literature, the Grothendieck classes of its graded pieces are well defined equivariantly by the equivariant blow-up relation. This is sufficient for the purpose of this paper (but see also \cite{MS4} for an alternative approach based on equivariant mixed Hodge modules).

As an application of the above formula \eqref{sr}, Weber \cite{W}[Thm.5.1, Rem.5.2] deduces the following result:
\bt 
 Let $X$ be a smooth complex algebraic  $\bT$-variety, and $U \subset $X an open subvariety so that  $X \setminus U=:D=\bigcup_{i=1}^m D_i$ is a $\bT$-invariant simple normal crossing divisor. Then: 
 \be\label{web}
mC^\bT_y([U \hookrightarrow X])=[\cO_X(-D) \otimes \Lambda_y \Omega^1_X(\log D)]_{\bT} \in K^0_\bT(X)[y] \simeq K_0^\bT(X)[y] \:.
\ee
\et

In fact Weber \cite{W}[Thm.5.1, Rem.5.2] formulated his results in the non-equivariant context, but his proof also works  $\bT$-equivariantly  as follows. By the inclusion-exclusion formula for the $\bT$-equivariant motivic Chern class one has
\be\label{Web5} mC^\bT_y([U \hookrightarrow X])=\sum_{I\subset \{1,\dots,m\}} (-1)^{|I|}mC^\bT_y([D_I \hookrightarrow X])
= \sum_{I\subset \{1,\dots,m\}} (-1)^{|I|} [\Lambda_y  \Omega^1_{D_I}]_{\bT} \:,\ee
with $D_I:=\bigcap_{i\in I} D_i$ smooth $\bT$-invariant closed subvarieties (and $D_{\emptyset}:=X$). Similarly
$$[\cO_X(-D)]_{\bT} = \sum_{I\subset \{1,\dots,m\}} (-1)^{|I|}[\cO_{D_I}]_{\bT} \in K^0_\bT(X) \simeq K_0^\bT(X)\:.$$
Then  Deligne's weight filtration of the sheaf of  logarithmic $p$-forms $$\Omega^p_X(\log D)=\Lambda^p \Omega^1_X(\log D)$$ is also $\bT$-equivariant, and similarly for the $\bT$-invariant normal crossing divisor $$D_{>I}:=\bigcup_{I\subset J} D_J \subset D_I,$$ so that the calculation of  Weber \cite{W}[Thm.5.1, Rem.5.2] holds $\bT$-equivariantly:
\be\label{web2b}
 \sum_{I\subset \{1,\dots,m\}} (-1)^{|I|} [\Lambda_y  \Omega^1_{D_I}]_{\bT} = [\cO_X(-D) \otimes \Lambda_y \Omega^1_X(\log D)]_{\bT} \in
 K^0_\bT(X)[y] \simeq K_0^\bT(X)[y] \:.
\ee

\subsubsection{Equivariant Hirzebruch class transformation.}
Let $X$ be a complex algebraic variety $X$ with an algebraic $\bT$-action.  An equivariant version of the Hirzebruch class transformation of \cite{BSY} can be defined as follows (see also \cite{AMSS2,W}). 
\bd 
The un-normalized {\it equivariant Hirzebruch class} transformation \index{equivariant Hirzebruch class} 
${T}^\bT_{y*}$ is the composition:
$${T}^\bT_{y*}:=\td_*^\bT \circ mC^\bT_y:K_0^\bT(var/X) \lra {\widehat{CH}_*^\bT}(X)\otimes \bQ[y] \lra \widehat{H}^\bT_{2*}(X) \otimes \bQ[y],$$
while a normalized version $\widehat{T}^\bT_{y*}$ is obtained by precomposing ${T}^\bT_{y*}$  with the normalization functor $$\Psi_{(1+y)}:{\widehat{CH}_*^\bT}(X)\otimes \bQ[y] \to {\widehat{CH}_*^\bT}(X; \bQ[y, (1+y)^{-1}]):=
\prod_{i \leq \dim(X)} {CH}_i^\bT(X) \otimes \bQ[y, (1+y)^{-1}]$$ 
given in degree $k$ by multiplication by $(1+y)^{-k}$. 
\ed
By construction, both transformations ${T}^\bT_{y*}$ and $\widehat{T}^\bT_{y*}$ are covariant for $\bT$-equivariant proper morphisms and they commute with exterior products.
The corresponding {\it equivariant Hirzebruch classes of $X$} are defined as:
$${T}^\bT_{y*}(X):={T}^\bT_{y*}([id_X]) \ , \ \ \widehat{T}^\bT_{y*}(X):=\widehat{T}^\bT_{y*}([id_X]) .$$
If $X$ is smooth, then  by \eqref{e52} we get
\be\label{38b}
{T}^\bT_{y*}(X)=\sum_{p=0}^{\dim(X)} \td_*^\bT([\Omega^p_X]_\bT) \cdot y^p
\ee

The following result is an equivariant version of \cite{MS1}[(1.1) and (2.13)], generalizing \eqref{38b} to the singular toric context:
\bp \label{p32}
Let $X$ be a  toric variety with torus $\bT$. Then 
\be\label{eq37}
mC^\bT_y(X)=\sum_{p=0}^{\dim(X)} [\widehat{\Omega}^p_X]_\bT \cdot y^p \in K_0^\bT(X)[y],
\ee
where $\widehat{\Omega}^p_X$ 
is the sheaf of Zariski $p$-forms on $X$. Moreover,
\be\label{38}
{T}^\bT_{y*}(X)=\sum_{p=0}^{\dim(X)} \td_*^\bT([\widehat{\Omega}^p_X]_\bT) \cdot y^p.
\ee
In particular, the top degree in $y$ of $mC^\bT_y(X)$, resp., ${T}^\bT_{y*}(X)$, is given by the equivariant class $[{\omega}_X]_\bT$ of the dualizing sheaf, resp., $\td_*^\bT([{\omega}_X]_\bT)$. Similarly, for $y=0$, we get $mC^\bT_0(X)=[\cO_X]_\bT$ and ${T}^\bT_{0*}(X)=\td_*^\bT(X)$.
\ep

\begin{proof}
First note that \eqref{38} is obtained by applying the transformation $ \td_*^\bT$ to \eqref{eq37}. 

Choose an equivariant cubic hyper-resolution $f_{\centerdot}:X_{\centerdot} \to X$ of $X$ (obtained by adapting to the equivariant context the results from \cite{GNPP}[Expos\'e V, Sect. 4]). 
The natural map $\widehat{\Omega}^p_X {\lra} Rf_{\centerdot *}\Omega^p_{X_{\centerdot}}$ of $\bT$-equivariant sheaf complexes is, by \cite{A}[Thm.4.3], a quasi-isomorphism after forgetting the $\bT$-action. In particular, $Rf_{\centerdot *}\Omega^p_{X_{\centerdot}}$ is concentrated in degree zero, and also given by a reflexive sheaf (since this is the case for $\widehat{\Omega}^p_X$, see \cite{CLS}[Prop.8.0.1, eqn.(8.0.5)]). 
This map 
then has to be a $\bT$-equivariant isomorphism since the restrictions of these two  
sheaves to the  $\bT$-invariant smooth locus $X_{sm}$ of $X$  are canonically isomorphic to $\Omega^p_{X_{sm}}$ as $\bT$-equivariant sheaves.
\end{proof}

\br\label{subsp}
The statement of Proposition \ref{p32} holds more generally, for 
$X$ a $\bT$-invariant closed subvariety (i.e., a closed union of torus orbits, corresponding to a star-closed subset $\Sig'\subset \Sig$) in a  toric variety with torus $\bT$ and fan $\Sig$.  
Instead of Zariski differential forms, one has to use the corresponding sheaves of Ishida differentials $\wti{\Omega}^p_X$ from \cite{Is}, with canonical sheaf $\omega_X:=\wti{\Omega}^
{\dim(X)}_X$. To endow each $\wti{\Omega}^p_X$ with a canonical $\bT$-action, one proceeds as indicated in \cite{A}[Introduction], by realizing them as subsheaves of $f_*\Omega^p_M$ with $f:M \to X$ any $\bT$-equivariant resolution of singularities. Note that these pushforward sheaves are independent of the choice of resolution by the birational invariance of $\Omega^p_M$. 
In the notations used in the proof of Proposition \ref{p32}, the natural map $\wti{\Omega}^p_X {\lra} Rf_{\centerdot *}\Omega^p_{X_{\centerdot}}$ of sheaf complexes is a quasi-isomorphism. If the $\bT$-equivariant  cubic hyper-resolution starts with a resolution of singularities $\epsilon_0:X_0 \to X$, then this morphism is even $\bT$-equivariant. This implies that
\be\label{eq37is}
mC^\bT_y(X)=\sum_{p=0}^{\dim(X)} [\wti{\Omega}^p_X]_\bT \cdot y^p \in K_0^\bT(X)[y].
\ee
Note also that $\wti{\Omega}^0_X\simeq \cO_X$, as $\bT$-equivariant sheaves. 
If $X$ is already a toric variety, these Ishida sheaves coincide with the Zariski sheaves, with the canonical $\bT$-action, as used above.
\er

Let $X_\Sig$ be a toric variety defined by a fan $\Sigma$, with $X:=X_{\Sig'} \subset X_\Sig$ a $\bT$-invariant closed algebraic subset of $X_\Sig$ defined by a star-closed subset $\Sig' \subset \Sig$. Let $O_\sig$ be the orbit of the cone $\sig \in \Sig'$, with $i_\sig:V_\sig \hookrightarrow X$ the orbit closure inclusion. Then, by additivity, the following equivariant version of  \cite{MS1}[Prop.3.1]  holds:

\bp\label{pad}
\be
mC^\bT_y(X)=\sum_{\sig \in \Sig'} mC^\bT_y([O_{\sig} \hookrightarrow X])=\sum_{\sig \in \Sig'} (i_\sig)_* mC^\bT_y([O_{\sig} \hookrightarrow V_{\sig}]).
\ee
A similar formula holds for the equivariant Hirzebruch classes ${T}^\bT_{y*}(X)$ and $\widehat{T}^\bT_{y*}(X)$.
\ep
Moreover, toric geometry can be used as in \cite{MS1}[Prop.3.2] to get the following result (see also \cite{W}[Cor.11.2]):
\bp\label{orbit} For any cone $\sig \in \Sig$, one has:
\be\label{11110} mC^\bT_y([O_{\sig} \hookrightarrow V_{\sig}])=(1+y)^{\dim(O_\sig)} \cdot [\omega_{V_\sig}]_\bT,
\ee
with $\omega_{V_\sig}$ the canonical sheaf of $V_\sig$, viewed as a $\bT=T_N$-equivariant sheaf via the quotient map $T_N\to T_{N(\sig)}$ given by \eqref{quot}. Therefore, 
\be\label{11111}
{T}^\bT_{y*}([O_{\sig} \hookrightarrow V_{\sig}])=(1+y)^{\dim(O_\sig)} \cdot \td_*^\bT([\omega_{V_\sig}]_\bT).
\ee
In particular, if $V_\sig$ is complete, then
\be\label{chiyt}
\chi_y^\bT([O_{\sig} \hookrightarrow V_{\sig}]):={T}^\bT_{y*}([O_{\sig} \hookrightarrow V_{\sig} \to pt]) = (-1-y)^{\dim(O_\sig)} \in \bQ[y] \subset (\widehat{\Lambda}_\bT)_\bQ [y].
\ee
\ep

\begin{proof}
Let $\wti{V}_{\sig} \overset{f_{\sig}}{\to} V_{\sig}$ be a toric resolution of singularities of $V_{\sig}$. Then $\wti{V}_{\sig}$ is a toric variety obtained by refining the fan of $V_{\sig}$. Let $O_{\sig} \hookrightarrow \wti{V}_{\sig}$ be the natural open inclusion, with complement the simple normal crossing divisor $D_{\sig}$ whose irreducible components correspond to the rays in the fan of $\wti{V}_{\sig}$. Note that $D_{\sig}$ is a $T_{N(\sig)}$-invariant (hence a $\bT$-invariant) divisor. Then formula (\ref{web}) applied to the open inclusion  $O_{\sig} \hookrightarrow \wti{V}_{\sig}$ yields by the functoriality of $mC^\bT_y$ with respect to the proper morphism $f_{\sig}$ that:
\be\label{6} 
mC^\bT_y\big([O_{\sig}\hookrightarrow V_{\sig}]\big)=(f_{\sig})_*\big([\cO_{\wti{V}_{\sig}}(-D_{\sig})
\otimes \Lambda_y \Omega_{\wti{V}_{\sig}}^1(\log D_{\sig})]_\bT\big).\ee
However, since $\wti{V}_{\sig}$ is a smooth toric variety, the locally free sheaf  $\Omega_{\wti{V}_{\sig}}^1(\log D_{\sig})$ is in fact $T_{N(\sig)}$-equivariantly, hence also $\bT$-equivariantly, a trivial sheaf of rank equal to $\dim(\wti{V}_{\sig})=\dim(O_{\sig})$, e.g., see \cite{CLS}[(8.1.5)]. 
So we get that:
\be
[\cO_{\wti{V}_{\sig}}(-D_{\sig})
\otimes \Lambda_y \Omega_{\wti{V}_{\sig}}^1(\log D_{\sig})]_\bT=(1+y)^{\dim(O_{\sig})} \cdot  [\cO_{\wti{V}_{\sig}}(-D_{\sig})]_\bT  \in K^{\bT}_0(\wti{V}_{\sig})[y].
 \ee

Furthermore, note that the canonical dualizing sheaf $\omega_{\wti{V}_{\sig}}$ on the toric variety $\wti{V}_{\sig}$ is precisely given by $\cO_{\wti{V}_{\sig}}(-D_{\sig})$, e.g., see \cite{CLS}[Thm.8.2.3]. And since the toric morphism $f_{\sig}:\wti{V}_{\sig} \to V_{\sig}$ is induced by a refinement of the fan of $V_{\sig}$, it follows from \cite{CLS}[Thm.8.2.15] that there is a $T_{N(\sig)}$-equivariant, hence also $\bT$-equivariant,  isomorphism:
\be\label{canonical}
{f_{\sig}}_*\omega_{\wti{V}_{\sig}} \simeq \omega_{V_{\sig}}.
\ee
Altogether, we obtain the following sequence of equalities:
\be
\begin{split}
 mC^\bT_y\big([O_{\sig}\hookrightarrow V_{\sig}]\big) &=(1+y)^{\dim(O_{\sig})} \cdot (f_{\sig})_*\big( [\cO_{\wti{V}_{\sig}}(-D_{\sig})]_\bT \big) \\
 &=(1+y)^{\dim(O_{\sig})} \cdot (f_{\sig})_* \big( [\omega_{\wti{V}_{\sig}}]_\bT \big) \\
   &=(1+y)^{\dim(O_{\sig})} \cdot  [\omega_{V_{\sig}}]_\bT ,
 \end{split}
\ee
where the last equality is a consequence of (\ref{canonical}) and the vanishing of the higher derived image sheaves of $\omega_{\wti{V}_{\sig}}$, i.e., 
$R^i(f_{\sig})_*\omega_{\wti{V}_{\sig}}=0$, for all $i >0$, 
see \cite{CLS}[Thm.9.3.12].

By applying the Todd class transformation $\td^\bT_*$ to equation (\ref{11110}), we obtain formula (\ref{11111}).

Finally, formula \eqref{chiyt} follows from
$$
\chi^\bT(V_\sig, \omega_{V_\sig})=(-1)^{\dim(O_\sig)}, 
$$
which is obtained by equivariant Serre duality and  $\chi^\bT(V_\sig, \cO_{V_\sig})=1$. The latter equality is just a special case of \eqref{am51} for $D$ the zero divisor, so that $\cO(D)=\cO$.
\end{proof}

Using the functoriality of the equivariant motivic Chern and Hirzebruch classes, we get a relative version of Proposition \ref{pad} as follows. Let $f:X \to X'$ be a proper toric morphism of  toric varieties, with the corresponding lattice homomorphism $f_N:N \to N'$ surjective (i.e., $f$ is a {\it toric fibration} \index{toric fibration}  in the sense of \cite{CMM}[Prop.2.1]),
 so that in particular $ker(\bT\to \bT')$ is connected, Let $f_{\bT}:\bT \to \bT'$ be the corresponding map of the associated tori, so that $\bT$ acts on $X'$ via $f$.
Let $\Sig$, $\Sig'$ be the fans of $X$, resp., $X'$. Since $f$ is a toric fibration, a $\bT$-orbit $O_\sig$ ($\sig \in \Sig$) is mapped by $f$ to a $\bT'$-orbit $f(O_\sig)=O_{\sig'}$ ($\sig' \in \Sig'$), such that the restriction map $f_\sig=f\vert_{O_\sig}:O_\sig \to O_{\sig'}$ is isomorphic to a projection $O_\sig\simeq O_{\sig'} \times O_{\sig/\sig'} \to O_{\sig'}$, with $O_{\sig/\sig'}\simeq (\bC^*)^\ell$ and $\ell=\dim(O_\sig) - \dim(O_{\sig'})$ the relative dimension of $f_\sig$ (see \cite{CMM}[Lem.2.6 and Prop.2.7]). Let $U\subset X$ be a locally closed $\bT$-invariant subset (i.e., a locally closed union of $\bT$-orbits of $X$), with 
\be\label{dl} d_\ell(U/\sig'):=\vert \Sigma_\ell(U/\sig') \vert \ee and 
\be\label{dl2} \Sigma_\ell(U/\sig'):=\{\sig \in \Sig \mid O_\sig \subset U, \ f(O_\sig)=O_{\sig'}, \ \ell=\dim(O_\sig) - \dim(O_{\sig'})\}.\ee

\bp\label{relmc} Under the above notations and assumptions, we have
\be\begin{split}\label{fn90}
f_* mC^\bT_y([U \hookrightarrow X])&=\sum_{\sig' \in \Sig'}  \sum_{\ell \geq 0} d_\ell(U/\sig') \cdot (-y-1)^\ell  \cdot mC^\bT_y([O_{\sig'}\hookrightarrow X'])\\
&=\sum_{\sig' \in \Sig'}  \sum_{\ell \geq 0} (-1)^\ell \cdot d_\ell(U/\sig') \cdot 
(1+y)^{\ell+\dim(O_\sig')} \cdot [\omega_{V_{\sig'}}]_\bT.
\end{split}
\ee
A similar formula holds for $f_* T^\bT_{y*}([U \hookrightarrow X])$.
\ep

\begin{proof}
By functoriality and additivity of $mC_y^\bT$, we have
$$f_* mC^\bT_y([U \hookrightarrow X])=mC^\bT_y([U \hookrightarrow X\overset{f}{\to} X'])
=\sum_{\sig' \in \Sig'}  \sum_{\ell \geq 0} \sum_{\sig \in \Sigma_\ell(U/\sig')}mC^\bT_y([O_\sig \hookrightarrow X\overset{f}{\to} X']) .$$
For $\ell \geq 0$, $\sig' \in \Sig'$ and $\sig \in \Sigma_\ell(U/\sig')$ fixed, we have
$$[O_\sig \hookrightarrow X\overset{f}{\to} X']=[O_\sig \overset{f_\sig}{\to} O_{\sig'} \hookrightarrow X'].$$
Let us choose a splitting of the surjection $f_\bT:\bT \to \bT'$ so that $\bT=\bT'\oplus \bT''$.
By using \cite{CMM}[Lem.2.6]  and the proof of \cite{CMM}[Prop.2.7], there is a $\bT$-equivariant isomorphism
$O_\sig \simeq O_{\sig'} \times O_{\sig/\sig'} \to O_{\sig'}$,
with $\bT$ acting on $O_{\sig'}$ by $f_\bT:\bT \to \bT'$ and on  $O_{\sig/\sig'}$ via the projection   $\bT\to \bT''\to O_{\sig/\sig'}$. Here, the surjective group homomorphism $\bT''\to O_{\sig/\sig'}$ is given in the proof of \cite{CMM}[Prop.2.7].
We then have
$$[O_\sig \overset{f_\sig}{\to} O_{\sig'} \hookrightarrow X']=
[O_{\sig'} \hookrightarrow X'] \times  [O_{\sig/\sig'} \to pt].$$
Using the multiplicativity of $mC_y^\bT$, we get 
$$mC_y^\bT([O_\sig \overset{f_\sig}{\to} O_{\sig'} \hookrightarrow X'])=
mC_y^\bT([O_{\sig'} \hookrightarrow X']) \boxtimes mC_y^\bT([O_{\sig/\sig'} \to pt]).$$ 
Here, $mC_y^\bT([O_{\sig'} \hookrightarrow X'])$ is calculated by \eqref{11110} using the factorization $O_{\sig'} \hookrightarrow V_{\sig'} \hookrightarrow X'$. Finally, $mC_y^\bT([O_{\sig/\sig'} \to pt])=(-1-y)^\ell$ by \eqref{chiyt} and using a projective toric compactification of $O_{\sig/\sig'}$.

The corresponding formula for the equivariant Hirzebruch class follows by applying the equivariant Todd transformation to \eqref{fn90}.
\end{proof}

\begin{rem}
By forgetting the action, the corresponding non-equivariant version of \eqref{fn90} holds with the same proof.
\end{rem}

We not get back to the general context, where a surjection of tori $\bT \to \bT'$ (as in the above results) yields an identification of equivariant characteristic classes with respect to $\bT$ and $\bT'$, as we explain next.
Let the complex torus $\bT$ act on a complex algebraic variety $X$, with a subtorus $\bT''$ acting trivially and fitting into a short exact sequence of algebraic tori
$$1\to \bT'' \to \bT \to \bT' \to 1$$
so that one has an induced action of $\bT'$ on $X$. Then a $\bT'$-equivariant coherent sheaf or vector bundle on $X$ gets via the surjective homomorphism $\bT\to \bT'$ a corresponding $\bT$-structure, inducing  natural ``change of group homomorphisms''
\be
K_{\bT'}^0(X)\to K_{\bT}^0(X) \quad \text{and} \quad K^{\bT'}_0(X)\to K^{\bT}_0(X) \:,
\ee
compatible with ring and module structure coming from the tensor product. Similarly one has in equivariant (co)homology 
natural ``change of group homomorphisms''
\be
H_{\bT'}^*(X;\bQ)\to H_{\bT}^*(X;\bQ) \quad \text{and} \quad H^{\bT'}_*(X;\bQ)\to H^{\bT}_*(X;\bQ) \:,
\ee
compatible with ring and module structure coming from the equivariant  cup and cap product, as well as compatible with the corresponding
completions. These ``change of group homomorphisms'' are induced via the following finite dimensional approximations of $E\bT$ and $E\bT'$.
Choose open subsets
$U\subset V$ (resp. $U'\subset V'$) in a linear $\bT$- (resp. $\bT'$-) representation $V$ (resp. $V'$) of high dimension so that $\bT$ (resp. $\bT'$) acts freely on $U$ (resp. $U'$), with $U\subset V$ (resp. $U'\subset V'$) of sufficiently high codimension.
Then one can also use $U\times  U'\subset V\times V'$ as an approximation of $E\bT$, together with the submersion
$$(U\times  U')\times_{\bT}X \to  U'\times_{\bT'}X$$
coming from the $\bT'$-equivariant submersion of free $\bT'$-spaces $(U/\bT'')\times U'\times X \to U'\times X$.
The induced smooth pullback (and similarly for cohomology)
$$H^{\bT'}_k(X;\bQ)=H_{k+\dim(U')-\dim(\bT')}(U'\times_{\bT'}X;\bQ) \to
H_{k+\dim(U\times U')-\dim(\bT)}((U\times U')\times_{\bT}X;\bQ) =H^{\bT}_k(X;\bQ)$$
is well defined and independent (up to isomorphism) of all choices, inducing the ``change of group homomorphisms'' as above.
Moreover, these ``change of group homomorphisms'' are injective in this context, as follows from the next result.

\bp \label{split}
Let the complex torus $\bT$ act on a complex algebraic variety $X$, with a subtorus $\bT''$ acting trivially. Let $\bT':=\bT/\bT''$ and choose a compatible  splitting of these tori and of their corresponding character lattices
$$\bT=\bT' \oplus \bT'' \ \ \text{\rm and} \ \ M=M' \oplus  M''.$$
These splittings then induce the following factorizations
$$K^0_\bT(X)  \simeq K^0_{\bT'}(X) \otimes_\bZ K^0_{\bT''}(pt)\quad \text{and} \quad 
K_0^\bT(X)  \simeq K_0^{\bT'}(X) \otimes_\bZ K_0^{\bT''}(pt)\:,$$
as well as
$$H^{*}_\bT(X,\bQ)  \simeq H^{*}_{\bT'}(X,\bQ) \otimes_\bQ (\Lambda_{\bT''})_\bQ \quad \text{and} \quad 
H_{*}^\bT(X,\bQ)  \simeq H_{*}^{\bT'}(X,\bQ) \otimes_\bQ (\Lambda_{\bT''})_\bQ\:, $$
so that the  ``change of group homomorphisms'' above are given by tensoring with the distinguished elements
 $$\bC_{\chi^0} \in K^0_{\bT''}(pt) \quad \text{ resp.} \quad   1 \in (\Lambda_{\bT''})_\bQ\:.$$
Finally, these factorizations and therefore also the ``change of group homomorphisms''
are compatible with the corresponding equivariant Chern character and Todd class transformations. 
\ep

\begin{proof}
Regard $X$ as $X\times pt$, with the $\bT$-action corresponding to the product action of $\bT'$ on $X$ and of $\bT''$ on $pt$.
 Then the factorizations in equivariant K-theory follow from suitable K\"unneth formulae.
Similarly in equivariant (co)homology, if one chooses $U\subset V$ above also of such a product type with respect to $\bT=\bT' \oplus \bT''$.
By their definitions these fit then with the corresponding  ``change of group homomorphisms'' as stated.
The last statement finally follows from the multiplicativity under cross-products of the equivariant Chern character and equivariant Todd class transformation,  see \cite{EG0}.
\end{proof}

As a consequence, we have the following.
\bc\label{csplit}
In the setup of Proposition \ref{split}, let $\cE, \cF$ be a $\bT'$-equivariant vector bundle, resp.,  coherent sheaf on $X$, with induced $\bT$-action via the projection $\bT \to \bT'$. Then 
\be
\ch^\bT([\cE]_\bT)=\ch^{\bT'}([\cE]_{\bT'}) \ \ \text{\rm and} \ \ \ \td_*^\bT([\cF]_\bT)=\td_*^{\bT'}([\cF]_{\bT'}) 
\ee
under the ``change of group homomorphisms'' as above.
\ec
\begin{proof}
Choose a splitting $\bT=\bT' \oplus \bT''$. Then we get by  Proposition \ref{split}:
\begin{equation*}\begin{split} 
\td_*^\bT([\cF]_\bT) &=\td_*^{\bT'}([\cF]_{\bT'}) \otimes \td_*^{\bT''} (\bC_{\chi^0})\\
&=\td_*^{\bT'}([\cF]_{\bT'}) \otimes 1,
\end{split}
\end{equation*}
and similarly for the equivariant Chern character.
\end{proof} 

In particular, in the notations of Proposition \ref{orbit},  the surjection of tori $\bT \to \bT':=T_{N(\sig)}$ yields an identification
\be\label{e68}
\td_*^\bT([\omega_{V_\sig}]_\bT)=\td_*^{\bT'}([\omega_{V_\sig}]_{\bT'}).
\ee

%%%%%%%%%%%%%%%%%%%%%%%

\subsection{Generalized equivariant Hirzebruch-Riemann-Roch}\label{gehrr}
Let $X_\Sig$ be a  complete toric variety defined by a fan $\Sigma$, with $X:=X_{\Sig'} \subset X_\Sig$ a $\bT$-invariant closed algebraic subset of $X_\Sig$ defined by a star-closed subset $\Sig' \subset \Sig$. 
Let $D$ be a $\bT$-invariant Cartier divisor on $X_\Sig$. The {\it equivariant Hirzebruch polynomial of $D\vert_X$} is defined by the formula:
\be
\chi^\bT_y(X,\cO_{X_\Sig}(D)\vert_X):=\sum _{p=0}^{\dim(X)} \chi^\bT(X,\wti{\Omega}^{p}_X \otimes  \cO_{X_\Sig}(D)\vert_X) \cdot y^p \in (\Lambda^{an}_\bT)_{\bb{Q}}[y]
\subset  (\widehat{\Lambda}_\bT)_{\bb{Q}}[y] \:,
\ee
with $\wti{\Omega}^{p}_X$ denoting as in Remark \ref{subsp} 
the sheaf of Ishida $p$-forms on $X$. 

Then we have the following result:
\bt\label{gHRR}(Generalized equivariant Hirzebruch-Riemann-Roch)\newline \index{eneralized equivariant Hirzebruch-Riemann-Roch} 
In the above setup, the equivariant Hirzebruch polynomial of $D\vert_X$ is computed by the formula:
\be\label{Hp}
\begin{split}
\chi^\bT_y(X,\cO_{X_\Sig}(D)\vert_X)&=\int_{X} \ch^\bT(\cO_{X_\Sig}(D)\vert_X) \cap T^\bT_{y*}(X).
\end{split}\ee
\et
\begin{proof}
This follows as in \cite{MS1}[Theorem 2.4], by making use of Proposition \ref{p32}, Remark \ref{subsp} and the module property (\ref{modpe}) of $\td_*^\bT$.
\end{proof}

As a consequence, we obtain the following weighted version of formula \eqref{am5}:

\bc\label{c38}
Let $P$ be a full-dimensional lattice polytope with associated projective toric variety $X_P$ and ample Cartier divisor $D=D_P$. Let $X:=X_{P'}$ be the $\bT$-invariant closed algebraic subset of $X_P$ corresponding to a polytopal subcomplex $P' \subset P$ (i.e., a closed union of faces of $P$). Then:
\be\label{ghrrc}
\chi^\bT_y(X,\cO_{X_P}(D)\vert_X)= \sum_{E \preceq P'} (1+y)^{\dim(E)} \cdot \sum_{m \in \Relint(E) \cap M} e^{s(m)}.
\ee
\ec

\begin{proof} For a face $E$ of $P'$, denote by $i_E:V_{\sig_E}:=X_E \hookrightarrow X$ the inclusion of the orbit closure associated to the (cone of the) face $E$. 
Note that we have $\dim(E)=\dim(O_{\sig_E}).$  Let $\bT':=T_{N(\sig_E)}$ be the quotient torus of $\bT$ corresponding to $X_E$.

Then, by Theorem \ref{gHRR}, Proposition \ref{pad} and Proposition \ref{orbit}, the following  equality holds:
\begin{equation*}\begin{split}
\chi^\bT_y(X,\cO_{X_P}(D)\vert_X) &= \int_{X} \ch^\bT (\cO_{X_P}(D)\vert_X) \cap T^\bT_{y*}(X) \\
&= \sum_{E \preceq P'} (1+y)^{\dim(E)} \int_{X} \ch^\bT(\cO_{X_P}(D)\vert_X) \cap (i_E)_*\td^\bT_*([\omega_{X_E}]_\bT).\end{split}
\end{equation*}
It remains to prove that for any face $E$ of $P'$, we have that:
\be\label{lem2}
 \int_{X} \ch^\bT(\cO_{X_P}(D)\vert_X) \cap (i_E)_*\td^\bT_*([\omega_{X_E}]_\bT) =\sum_{m \in \Relint(E) \cap M} e^{s(m)}.
\ee
This follows from the functorial properties of the cap product and formula \eqref{tran}. Indeed,
\begin{equation*}
\begin{split}
 \int_{X} \ch^\bT(\cO_{X_P}(D)\vert_X) \cap (i_E)_*\td^\bT_*([\omega_{X_E}]_\bT) 
 &= \int_{X_E} (i_E)^* \ch^\bT(\cO_{X_P}(D)\vert_X) \cap \td^\bT_*([\omega_{X_E}]_\bT) \\
 &= \int_{X_E}  \ch^\bT((i_E)^*(\cO_{X_P}(D)\vert_X)) \cap \td^\bT_*([\omega_{X_E}]_\bT) \\
&\overset{(\ref{tran})}{=} \int_{X_E}  \ch^\bT(\cO_{X_E}(D_{E_0}-div(\chi^{m_0}))) \cap \td^\bT_*([\omega_{X_E}]) \\
&=e^{-c(m_0)} \int_{X_E}  \ch^\bT(\cO_{X_E}(D_{E_0})) \cap \td^\bT_*([\omega_{X_E}]_\bT) \\
&\overset{({*})}{=} e^{-c(m_0)} \int_{X_E}  \ch^{\bT'}(\cO_{X_E}(D_{E_0})) \cap \td^{\bT'}_*([\omega_{X_E}]_{\bT'}) \\
&\overset{(\ref{lpc2})}{=} e^{s(m_0)} \cdot \left(   \sum_{m \in \Relint(E-m_0) \cap M} e^{s(m)}  \right)\\
&=\sum_{m \in \Relint(E) \cap M} e^{s(m)},
 \end{split}
\end{equation*}
where (*) uses Proposition \ref{split} and Corollary \ref{csplit}. Here $m_0$ is a vertex of $E$ as in (\ref{tran}), so that $E_0:=E-m_0$ is a full-dimensional lattice polytope  in $Span(E_0)$ relative to the lattice $Span(E_0)\cap M$, with $X_E$ the associated toric variety.
\end{proof}

\br
For future use, we include here the following formula, which one gets as in the proof of  \eqref{lem2}, but using $\cO_{X_E}$ instead of $\omega_{X_E}$:
\be\label{lem2b}
\int_{X} \ch^\bT(\cO_{X_P}(D)\vert_X) \cap \td^\bT_*((i_E)_*[\cO_{X_E}]_\bT) =\sum_{m \in E \cap M} e^{s(m)}.
\ee
This is also a special case of Corollary \ref{c38} for $y=0$ and $P'=E$.
\er

\medskip
 As another application,  we use Theorem \ref{gHRR} in the context of a {\it globally generated} $\bT$-invariant Cartier divisor $D$ on a complete toric variety $X$ with associated torus $\bT$. Let $P_D\subset M_\bR$ be the lattice polytope corresponding to $D$, and let $X_D$ be the toric variety of the lattice polytope $P_D$, defined via the corresponding \index{generalized fan} {\it generalized fan} as in \cite{CLS}[Prop.6.2.3]. By \cite{CLS}[Thm.6.2.8], there is a proper toric morphism $f:X \to X_D$, induced by the corresponding lattice projection $N \to  N_D$ given by dividing out by the minimal cone of the generalized fan of $P_D$. In particular, $f:X \to X_D$ is a toric fibration. Let $M_D \hookrightarrow M$ be the associated inclusion of dual lattices. Choosing a vertex $m_0$ of $P_D$, we get that $P':=P_D-m_0 \subset M_D$ is a full-dimensional lattice polytope. Let $D'$ be the ample divisor on $X_D$ associated to $P'$, with $\Sigma'$ the inner normal fan of $P'$ (defining $X_D$), so that there is a one-to-one correspondence between cones $\sig' \in \Sig'$ and faces $E'$ of the lattice polytope $P'$, and by translation to the faces $E$ of $P_D$. Then, by the proof of \cite{CLS}[Thm.6.2.8], one gets that $\cO(D-div(\chi^{m_0}))\simeq f^*\cO(D')$, as $\bT$-equivariant sheaves. 

Following the notations of Proposition \ref{relmc}, let $Z \subset X$ be a closed $\bT$-invariant subset with associated multiplicities $d_\ell(Z/\sig')$ as in \eqref{dl}, denoted here via the above correspondence as $d_\ell(Z/E)$. We can now prove the following.
\bc\label{gHRRnef}
With the above notations,
\be\label{nefghrr}
\chi^\bT_y(Z,\cO_{X}(D)\vert_Z)= \sum_{E \preceq P_D} \left( \sum_{\ell \geq 0} 
(-1)^\ell \cdot d_\ell(Z/E) \cdot (1+y)^{\ell + \dim(E)} \right) \cdot  \sum_{m \in \Relint(E) \cap M} e^{s(m)}.
\ee
\ec

\begin{proof}
By the generalized Hirzebruch-Riemann-Roch Theorem \ref{gHRR}, the projection formula, and functoriality of Hirzebruch classes, we have
\begin{equation*}
\begin{split}
\chi^\bT_y &(Z,\cO_{X}(D)\vert_Z) = \int_{Z} \ch^\bT (\cO_{X}(D)\vert_Z) \cap T^\bT_{y*}(Z) \\
&= \int_{X} \ch^\bT (\cO_{X}(D)) \cap T^\bT_{y*}([Z\hookrightarrow X]) \\
&= e^{s(m_0)} \int_{X} \ch^\bT (f^*\cO_{X}(D')) \cap T^\bT_{y*}([Z\hookrightarrow X]) \\
&=e^{s(m_0)} \int_{X_D} \ch^\bT (\cO_{X}(D')) \cap f_*T^\bT_{y*}([Z\hookrightarrow X])\\
&\overset{(\ast)}{=} e^{s(m_0)} \left( \sum_{E' \preceq P'} 
 \sum_{\ell \geq 0} 
(-1)^\ell \cdot d_\ell(Z/E') \cdot (1+y)^{\ell +\dim(E')}  \cdot  \sum_{m \in \Relint(E') \cap M} e^{s(m)} \right)\\
&= \sum_{E \preceq P_D} \sum_{\ell \geq 0} 
(-1)^\ell \cdot d_\ell(Z/E) \cdot (1+y)^{\ell + \dim(E)}  \cdot  \sum_{m \in \Relint(E) \cap M} e^{s(m)},
 \end{split}
\end{equation*}
where the equality $({\ast})$ follows by using \eqref{fn90} and \eqref{lem2}.
\end{proof}

\bex[Rigidity of the equivariant $\chi_y$-genus]\label{rigid}
If, in Corollary \ref{gHRRnef}, we take $D=0$, then $P_D=\{ 0 \} \subset M \subset M_\bR$, so 
$d_\ell(Z/\{0\})$ is just the number of $\ell$-dimensional torus orbits in $Z$. Hence, formula 
\eqref{nefghrr} becomes in this case
\be\label{nefex} \chi^\bT_y(Z,\cO_{Z}) = \left(  \sum_{\ell \geq 0} 
(-1)^\ell \cdot d_\ell(Z/\{0\}) \cdot (1+y)^{\ell} \right) \cdot e^{s(0)}.
\ee
Note that this formula does not take into account any equivariant information. So, forgetting the $\bT$-action,
$$  \chi_y(Z,\cO_{Z}) =
\left(  \sum_{\ell \geq 0} 
(-1)^\ell \cdot d_\ell(Z/\{0\}) \cdot (1+y)^{\ell} \right) \:.
$$
In particular, for $Z=X$ and $y=0$, this becomes $\chi(X,\cO_{X})=1$. \qed
\eex

\begin{rem}
Note that if $D$ is an ample $\bT$-invariant Cartier divisor, then the above morphism $f$ is the identity of $X$, so formula \eqref{nefghrr} reduces to 
\eqref{ghrrc} since the multiplicities $d_\ell(Z/E)$ are given as follows: $d_\ell(Z/E)=0$ for $\ell>0$, and $d_0(Z/E)$ is either $1$ or $0$, depending wether the orbit associated to $E$ is included or not in $Z$.
\end{rem}

\begin{rem}
By forgetting the $\bT$-action in formula \eqref{nefghrr}, we get the following weighted lattice point counting for lattice polytopes associated to globally generated $\bT$-invariant Cartier divisors. More precisely, in the above notations, we get the following generalization of Theorem \ref{wco} and of \cite{MS1}[Cor.1.8] for $Z$ a $\bT$-invariant closed subset in a complete toric variety $X$:
\be\label{nefghrrc}
\chi_y(Z,\cO_{X}(D)\vert_Z)= \sum_{E \preceq P_D}\left(  \sum_{\ell \geq 0} 
(-1)^\ell \cdot d_\ell(Z/E) \cdot (1+y)^{\ell + \dim(E)}  \right) \cdot  \vert \Relint(E) \cap M \vert .
\ee
\end{rem}

%%%%%%%%%%%%%%%%%%%%

\subsection{Equivariant Hirzebruch and Todd characteristic classes of simplicial toric varieties}\label{calch}
The previously mentioned formulae of Subsection \ref{necl} for the Todd class and Hirzebruch classes of a simplicial toric variety hold with the same proofs (up to some small modifications, as explained below) for the equivariant versions of these characteristic classes, compare also with \cite{EG}[Rem.4.3]. {\it Different proofs of all these results}, independent of the equivariant version of the Lefschetz-Riemann-Roch from \cite{EG}[Thm.3.1 and Rem.4.3], will be given in Subsection \ref{reprloc} via localization at the torus fixed points. Nevertheless, the approach of this subsection is useful to derive global expressions for the equivariant characteristic classes of interest, which will then be localized.
 The equivariant version of the Lefschetz-Riemann-Roch of Edidin-Graham \cite{EG}[Thm.3.1 and Rem.4.3] in the context of the Cox construction should be seen as a counterpart to the corresponding Lefschetz-Riemann-Roch Theorem of Baum-Fulton-Quart \cite{BFQ}
in the context of an action by a finite group $G$, see especially \cite{BFQ}[Section 4] for the calculation of the Todd class $td_*(X/G)$ of a corresponding quotient variety $X/G$. These methods have also  been used in \cite{CMSS} for the calculation of the Hirzbruch classes
of such a quotient variety $X/G$ for $G$ finite.\\

Let $X:=X_\Sig$ be a simplicial  toric variety associated to a fan $\Sig$ of full-dimensional convex support, e.g., $X=X_P$ is the semi-projective simplicial toric variety corresponding to a simple full-dimensional lattice polyhedron $P \subset M_\bR$. This includes the cases of full-dimensional simple lattice polytopes, as well as full-dimensional rational pointed polyhedral cones. The assumption that $\Sig$ is of full-dimensional convex support implies that $X$ contains no torus factors, so we can use the Cox construction as described in Subsection \ref{Cox}. Moreover, this allows us to make use of Remark \ref{reminj}. 

In addition, we  prove formulae for 
the equivariant Hirzebruch classes  $T_{y*}^\bT([U \hookrightarrow X])$, with $U$ the open complement of a $\bT$-invariant divisor $D_K:=\bigcup_{\rho\in K} D_\rho$, for $K \subset \Sigma(1)$. Similarly, we 
indicate here the argument even for 
the Hirzebruch classes of an orbit closure $V_{\tau}$, for $\tau \in \Sig$, pointing out the needed modifications in the proof.

The arguments here are based on the Cox construction. With the notations from Subsection \ref{Cox}, let 
$\pi:W:=\bC^r \setminus Z(\Sig) \to X$ be the toric morphism, with $G$ the kernel of the corresponding map of tori $\gamma:\wti{\bT}:= (\bC^*)^r\to \bT$ (here $r=\vert \Sigma(1) \vert$). Since $X$ is a simplicial toric variety  containing no torus factor,
it follows that $X$ is the geometric quotient $W/G$, with $G$ acting with finite stabilizers.
If $\tau$ is a cone of $\Sig$, then the orbit closure $V_{\tau}$ is the image under $\pi$ of a linear subspace $W_{\tau}\subset W$. We have the following formula (in which we omit the symbols for pushforwards under closed embeddings).
\begin{lem}
\be\label{thi}
\td_*^\bT\left((\pi_*\Lambda_yT^*_{W_{\tau}})^G\right)=(1+y)^{r-n}\cdot  \td_*^\bT(mC_y^\bT(V_{\tau}))\in \widehat{H}_*^\bT(X;\bQ)[y]\simeq \widehat{H}^*_\bT(X;\bQ)[y].
\ee
\end{lem}
\begin{proof}
By \eqref{loc}, it suffices to check formula \eqref{thi} after restriction to each $U_\sig$, with $U_\sig$ the $\bT$-invariant open affine subset of $X$ containing the corresponding $\bT$-fixed point $x_\sig$, for $\sigma \in \Sigma(n)$. We have $\pi^{-1}(U_\sig)\simeq \bC^n \times (\bC^*)^{r-n}$, with $\wti{\bT}\simeq \bT \times (\bC^*)^{r-n}=:\bT \times \bT'$ acting on the respective factors, and the factor $\bC^n$ corresponding to the rays of $\sigma$ and $\bT'=(\bC^*)^{r-n}$ acting freely by multiplication on itself. Similarly, $G\simeq G_\sigma \times \bT'$, with $G_\sig \subset \bT$ a finite subgroup. So, above $U_\sig$, $\pi$ can be factorized as a composition of the free quotient $\wti{\pi}:\bC^n \times (\bC^*)^{r-n} \to \bC^n$ by the $\bT'$-action, followed by a finite quotient map $\pi=\pi_\sig:\bC^n \to \bC^n/{G_\sig}=U_\sig$.
Similarly, $\pi^{-1}(U_\sig \cap V_{\tau})\simeq L_\tau \times (\bC^*)^{r-n}$, with $L_\tau \subset \bC^n$ a linear subspace. Then
$$\big(\wti{\pi}_*\Lambda_y T^*_{L_\tau \times (\bC^*)^{r-n}}\big)^{\bT'}=\Lambda_y T^*_{L_\tau} \cdot (1+y)^{r-n},$$
since $T^*_{(\bC^*)^{r-n}}\simeq (\bC^*)^{r-n} \times T^*_{id}$, with $\bT'$ acting by the co-adjoint action on the cotangent space $T^*_{id}$ of $\bT'$ at the identity element $id \in \bT'$. This is the trivial $\bT'$-bundle of rank $r-n$, since $\bT'$ is an abelian group, so that 
$\Lambda_yT^*_{(\bC^*)^{r-n}}=(1+y)^{r-n}$.
Finally, in $K_0^\bT(U_\sig)[y]$, we have by Proposition \ref{p32} and Corollary \ref{csplit} that
$$\left((\pi_\sig)_* \Lambda_y T^*_{L_\tau}\right)^{G_\sig}=mC^\bT_y(V_{\tau}\cap U_\sig).$$
The desired formula \eqref{thi} follows by applying the equivariant Todd transformation. 
\end{proof}

 By the equivariant version of the Lefschetz-Riemann-Roch theorem of Edidin-Graham (see \cite{EG}[Thm.3.1 and Rem.4.3]), applied to the left-hand side of formula \eqref{thi}, the proof of the formulae of Subsection \ref{necl} for the Todd class and Hirzebruch classes of 
projective simplicial toric varieties hold with the same proofs in the equivariant setting.
For later applications, we work with the cohomological images $\td_*^\bT(X) \in \widehat{H}^*_\bT(X;\bQ)$ and ${T}^\bT_{y*}(X) \in \widehat{H}^*_\bT(X;\bQ)[y]$,
resp., $\widehat{T}^\bT_{y*}(X) \in \widehat{H}^*_\bT(X;\bQ[y]):=\prod_{i \geq 0} H^i_\bT(X;\bQ[y])$
 of these classes under equivariant Poincar\'e duality.
In the notations of Subsection \ref{Cox}, one has the following equivariant counterpart of Theorem \ref{MSt1}:
\bt\label{eqHirz} Let $X:=X_\Sig$ be a simplicial  toric variety associated to a fan $\Sig$ of full-dimensional convex support.
Then the equivariant Hirzebruch classes ${T}^\bT_{y*}(X)$ and $\widehat{T}^\bT_{y*}(X)$ are computed by: \index{equivariant Hirzebruch class}
\be\label{eHirz0}
{T}^\bT_{y*}(X)= (1+y)^{n-r} \cdot \sum_{g \in G_{\Sig}}  \prod_{\rho \in \Sig(1)} \frac{ F_{\rho} \cdot 
\big( 1+y  \cdot a_{\rho}(g)  \cdot e^{-F_{\rho}}\big)}{1-a_{\rho}(g) \cdot e^{-F_{\rho}}}  \in \widehat{H}^*_\bT(X;\bQ)[y] \:,
\ee

\be\label{eHirz}
\widehat{T}^\bT_{y*}(X)=  \sum_{g \in G_{\Sig}}  \prod_{\rho \in \Sig(1)} \frac{ F_{\rho} \cdot 
\big( 1+y  \cdot a_{\rho}(g)  \cdot e^{-F_{\rho}(1+y)}\big)}{1-a_{\rho}(g) \cdot e^{-F_{\rho}(1+y)}}  \in \widehat{H}^*_\bT(X;\bQ[y]) \:,
\ee
with $r=\vert \Sigma(1)\vert$, and $F_\rho=[D_\rho]_\bT$ denoting the equivariant fundamental class of the $\bT$-invariant divisor $D_\rho$ corresponding to the ray $\rho \in \Sig(1)$.
\et

\bc\label{ceH}
If $X:=X_\Sig$ is a smooth  toric variety associated to a fan $\Sig$ of full-dimensional convex support, then $G_\Sig$ is the trivial group.
Hence, the equivariant Hirzebruch classes of $X$ are given by:
\be
{T}^\bT_{y*}(X)= (1+y)^{n-r} \cdot  \prod_{\rho \in \Sig(1)} \frac{ F_{\rho} \cdot 
\big( 1+y  \cdot e^{-F_{\rho}}\big)}{1- e^{-F_{\rho}}}  
\ee
\be\label{enHirz}
\widehat{T}^\bT_{y*}(X)=   \prod_{\rho \in \Sig(1)} \frac{ F_{\rho} \cdot 
\big( 1+y  \cdot e^{-F_{\rho}(1+y)}\big)}{1- e^{-F_{\rho}(1+y)}} \:.
\ee
\ec

\br
In particular, by setting $y=0$ in Theorem \ref{eqHirz}, we recover the equivariant Todd class formula of \cite{BrV2} given by the following expression in $\widehat{H}^*_\bT(X;\bQ)$: \index{equivariant Todd class}
\begin{equation}\label{Todd0}
\begin{split}
\td_*^\bT(X)&= \sum_{g\in G_{\Sigma}} \prod_{\rho\in \Sigma(1)} \frac{F_{\rho}}{1-a_{\rho}(g)\cdot e^{-F_{\rho}}} .
\end{split}
\end{equation}
Similarly, setting $y=0$ in Corollary \ref{ceH}, one gets the equivariant Todd class 
for such a smooth  toric variety. 
Finally, if $X$ is projective, the specialization at $y=1$ fits with suitable $L$-classes, i.e., $\widehat{T}_{1*}(X)=L(X)$ is the Thom-Milnor $L$-class of a projective toric variety $X$ (cf. \cite{CMSS}[Cor.1.2] and \cite{MS1}[Cor.1.2]).
There is also an equivariant version of formula (\ref{todd-EM-intro}), which can be proved via the  the specialization for $y=1$ of \eqref{eHirz} exactly  as in \cite{MS1}[Sec. 6],  by using the corresponding $\bT$-equivariant Mock Hirzebruch and $T$-classes. We leave the details of this calculation to the reader.

\er

Moreover, taking the top degree of $y$ in \eqref{eHirz0} and using \eqref{38}, yields the following.
\bc\label{c311}
Under the assumptions and notations of Theorem \ref{eqHirz}, we have:
\be\label{dueq}
\td_*^\bT([\omega_X]_\bT)= \sum_{g\in G_{\Sigma}} \prod_{\rho\in \Sigma(1)} \frac{a_{\rho}(g) \cdot F_{\rho} \cdot e^{-F_{\rho}}}{1-a_{\rho}(g) e^{-F_{\rho}}} \in  \widehat{H}^*_\bT(X;\bb{Q}).
\ee
\ec

\begin{rem}\label{rem:duality}
Note that formula \eqref{dueq} coincides with \eqref{Todd0} upon substituting $-F_\rho$ for $F_\rho$, for each $\rho \in \Sig(1)$, i.e., these classes are exchanged by the {\it cohomological duality involution} on $\widehat{H}^*_\bT(X;\bb{Q})$ given by multiplication with $(-1)^i$ in degree $2i$. (Recall that under our assumptions $\widehat{H}^*_\bT(X;\bb{Q})$  is even, by \eqref{mm}.
 In fact this is true in greater generality, e.g., $\widehat{H}^*_\bT(X;\bb{Q})$  is even in case $\bT$ acts on a complex algebraic rational homology manifold with finitely many orbits, see \cite{W2}.) 
\end{rem}

In Section \ref{emf} below, formula \eqref{dueq} will be used for proving Euler-Maclaurin type formulae for lattice points in the interior of a full-dimensional polytope, generalizing formula \eqref{4b} for lattice point counting. The next results are motivated by the fact that, instead of deleting all facets of the polytope $P$, one can just delete some of the facets $F_i$, $i\in K\subset \Sigma_P(1)$. 

\begin{lem} 
Let $X:=X_\Sig$ be a simplicial toric variety associated to a fan $\Sig$ of full-dimensional convex support.
Let $U\subset X$ be the open complement of the divisor $D_K:=\bigcup_{\rho\in K} D_\rho$, for $K \subset \Sigma(1)$. 
In the notations 
of the Cox construction, let $W_\rho=\{x_\rho=0\}\subset W$ be the inverse image of $D_\rho$ under the quotient map $\pi:W\to X$. Then the preimage $\wti{U}$ of $U$ under $\pi$ is the complement of the $\wti{\bT}$-invariant normal crossing divisor $W_K=\bigcup_{\rho\in K} W_\rho$ in $W$, and
\be\label{thi2}
\td_*^\bT\big((\pi_*mC_y^{\wti{\bT}}([\wti{U}\hookrightarrow W]))^G\big)=
(1+y)^{r-n}\cdot  \td_*^\bT(mC_y^\bT([U \hookrightarrow X]))\in  \widehat{H}^*_\bT(X;\bQ)[y].
\ee
\end{lem}

\begin{proof}
By the inclusion-exclusion formula for the equivariant motivic Chern classes, one has as in \eqref{Web5} that
\be mC^\bT_y([U \hookrightarrow X])=\sum_{I\subset K} (-1)^{|I|}mC^\bT_y([D_I \hookrightarrow X]) \:,\ee
with $D_I=\bigcap_{\rho \in I} D_\rho$. Similarly, 
\be mC^{\wti{\bT}}_y([\wti{U} \hookrightarrow W])=
\sum_{I\subset K} (-1)^{|I|}mC^{\wti{\bT}}_y([W_I \hookrightarrow W]),\ee
with $W_I=\bigcap_{\rho \in I} W_\rho$.
The assertion follows now by applying formula \eqref{thi} to the summands on the right-hand side of the two identities above.\end{proof}

By formula \eqref{web}, in the above notations we have:
\be\label{web2}
mC^{\wti{\bT}}_y([\wti{U} \hookrightarrow W])=[\cO_W(-W_K) \otimes \Lambda_y \Omega^1_W(\log W_K)]_{\wti{\bT}} \in K_0^{\wti{\bT}}(W)[y] \:.
\ee
Then by the equivariant version of the Lefschetz-Riemann-Roch theorem of Edidin-Graham (see \cite{EG}[Thm.3.1 and Rem.4.3]), applied to the right-hand side of formula \eqref{web2}, we get the following generalization of Theorem \ref{eqHirz}: 

\bt\label{eqHirzcomp} 
Let $X:=X_\Sig$ be a simplicial  toric variety associated to a fan $\Sig$ of full-dimensional convex support.
Let $U\subset X$ be the open complement of the divisor $D_K:=\bigcup_{\rho\in K} D_\rho$, for $K \subset \Sigma(1)$. 
The equivariant Hirzebruch classes ${T}^\bT_{y*}([U\hookrightarrow X])$ and $\widehat{T}^\bT_{y*}([U\hookrightarrow X])$ are computed by:
\be\label{eHirz0comp}
\begin{split}
{T}^\bT_{y*}&([U\hookrightarrow X]) \\ &= (1+y)^{n-r} \cdot \sum_{g \in G_{\Sig}}  \prod_{\rho \in K} \frac{ F_{\rho} \cdot 
( 1+y ) \cdot a_{\rho}(g) \cdot e^{-F_{\rho}}}{1-a_{\rho}(g) \cdot e^{-F_{\rho}}} 
\prod_{\rho \notin K} \frac{ F_{\rho} \cdot 
\big( 1+y  \cdot a_{\rho}(g)  \cdot e^{-F_{\rho}}\big)}{1-a_{\rho}(g) \cdot e^{-F_{\rho}}} \:,
\end{split}
\ee

\be\label{eHirzcomp}
\begin{split}
\widehat{T}^\bT_{y*}&([U\hookrightarrow X]) \\ &=  \sum_{g \in G_{\Sig}}  \prod_{\rho \in K} \frac{ F_{\rho} \cdot 
( 1+y ) \cdot a_{\rho}(g)  \cdot e^{-F_{\rho}(1+y)}}{1-a_{\rho}(g) \cdot e^{-F_{\rho}(1+y)}} 
\prod_{\rho \notin K} \frac{ F_{\rho} \cdot 
\big( 1+y  \cdot a_{\rho}(g)  \cdot e^{-F_{\rho}(1+y)}\big)}{1-a_{\rho}(g) \cdot e^{-F_{\rho}(1+y)}}
\:,
\end{split}
\ee
with $r=\vert \Sigma(1)\vert$, and $F_\rho=[D_\rho]_\bT$ denoting the equivariant fundamental class of the $\bT$-invariant divisor $D_\rho$ corresponding to the ray $\rho \in \Sig(1)$.
\et
\begin{proof}
The proof is similar to that of Theorem \ref{eqHirz} and is based on the equivariant Lefschetz-Riemann-Roch theorem. We only indicate here the changes. Instead of calculating the equivariant twisted Chern character $ch^{\wti{\bT}}(i_g^*-)(g)$, $g \in G$, of $\Lambda_yT^*_W$ for $i_g:W^g \hookrightarrow W$ the fixed point set inclusion, one now has to calculate this for $\cO_W(-W_K) \otimes \Lambda_y \Omega^1_W(\log W_K)$. 
Here, for $\cE \in K^0_{\wti{\bT}}(W)$, the twisted equivariant Chern character is defined as
$$ch^{\wti{\bT}}(i_g^*\cE)(g)=\sum_\chi \chi(g) \cdot ch^{\wti{\bT}}(\cE_\chi),$$
where $\cE\simeq \bigoplus_\chi \cE_\chi$ is the finite decomposition of $\cE$ into sheaves $\cE_\chi$ on which $g$ acts by a (complex-valued) character $\chi$. These sheaves $\cE_\chi$ are also $\wti{\bT}$-equivariant since $\wti{\bT}$ is an abelian group.

We now have to evaluate $ch^{\wti{\bT}}(i_g^*-)(g)$ on 
$$\cO_W(-W_K) \otimes \Lambda_y \Omega^1_W(\log W_K)\simeq \prod_{\rho \in K} \big(\cO_W(-W_\rho) \otimes pr_\rho^* \Lambda_y \Omega^1_\bC(\log \{0\}) \big)
\prod_{\rho \notin K} \Lambda_y\cO_W(-W_\rho).$$
Here $pr_\rho:W \to \bC$ is the projection to the $\rho$-th factor, with $\wti{\bT} \simeq (\bC^*)^r$ acting factorwise via the projection $a_\rho:\wti{\bT} \to \bC^*$ to the corresponding $\rho$-th factor. So the calculation can be done factorwise, with 
$$ch^{\wti{\bT}}(i_g^*\Lambda_y\cO_W(-W_\rho))(g)=i_g^*\left( 1+y  \cdot a_{\rho}(g^{-1})  \cdot e^{-z_{\rho}}\right),$$
for $\rho \notin K$, and $z_\rho:=[W_\rho]_{\wti{\bT}} \in H^*_{\wti{\bT}}(W;\bQ)$ the corresponding equivariant fundamental class as in \cite{MS1}[eqn.(5.13)]. By the inclusion-exclusion formula, for $\{0\} \subset \bC$ and $\rho \in K$, we get:
$$
[\cO_W(-W_\rho) \otimes pr_\rho^* \Lambda_y \Omega^1_\bC(\log \{0\})] =
[\Lambda_y\cO_W(-W_\rho)]-[\cO_{D_\rho}].$$
So 
\begin{equation*}
\begin{split} 
& ch^{\wti{\bT}}\left(i_g^*\big(\cO_W(-W_\rho) \otimes pr_\rho^* \Lambda_y \Omega^1_\bC(\log \{0\})\big)\right)(g) \\
&=i_g^*\big( 1+y  \cdot a_{\rho}(g^{-1})  \cdot e^{-z_{\rho}}\big) - ch^{\wti{\bT}}(i_g^* \left([\cO_W]-[\cO_W(-D_{\rho})]\right))(g) \\
& = i_g^*\big( 1+y  \cdot a_{\rho}(g^{-1})  \cdot e^{-z_{\rho}}\big) - i_g^*\big( 1 -a_\rho(g^{-1}) \cdot e^{-z_{\rho}} \big) \\
& = i_g^*\left(( 1+y ) \cdot a_{\rho}(g^{-1})  \cdot e^{-z_{\rho}}\right).
\end{split} 
\end{equation*}
Then these formulae are pushed forward via $(i_g)_*$ using the projection formula, and one applies the ring isomorphism $\phi:H^{*}_{\wti{\bT}}(W;\bQ) \simeq H^{*}_{{\bT}}(X;\bQ)$.
Note that $z_\rho$ maps to $F_\rho$ under the identification $\phi:H^2_{\wti{\bT}}(W;\bQ) \simeq H^2_{{\bT}}(X;\bQ)$. 
Finally, note that $g \in G_\Sig$ if and only if $g^{-1} \in G_\Sig$. This completes the proof for the un-normalized equivariant Hirzebruch classes.
 To get the formula for the normalized equivariant Hirzebruch classes, we just have to substitute $(1+y)^{-n}\cdot [X]_{\bT}$ for $[X]_{\bT}\in H^\bT_{2n}(X;\bQ)$ (implicitly used in the equivariant Poincar\'e duality), and
 $(1+y)\cdot F_{\rho}$ for $F_{\rho}\in H^2(X;\bQ)$.
\end{proof}

For later use, and since the needed notation was already introduced in the previous result, let us sketch a proof of the following extension of the Todd class formula \eqref{Todd0} to the $\bT$-equivariant coherent sheaf $\pi_*(\cO_W \otimes \bC_{\chi^{\wti{m}}})^G$, for $\pi:W \to X$ the quotient map of the Cox construction and $\chi^{\wti{m}}$ a character of $\wti{\bT}$.
\bl\label{l329} With the above notations, we have
\be\label{la22}
\td_*^\bT\left(\pi_*(\cO_W \otimes \bC_{\chi^{\wti{m}}})^G\right)= \sum_{g\in G_{\Sigma}} \chi^{\wti{m}}(g^{-1}) \prod_{\rho\in \Sigma(1)} \frac{F_{\rho} \cdot e^{\langle \wti{m}, e_\rho \rangle \cdot F_\rho}}{1-a_{\rho}(g) e^{-F_{\rho}}} ,
\ee
with $\{e_\rho\}_{\rho \in \Sig(1)}$ the standard basis of the lattice $\wti{N}=\bZ^{\vert \Sig(1) \vert}$.
\el
\begin{proof}
Comparing to the proof of the classical Todd class formula \eqref{Todd0}, we need to calculate 
$ch^{\wti{\bT}}(i_g^*\big(\cO_W \otimes \bC_{\chi^{\wti{m}}}\big))(g)$ instead of $ch^{\wti{\bT}}(i_g^*(\cO_W))(g)=1$. This is given by
$$ch^{\wti{\bT}}(i_g^*\big(\cO_W \otimes \bC_{\chi^{\wti{m}}}\big))(g)=i_g^* \left( \chi^{\wti{m}}(g) e^{c(\wti{m})} \right),$$
with $c(\wti{m})=\sum_\rho \langle \wti{m}, e_\rho \rangle z_\rho \in H^2_{\wti{\bT}}(W;\bb{Q})$, 
e.g., see \cite{CLS}[Prop.12.4.13(b)].
\end{proof}

\br
Note that if $K=\emptyset$, Theorem \ref{eqHirzcomp} reduces to Theorem \ref{eqHirz}. At the other extreme, if $K=\Sig(1)$, one gets for the un-normalized class first by specializing to $y=0$ just Corollary \ref{c311}, and using this one then recovers for arbitrary $y$ formula \eqref{11111} in the case of the zero cone.
\er

\br
For $y=0$, both formulae of Theorem \ref{eqHirzcomp} specialize to the following Todd class type formula, which will fit later on with the Euler-Maclaurin formulae for polytopes with some facets removed:
\be\label{t0comp}
{T}^\bT_{0*}([U\hookrightarrow X])= \sum_{g \in G_{\Sig}}  \prod_{\rho \in K} \frac{ F_{\rho} \cdot 
 a_{\rho}(g) \cdot e^{-F_{\rho}}}{1-a_{\rho}(g) \cdot e^{-F_{\rho}}} 
\prod_{\rho \notin K} \frac{ F_{\rho}}{1-a_{\rho}(g) \cdot e^{-F_{\rho}}} \:.
\ee
Note that in view of Remark \ref{subsp}, one also has 
\be\label{difer} {T}^\bT_{0*}([U\hookrightarrow X])=td_*^\bT(X) - td_*^\bT(D_K).\ee
\er

%%%%%%%%%%%%%%%%%%%%%%%%%%%%%%

\section{Localization in equivariant $K$-theory and applications}\label{sec:4}
In this section we apply localization techniques in $\bT$-equivariant $K$-theory (Subsection \ref{sec4.1}) and $\bT$-equivariant homology theory (Subsection \ref{sec4.2}) of toric varieties, due to Brion-Vergne \cite{BrV2} and, resp., Brylinski-Zhang \cite{BZ}, for the calculation of the $\bT$-equivariant motivic Chern and Hirzebruch classes in this toric context. 
Here the calculation of the  localized equivariant Hirzebruch class for a toric variety is due to Weber \cite{W}[Thm.11.3] (see also \cite{R}[Thm.6.1]). Our localization techniques in $\bT$-equivariant homology theory are selfcontained and independent of the unpublished paper of Brylinski-Zhang \cite{BZ}.
In the simplicial context, these localizations can be made explicit by using a Lefschetz type Euler characteristic in the local affine context, 
 instead of using the global equivariant Lefschetz-Riemann-Roch theorem.
Finally, we recast in Subsection \ref{reprloc} the formulae of the previous section for the 
equivariant Hirzebruch classes of simplicial toric varieties via localization at the torus fixed points.

We restrict ourselves for simplicity to the toric context, although many results hold in greater generality and have a long history.
We refer to Goresky-MacPherson-Kottwitz \cite{GKM} for a review of the  localization at torus fixed points.
General versions of these localization results can be found in Anderson-Fulton \cite{AF}[Sections 5, 7 and 17] for equivariant (co)homology, Brion 
\cite{Br2}[Cor.2, Sec.2.3] for equivariant Chow groups, and Thomason \cite{Th}[Thm.2.1] for equivariant algebraic K-theory.
See in particular \cite{AF}[Section 17.4] and \cite{Br2}[Section 4] for the general theory of {\it equivariant multiplicities} for ``nondegenerate'' torus fixed points.

\subsection{Localization in equivariant K-theory}\label{sec4.1}
Let $X=X_\Sig$ be an $n$-dimensional toric variety with torus $\bT=T_N$ such that the fixed-point set $X^\bT \neq \emptyset$, e.g., $X$ is projective.

By \cite{BrV2}[Prop.1.5], the inclusion $i: X^\bT \hookrightarrow X$ induces an injective morphism of $K^0_\bT(pt)$-modules
$$i_*:K_0^\bT(X^\bT)  \hookrightarrow K_0^\bT(X)$$
which becomes an isomorphism $$i_*:K_0^\bT(X^\bT)_S  \simeq K_0^\bT(X)_S$$
upon localization at the multiplicative subset $S\subset \bZ[M]=K_0^\bT(pt)$ generated by the elements $1-\chi^m$, for $0\neq m \in M$. Note that $$K_0^\bT(X^\bT) = \bigoplus_{x \in X^\bT} K_0^\bT(x),$$
from which one gets via the isomorphism $i_*$ a projection map of $K_0^\bT(pt)$-modules, called here the {\it $K$-theoretic localization map at $x$}, $$pr_x:K_0^\bT(X)_S \simeq  \bigoplus_{x \in X^\bT} K_0^\bT(x)_S \lra K_0^\bT(x)_S=\bZ[M]_S.$$
Let $x=x_\sig \subset U_\sig$ be the $\bT$-fixed point corresponding to a cone $\sig\in \Sig(n)$ of maximal dimension, as in \eqref{loc}, where $U_\sig \subset X$ is the unique $\bT$-invariant open affine subset containing $x_\sig=O_\sig$. Then the localization map at $x$ factorizes via restriction over $U_\sig$ as:
$$pr_x:K_0^\bT(X)_S \lra K_0^\bT(U_\sig)_S \simeq  K_0^\bT(x_\sig)_S=\bZ[M]_S.$$

We now explain a different description of the isomorphism (denoted also by $pr_x$) $$pr_x : K_0^\bT(U_\sig)_S\simeq K_0^\bT(x_\sig)_S=\bZ[M]_S$$ by using the eigenspace decomposition \eqref{eigen}. As in \cite{BrV2}[Sect.1.3] (or \cite{CLS}[Def.13.2.2]), a formal power series $f \in \bZ[[M]]$ is called {\it summable} if there is $g \in \bZ[M]$ and a finite subset $I \subset M \setminus \{0\}$ such that in $\bZ[[M]]$ one has: $f \cdot \prod_{m \in I}(1-\chi^m)=g$. Let $$\mathbb{S}(f):=g \cdot  \prod_{m \in I}(1-\chi^m)^{-1} \in \bZ[M]_S$$
be the {\it sum} of $f$, which is easily seen to be independent of the factorization. 
Let $$\bZ[[M]]_{\rm Sum} \subset \bZ[[M]]$$ be the subset of summable elements in $\bZ[[M]]$. This is a $\bZ[M]$-submodule of $\bZ[[M]]$, and the summation map $\mathbb{S}$ induces a homomorphism of $\bZ[M]$-modules $$\mathbb{S}: \bZ[[M]]_{\rm Sum} \to \bZ[M]_S.$$
Let $\cF$ be a $\bT$-equivariant coherent sheaf on $U_\sig$. Then $W:=H^0(U_\sig;\cF)$ has an eigenspace decomposition $W=\bigoplus_{m\in M} W_{\chi^m}$ (see \cite{BrV2}[Sect.1.3]) with eigenspaces $W_{\chi^m}$ of finite dimension as in \eqref{eigen} (although $W=H^0(U_\sig;\cF)$ could be infinite dimensional). Then 
\be\label{si} \chi_\sig^\bT(\cF):=\sum_{m \in M} \dim_\bC W_{\chi^m} \cdot \chi^{m} \in \bZ[[M]]
\ee
is summable (cf. \cite{BrV2}[Prop.1.3]). This then induces a homomorphism of $\bZ[M]$-modules (see \cite{BrV2}[Rem.1.3])
\be\label{bunn}
\chi_\sig^\bT : K_0^\bT(U_\sig) \lra \bb{Z}[M]_{\rm sum}
\ee
such that the composition $\mathbb{S} \circ \chi_\sig^\bT$ induces after localization the map $pr_x$, since 
\be\label{locf} \chi_\sig^\bT(i_*([\bC_{\chi^m}])=\chi^m\in \bZ[M]\subset \bb{Z}[M]_{\rm sum}.\ee 

\br\label{sum-global}
Let $\cF$ be a $\bT$-equivariant coherent sheaf on the toric variety $X$.  Then also the global cohomology 
$H^*(X;\cF)$ has an eigenspace decomposition  with all eigenspaces  of finite dimension so that the corresponding equivariant Euler characteristic $ \chi^\bT(\cF)$ is summable (see \cite{BrV2}[Cor.1.3]).
\er
\medskip

%%%%%%%%%%%%%%
For applications to the case of simplicial cones, we next introduce a Lefschetz type variant of the Euler characteristic $\chi_\sig^\bT$, and a corresponding summation map $\mathbb{S}$. Let $\sig \in \Sig(n)$ be a simplicial cone with $u_1,\ldots, u_n \in N=N_\sig$ the generators of the the rays $\rho_j \in \sig(1)$, $j=1,\ldots,n$. Let $N'=N'_\sig$ be the finite index sublattice of $N$ generated by $u_1,\ldots, u_n$, and consider $\sig \in N'_{\bR}=N_\bR$ so that it is smooth with respect to the lattice $N'$. With $\bT$, $\bT'$ the corresponding $n$-dimensional tori of the lattices $N$, resp., $N'$, the inclusion $N' \hookrightarrow N$ induces a toric morphism $\pi: U'_\sig \to U_\sig$ of the associated affine toric varieties. Let $G_\sig$ be the finite kernel of the epimorphism $\pi:\bT'\to \bT$, so that $U'_\sig/G_\sig \simeq U_\sig$ (e.g., see \cite{CLS}[Prop.1.3.18]).

 Let $m_i\in M=M_{\sig}$, $1\leq i\leq n$, be the unique primitive elements in the dual lattice $M$ of $N$ satisfying $ \langle m_i,u_j \rangle = 0$ for $i\neq j$ and
 $q_i:= \langle m_i,u_i \rangle > 0$
 so that the dual lattice $M'=M_{\sig}'$ of $N'$ is generated by the elements $m'_j:=\frac{m_j}{q_j}.$
Let $a_{\rho_j}:G_\sig \to \bC^*$ be the characters of $G_\sig$ as introduced in \eqref{a-intr}.  
  
For $\cF'$ a $\bT'$-equivariant coherent sheaf on $U'_\sig$, the vector space $W':=H^0(U'_\sig;\cF')$ has an eigenspace decomposition $W'=\bigoplus_{m'\in M'} W'_{\chi^{m'}}$ as before. Since $\bT'$ is abelian, its finite subgroup $G_\sig$ acts on $W'$ respecting this eigenspace decomposition. We can then introduce the Lefschetz type Euler characteristic
  \be
  tr^{\bT'}_\sig(\cF'):=\frac{1}{\vert G_\sig \vert} \bigoplus_{m'\in M'} \sum_{g \in G_\sig}
  tr(g: W'_{\chi^{m'}} \to W'_{\chi^{m'}}) \cdot \chi^{m'} \in \bC[[M']]_{\rm sum}.
  \ee
  In this context, the notion of {\it summable} is defined almost as above, but using the multiplicative subset $S' \subset \bC[M']$ generated by elements $1-a\cdot \chi^{m'}$, for $0 \neq m' \in M'$ and $a \in \bC^*$. This induces a homomorphism of $\bZ[M']$-modules
  $$ tr^{\bT'}_\sig: K_0^\bT(U'_\sig) \lra \bb{C}[M']_{\rm sum}.$$
  The fact that $tr^{\bT'}_\sig(\cF')$ is summable will be explained below for sheaves $\cF'=\cO_{U'_\sig}\otimes \bC_{\chi^{m'}}$, with $m' \in M'$, whose classes generate $K_0^\bT(U'_\sig)$ (as in the proof of \cite{BrV2}[Cor.1.2]). There is also a  corresponding summation map
  $$\mathbb{S}': \bC[[M']]_{\rm sum} \to \bC[M']_{S'},$$
  so that the following diagram of $\bZ[M]$-linear maps commutes:
  \begin{equation}\begin{CD}\label{trace}
K_0^{\bT'}(U'_\sig)  @>tr^{\bT'}_\sig>> \bC[[M']]_{\rm sum} @>\mathbb{S}'>> \bC[M']_{S'} \\
@V\pi_*^{G_\sig}VV   @AAA @AAA \\
K_0^\bT(U_\sig)  @>\chi^{\bT}_\sig>> \bZ[[M]]_{\rm sum} @>\mathbb{S}>> \bZ[M]_{S} \ .
\end{CD}
\end{equation}
Here, $\pi_*^{G_\sig}$ is induced by the corresponding exact functor given by taking the $G_\sig$-invariant part of the pushforward for the finite map $\pi$, which is $\bT=\bT'/G_\sig$-equivariant. The monomorphism $\bZ[[M]]_{\rm sum} \to \bC[[M']]_{\rm sum}$  and the algebra map $\bZ[M]_{S} \to \bC[M']_{S'}$ are induced by the lattice injection $M \hookrightarrow M'$.

%%%%%%%%%%%%%%

\bex
 If $\cF=\cO_X\vert_{U_\sig}$, then 
 $$\chi_\sig^\bT(\cO_X\vert_{U_\sig})=\sum_{m \in \sigma^\vee \cap M} \chi^{-m} \in  \bb{Z}[M]_{\rm sum}.$$
If, moreover, $\sig$ is a smooth cone with $m_{\sig,i}$, $i=1,\ldots,n$, the minimal generators of $\sig^\vee$, then
$$
\chi_\sig^\bT(\cO_X\vert_{U_\sig})=\prod_{i=1}^n (\sum_{k\geq 0} 
(\chi^{-m_{\sig,i}})^k ),
$$
hence  (as in \cite{CLS}[Lem.13.2.4])
\be\label{sumsm}
\mathbb{S}(\chi_\sig^\bT(\cO_X\vert_{U_\sig}))=\prod_{i=1}^n \frac{1}{1-\chi^{-m_{\sig,i}}}.
\ee
If $\sig$ is only a simplicial cone, with the same notation for the minimal generators of $\sig^\vee$, we get for  $\cF'=\cO_{U'_\sig}$ that
$$
tr_\sig^{\bT'}(\cO_{U'_\sig})=\frac{1}{\vert G_\sig \vert} \sum_{g \in G_\sig} \prod_{i=1}^n \left(\sum_{k\geq 0} 
(a_{\rho_i}(g^{-1}) \cdot \chi^{-m'_{\sig,i}})^k \right),
$$
hence this series is summable, and 
by applying diagram \eqref{trace} to $\cF'=\cO_{U'_\sig}$, with $\cO_{U_\sig}\simeq \pi_*^{G_\sig}(\cO_{U'_\sig})$, we recover {\it Molien's formula} (see \cite{BrV2}[page 24]):
\be\label{sumsmu}
\mathbb{S}(\chi_\sig^\bT(\cO_X\vert_{U_\sig}))=\frac{1}{\vert G_\sig \vert} \sum_{g \in G_\sig}  \prod_{i=1}^n \frac{1}{1-a_{\rho_i}(g^{-1}) \cdot \chi^{-m'_{\sig,i}}} \in \bC[M']_{S'}.
\ee
Recall here that, by our convention (following \cite{CLS}), $\bT$ acts on $\chi^m$ by $t \cdot \chi^m = \chi^m(t^{-1}) \chi^m$. Note that \cite{BrV2} uses Oda's convention \cite{O}[pag.6]: $t \cdot \chi^m = \chi^m(t) \chi^m$, which explains the sign difference. \qed
\eex

\bex
Consider the case of a simplicial cone  $\sig \in \Sig(n)$ with minimal generators of $\sig^\vee$ as before. We get for   $\cF'=\cO_{U'_\sig} \otimes \bC_{\chi^{m'}}$ (with $m' \in M'$) that
$$
tr_\sig^{\bT'}(\cO_{U'_\sig} \otimes \bC_{\chi^{m'}})=\frac{1}{\vert G_\sig \vert} \sum_{g \in G_\sig} \chi^{m'}(g) \cdot \chi^{m'} \prod_{i=1}^n (\sum_{k\geq 0} 
\left(a_{\rho_i}(g^{-1}) \cdot \chi^{-m'_{\sig,i}})^k \right),
$$
hence this series is summable, and 
by applying diagram \eqref{trace} to $\cF'=\cO_{U'_\sig}\otimes \bC_{\chi^{m'}}$, with $\cF:=\pi_*^{G_\sig}(\cO_{U'_\sig}\otimes \bC_{\chi^{m'}})$, we get
\be\label{sumsmus}
\mathbb{S}(\chi_\sig^\bT(\cF))=\frac{1}{\vert G_\sig \vert} \sum_{g \in G_\sig}   \chi^{m'}(g) \cdot \chi^{m'} \prod_{i=1}^n \frac{1}{1-a_{\rho_i}(g^{-1}) \cdot \chi^{-m'_{\sig,i}}} \in \bC[M']_{S'}.
\ee
\qed
\eex

\bex
Back to the general case, for $\sig \in \Sig(n)$, consider $\cF=\cO_X(D)\vert_{U_\sig}$, with $D$ a $\bT$-invariant Cartier divisor on $X$. Then one gets
$$\chi_\sig^\bT(\cO_X(D)\vert_{U_\sig})=\chi^{-m_\sig} \sum_{m \in \sigma^\vee \cap M} \chi^{-m} \in  \bb{Z}[M]_{\rm sum},$$ where $m_\sig \in M$ is uniquely defined so that $D\vert_{U_\sig}=div(\chi^{-m_\sig})\vert_{U_\sig}$. The sequence $\{m_\sig \in M \mid \sig \in \Sig(n)\}$ is the Cartier data of $D$, in the sense of \cite{CLS}[Thm.4.2.8]. In other words, $\cO_X(D)\vert_{U_\sig} \simeq U_\sig \times \bC_{\chi^{-m_\sig}}$ as $\bT$-equivariant line bundles (e.g., see \cite{CLS}[page 609]), so that for any $\bT$-equivariant coherent sheaf on $X$, we have 
\be\label{lun}
\chi_\sig^\bT\left((\cF \otimes \cO_X(D))\vert_{U_\sig}\right)=\chi^{-m_\sig} \cdot \chi_\sig^\bT(\cF \vert_{U_\sig}).
\ee
\qed
\eex

\bex\label{extw}
Generalizing the case of $\cO_X$, by \cite{CLS}[Prop.8.2.18], one has for any $0 \leq p \leq n=\dim(X)$:
$$\chi_\sig^\bT(\widehat{\Omega}^p_X\vert_{U_\sig})=
\sum_{\tau \preceq \sigma} {\dim(O_\tau) \choose p} 
\sum_{m \in {\rm Relint}(\sigma^\vee \cap \tau^\perp) \cap M} \chi^{-m} \in  \bb{Z}[M]_{\rm sum}.$$
Hence,
\be\label{cou}\chi_\sig^\bT(\omega_X\vert_{U_\sig})=
\sum_{m \in {\rm Relint}(\sigma^\vee) \cap M} \chi^{-m} \in  \bb{Z}[M]_{\rm sum}.\ee
In particular, if $X$ is quasi-projective, then
$$\chi_\sig^\bT(mC^\bT_y(X)\vert_{U_\sig})=
\sum_{\tau \preceq \sigma} (1+y)^{\dim(O_\tau)} 
\sum_{m \in {\rm Relint}(\sigma^\vee \cap \tau^\perp) \cap M} \chi^{-m} \in  \bb{Z}[M]_{\rm sum} \otimes_\bZ \bZ[y].$$
If, moreover, $\sig$ is a smooth cone with $m_{\sig,i}$, $i=1,\ldots,n$, the minimal generators of $\sig^\vee$, then
$$
\chi_\sig^\bT\left(mC^\bT_y(X)\vert_{U_\sig}\right)=\prod_{i=1}^n \left(1+(1+y) \cdot \sum_{k\geq 1} 
(\chi^{-m_{\sig,i}})^k \right),
$$
hence  
\be\label{sumsmmc}
\mathbb{S}\left(\chi_\sig^\bT(mC^\bT_y(X)\vert_{U_\sig})\right)=\prod_{i=1}^n \frac{1+y \cdot \chi^{-m_{\sig,i}}}{1-\chi^{-m_{\sig,i}}} \in \bZ[M]_S.
\ee
Consider now the case of a simplicial cone  $\sig$ with minimal generators of $\sig^\vee$ as before, and let $\cF'=\Omega^p_{U'_\sig}$, $p \geq 0$, resp., $[\cF']=mC^{\bT'}_y(U'_{\sig})=\sum_{p=0}^n [\Omega^p_{U'_\sig}]\cdot y^p \in K_0^{\bT'}(U'_\sig)[y]$.  Then 
$$
tr_\sig^{\bT'}\left(mC^{\bT'}_y(U'_\sig)\right)=\frac{1}{\vert G_\sig \vert} \sum_{g \in G_\sig}  \prod_{i=1}^n \left(1+(1+y) \cdot \sum_{k\geq 1} 
(a_{\rho_i}(g^{-1}) \cdot \chi^{-m'_{\sig,i}})^k \right).
$$
By applying diagram \eqref{trace} to $\cF'=\Omega^p_{U'_\sig}$, with $\cF=\pi_*^{G_\sig}(\Omega^p_{U'_\sig}) \simeq \widehat{\Omega}^p_X\vert_{U_\sig}$, resp., $[\cF']=mC^{\bT'}_y(U'_{\sig})$, with $[\cF]= mC^\bT_y(X)\vert_{U_\sig}$, 
 we get
\be\label{sumsmusd}
\mathbb{S}\left(\chi_\sig^\bT(mC^\bT_y(X)\vert_{U_\sig})\right)=\frac{1}{\vert G_\sig \vert} \sum_{g \in G_\sig}  \prod_{i=1}^n \frac{1+y \cdot a_{\rho_i}(g^{-1}) \cdot \chi^{-m'_{\sig,i}}}{1-a_{\rho_i}(g^{-1}) \cdot \chi^{-m'_{\sig,i}}} \in \bC[M']_{S'}.
\ee

Finally, one can twist the sheaves $\widehat{\Omega}^p_X$ by $\cO_X(D)$, for $D$ a $\bT$-invariant Cartier divisor, and use \eqref{lun} to get the corresponding identities. \qed
\eex

\br \label{Molien}
Let $V=\bC^n$ be a $\bT\times G$-represention with $G$ a finite group and $0$  an ``attractive'' torus fixed point (i.e., the  torus weights $m'_i$ of the coordinate $\bT$-characters  for $i=1,\dots,n$ are contained in some open half space).  For $g\in G$, denote  
by $a_{\rho_i}(g)\in \bC^*$ the values for $g$ of the corresponding characters of the cyclic subgroup $\langle g\rangle \subset G$ generated by $g$. Then the same argument of proof shows for the $\bT$-variety $X:=\bC^n/G$ (with its isolated 
``attractive'' torus fixed point $x:=0\in \bC^n/G$):
$$\mathbb{S}\left(\chi^\bT(mC^\bT_y(X)\right)=\frac{1}{\vert G\vert} \sum_{g \in G}  \prod_{i=1}^n \frac{1+y \cdot a_{\rho_i}(g^{-1}) \cdot \chi^{-m'_{i}}}{1-a_{\rho_i}(g^{-1}) \cdot \chi^{-m'_{i}}}\:.$$
In the special case of a $\bT=\bC^*$-action, such a formula is also discussed in \cite{DBW}[Thm.1] for a comparison of the  localized equivariant Hirzebruch class of $X$ with an (extended) Molien series. 
\er

\br \label{Ishida-loc}
If, more generally, we work with $X=X_{\Sig'}$ a closed $\bT$-invariant algebraic subset of 
the toric variety $X_\Sig$ corresponding to a star-closed subset $\Sig' \subset \Sig$, similar localization formulae hold for the sheaves $\wti{\Omega}^p_X$ of Ishida $p$-forms of $X$ (extended by zero to $X_\Sig$). More precisely, using now \cite{A}[Sect.3], one gets for $0 \leq p \leq \dim(X)$:
\be\label{isloc} \chi_\sig^\bT(\wti{\Omega}^p_X\vert_{U_\sig})=
\sum_{\tau \preceq \sigma, \tau \in \Sig'} {\dim(O_\tau) \choose p} 
\sum_{m \in {\rm Relint}(\sigma^\vee \cap \tau^\perp) \cap M} \chi^{-m} \in  \bb{Z}[M]_{\rm sum}.\ee
In particular, if $X=V_\tau\subset X_\Sig$ is the orbit closure for $\tau \preceq \sig$ of dimension $d$, then $\wti{\Omega}^d_X=\omega_{V_\tau}$, so we get by \eqref{11110} that
\be\label{eq111}
\chi_\sig^\bT\left(mC_0([O_\tau \hookrightarrow X_\Sig])\vert_{U_\sig}\right)=\chi_\sig^\bT(\omega_{V_\tau}\vert_{U_\sig})=
\sum_{m \in {\rm Relint}(\sigma^\vee \cap \tau^\perp) \cap M} \chi^{-m} \in  \bb{Z}[M]_{\rm sum}.
\ee
\er

\bex Let $Z=X_{\Sig'}$ be a closed $\bT$-invariant algebraic subset of $X=X_\Sig$ corresponding to a star-closed subset $\Sig' \subset \Sig$, with open complement $V:=X \setminus Z$. Then
\be\label{mclocal}
\begin{split}
\chi_\sig^\bT&\left(mC^\bT_y([V\hookrightarrow X])\vert_{U_\sig}\right)\\ &=
\sum_{\tau \preceq \sigma; \tau \notin \Sig'} (1+y)^{\dim(O_\tau)} 
\sum_{m \in {\rm Relint}(\sigma^\vee \cap \tau^\perp) \cap M} \chi^{-m} \in  \bb{Z}[M]_{\rm sum} \otimes_\bZ \bZ[y].
\end{split}
\ee
We now consider the case when $Z=D_K:=\bigcup_{\rho\in K} D_\rho$, for $K \subset \Sigma(1)$, and $\sig \in \Sig(n)$ is a smooth, resp., simplicial cone.\newline
Let us first assume that $\sig$ is smooth, with $m_{\sig,i}$, $i=1,\ldots,n$, the minimal generators of $\sig^\vee$. Then 
\begin{multline*}
\chi_\sig^\bT\left(mC^\bT_y([V\hookrightarrow X])\vert_{U_\sig}\right)\\
=\prod_{\rho_i \in \sig(1)\cap K} \left((1+y) \cdot \sum_{k\geq 1} 
(\chi^{-m_{\sig,i}})^k \right) \cdot \prod_{\rho_i \in  \sig(1)\setminus K} \left(1+(1+y) \cdot \sum_{k\geq 1} 
(\chi^{-m_{\sig,i}})^k \right),
\end{multline*}
hence  
\be\label{sumsmme}
\mathbb{S}\left(\chi_\sig^\bT(mC^\bT_y([V\hookrightarrow X])\vert_{U_\sig})\right)=
\prod_{\rho_i \in \sig(1)\cap K} 
\frac{(1+y) \cdot \chi^{-m_{\sig,i}}}{1-\chi^{-m_{\sig,i}}} 
\prod_{\rho_i \in  \sig(1)\setminus K} 
 \frac{1+y \cdot \chi^{-m_{\sig,i}}}{1-\chi^{-m_{\sig,i}}} \in \bZ[M]_S.
\ee
Consider now the case of a simplicial cone  $\sig$ with minimal generators of $\sig^\vee$ as before, and let $V'=U'_\sig \setminus D_{K'}$, with 
$D_{K'}=\bigcup_{\rho\in K'} D'_\rho$, for $K':=\sig(1) \cap K$. Let $[\cF']=mC^{\bT'}_y([V' \hookrightarrow U'_{\sig}])$.  Then 
\begin{equation*}\begin{split}
tr_\sig^{\bT'} &\left(mC^{\bT'}_y([V \hookrightarrow U'_\sig])\right)\\
&=\frac{1}{\vert G_\sig \vert} \sum_{g \in G_\sig}  \prod_{\rho_i \in K'} \left((1+y)  \sum_{k\geq 1} 
(a_{\rho_i}(g^{-1})  \chi^{-m'_{\sig,i}})^k \right) \\ & \cdot \prod_{\rho_i \in  \sig(1)\setminus K'} \left(1+(1+y)  \sum_{k\geq 1} 
(a_{\rho_i}(g^{-1})  \chi^{-m'_{\sig,i}})^k \right).
\end{split}
\end{equation*}
By applying \eqref{trace} to $[\cF']=mC^{\bT'}_y([V' \hookrightarrow U'_{\sig}])$, with $[\cF]=\pi_*^{G_\sig}([\cF'])= mC^\bT_y([V \hookrightarrow X])\vert_{U_\sig}$, 
 we get
 \begin{multline}\label{sumsmusf}
\mathbb{S}\left(\chi_\sig^\bT(mC^\bT_y([V \hookrightarrow X])\vert_{U_\sig}\right)=\\
\frac{1}{\vert G_\sig \vert} \sum_{g \in G_\sig}  \prod_{\rho_i\in \sig(1) \cap K} \frac{(1+y) \cdot a_{\rho_i}(g^{-1}) \cdot \chi^{-m'_{\sig,i}}}{1-a_{\rho_i}(g^{-1}) \cdot \chi^{-m'_{\sig,i}}} \cdot \prod_{\rho_i\in \sig(1) \setminus K} \frac{1+y \cdot a_{\rho_i}(g^{-1}) \cdot \chi^{-m'_{\sig,i}}}{1-a_{\rho_i}(g^{-1}) \cdot \chi^{-m'_{\sig,i}}}.
\end{multline}
\qed
\eex

We now explain the following equality, originally due to Brion \cite{Br1}, where we follow the approach of \cite{BrV2}[Cor.1.3] (see also \cite{CLS}[Thm.13.2.8] for a special case):
\bt\label{thloc}
Let $X=X_\Sig$ be a complete toric variety of complex dimension $n$, with $\cF$ a $\bT$-equivariant coherent sheaf on $X$. Then the $K$-theoretic Euler characteristic of $\cF$ can be calculated via localization at the $\bT$-fixed points as:
\be
\chi^\bT(X,\cF)=\sum_{\sig \in \Sigma(n)} \mathbb{S}\left(\chi_\sig^\bT(\cF\vert_{U_\sig})\right) \in \bb{Z}[M]_S.
\ee
\et

\begin{proof} Let $x_\sig=O_\sig \subset U_\sig$ be the torus fixed point corresponding to $\sig \in \Sig(n)$. The assertion follows from the commutativity of the lower  right square of the following diagram:
$$\begin{CD}
\bigoplus_{\sig  \in \Sigma(n)} K^\bT_0(x_\sig) @>i_*>> K_0^\bT(X) @>\chi^\bT>> \bZ[M] \\
@VVV @VVV @VVV \\
\bigoplus_{\sig  \in \Sigma(n)} K^\bT_0(x_\sig)_S @>\sim>> K_0^\bT(X)_S @> \chi^\bT>> \bZ[M]_S \\
@|  @V \wr VV @|\\
\bigoplus_{\sig  \in \Sigma(n)} K^\bT_0(x_\sig)_S @>\sim>> \bigoplus_{\sig  \in \Sigma(n)} K^\bT_0(U_\sig)_S @>\sum_\sig (\mathbb{S} \circ 
\chi_\sig^\bT)>> \bZ[M]_S \:.
\end{CD}$$
The vertical maps are the natural localization maps (upon restriction to the $U_\sig$'s for the middle map). The commutativity of the outer square follows from \eqref{locf}. The lower left square commutes by the functoriality of restriction and localization. This yields the desired commutativity of the lower right square.
\end{proof}

As a consequence, we get the following weighted version of Brion's equality \cite{Br1}[page 655], see also \cite{CLS}[Cor.13.2.10(a)].
\bc\label{wBr}
Let $P$ be a full-dimensional lattice polytope with associated projective toric variety $X=X_P$ and ample Cartier divisor $D=D_P$. For each vertex $v$ of $P$, consider the cone $C_v={\rm Cone}(P \cap M -v)=\sigma_v^\vee$, with faces $E_v={\rm Cone}(E \cap M -v)$ for $v\in E$. 
Then the following identity holds in $\bZ[M]_S \otimes_\bZ \bZ[y]$:
\be\label{f94} \chi^\bT\left(X, mC^\bT_y(X) \otimes \cO_X(D)\right)= \sum_{v \ \text{\rm vertex}} \chi^{-v} \cdot \mathbb{S} \left( \sum_{v \in E \preceq P} (1+y)^{\dim(E)} \cdot \sum_{m \in \Relint(E_v) \cap M} \chi^{-m} \right).
\ee
\ec

\begin{proof}
By the definition of the inner normal fan $\Sig_P$ of $P$, one has for any vertex $v$ of $P$ that $\sigma_v=C_v^\vee$, with $C_v$ defined as above (see \cite{CLS}[page 76]). By dualizing, it follows that $C_v=\sigma_v^\vee$ (cf. \cite{CLS}[Prop.1.2.4]). By \cite{CLS}[Prop.1.2.10], the faces $E_v$ of $C_v$ are exactly of the form $\sigma_v^\vee \cap \tau ^\perp$, with $\tau \preceq \sigma_v$. Note also that the Cartier data for $D=D_P$ over $U_{\sig_v}$ is exactly given by $m_{\sig_v}=v\in M$, see \cite{CLS}[(4.2.8)]. 

The desired identity \eqref{f94} follows now from Theorem \ref{thloc} together with Example \ref{extw} and formula \eqref{lun}.
\end{proof}

\br
A direct calculation of $\chi^\bT\left(X, mC^\bT_y(X) \otimes \cO_X(D)\right)$, without using localization techniques, can be obtained from \cite{CLS}[Lem.9.4.8 and Thm.9.3.1], as follows:
\be\label{f95}
\chi^\bT\left(X, mC^\bT_y(X) \otimes \cO_X(D)\right)=\sum_{E \preceq P} (1+y)^{\dim(E)} \cdot \sum_{m \in \Relint(E) \cap M} \chi^{-m} \in \bZ[M]\otimes_\bZ \bZ[y].
\ee
This is a $K$-theoretic version of Corollary \ref{c38}. For $y=0$, formula \eqref{f95} reduces to
$$\chi^\bT(X, \cO_X(D))=\sum_{m\in P \cap M} \chi^{-m} \in \bZ[M],$$
which can already be deduced from \eqref{am1} and \eqref{am2}.

For $y=0$, formula \eqref{f94} therefore reduces to Brion's equality:
\be
\sum_{m\in P \cap M} \chi^{-m} = \chi^\bT(X, \cO_X(D)) = \sum_{v \ \text{\rm vertex}} \chi^{-v} \cdot \mathbb{S} \left( \sum_{m \in C_v \cap M} \chi^{-m} \right).
\ee
\er

\br
Equation \eqref{f94} also holds for  $P$ be a full-dimensional lattice polyhedron with associated semi-projective toric variety $X=X_P$ and nef  Cartier divisor $D=D_P$ (as in Remark \ref{polyhedra}), if one uses on the left side
$$\mathbb{S} \left( \chi^\bT\left(X, mC^\bT_y(X) \otimes \cO_X(D)\right)\right)$$
as in Remark \ref{sum-global} (using \cite{BrV2}[Cor.1.3]).
\er

%%%%%%%%%%%%%%%%%%

\subsection{Localization in equivariant homology}\label{sec4.2}
Let $X_\Sig$ be an $n$-dimensional toric variety with torus $\bT=T_N$.
Let $X=X_{\Sig'}$ be a $\bT$-invariant closed algebraic subset of $X_\Sig$, defined by a star-closed subset $\Sig' \subset \Sig$,
such that the fixed-point set $X^\bT \neq \emptyset$.

By \cite{BZ}[Lem.8.4, Lem.8.5], the inclusion $i: X^\bT \hookrightarrow X$ induces an injective morphism of $H^{*}_\bT(pt;\bQ)$-modules $$i_*:\widehat{H}_{*}^\bT(X^\bT;\bQ)  \hookrightarrow \widehat{H}_{*}^\bT(X;\bQ)$$
which becomes an isomorphism $$i_*:\widehat{H}_{*}^\bT(X^\bT;\bQ)_L  \simeq \widehat{H}_{*}^\bT(X;\bQ)_L$$
upon localization at the multiplicative subset $L\subset (\Lambda_\bT)_\bQ=H^{*}_\bT(pt;\bQ)$ generated by the elements $\pm c(m)$, for $0\neq m \in M$ (cf. \cite{BZ}[Cor.8.9]).

 In fact, for a $\bT$-orbit $O_\sigma$ with stabilizer torus $\bT_\sigma\subset \bT$ one has
$$O_\sigma \simeq  \bT / \bT_\sigma\simeq \bT\times_{\bT_\sigma}\{pt\}$$
so that 
$$H_*^{\bT, BM}(O_\sig;\bQ) \simeq H_*^{\bT_\sig, BM}(pt;\bQ) \simeq Sym_{\bQ}(M_\sig)$$
is concentrated in even degrees, with $M_\sig:=M/\sig^{\perp} \cap M$ the character lattice of $\bT_\sig$.
So for $\sig$ of positive dimension (i.e., $O_\sig$ is not a $\bT$-fixed point), there is some $0\neq m\in \sig^{\perp} \cap M$
with $\pm c(m)$ acting by zero on $H_*^{\bT, BM}(O_\sig;\bQ)$. Letting $X_i\subset X$ be the closed union of $\bT$-orbits in $X$ of dimension at most $i$, with $X_i\backslash X_{i-1}$ a disjoint union of orbits  $O_\sigma$ of dimension $i$ (and $X_{-1}:=\emptyset$), the stated results follow by induction on $i$ (with $X^\bT=X_0$) from the short exact sequences (coming from \eqref{long-BM}):
$$0 \to H^{\bT, BM}_{2k}(X_{i-1};\bQ) \to H^{\bT, BM}_{2k}(X_i;\bQ) \to H^{\bT, BM}_{2k}(X_i\backslash X_{i-1};\bQ) \to 0 \:.$$

\br In particular, by applying the localization isomorphism to both $X$ and $X_\Sig$, it follows that the homomorphism $\widehat{H}_{*}^\bT(X;\bQ)_L \to \widehat{H}_{*}^\bT(X_\Sig;\bQ)_L$ induced by inclusion is injective. So in the following it suffices to work in the localized homology $\widehat{H}_{*}^\bT(X_\Sig;\bQ)_L$ of the ambient toric variety.
\er

Let us now assume that $X=X_\Sig$. Since $$\widehat{H}_{*}^\bT(X^\bT;\bQ) = \bigoplus_{x \in X^\bT} \widehat{H}_{*}^\bT(x;\bQ),$$
one gets via the isomorphism $i_*$ a projection map of $H^{*}_\bT(pt;\bQ)$-modules, called here  the {\it homological localization map at $x$}, 
$$pr_x:\widehat{H}_{*}^\bT(X;\bQ)_L \simeq  \bigoplus_{x \in X^\bT} \widehat{H}_{*}^\bT(x;\bQ)_L \lra \widehat{H}_{*}^\bT(x;\bQ)_L = L^{-1}(\widehat{\Lambda}_\bT)_\bQ.$$

\br
If $X$ is a toric variety with ${H}_{*}^\bT(X;\bQ)$  a finitely generated free $(\Lambda_\bT)_\bQ$-module (e.g., $X$ is a complete simplicial toric variety, see \cite{CLS}[Prop.12.4.7]), then the  algebraic localization map 
$\widehat{H}_{*}^\bT(X;\bQ) \to \widehat{H}_{*}^\bT(X;\bQ)_L$ is injective.
\er

 Before giving a more classical description of the homological localization map $pr_x$ in the context of simplicial toric varieties, let us formulate the following compatibility of $pr_x:\widehat{H}_{*}^\bT(X;\bQ)_L \to \widehat{H}_{*}^\bT(x;\bQ)_L$ with cap products. As before, $X=X_\Sig$ is a toric variety, with $x \in X^\bT$, and let $i_x:\{x\} \hookrightarrow X$ be the inclusion map. Then, for any $a \in \widehat{H}^{*}_\bT(X;\bQ)_L$ and $b \in \widehat{H}_{*}^\bT(X;\bQ)_L$, we have:
\be\label{proj}
pr_x(a \cap b)=i_x^*(a) \cap pr_x(b).
\ee
This follows from the definition of $pr_x$, since the cap product $\cap$ commutes with pullback for open inclusions, together with the projection formula for a closed inclusion.

\medskip

Assume now that $X=X_\Sig$ is a {\it simplicial} toric variety, so that one has equivariant Poincar\'e duality with $\bQ$-coefficients (recall that in the non-complete case, ${H}_{*}^\bT$ denotes equivariant Borel-Moore homology):
$$\cap [X]_\bT : \widehat{H}^{*}_\bT(X;\bQ)_L  \overset{\sim}{\lra}   \widehat{H}_{*}^\bT(X;\bQ)_L.$$
For this reason, in the following we use  in this context interchangeably homology and cohomology, with a cap product in homology corresponding to multiplication in cohomology.

\bp\label{pr17}
Let $X=X_\Sig$ be a simplicial toric variety of dimension $n$, with $x_\sig \in X^\bT$ a torus fixed point and inclusion map $i_\sig:\{x_\sig\} \hookrightarrow X$. Then 
\be\label{f106}
pr_{x_\sig}=\frac{i_\sig^*}{Eu^\bT_X(x_\sig)}: \widehat{H}^{*}_\bT(X;\bQ)_L  \to \widehat{H}^{*}_\bT(x_\sig;\bQ)_L,
\ee
with the generalized Euler class  of the fixed point $x_\sig$ in $X$ defined by $$0\neq Eu^\bT_X(x_\sig):=i_\sig^* \left( \mult(\sig) \cdot \prod_{\rho \in \sig(1)} [D_\rho]_\bT\right) \in \bQ\cdot L.$$
If $X$ is, moreover, a complete simplicial toric variety, let  
 $\int_X:\widehat{H}^*_\bT(X;\bb{Q})_L \to \widehat{H}^*_\bT(pt;\bb{Q})_L=L^{-1}(\widehat\Lambda_\bT)_\bQ$ be the equivariant Gysin map (or, equivalently, the equivariant pushforward) for the constant map $X \to pt$. Then 
\be\label{f108n}
\int_X = \sum_{\sig \in \Sig(n)} \frac{i_\sig^*}{Eu^\bT_X(x_\sig)}: \widehat{H}^{*}_\bT(X;\bQ)_L  \to \widehat{H}^{*}_\bT(pt;\bQ)_L.
\ee
\ep

\begin{proof}
For $\sig \in \Sig(n)$ a cone of maximal dimension, formula \eqref{23} written in homological terms
becomes:
\[ \left( \mult(\sig) \cdot \prod_{\rho \in \sig(1)} [D_\rho]_\bT \right) \cap [X]_\bT = {(i_\sig)}_*[x_\sig]_\bT.\]
Using \eqref{proj}, we get
\[ i_\sig^* \left( \mult(\sig) \cdot \prod_{\rho \in \sig(1)} [D_\rho]_\bT \right) \cap pr_{x_\sig} ([X]_\bT)=
pr_{x_\sig}\left({(i_\sig)}_*[x_\sig]_\bT\right)=[x_\sig]_\bT,
\]
where the last equality follows from the fact that the composition $pr_{x_\sig} \circ {(i_\sig)}_*$ is by definition just the identity. The factor $ i_\sig^* \left( \mult(\sig) \cdot \prod_{\rho \in \sig(1)} [D_\rho]_\bT \right)$ is by definition the generalized Euler class $0\neq Eu^\bT_X(x_\sig)$
(this is non-zero by the formula above, since $[x_\sig]_\bT\neq 0$), and it is an element in $\bQ \cdot L$ (see \cite{CLS}[Lem..13.3.5]) since any divisor on a simplicial toric variety is $\bQ$-Cartier (\cite{CLS}[Prop.4.2.7]). So formula \eqref{f106} follows now from the projection formula \eqref{proj}, i.e.,  for $a \in \widehat{H}^{*}_\bT(X;\bQ)_L$, we have:
$$pr_{x_\sig} \left(a \cap [X]_\bT\right)=i_\sig^*(a) \cap \frac{[x_\sig]_\bT}{Eu^\bT_X(x_\sig)}.$$

For the second assertion, note that the map
 $$\int_X:\widehat{H}^*_\bT(X;\bb{Q})_L \to \widehat{H}^*_\bT(pt;\bb{Q})_L=L^{-1}(\widehat\Lambda_\bT)_\bQ$$
 corresponds (by the functoriality of pushforward)  under the localization isomorphism  $$ \widehat{H}_{*}^\bT(X;\bQ)_L \simeq \widehat{H}_{*}^\bT(X^\bT;\bQ)_L  =  \bigoplus_{x \in X^\bT} \widehat{H}_{*}^\bT(x;\bQ)_L$$
 to the sum of localization maps $pr_x$, $x \in X^\bT$, i.e.,
 $$\int_X = \sum_{\sig \in \Sig(n)} pr_{x_\sig} : \widehat{H}^{*}_\bT(X;\bQ)_L  \to \widehat{H}^{*}_\bT(pt;\bQ)_L.$$
\end{proof}

\br
If the fixed point $x_\sig$ corresponds to a smooth cone $\sig \in \Sig$, then $\mult(\sig)=1$ and the divisors $D_\rho$ ($\rho \in \sig(1)$) are Cartier divisors. Moreover, the normal bundle 
$N_{x_\sig} U_\sig$ of $x_\sig$ in the smooth affine variety $U_\sig$ is given by
$$N_{x_\sig} U_\sig=\bigoplus_{\rho \in \sig(1)} \cO(D_\rho)\vert_{x_\sig}.$$
Hence $Eu^\bT_X(x_\sig)=c^\bT_n(N_{x_\sig} U_\sig)$ is the classical equivariant Euler class of the fixed point $x_\sig$, given by the top equivariant Chern class of the normal bundle $N_{x_\sig} U_\sig$.
\er

\br The localization formula $pr_{x}=\frac{i_x^*}{Eu^\bT_X(x)}$, with  $Eu^\bT_X(x)=c^\bT_n(N_{x} X)$, holds  in general for a ``nondegenerate'' (and therefore also isolated) $\bT$-fixed point $x$ in a smooth $\bT$-variety $X$ of dimension $n$,
with $i_x: \{x\}\to X$ the inclusion and $N_xX=T_xX$ the corresponding normal or tangent space in $x$ (see, e.g., 
\cite{AF}[Prop.17.4.4 (ii)]). Here ``nondegenerate'' just means  that $c^\bT_n(N_{x} X)\neq 0$, i.e., all $\bT$-weights on $N_xX=T_xX$ are non-zero.
Similarly, this applies in an ``orbifold'' context like in Remark \ref{Molien}, but this time with a $\bT\times G$ representation $V=\bC^n$ such
that $0$ is only a ``nondegenerate'' $\bT$-fixed point with the finite group $G$ acting effectively on $V$, so that the projection map
$\pi: V\to X:=V/G$ is of degree $|G|$, i.e., $\pi_*([V]_\bT)=|G|\cdot [X]_\bT$. Here one should use then for the isolated $\bT$-fixed point
$x:=0\in V/G$ the ``generalized Euler class''   (see, e.g., \cite{AF}[Prop.17.4.4 (vi)])
$$0\neq Eu^\bT_X(x):=|G|\cdot c^\bT_n(V)\in H^{*}_\bT(0;\bQ)^G=H^{*}_\bT(x;\bQ)\:.$$
The corresponding integration formula \eqref{f108n} is often attributed to Atiyah–Bott and Berline–Vergne (see, e.g., \cite{AF}[Notes to Chapter 5]).
\er

 Let us now get back to the context of a general toric variety $X=X_\Sigma$ with $X^\bT\neq \emptyset$.
The equivariant Todd class transformation $\td_*^\bT$ commutes with pushforward for closed inclusions $i_*$ and pullback $j^*$ under open inclusions. So $\td_*^\bT$ is compatible with the two diagrams below:
\be\label{diagr} \begin{CD}
 K^\bT_0(X^\bT) @>i_*>> K_0^\bT(X)  @.  @.  \widehat{H}_{2*}^\bT(X^\bT;\bQ) @>i_*>> \widehat{H}_{2*}^\bT(X;\bQ)  \\
@VVV @VVj^*V                           @>\td_*^\bT>>           @VVV @VVj^*V      \\
 K^\bT_0(x_\sig) @>i_*>> K^\bT_0(U_\sig)  @.  @.  \widehat{H}_{2*}^\bT(x_\sig;\bQ) @>i_*>>  \widehat{H}_{2*}^\bT(U_\sig;\bQ) 
 \end{CD}\ee
 To extend this property to the localized versions, we need to show that the ring homomorphism  $\ch^\bT:\bZ[M] \simeq K_\bT^0(pt) \to (\Lambda^{an}_\bT)_\bQ \subset (\widehat{\Lambda}_\bT)_\bQ$  given by  the equivariant Chern character is compatible with the multiplicative subsets used for localization. This fact follows as in \cite{CLS}[pag.644]; indeed, if $0 \neq m\in M$, then 
 $$\ch^\bT(1-\chi^m)=1-e^{c(m)}=c(m) \frac{1-e^{c(m)}}{c(m)} \in (\Lambda^{an}_\bT)_\bQ \subset (\widehat{\Lambda}_\bT)_\bQ\:,$$
 with $c(m) \in L$ and the {\it convergent} power series $\frac{1-e^{c(m)}}{c(m)} \in  (\Lambda^{an}_\bT)_\bQ \subset (\widehat{\Lambda}_\bT)_\bQ$  being invertible (since its constant coefficient is a unit). So $\ch^\bT$ induces a ring homomorphism of localized rings
 $$\ch^\bT:\bZ[M]_S \simeq K_\bT^0(pt)_S \longrightarrow L^{-1}(\Lambda^{an}_\bT)_\bQ \subset L^{-1}(\widehat{\Lambda}_\bT)_\bQ \:.$$
 In particular, this localized version of the equivariant Chern character $\ch^\bT$ can be applied to the image of the summation map $\mathbb{S}:\bZ[[M]]_{\rm Sum} \to \bZ[M]_S$, where by abuse of notation we denote
 $$(\ch^\bT \circ \mathbb{S})\left(\sum_{m \in M} a_m \cdot \chi^m\right)=: \mathbb{S}\left(\sum_{m \in M} a_m \cdot e^{c(m)}\right),$$
 extending the corresponding formula $\ch^\bT (\sum_{m \in M} a_m  \cdot  \chi^m)=\sum_{m \in M} a_m  \cdot  e^{c(m)}$ for finite sums in $\bZ[M] \subset \bZ[[M]]_{\rm Sum}$.
 
 \begin{rem}
 In the context of diagram \eqref{trace}, we have  also to consider the extension of the Chern character to the complexified $K$-theory
$$\ch^{\bT'}:\bC[M'] \simeq K_{\bT'}^0(pt) \otimes_\bZ \bC \to (\Lambda^{an}_{\bT'})_\bC \subset (\widehat{\Lambda}_{\bT'})_\bC,$$ 
which is compatible with localization 
$$\ch^{\bT'}:\bC[M']_{S'} \simeq K_{\bT'}^0(pt)_{S'}  \otimes_\bZ \bC \longrightarrow {L'}^{-1}(\Lambda^{an}_{\bT'})_\bC \subset {L'}^{-1}(\widehat{\Lambda}_{\bT'})_\bC \:.$$
Here, $S' \subset \bC[M']$ is as before the multiplicative subset  generated by elements $1-a\cdot \chi^{m'}$, for $0 \neq m' \in M'$ and $a \in \bC^*$, and $L' \subset (\Lambda_{\bT'})_\bC=H^{*}_{\bT'}(pt;\bC)$ is the multiplicative set generated by the elements $\pm a \cdot c(m')$, for $0\neq m' \in M'$ and $a \in \bC^*$. The compatibility with localization follows as above, using in addition that 
 $$\ch^{\bT'}(1-a \cdot \chi^{m'})=1-a \cdot e^{c(m')} \in (\Lambda^{an}_{\bT'})_\bC \subset (\widehat{\Lambda}_{\bT'})_\bC\:,$$
 for $a \neq 1$ and $0 \neq m' \in M'$, is an invertible convergent power series since its constant coefficient $1-a$ is a unit
(which is also true for $a\neq 1$ and $0=m'\in M'$).  
These two Chern characters $\ch^{\bT'}$ and $\ch^{\bT}$  fit into a commutative diagram
\begin{equation}\begin{CD}\label{trace2}
 \bC[M']_{S'} @>\ch^{\bT'}>> {L'}^{-1}(\Lambda^{an}_{\bT'})_\bC @>>> {L'}^{-1}(\widehat{\Lambda}_{\bT'})_\bC \\
@AAA @AAA @AAA\\
\bZ[M]_{S} @>\ch^{\bT}>> {L}^{-1}(\Lambda^{an}_{\bT})_\bQ @>>> {L}^{-1}(\widehat{\Lambda}_{\bT})_\bQ \ .
\end{CD}
\end{equation}
The last two vertical arrows are injections (between convergent, resp., formal Laurent series with rational or complex coefficients).
 \end{rem}

 By the module property \eqref{modpe} of the equivariant Todd class transformation and the functoriality of  $\ch^\bT$ under pullbacks, it follows that $\td_*^\bT$ is compatible with the corresponding localized versions of the diagrams appearing in \eqref{diagr}, with the corresponding maps labeled $i_*$ being now isomorphisms, as used in the definitions of the localization maps $pr_x$ in $K$-theory and homology, respectively.
 
 \bp\label{pr47}
 For a $\bT$-equivariant coherent sheaf $\cF$ on $X=X_\Sig$, and $x_\sig \in X^\bT$ a given fixed point of the $\bT$-action, we have:
 \be\label{f98}
\td_*^\bT([\cF])_{x_\sig}:= pr_{x_\sig} (\td_*^\bT([\cF]) = \ch^\bT \left(( \mathbb{S} \circ \chi_\sig^\bT)(\cF)\right)
\in  L^{-1}(\Lambda^{an}_\bT)_\bQ \subset L^{-1}(\widehat{\Lambda}_\bT)_\bQ
 \ee
 and 
 \be\begin{split}
L^{-1} \td_*^\bT([\cF]) &= \bigoplus_{\sig \in \Sigma(n)} \td_*^\bT([\cF])_{x_\sig} \cdot [x_\sig]_{\bT} \\ &= \bigoplus_{\sig \in \Sigma(n)} \ch^\bT \left(( \mathbb{S} \circ \chi_\sig^\bT)(\cF)\right) \cdot [x_\sig]_{\bT} \in \widehat{H}_{*}^\bT(X;\bQ)_L.
\end{split}
 \ee
 \ep
 
 \begin{proof}
 Indeed, using the explicit calculation of the localization map in equivariant $K$-theory (as in the previous subsection), we have:
 $$ pr_{x_\sig} (\td_*^\bT([\cF]) = \td_*^\bT (pr_{x_\sig} ([\cF])) =\ch^\bT \left(( \mathbb{S} \circ \chi_\sig^\bT)(\cF)\right). $$
 The second formula just follows from the isomorphism $i_*:\widehat{H}_{*}^\bT(X^\bT;\bQ)_L  \simeq \widehat{H}_{*}^\bT(X;\bQ)_L$.
 \end{proof}
 
 \bex[Localized equivariant Hirzebruch and Todd classes] Assume now that $X_\Sig$ is a  toric variety. Let, moreover, $X:=X_{\Sig'}$ be the $\bT$-invariant closed algebraic subset defined by a star-closed subset $\Sig' \subset \Sig$, with $x_\sig \in X^\bT$ a given fixed point of the $\bT$-action.  By using \eqref{f98} and  \eqref{isloc}, the following identity holds in $L^{-1}(\Lambda^{an}_\bT)_\bQ \otimes_\bZ \bZ[y] \subset L^{-1}(\widehat{\Lambda}_\bT)_\bQ \otimes_\bZ \bZ[y]$:
 \be\label{eq147}\begin{split}
 T^\bT_{y*}(X)_{x_\sig}&:= \td_*^\bT([mC_y^\bT(X)])_{x_\sig} \\ &=\sum_{\tau \preceq \sigma, \tau \in \Sig'} (1+y)^{\dim(O_\tau)} 
\cdot (\ch^\bT \circ \mathbb{S}) \left( \sum_{m \in {\rm Relint}(\sigma^\vee \cap \tau^\perp) \cap M} \chi^{-m} \right) \\
&=\sum_{\tau \preceq \sigma, \tau \in \Sig'} (1+y)^{\dim(O_\tau)} 
\cdot \mathbb{S} \left( \sum_{m \in {\rm Relint}(\sigma^\vee \cap \tau^\perp) \cap M} e^{s(m)} \right) \:.
\end{split}
 \ee
 We thus have in $\widehat{H}_{*}^\bT(X;\bQ)_L \otimes_\bZ \bZ[y]$ (compare also with \cite{R}[Thm.5.3], \cite{W}[Thm.11.3]):
\be  L^{-1}  T^\bT_{y*}(X)= \bigoplus_{\sig \in \Sigma'(n)} \left(\sum_{\tau \preceq \sigma, \tau \in \Sig'} (1+y)^{\dim(O_\tau)} 
\cdot \mathbb{S} \left(  \sum_{m \in {\rm Relint}(\sigma^\vee \cap \tau^\perp) \cap M} e^{s(m)} \right) \right) \cdot [x_\sig]_\bT  .
\ee
Specializing the above formulae to $y=0$, or applying \eqref{f98} to the structure sheaf $\cF=\cO_X$ (extended by zero to the ambient toric variety, which now can be arbitrary) we get the following result:
\be\label{bz1}
 \td_*^\bT(X)_{x_\sig} 
=\sum_{\tau \preceq \sigma, \tau \in \Sig'}  \mathbb{S} \left( \sum_{m \in {\rm Relint}(\sigma^\vee \cap \tau^\perp) \cap M} e^{s(m)} \right) \in L^{-1}(\Lambda^{an}_\bT)_\bQ \subset L^{-1}(\widehat{\Lambda}_\bT)_\bQ \:,
 \ee
and 
\be\label{bz2}
L^{-1} \td_*^\bT(X) = \bigoplus_{\sig \in \Sigma'(n)} \left(\sum_{\tau \preceq \sigma, \tau \in \Sig'} \mathbb{S} \left(  \sum_{m \in {\rm Relint}(\sigma^\vee \cap \tau^\perp) \cap M} e^{s(m)} \right) \right) \cdot [x_\sig]_\bT \in \widehat{H}_{*}^\bT(X;\bQ)_L. \ee
In the case $\Sigma'=\Sig$, formulae \eqref{bz1} and \eqref{bz2} reduce to 
 a result of Brylinski-Zhang \cite{BZ}[Thm.9.4]:
 \be\label{bzin}
 \td_*^\bT(X_\Sig)_{x_\sig} =
\mathbb{S} \left( \sum_{m \in \sigma^\vee \cap M}  e^{s(m)} \right) \in L^{-1}(\Lambda^{an}_\bT)_\bQ \subset L^{-1}(\widehat{\Lambda}_\bT)_\bQ \:,
 \ee
 and 
  \be
L^{-1} \td_*^\bT(X_\Sig) = \bigoplus_{\sig \in \Sigma(n)} \mathbb{S} \left( \sum_{m \in \sigma^\vee \cap M}  e^{s(m)} \right)  \cdot [x_\sig]_\bT \in \widehat{H}_{*}^\bT(X_\Sig;\bQ)_L. \ee
\qed
\eex

\br We refer to the work of Rychlewicz \cite{R}[Section 7] for the dicussion of interesting {\it positiviy properties}
of the localized Hirzebruch classes $T^\bT_{y*}(X)_{x_\sig}$.
\er

\bex[Localized equivariant Hirzebruch classes for complements of divisors]
Let $X=X_\Sig$ be a toric variety with $V:=X \setminus D_K$, and $D_K=\bigcup_{\rho \in K} D_\rho$ for $K \subset \Sigma(1)$. Let $\sig \in \Sig(n)$ be a simplicial cone with $x_\sig \in X^\bT$ the corresponding $\bT$-fixed point.
Then we get by \eqref{sumsmusf} and Proposition \ref{pr47} the following identity in 
${L'}^{-1}(\Lambda^{an}_{\bT'})_\bC[y] \subset {L'}^{-1}(\widehat{\Lambda}_{\bT'})_\bC[y]$:
 \begin{multline}\label{sumsmusg}
 T^\bT_{y*}\left([V \hookrightarrow X]\right)_{x_\sig}:= \td_*^\bT\left([mC_y^\bT([V \hookrightarrow X])]\right)_{x_\sig} \\
=\frac{1}{\vert G_\sig \vert} \sum_{g \in G_\sig}  \prod_{\rho_i\in \sig(1) \cap K} \frac{(1+y) \cdot a_{\rho_i}(g^{-1}) \cdot e^{-c(m'_{\sig,i})}}{1-a_{\rho_i}(g^{-1}) \cdot e^{-c(m'_{\sig,i})}} \cdot \prod_{\rho_i\in \sig(1) \setminus K} \frac{1+y \cdot a_{\rho_i}(g^{-1}) \cdot e^{-c(m'_{\sig,i})}}{1-a_{\rho_i}(g^{-1}) \cdot e^{-c(m'_{\sig,i})}}.
\end{multline}
In particular, if $\sig$ is a smooth cone, then
 \begin{equation*} 
 \begin{split}
 T^\bT_{y*}\left([V \hookrightarrow X]\right)_{x_\sig} = \prod_{\rho_i\in \sig(1) \cap K} \frac{(1+y) \cdot e^{-c(m_{\sig,i})}}{1- e^{-c(m_{\sig,i})}} \cdot \prod_{\rho_i\in \sig(1) \setminus K} \frac{1+y \cdot e^{-c(m_{\sig,i})}}{1- e^{-c(m_{\sig,i})}},
\end{split}
\end{equation*}
with $m'_{\sig,i}=m_{\sig,i}$ for all $i$ (since $\sig$ is a smooth cone).

By specializing \eqref{sumsmusg} to $y=0$, we get a local version of the Todd type formula \eqref{t0comp}. In particular, for $y=0$ and $\sig$ smooth, we get:
\be\label{toddloc}
 T^\bT_{0*}\left([V \hookrightarrow X]\right)_{x_\sig}
 = \prod_{\rho_i\in \sig(1) \cap K} \frac{e^{-c(m_{\sig,i})}}{1- e^{-c(m_{\sig,i})}} \cdot \prod_{\rho_i\in \sig(1) \setminus K} \frac{1}{1- e^{-c(m_{\sig,i})}}.
\ee

In the case when $K=\emptyset$, we thus obtain a more explicit expression for the equivariant Hirzebruch and Todd classes of $X$ localized at $x_\sig$. \qed
\eex

 As a consequence of Proposition \ref{pr47} and Theorem \ref{thloc}, we also have the following.
 \bc\label{thlocb}
Let $X=X_\Sig$ be a complete toric variety of complex dimension $n$, with $\cF$ a $\bT$-coherent sheaf on $X$. Then the cohomological Euler characteristic of $\cF$ can be calculated via localization at the $\bT$-fixed points as:
\be\label{f104}\begin{split}
\chi^\bT(X,\cF) & =\sum_{\sig \in \Sigma(n)} (\ch^\bT \circ \mathbb{S}) \left(\chi_\sig^\bT(\cF\vert_{U_\sig})\right) \\
&=\sum_{\sig \in \Sigma(n)} \td_*^\bT([\cF])_{x_\sig} \in L^{-1}(\Lambda^{an}_\bT)_\bQ \subset L^{-1}(\widehat{\Lambda}_\bT)_\bQ \: .
\end{split}
\ee
\ec

\medskip

Translating Theorem \ref{thloc} into the (co)homological context, using Proposition \ref{pr47} and, in the simplicial context,
Proposition \ref{pr17}, we get the following consequence.
\bc\label{rigcor}
Let $X=X_\Sig$ be a complete toric variety, $D$ a $\bT$-invariant Cartier divisor, with $\cF$ a $\bT$-equivariant coherent sheaf on $X$. Then the cohomological equivariant Euler characteristic of $\cO_X(D) \otimes \cF$ is computed by
\be\label{eqec}
\chi^\bT\left(X,\cO_X(D) \otimes \cF\right)=\sum_{\sig \in \Sigma(n)} 
 pr_{x_\sig} \left(\td_*^\bT([\cO_X(D) \otimes \cF]\right) 
\in   L^{-1}(\Lambda^{an}_\bT)_\bQ .
\ee
If, moreover, $X$ is a simplicial toric variety, then 
\be\label{eqecc}
\chi^\bT\left(X,\cO_X(D) \otimes \cF\right)=\sum_{\sig \in \Sigma(n)} 
\frac{i_\sig^* (e^{[D]_\bT} \td_*^\bT([\cF]))}{Eu_X^\bT(x_\sig)} 
\in   L^{-1}(\Lambda^{an}_\bT)_\bQ .
\ee
\ec 

%%%%%%%%%%%

\subsection{Equivariant Hirzebruch classes of simplicial toric varieties via localization}\label{reprloc}
In this sub\-section, we explain how to reprove Theorems \ref{eqHirz}, \ref{eqHirzcomp} and Lemma \ref{l329}
by localization techniques, instead of using the equivariant Lefschetz-Riemann-Roch theorem of \cite{EG}[Thm.3.1 and Rem.4.3] (as done in Subsection \ref{calch}).

Let $X:=X_\Sig$ be a simplicial  toric variety associated to a fan $\Sig$ of full-dimensional convex support. Then the equivariant cohomology of $X$ is a free $(\Lambda_\bT)_\bQ$-module (by \eqref{mm}), so that the injectivity part of the sequence \eqref{loc}, resp.,  \eqref{pp}, still holds for the equivariant cohomology (as pointed out in Remark \ref{reminj}). Hence the corresponding map for the completed equivariant cohomology rings is injective, i.e.,
$$\bigoplus_{\sig \in \Sig(n)} i_\sig^* : \widehat{H}^*_\bT(X;\bQ) \hookrightarrow \bigoplus_{\sig \in \Sig(n)} \widehat{H}^*_\bT(x_\sig;\bQ). $$
By localizing at the multiplicative set $L$, and using the exactness of localization, we get an injective map
$$\bigoplus_{\sig \in \Sig(n)} i_\sig^* : \widehat{H}^*_\bT(X;\bQ)_L \hookrightarrow \bigoplus_{\sig \in \Sig(n)} \widehat{H}^*_\bT(x_\sig;\bQ)_L. $$
 So it is enough to check these characteristic class formulae by using  for each fixed point $x_\sig$, $\sig \in \Sig(n)$, the induced restriction  map $pr_{x_\sig}$ 
 \be\label{indpr} \widehat{H}^*_\bT(X;\bQ) \to \widehat{H}^*_\bT(X;\bQ)_L  \overset{pr_{x_\sig}}{\to}  \widehat{H}^*_\bT(x_\sig;\bQ)_L \to \widehat{H}^*_{\bT'}(x_\sig;\bC)_{L'},\ee
 with the middle arrow 
 $pr_{x_\sig}$  as in  \eqref{f106}.
 Even the direct sum $\bigoplus_{\sig \in \Sig(n)} pr_{x_\sig}$ of these induced restriction maps is still injective, since the localization map on the integral domain $\widehat{H}^*_\bT(x_\sig;\bQ) \to  \widehat{H}^*_\bT(x_\sig;\bQ)_L$ is injective, and $pr_{x_\sig}=\frac{i_\sig^*}{Eu^\bT_X(x_\sig)}$
 differs from $i_\sig^*$ by the unit $$Eu^\bT_X(x_\sig)=i_\sig^* \left( \mult(\sig) \cdot \prod_{\rho \in \sig(1)} [D_\rho]_\bT\right)=\vert G_\sig \vert \prod_{\rho \in \sig(1)} i_\sig^*F_\rho \in \widehat{H}^*_\bT(x_\sig;\bQ)_L.$$
Moreover, as mentioned after diagram \eqref{trace2}, no information is lost if we consider complex instead  of rational coefficients.

To reprove formula \eqref{la22} of Lemma \ref{l329}   by localization, we start with the local version discussed in the following example.
\bex\label{ex420}
Let $\sig \in \Sig(n)$ be a simplicial cone  with minimal generators $\{m_{\sig,i} \}_i$ of $\sig^\vee$ as in the context of diagram \eqref{trace}. For $\cF:=\pi_*^{G_\sig}(\cO_{U'_\sig}\otimes \bC_{\chi^{m'}})$ with $m' \in M'$, we get  by Proposition \ref{pr47} and formula \eqref{sumsmus} the identity
\be\label{s99}
\td^\bT_*([\cF])_{x_\sig}=\frac{1}{\vert G_\sig \vert} \sum_{g \in G_\sig}   \chi^{m'}(g)  e^{c(m')} \prod_{i=1}^n \frac{1}{1-a_{\rho_i}(g^{-1}) \cdot e^{-c(m'_{\sig,i})}} \in \widehat{H}^*_{\bT'}(x_\sig;\bC)_{L'}.
\ee
\qed
\eex

Using Example \ref{ex420}, we can now reprove formula \eqref{la22} of Lemma \ref{l329} by localization techniques. Recall  that the global formula 
\be\label{la22b}
\td_*^\bT\left(\pi_*(\cO_W \otimes \bC_{\chi^{\wti{m}}})^G\right)= \sum_{g\in G_{\Sigma}} \chi^{\wti{m}}(g^{-1}) \prod_{\rho\in \Sigma(1)} \frac{F_{\rho} \cdot e^{\langle \wti{m}, e_\rho \rangle \cdot F_\rho}}{1-a_{\rho}(g) e^{-F_{\rho}}} \in \widehat{H}^*_\bT(X;\bQ),
\ee
with $\{e_\rho\}_{\rho \in \Sig(1)}$ the standard basis of the lattice $\wti{N}=\bZ^{\vert \Sig(1) \vert}$ and $\widetilde{m}$ in the dual character lattice $\widetilde{M}$, is formulated in terms of  the Cox construction $\pi:W \to X$. 
So let us now compare, for a fixed $\sig \in \Sig(n)$ with $U_\sig$ the $\bT$-invariant open affine subset of $X$ containing the corresponding fixed point $x_\sig$, the local and global formulae \eqref{s99} and, resp., \eqref{la22b}. 

To simplify the notation, assume $\sig(1)=\{\rho_1,\ldots, \rho_n\}$. We next adapt the arguments of \cite{BrV2}[page 24] to our context. If $g \in G_\Sig \setminus G_\sig$, by \eqref{stab} there exists a $\rho \in \Sigma(1)$ with $a_\rho(g) \neq 1$, so that the restriction of 
$\frac{F_{\rho} \cdot e^{\langle \wti{m}, e_\rho \rangle \cdot F_\rho}}{1-a_{\rho}(g) e^{-F_{\rho}}} $ to $U_\sig$ becomes $0$. Thus, the summation on the right-hand side of \eqref{la22b} reduces after restriction to $U_\sig$ to a summation over ${g \in G_\sig}$. If $g \in G_\sig$ and $\rho \notin \sig(1)$, then $a_\rho(g)=1$, so that the restriction of 
$\frac{F_{\rho} \cdot e^{\langle \wti{m}, e_\rho \rangle \cdot F_\rho}}{1-a_{\rho}(g) e^{-F_{\rho}}} $ to $U_\sig$ becomes $1$. So the product on the right hand side of \eqref{la22b} reduces to a product over $\rho \in \sig(1)$.
As in the proof of formula \eqref{thi}, we have 
$\pi^{-1}(U_\sig)=U'_\sig \times (\bC^*)^{r-n}$, with $\wti{\bT}\simeq \bT' \times (\bC^*)^{r-n}$, $G\simeq G_\sigma \times (\bC^*)^{r-n}$, and $U'_\sig/G_\sig=U_\sig$. Let $\widetilde{m}=m' + m''\in M' \oplus M''$ be the corresponding character decomposition. Note that $\bC_{\chi^{m''}}$ is a trivial $(\bC^*)^{r-n}$-equivariant line bundle on $(\bC^*)^{r-n}$, so we can assume that $m''=0$ is the trivial character. By the projection formula, the restriction of $\pi_*(\cO_W \otimes \bC_{\chi^{\wti{m}}})^G$ to $U_\sig$ is isomorphic to the $\cF$ appearing in \eqref{s99}. Finally, 

\begingroup
\allowdisplaybreaks
\begin{align*}
& pr_{x_\sig}\left( \sum_{g\in G_{\Sigma}} \chi^{\wti{m}}(g^{-1}) \prod_{\rho\in \Sigma(1)} \frac{F_{\rho} \cdot e^{\langle \wti{m}, e_\rho \rangle \cdot F_\rho}}{1-a_{\rho}(g) e^{-F_{\rho}}} \right) \\
&= \frac{1}{Eu_X^\bT(x_\sig)} \cdot {i_\sig^*}  \left( \sum_{g\in G_{\Sigma}} \chi^{\wti{m}}(g^{-1}) \prod_{\rho\in \Sigma(1)} \frac{F_{\rho} \cdot e^{\langle \wti{m}, e_\rho \rangle \cdot F_\rho}}{1-a_{\rho}(g) e^{-F_{\rho}}} \right) \\
&=\frac{1}{Eu_X^\bT(x_\sig)} \cdot {i_\sig^*} \left( \sum_{g\in G_{\sigma}} \chi^{m'}(g^{-1}) \prod_{i=1}^n \frac{F_{\rho_i} \cdot e^{\langle {m'}, e_{\rho_i} \rangle \cdot F_{\rho_i}}}{1-a_{\rho_i}(g) e^{-F_{\rho_i}}} \right) \\
&=\frac{1}{\vert G_\sig \vert} \sum_{g\in G_{\sigma}} \chi^{m'}(g^{-1})  \cdot 
i_\sig^* \left(  \prod_{i=1}^n \frac{e^{\langle {m'}, e_{\rho_i} \rangle \cdot F_{\rho_i}}}{1-a_{\rho_i}(g) e^{-F_{\rho_i}}}  \right) \\
&=\frac{1}{\vert G_\sig \vert} \sum_{g \in G_\sig}   \chi^{m'}(g) e^{c(m')} \prod_{i=1}^n \frac{1}{1-a_{\rho_i}(g^{-1}) \cdot e^{-c(m'_{\sig,i})}}, 
\end{align*}\endgroup
where the last equality uses \eqref{alg} for the torus $\bT'$ to show that $i_\sig^* F_{\rho_i}=c(m'_{\sig,i})$ and $c(m')=\sum_{i=1}^n \langle m',  e_{\rho_i} \rangle i^*_{\sig} F_{\rho_i}$, as well as changing $g$ by $g^{-1}$ in $G_\sig$. Also note that $a_\rho(g)$ for $g \in G_\Sig$ and $\rho \in \Sig(1)$ fits with the corresponding  $a_\rho(g)$ for $g \in G_\sig$ and $\rho \in \sig(1)$, as in \eqref{a-intr}.

Altogether, the local formula \eqref{s99} is obtained from the global formula \eqref{la22b} upon applying the homological localization map $pr_{x_\sig}$ of \eqref{indpr}. \qed

 \br We leave it to the reader to give a similar proof of  formula \eqref{eHirz0comp}
of Theorem \ref{eqHirzcomp} via localization using the corresponding localized classes from  \eqref{sumsmusg}.
For the special case of  formula \eqref{eHirz0} of Theorem \ref{eqHirz} this is also worked out in 
\cite{CMSSEM}[Section 4.2]. \er

%%%%%%%%%%%%%%%%%%%%%%%%%%%%%

\section{Euler-Maclaurin formulae via equivariant Hirzebruch-Riemann-Roch}\label{emf}
An \index{Euler-Maclaurin formula} Euler-Maclaurin formula relates the sum $\sum_{m\in P \cap M} f(m)$ of the values of a suitable function $f$ at the lattice points in a polyhedron $P \subset M_\bR:=M\otimes\bb{R}$ to integrals over suitable polyhedra. Here, we are interested in the case where $P$ is a polytope or a cone. The corresponding results for pointed cones will fit with the localization results developed in the previous section.

We begin this section with a short overview in Subsection \ref{sec5.1} of relations between the equivariant Hirzebruch-Riemann-Roch theorem and Euler-Maclaurin type formulae obtained via polytope dilations, as indicated in works of Khovanskii-Pukhlikov \cite{KP} and Brion-Vergne \cite{BrV2}. 

In Subsection \ref{sec5.2} we extend the approach of Brion-Vergne from \cite{BrV2} by allowing {\it arbitrary} $\bT$-equivariant coherent sheaf coefficients.  The key formulae for relating exponentional integrals to toric geometry are Theorem \ref{them1}
for dilated polytopes and Theorem \ref{them1loc} for dilated tangent cones.
We derive an abstract Euler-Maclaurin type formula (even for exponential functions instead of polynomials), 
based on the equivariant Hirzebruch-Riemann-Roch theorem and the motivic formalism of equivariant characteristic classes developed in the previous sections. See  Theorem \ref{abstrEM} and Corollary \ref{abstrEMc}
for the case of dilated polytopes and Proposition \ref{pr520} for the case of dilated tangent cones.
 To deal with various convergence issues in the context of exponential functions, we introduce and use an analytic version of the $\bT$-equivariant cohomology ring (see equation \eqref{ans}).

The abstract Euler-Maclaurin formula is specialized in Subsection \ref{eEM} to various situations, recovering many known Euler-Maclaurin type formulae, as well as obtaining several new ones in a uniform way dictated by toric geometry. 
Example \ref{ex:classical1}  (resp., Example  \ref{ex:classical2}) study first the classical case of (the interior of) a simple lattice polytope.
Theorem \ref{EMfacesdel} studies the case of a simple lattice polytope with some facets removed, whereas  Theorem \ref{EMintfaceb} (resp., Theorem \ref{EMintface})
considers the case of (the relative interior of) a fixed face of a simple lattice polytope.
The final Example \ref{ex526} deals with the case of a lattice polytope associated to a globally generated $\bT$-invariant Cartier divisor
on a projective simplicial toric variety.

%%%%%%%%%%%%%

\subsection{Brief overview of Euler-Maclaurin formulae} \label{sec5.1}
The theory of valuations for rational polyhedra $P\subset M_\bR$ was applied by 
Khovanskii and Pukhlikov in \cite{KP} to obtain Euler-Maclaurin type formulae for Delzant polytopes 
(corresponding to smooth projective toric varieties) in terms of infinitesimal movements \index{dilated polytope} $P(h)$ of the polytope. More precisely, for a polynomial $f$ one has the identity:
\begin{equation}\label{EM-h}
Todd(\frac{\partial}{\partial h}) \left( \int_{P(h)} f(m) \ dm \right)_{|_{h=0}} = \sum_{m\in P\cap M} f(m) \:.
\end{equation}
Here, if the polytope $P$ is the lattice polytope defined by inequalities of the form
$$\langle m, u_F\rangle +c_F \geq 0,$$
with $u_F$ the facet normals, the polytope $P(h)$ with shifted faces is defined by inequalities 
$$\langle m, u_F\rangle +c_F +h_F \geq 0,$$
with $h=(h_F)_F$ a vector with real entries indexed by the facets $F$ of $P$. Moreover, $Todd(\frac{\partial}{\partial h})$ is the \index{Todd differential operator} differential operator defined by:
$$Todd(\frac{\partial}{\partial h}):=\prod_{F \ {\rm facet}} \frac{\frac{\partial}{\partial h_F}}{1-e^{-\frac{\partial}{\partial h_F}}}.$$

The relation between formula (\ref{EM-h}) and the equivariant Hirzebruch-Riemann-Roch theorem is only indicated in \cite{KP},
but not used (but see, for instance, the proof of \cite{CLS}[Thm.13.5.6]). Moreover, it is clearly pointed out that exponential functions $f(x)=e^{\langle x,z\rangle}$ are needed to work with rational  polyhedral cones. From this, one then  gets the result for polynomials $f$ (compare also with \cite{CLS}[Sect.13.5]).

For a direct extension of (\ref{EM-h}) to simple lattice polytopes, with a corresponding more complicated Todd operator
$Todd(\frac{\partial}{\partial h})$, see \cite{BrV1} as well as the nice survey \cite{KSW}.
The approach to the Euler-Maclaurin formula (\ref{EM-h}) through the equivariant Hirzebruch-Riemann-Roch theorem for projective simplicial toric varieties corresponding to simple lattice polytopes is due to Brion-Vergne, see \cite{BrV2}[Thm.4.5]. 

\subsection{Euler-Maclaurin formulae via polytope dilations}\label{sec5.2}
In the following, we explain a uniform approach to various Euler-Maclaurin formulae for \index{simple lattice polytope} simple lattice polytopes or \index{pointed cone} pointed cones, formulated in terms of dilations of the polytope or cone.
To have a unified language, assume $P$ is a full-dimensional simple lattice polyhedron in $M_\bR$ with associated \index{semi-projective} (semi-projective) toric variety $X=X_P$ and inner normal fan $\Sigma=\Sig_P$. 
Let $P(h)$ be the \index{dilation} dilation of $P$ with respect to the vector $h=(h_\rho)_{\rho \in \Sigma(1)}$ with real entries indexed by the rays of $\Sigma$ (since we use the related toric geometry), or equivalently, by the facets of $P$ (if one would like to formulate everything directly in terms of $P$). So, if $P$ is defined by inequalities of the form
$$\langle m, u_\rho\rangle +c_\rho \geq 0,$$
with $u_\rho$ the ray generators and $c_\rho \in \bZ$, for each $\rho \in \Sig(1)$, then $P(h)$  is defined by inequalities 
$$\langle m, u_\rho \rangle +c_\rho +h_\rho \geq 0,$$
for each $\rho \in \Sig(1)$.
For later use in the context of weighted Euler-Maclaurin formulae, we also define a \index{parametrized dilation} parametrized dilation $P_y(h)$ of $P$ by:
 $$P_y(h):=\{ m \in M \mid \langle m, u_\rho \rangle + (1+y) \cdot c_\rho +h_\rho \geq 0, \ {\rm for \ all } \ \rho \in \Sig(1) \},$$
with $h=(h_\rho)_\rho$ as before a vector with real entries indexed by the rays $\rho$ of $\Sigma$. Note that $P_0(h)=P(h)$, and $P_y(0)=:P_y$ is the dilation of $P$ by the factor $1+y$  (with $y$ now treated as a complex number).

\medskip

Let us first consider the case when $P$ is a full-dimensional simple lattice polytope in $M_\bR$. 
In what follows, we adapt the arguments of Brion-Vergne \cite{BrV2}[Thm.4.5] to our context. 
If $h$ is in a small enough open neighborhood $U$ of the origin in $\bR^r$ (with $r$ the number of facets of $P$), then $P(h)$ is again a simple polytope of the same combinatorial type, and similarly for $P_y(h)$.
Let us fix $h \in U\cap \bQ^r$ and $y \in \bQ$, and choose $k \in \bN$ so that $k\cdot P_y(h)$ is a lattice polytope in $M_\bR$ with ample Cartier divisor $$D_{k\cdot P_y(h)}=k \cdot D_{P_y(h)}=k \left( \sum_{\rho \in \Sig(1)} ((1+y) \cdot c_\rho+h_\rho) \cdot D_\rho \right),$$
where $D_P=\sum_{\rho \in \Sig(1)} c_\rho \cdot D_\rho$  is the ample Cartier divisor of $P$.
On a simplicial toric variety any divisor is $\bQ$-Cartier (\cite{CLS}[Prop.4.2.7]), so that $D_{P_y(h)}$ has an equivariant fundamental class with rational coefficients satisfying $k\cdot [D_{P_y(h)}]_\bT := [k\cdot D_{P_y(h)}]_\bT=c^\bT_1(\cO_X(k \cdot D_{P_y(h)}))$, e.g., see \cite{CLS}[Prop.12.4.13].
 Applying formula \eqref{lpc1} for $k\cdot P(h)$, and using the fact that the associated toric variety does not change under polytope dilation, we get
\be\label{lpc10}
\sum_{m \in k\cdot P_y(h) \cap M} e^{s(m)}=\int_X e^{k\cdot [D_{P_y(h)}]_\bT} \cap \td^\bT_*(X) \in (\Lambda^{an}_\bT)_\bQ
\subset (\widehat{\Lambda}_\bT)_\bQ \:.
\ee
A localized version of formula \eqref{lpc10} can be derived from \eqref{f108n} as:
\be\label{f110}
\sum_{m \in k\cdot P_y(h) \cap M} e^{s(m)}= \sum_{\sig \in \Sig(n)} \frac{i_\sig^* \left( e^{k\cdot [D_{P_y(h)}]_\bT}  \right)}{Eu^\bT_X(x_\sig)} \cdot i_\sig^*(\td^\bT_*(X))  \in  L^{-1}(\Lambda^{an}_\bT)_\bQ \subset L^{-1}(\widehat{\Lambda}_\bT)_\bQ \:.
\ee 
Let us now recall the identification
$s: Sym_{\bb{Q}}(M)  \simeq  H_\bT^*(pt;\bb{Q})=:(\Lambda_\bT)_{\bb{Q}}$, 
and let $z \in N_\bC:=N \otimes_\bZ \bC= \Hom_\bR(M_\bR, \bC)$ be given. By the universal property of $Sym$, $z$ induces  $\bR$-algebra homomorphisms
$$\langle -, z \rangle : Sym_{\bb{R}}(M) \to \bC \:,$$
by which we can view $\langle p, z \rangle$ for $z$ now variable and $p\in  Sym_{\bb{R}}(M)$, resp., $p\in  (Sym_{\bb{R}}(M))^{an}$ fixed,
as a $\bC$-valued polynomial on $N_\bR$, resp., as a {\it convergent  power series function} around zero in  $N_\bR$.
Assume now that $z$ is chosen so  that $\langle Eu^\bT_X(x_\sig), z \rangle \neq 0$ for each $\sig \in \Sig(n)$, i.e., 
 $\langle i_\sig^* F_\rho, z \rangle \neq 0$, for each ray $\rho\in \sigma(1)$ of $\sigma\in\Sig(n)$. Applying the function $\langle -, \frac{1}{k} \cdot z \rangle$ to \eqref{f110}, we get
 \be\label{f115}
\sum_{m \in k\cdot P_y(h) \cap M} e^{\langle m, \frac{1}{k} \cdot z\rangle}= 
\sum_{\sig \in \Sig(n)} \frac{ e^{\langle i_\sig^* [D_{P_y(h)}]_\bT,z\rangle}}{\langle Eu^\bT_X(x_\sig), \frac{1}{k} \cdot z \rangle} \cdot \langle i^*_\sig (\td^\bT_*(X)), \frac{1}{k} \cdot z \rangle. 
\ee 
Note that by the \index{Riemann sum approximation} Riemann sum approximation of an integral, we have:
$$\lim_{k\to \infty} \frac{1}{k^{n}}\sum_{m \in k\cdot P_y(h) \cap M} e^{\langle m, \frac{1}{k} \cdot z\rangle}= \lim_{k\to \infty} \frac{1}{k^{n}}\sum_{m \in P_y(h) \cap \frac{1}{k} M} e^{\langle  m, z\rangle} = \int_{P_y(h)} e^{\langle  m, z\rangle} \ dm,$$
with the Lebesgue measure $dm$ normalized so that the unit cube in $M \subset M_\bR$ has volume $1$ (which explains the use of the factor $\frac{1}{k^n}$). 
We next study the limits on the right-hand side of \eqref{f115}. Note that
$${k^n} \cdot \langle Eu^\bT_X(x_\sig), \frac{1}{k} \cdot z \rangle = \langle Eu^\bT_X(x_\sig), z \rangle$$ since 
$\langle Eu^\bT_X(x_\sig), z \rangle$ is a (non-zero) homogeneous polynomial of degree $n$ in $z$. Finally, we have that
\be\label{to} \lim_{k\to \infty} \langle i^*_\sig (\td^\bT_*(X)), \frac{1}{k} \cdot z \rangle =1.\ee
Indeed, by formula \eqref{Todd0}, we have
$$\langle i^*_\sig (\td^\bT_*(X)), \frac{1}{k} \cdot z \rangle=
\sum_{g\in G_{\Sigma}} \prod_{\rho\in \Sigma(1)} \frac{\langle i^*_\sig  F_{\rho},\frac{1}{k} \cdot z \rangle }{1-a_{\rho}(g) \cdot e^{-\langle i^*_\sig F_{\rho}, \frac{1}{k} \cdot z \rangle}},$$
with $F_\rho=[D_\rho]_\bT$ denoting the equivariant fundamental class of the $\bT$-invariant divisor $D_\rho$ corresponding to the ray $\rho \in \Sig(1)$, and note that for $k$ fixed (or $k=1$) this is a {\it convergent power series function} around zero in $z$.
Finally,  $$\lim_{k \to \infty} \frac{\langle i^*_\sig  F_{\rho},\frac{1}{k} \cdot z \rangle }{1-a_{\rho}(g) \cdot e^{-\langle i^*_\sig F_{\rho}, \frac{1}{k} \cdot z \rangle}}=
\begin{cases}
0, \ \  \text{if} \ a_\rho(g)\neq 1,\\
1, \ \  \text{if} \ a_\rho(g) = 1.
\end{cases}
$$
So only $g=id_G$ contributes a non-zero limit to \eqref{to}, and this contribution is $1$.
Altogether, we get the following result.
\bt\label{them1} In the above notations, we have
\be\label{f113}
\begin{split}
\int_{P_y(h)} e^{\langle  m, z\rangle} \ dm &= \sum_{\sig \in \Sig(n)} \frac{ e^{\langle i_\sig^* [D_{P_y(h)}],z\rangle}}{\langle Eu^\bT_X(x_\sig),  z \rangle} \\
&=
 \sum_{\sig \in \Sig(n)} \frac{ e^{\langle (1+y) \cdot i_\sig^* c^\bT_1(\cO_X(D_{P})),z\rangle}}{\langle Eu^\bT_X(x_\sig),  z \rangle} \cdot  e^{ \sum_\rho h_\rho \langle i_\sig^* F_\rho,z\rangle} \:.
 \end{split}
\ee
\et
\br\label{regular} The left hand side of \eqref{f113} is a {\it continuous} function in $h$  near zero, and for all $z \in N_\bC$, resp., in $y$, whereas the right hand side is an {\it analytic} function in $h$  near zero,
and for $z \in N_\bC$ away from the linear  hyperplanes  $\langle i_\sig^* F_\rho, z \rangle = 0$ for each ray $\rho\in \sigma(1)$ of $\sigma\in \Sig(n)$, resp., in $y$. But then both sides of this equality have to be 
{\it analytic  functions in $h$ near zero and all $z\in N_\bC$}, resp., in $y\in \bR$, with the corresponding Taylor series around zero converging uniformly on small complete neighborhoods of zero in the variables $h$ and $z$ (cf. also \cite{KSW}[page 27]).
\er

For later use, let us now state the following consequence of Theorem \ref{them1} (compare also with \cite{BrV1}[Lem.3.11]).
\bc\label{polin}
Let $f$ be a polynomial function of degree $d$ on $M_\bR$. Then $\int_{P_y(h)} f(m) \ dm$ is a polynomial function in $h$ of degree bounded above by $d+n$.
\ec
\begin{proof}
We follow the idea of proof from \cite{BrV1}[Lem.3.11], adapted to our language. By replacing in \eqref{f113} $z$ by $tz$, with $0\neq t \in \bR$ small (so that also
 $\langle i_\sig^* F_\rho, tz \rangle \neq 0$ for each ray $\rho\in \sigma(1)$ of $\sigma\in \Sig(n)$), and multiplying both sides by $t^n$ , we get an equality of analytic functions in $t$, with each term on the right hand side also analytic even in zero, since $\langle Eu^\bT_X(x_\sig),  z \rangle\neq 0$ is homogeneous in $z$ of order $n$.
Now taking the Taylor expansion at $t=0$ of these two analytic functions, we get for $z$ small and away from the linear hyperplanes $\langle i_\sig^* F_\rho, z \rangle = 0$ for each ray $\rho\in \sigma(1)$ of $\sigma\in \Sig(n)$, that the assertion holds for $f(m)=\langle m, z \rangle^k$, for any given non-negative integer $k$.
This then implies the statement for any polynomial $f(m)$.
\end{proof}

\bc\label{strc}
Let $X=X_P$ be the toric variety associated to a simple full-dimensional lattice polytope $P \subset M_\bR$, with  
$D=D_P$ the corresponding ample divisor. Let $\cF$ be a $\bT$-equivariant 
coherent sheaf on $X$. Then, for a polynomial  function $f$ on $M_\bR$, the expression
\be\label{stre} \sum_{m\in M} \left( \sum_{i=0}^n (-1)^i \cdot \dim_\bC H^i(X;\cO_X((1+y) D) \otimes \cF)_{\chi^{-m}}\right) \cdot f(m) \ee
is a polynomial in $1+y$. Moreover, the value of this polynomial at $0$ (i.e., for $y=-1$) is given by 
\[ \sum_{m\in M} \left( \sum_{i=0}^n (-1)^i \cdot \dim_\bC H^i(X;\cF)_{\chi^{-m}}\right) \cdot f(m). \]
\ec

\begin{proof}
Using the ideas of the above corollary (as in \cite{BrV1}[Prop.4.1]), the assertion follows from the formula
\begin{equation*}\begin{split}
&\sum_{m\in M} \left( \sum_{i=0}^n (-1)^i \cdot \dim_\bC H^i(X;\cO_X((1+y) D) \otimes \cF)_{\chi^{-m}}\right) \cdot e^{\langle m, z \rangle}\\
&=
\sum_{\sig \in \Sig(n)} \frac{ e^{\langle i_\sig^* [D_{P_y}],z\rangle}}{\langle Eu^\bT_X(x_\sig),  z \rangle} \cdot  \langle i_\sig^* (td_*^\bT([\cF])), z 
\rangle \\
&=
 \sum_{\sig \in \Sig(n)} \frac{ e^{\langle (1+y) \cdot i_\sig^* c^\bT_1(\cO_X(D_{P})),z\rangle}}{\langle Eu^\bT_X(x_\sig),  z \rangle} \cdot  \langle i_\sig^* (td_*^\bT([\cF])), z 
\rangle \:,
\end{split}
\end{equation*}
which can be deduced from \eqref{eqecc}, upon pairing with $z \in N_\bC$.
\end{proof}

\br
For $\cF=\cO_X$, the expression \eqref{stre} becomes
\[ \sum_{m\in M\cap P_y}f(m) \]
as in \cite{BrV1}. For the function $f=1$, this further specializes to the classical \index{Ehrhart polynomial} {\it Ehrhart polynomial}.
As another special case, for $[\cF]=mC^\bT_0([O_\tau \hookrightarrow X]) \in K_0^\bT(X)$, \eqref{stre} calculates \[ \sum_{m\in M\cap {\rm Relint}(E_y)}f(m) \]
with $E$ the face of $P$ corresponding to a cone $\tau \in \Sig$.
\er

We next explain (in Theorem \ref{them1loc} below) a result analogous to Theorem \ref{them1} for the dilation  
$$Tan(P,v)(h):=(v+C_v)(h)=v(h)+C_v$$ 
of the \index{tangent cone} {\it tangent cone}  $Tan(P,v)=v+C_v$ of a vertex $v$ of the polytope $P$, with $C_v=Cone(P \cap M-v) \subset M_\bR$ a full-dimensional simple lattice cone with vertex $0$. The arguments from the case of a full-dimensional lattice polytope apply similarly to the tangent cone, once the approximation of the integral $\int_{{C_v}(h)} e^{\langle  m, z\rangle} \ dm$ by the corresponding  Riemann sum is explained. The corresponding equivariant Hirzebruch-Riemann-Roch formula in this setting is \eqref{bzin}, with $\sig^\vee=C_v$ and $v$ corresponding to the torus fixed point $x_\sig$ in the $\bT$-invariant open affine neighborhood $U_\sig$ associated to $\sig \in \Sig(n)$.  In this local context $h=(h_\rho)_\rho$ is now a vector with real entries indexed by the rays $\rho$ of $\sig=C_v^\vee$.

\bp\label{pr54}
With the above notations, 
one has that $\sum_{m \in C_v \cap  M} e^{\langle  m, z\rangle}$ and $\int_{C_v} e^{\langle  m, z\rangle} \ dm$ are convergent (locally uniformly) to meromorphic functions in $L^{-1}(\Lambda^{an}_\bT)_\bC$, for $z\in N_\bC$ satisfying $-z \in Int(\sig)$. Moreover,
$$\lim_{k\to \infty} \frac{1}{k^{n}}\sum_{m \in k\cdot {C_v}(h) \cap M} e^{\langle m, \frac{1}{k} \cdot z\rangle}= \lim_{k\to \infty} \frac{1}{k^{n}}\sum_{m \in {C_v}(h) \cap \frac{1}{k} M} e^{\langle  m, z\rangle} = \int_{{C_v}(h)} e^{\langle  m, z\rangle} \ dm,$$
with $h\in  \bQ^r$ a small enough dilation vector, and $k \in \bN$ so that $k \cdot C_v(h)$ is a simple pointed lattice cone in $M_\bR$.  
\ep
\begin{proof}
Let $\sig\in \Sig(n)$ be a smooth $n$-dimensional cone with generators $u_{\sig,1},\ldots, u_{\sig,n}$ of the rays $\rho_1,\ldots, \rho_n \in \sig(1)$, and let $m_{\sig,1}, \ldots, m_{\sig,n}$ be the dual basis of $M=\bigoplus_{i=1}^n \bN \cdot m_{\sig,i}$. Then
$$\sum_{m \in C_v \cap  M} e^{\langle  m, z\rangle}=\prod_{i=1}^n \left( \sum_{j_i=0}^\infty (e^{\langle  m_{\sig,i}, z\rangle})^{j_i} \right) 
=\prod_{i=1}^n \frac{1}{1-e^{\langle  m_{\sig,i}, z\rangle}}=\frac{\langle i_\sig^*(\td_*^\bT(X)),z \rangle}{\langle  Eu_X^\bT(x_\sig), z \rangle },$$
with all $\langle  m_{\sig, i}, z\rangle < 0$ for $-z \in Int(\sig)$, and the corresponding geometric series locally uniformly convergent in these $z$. The last equality follows from \eqref{toddloc} (for $K=\emptyset$), using the identification $s:Sym_\bQ(M) \simeq (\Lambda_\bT)_\bQ$.
Similarly, $$\int_{C_v} e^{\langle  m, z\rangle} \ dm=\prod_{i=1}^n \frac{-1}{\langle  m_{\sig,i}, z\rangle},$$
for $-z \in Int(\sig)$ and convergence as before, with the Lebesgue measure $dm$ normalized so that the unit cube in $M \subset M_\bR$ has volume $1$ (see also \cite{GP}[eq.(8)] and \cite{KP}[Prop.1]). By the multiplicativity of the equivariant Todd class transformation, we get as in the case of a polytope that 
\begin{equation*}\begin{split}
\lim_{k\to \infty} \frac{1}{k^{n}} \sum_{m \in k\cdot {C_v}(h) \cap M} e^{\langle m, \frac{1}{k} \cdot z\rangle}
& =\lim_{k\to \infty}  \frac{1}{k^{n}} \prod_{i=1}^n \left( \frac{e^{h_i \langle m_{\sig, i}, z \rangle}}{1-e^{\langle  m_{\sig,i}, \frac{1}{k} \cdot z\rangle}} \right) \\
&= \lim_{k\to \infty}  \prod_{i=1}^n \left( \frac{e^{h_i \langle m_{\sig,i}, z \rangle} \cdot \langle  m_{\sig,i}, \frac{1}{k} \cdot z\rangle }{(1-e^{\langle  m_{\sig,i}, \frac{1}{k} \cdot z\rangle} ) \cdot \langle  m_{\sig,i}, z \rangle} \right) \\
&=\prod_{i=1}^n  \frac{- e^{h_i \langle m_{\sig,i}, z \rangle} }{ \langle  m_{\sig,i}, z \rangle} \\
&= e^{\sum_{i=1}^n h_i \langle m_{\sig,i}, z \rangle} \cdot \prod_{i=1}^n \frac{-1}{\langle  m_{\sig,i}, z\rangle} \\
&= e^{\sum_{i=1}^n h_i \langle m_{\sig,i}, z \rangle} \cdot \int_{{C_v}} e^{\langle  m, z\rangle} \ dm \\
 &=\int_{{C_v}(h)} e^{\langle  m, z\rangle} \ dm.
\end{split}
\end{equation*}

For the proof of the case of a simplicial cone by smooth subdivisions, let us remark that the above approximation by Riemann sums for the full cone $C_v$ holds similarly for the cone $C_v^K:=C_v \setminus \bigcup_{i\in K} F_i$ with some facets $F_i$ ($i \in K \subset \Sigma(1)$) removed. More precisely, one gets 
$$\sum_{m \in C^K_v \cap  M} e^{\langle  m, z\rangle}
=\prod_{i\in \sig(1) \cap K} \frac{e^{\langle  m_{\sig,i}, z\rangle}}{1-e^{\langle  m_{\sig,i}, z\rangle}} \cdot 
\prod_{i\in \sig(1) \setminus K}
\frac{1}{1-e^{\langle  m_{\sig,i}, z\rangle}}
=\frac{\langle i_\sig^*(T_{0*}^\bT([V \hookrightarrow X])),z \rangle}{\langle  Eu_X^\bT(x_\sig), z \rangle },$$
with the last equality following from \eqref{toddloc}, for $V=X \setminus \bigcup_{i\in K} D_{\rho_i}$ (with the divisors $D_{\rho_i}$ corresponding to the divisors of the original polytope $P$). The series on the left hand side are locally uniformly convergent to a meromorphic function in $L^{-1}(\Lambda^{an}_\bT)_\bC$, for $z\in N_\bC$ satisfying $-z \in Int(\sig)$. Moreover, a similar proof yields in this case that 
$$\lim_{k\to \infty} \frac{1}{k^{n}} \sum_{m \in k\cdot {C^K_v}(h) \cap M} e^{\langle m, \frac{1}{k} \cdot z\rangle}=\int_{{C_v}(h)} e^{\langle  m, z\rangle} \ dm.$$

In the general case, let $C_v$ be as before a full-dimensional  simple lattice cone with vertex $0$. The assumption that $C_v$ is simple will  be needed in the following arguments only for the property that $C_v(h)$ is then also simple for $h=(h_\rho)$ small enough.
Consider a refinement $C_v=\bigcup_{i=1}^{s} C_i$ of $C_v$ by smooth full-dimensional lattice cones with vertex $0$, see, e.g., \cite{CLS}[Thm.11.1.9] in the dual context of toric resolutions by fan refinements.
Then, considering each newly introduced facet $F$ 
for only one of the cones it belongs to, we get a disjoint union $C_v=\bigsqcup_{i=1}^{s} C^{K_i}_i$ into smooth full-dimensional lattice cones with vertex $0$ and with some facets removed. Let $z \in N_\bC$ be so that $-z \in Int(\sig)$, with $\sig=C_v^\vee$. Then we also have that $-z \in Int(\sig_i)$ with $\sig_i=C_i^\vee$. So we can apply the case of smooth cones to all of the $C_i$. Then
$$\sum_{m \in C_v \cap  M} e^{\langle  m, z\rangle}=\sum_{i=1}^s \left(\sum_{m \in C^{K_i}_i \cap  M} e^{\langle  m, z\rangle} \right)$$ is convergent (locally uniformly) to a meromorphic function in $L^{-1}(\Lambda^{an}_\bT)_\bC$, for $z\in N_\bC$ satisfying $-z \in Int(\sig)$. Moreover,
$$\lim_{k\to \infty} \frac{1}{k^{n}}\sum_{m \in k\cdot {C_v}(h) \cap M} e^{\langle m, \frac{1}{k} \cdot z\rangle}
=\sum_{i=1}^s \int_{{C_i}(h)} e^{\langle  m, z\rangle} \ dm
= \int_{{C_v}(h)} e^{\langle  m, z\rangle} \ dm,$$
with $h\in M_\bQ \simeq \bQ^r$ a small enough dilation vector, and $k \in \bN$ so that $k\cdot C_v(h)$ is a full-dimensional pointed lattice cone in $M_\bR$. 
\end{proof}

\br\label{rem55}
By the above smooth decomposition method, one can get more explicit expressions for the summation appearing in formula \eqref{bzin}. In the context of the tangent cone $Tan(P,v)$ of a simple full-dimensional lattice polytope $P$ at a vertex $v$, we also get the identification
\be\label{rem55a}
\sum_{m\in C_v\cap M} e^{\langle  m, z\rangle}  =\langle \td_*^\bT (X)_{x_\sig}, z \rangle=
\left\langle \mathbb{S}\left(\sum_{m\in \sig^\vee \cap M} e^{s(m)}\right) , z \right\rangle \in L^{-1}(\Lambda^{an}_\bT)_\bQ,
\ee
for $-z\in Int(\sig)$. The left hand side is a convergent series for such $z$, whereas the right hand side is a corresponding meromorphic function on $N_\bC$. Furthermore, still assuming $-z\in Int(\sig)$, we get more generally, for $\tau \preceq \sig$ corresponding to a face $E$ of $P$ containing the vertex $v$, a similar interpretation for the following equality:
\be\label{eq111b}
\begin{split}
\sum_{m \in {\rm Relint}(C_v \cap E) \cap M} e^{\langle  m, z\rangle} &=
\langle T_{0*}^\bT ([O_\tau \hookrightarrow  X])_{x_\sig}, z \rangle \\ &=
\left\langle \mathbb{S}\left( \sum_{{\rm Relint}(\sigma^\vee \cap \tau^\perp) \cap M} e^{s(m)} \right) , z \right\rangle \in L^{-1}(\Lambda^{an}_\bT)_\bQ.
\end{split}
\ee
\er

As a consequence of the above discussion, we can now give a geometric meaning to the localized summands of formula \eqref{f113}.  
\bt\label{them1loc} 
Let $v \in P$ be a vertex of the full-dimensional simple lattice polytope $P \subset M_\bR$, with  tangent cone $Tan(P,v)$, for $C_v=Cone(P \cap M-v) \subset M_\bR$ a full-dimensional simple lattice cone with vertex $0$. Consider its dilation 
$$Tan(P,v)(h)=(v+C_v)(h).$$
 Let $(X,D)$ be the projective toric variety with ample Cartier divisor 
associated to $P$. Denote by $x_\sig \in X$ the torus fixed point corresponding to $v \in P$, with associated cone $\sig=C_v^\vee \in \Sig(n)$. Then for $z \in N_\bC$ satisfying $-z \in Int(\sig)$, we have
\be\label{f113loc}
\begin{split}
\int_{Tan(P,v)(h)} e^{\langle  m, z\rangle} \ dm =
 \frac{ e^{\langle  v ,z\rangle}}{\langle Eu^\bT_X(x_\sig),  z \rangle} \cdot  e^{ \sum_{\rho\in \sig(1)} h_\rho \langle i_\sig^* F_\rho,z\rangle}  \:.
 \end{split}
\ee
\et

\begin{proof}
In light of the above proposition, one only has to note that 
$i_\sig^*c_1^\bT(\cO(D))=-c(v)=s(v) \in H^2_\bT(pt;\bQ)\simeq M_\bQ$.
\end{proof}

\br
The left hand side of the equality  \eqref{f113loc} is an
 analytic  function in $h$ near zero and $-z\in Int(\sig)$, resp., in $y\in \bR$. 
 Moreover, this integral is locally uniformly convergent in $z$. 
 The right hand side of \eqref{f113loc} is an
 analytic  function in $h$ near zero and $y\in \bR$, as well as a meromorphic function in $z$ which is holomorphic outside the hyperplanes $\langle i_\sig^* F_\rho, z \rangle =0$ for each ray $\rho \in \sig(1)$. 
\er

\medskip

Back to the context of a projective simplicial toric variety associated to a full-dimensional simple lattice polytope $P \subset M_\bR$, 
in order to relate the right-hand side of formula \eqref{f113} to the equivariant Hirzebruch-Riemann-Roch formulae we have to introduce the {\it analytic subring} \index{analytic subring}
\be\label{ans} \left( H^*_\bT(X;\bQ) \right)^{an} \subset \widehat{H}^*_\bT(X;\bQ)\ee
defined as the image of the \index{analytic Stanley-Reisner subring} {\it analytic Stanley-Reisner subring} (depending only on the fan $\Sig$)
$$SR^{an}_\bQ(\Sig):=\bQ\{x_\rho \mid  \rho \in \Sig(1) \}/\sim_{SR}  \ \ \subset \ \bQ[[x_\rho \mid  \rho \in \Sig(1)]]/\sim_{SR}=:\widehat{SR_\bQ(\Sig)} $$ under the isomorphism $\widehat{SR_\bQ(\Sig)} \ \simeq \widehat{H}^*_\bT(X;\bQ)$ given on generators by $x_\rho \mapsto F_\rho$, $\rho \in \Sig(1)$. Given an element $p(F_\rho) \in \left( H^*_\bT(X;\bQ) \right)^{an}$, with $p(x_\rho) \in \bQ\{x_\rho \mid  \rho \in \Sig(1) \}$ a convergent power series, one gets for the restriction $i_\sig^*$ to a fixed point $x_\sig$ that
\be\label{flip} p(i_\sig^* F_\rho)=i_\sig^*(p(F_\rho)).\ee
Here, the convergent power series on the left side corresponds to the image of $p(x_\rho)$ under the evaluation homomorphism
$$\bQ\{x_\rho \mid  \rho \in \Sig(1) \}\mapsto  (\Lambda_\bT^{an})_\bQ:\: x_\rho\mapsto i_\sig^* F_\rho \in H^2_\bT(pt;\bQ)\simeq M_\bQ\:.$$

\begin{rem}\label{trans}
Regarding $\sig \in \Sig(n)$ as a fan, we have a corresponding (analytic) Stanley-Reisner ring
$$SR^{an}_\bQ(\sig):=\bQ\{x_\rho \mid  \rho \in \sig(1) \} \ \ \subset \ \bQ[[x_\rho \mid  \rho \in \sig(1)]]=:\widehat{SR_\bQ(\sig)},$$ 
with an isomorphism $\widehat{SR_\bQ(\sig)} \ \simeq \widehat{H}^*_\bT(U_\sig;\bQ)
\overset{i_\sig^*}{\simeq} \widehat{H}^*_\bT(x_\sig;\bQ)$ given on generators by $x_\rho \mapsto i_\sig^* F_\rho$, $\rho \in \sig(1)$. Then the restriction map $\left( H^*_\bT(X;\bQ) \right)^{an} \to \left( H^*_\bT(U_\sig;\bQ) \right)^{an}$ is induced from a corresponding restriction map $SR^{an}_\bQ(\Sig) \to SR^{an}_\bQ(\sig)$, sending $x_\rho \mapsto 0$ if $\rho \notin \sig(1)$. We get the following factorization of the right hand side of \eqref{flip}:
$$i_\sig^*:\left( H^*_\bT(X;\bQ) \right)^{an} \to \left( H^*_\bT(U_\sig;\bQ) \right)^{an} \to \left( H^*_\bT(x_\sig;\bQ) \right)^{an}.$$
\end{rem}

As a consequence of Theorem \ref{them1} and Remark \ref{regular}, we get by differentiation and convergence the following.
\bc\label{convergence1}
Let $p(\frac{\partial}{\partial h}) \in \bQ\{\frac{\partial}{\partial h_\rho} \mid \rho \in \Sig(1) \} \subset \bQ[[\frac{\partial}{\partial h_\rho} \mid \rho \in \Sig(1) ]]$ be an infinite order differential operator with constant rational coefficients, i.e., obtained by substituting $x_\rho \mapsto \frac{\partial}{\partial h_\rho}$ into a convergent (near zero) power series with rational coefficients $p(x) \in \bQ\{x_\rho \mid \rho \in \Sig(1) \} \subset \bQ[[x_\rho \mid \rho \in \Sig(1) ]]$. Then, in the above notations, we get for $z$ small enough and away from the hyperplanes $\langle i_\sig^* F_\rho, z \rangle = 0$ for each ray $\rho\in \sigma(1)$ of $\sigma\in \Sig(n)$ the following formula:
\be\label{f1133}
\begin{split}
p\left(\frac{\partial}{\partial h}\right) \left( \int_{P_y(h)} e^{\langle  m, z\rangle} \ dm \right)_{\vert_{h=0}} &=
 \sum_{\sig \in \Sig(n)} \frac{ e^{\langle (1+y) \cdot i_\sig^* c^\bT_1(\cO_X(D_{P})),z\rangle}}{\langle Eu^\bT_X(x_\sig),  z \rangle} \cdot  \langle p(i_\sig^* F_\rho), z\rangle \\
&= \left\langle 
\sum_{\sig \in \Sig(n)} \frac{i_\sig^* \left( e^{(1+y) \cdot c^\bT_1(\cO_X(D))} p(F_\rho)\right)}{Eu^\bT_X(x_\sig)}  
, z  \right\rangle \\
& =\left\langle 
\int_X   e^{(1+y) \cdot c^\bT_1(\cO_X(D))} p(F_\rho), z  \right\rangle \:.
\end{split}\ee
If $p_k\in \bQ[\frac{\partial}{\partial h_\rho} \mid \rho \in \Sig(1) ]$ is the corresponding truncation of $p$ up to order $k$, then both sides of \eqref{f1133} applied to $p_k$ converge for $k\to \infty$ locally uniformly in these $z$ to \eqref{f1133} applied to $p$.
\ec

\br
By \eqref{f1133}, the operator $p(\frac{\partial}{\partial h}) \left( \int_{P_y(h)} e^{\langle  m, z\rangle} \ dm \right)_{\vert_{h=0}}$ depends only on the equivalence class of $[p(x)] \in SR^{an}_\bQ(\Sig)$ and not on the chosen convergent power series representative.
\er

\br\label{convergence2}
Assume in addition in Corollary \ref{convergence1} that $\langle \int_X   e^{(1+y) \cdot c^\bT_1(\cO_X(D))} p(F_\rho), z  \rangle$
is a convergent power series in $z$ near zero (for $y\in \bR$ a fixed parameter, e.g., $y=0$). Then one gets as an application of Cauchy's integral formula (see also \cite{KSW}[page 27]),
that  both sides of \eqref{f1133} applied to $p_k$ converge for $k\to \infty$ and $z$ small locally  uniformly  to \eqref{f1133} applied to $p$. In particular, this limit commutes with finite order  differentiations with respect to $z$ (and $z$ small enough).
\er

\bc\label{cformal1}
Let $p(\frac{\partial}{\partial h}) \in \bQ[[\frac{\partial}{\partial h_\rho} \mid \rho \in \Sig(1)]]$ be an infinite order differential operator with constant rational coefficients, i.e., obtained by substituting $x_\rho \mapsto \frac{\partial}{\partial h_\rho}$ into a formal power series with rational coefficients $p(x) \in \bQ[[x_\rho \mid \rho \in \Sig(1) ]]$. Then for a polynomial  function $f$ on $M_\bR$, we have the following formula:
\be\label{form1133}
p\left(\frac{\partial}{\partial h}\right) \left( \int_{P_y(h) }f(m) \ dm \right)_{\vert_{h=0}} 
=  f\left(\frac{\partial}{\partial z}\right) \left(
\langle \int_X   e^{(1+y) \cdot c^\bT_1(\cO_X(D))} p(F_\rho), z  \rangle \right)_{\vert_{z=0}} \:,
\ee
where on the right hand side the operator $\left( f(\frac{\partial}{\partial z}) \right)_{\vert_{z=0}}$ acts on a formal power series in $z$.
\ec

\begin{proof}
Let first  $p(\frac{\partial}{\partial h}) \in \bQ[\frac{\partial}{\partial h_\rho} \mid \rho \in \Sig(1)]$ be a finite order differential operator with constant rational coefficients, i.e., obtained by substituting $x_\rho \mapsto \frac{\partial}{\partial h_\rho}$ into a polynomial with rational coefficients 
$p(x) \in \bQ[x_\rho \mid \rho \in \Sig(1) ]$. Then the result follows from applying the operator $\left( f(\frac{\partial}{\partial z}) \right)_{\vert_{z=0}}$
to \eqref{f1133}. Here we first need to assume that $z$ is small enough and away from the hyperplanes $\langle i_\sig^* F_\rho, z \rangle = 0$ for each ray $\rho\in \sigma(1)$ of $\sigma\in \Sig(n)$, since the localization formula is used; however, formula \eqref{f1133} then holds for all $z$ small enough (and $y\in \bR$ fixed) by continuity, since the left hand side of \eqref{f1133} is analytic in $z$ near zero by Remark \ref{regular}.
 Moreover, by Corollary \ref{polin} the left hand side and therefore also the right hand side of \eqref{form1133} only depend on a truncation of $p$ up to   order $n+deg(f)$. Especially the left hand side is well defined and the stated equality holds for a formal power series with rational coefficients $p(x) \in \bQ[[x_\rho \mid \rho \in \Sig(1) ]]$. 
\end{proof}

\br
By \eqref{form1133}, the operator $p(\frac{\partial}{\partial h}) \left( \int_{P_y(h)} f(z) \ dm \right)_{\vert_{h=0}}$ depends only on the equivalence class of $[p(x)] \in \widehat{SR_\bQ(\Sig)}$ and not on the chosen formal power series representative.
\er

As a consequence of Theorem \ref{them1loc} and Remark \ref{trans}, we get the following local version of Corollary \ref{convergence1}:
\bc\label{convergence11}
Let $p(\frac{\partial}{\partial h}) \in \bQ\{\frac{\partial}{\partial h_\rho} \mid \rho \in \Sig(1) \} \subset \bQ[[\frac{\partial}{\partial h_\rho} \mid \rho \in \Sig(1) ]]$ be an infinite order differential operator with constant rational coefficients, i.e., obtained by substituting $x_\rho \mapsto \frac{\partial}{\partial h_\rho}$ into a convergent (near zero) power series with rational coefficients $p(x) \in \bQ\{x_\rho \mid \rho \in \Sig(1) \} \subset \bQ[[x_\rho \mid \rho \in \Sig(1) ]]$. Fix a vertex $v \in P$ with  tangent cone $Tan(P,v)$, and let 
$\sig \in \Sig(n)$ be the corresponding cone.
Then, for $z$ small enough and $-z \in Int(\sig)$, we have the following formula:
\be\label{f1133b}
p\left(\frac{\partial}{\partial h}\right) \left( \int_{Tan(P,v)(h)} e^{\langle  m, z\rangle} \ dm \right)_{\vert_{h=0}} =
  \left\langle 
\frac{ e^{v} \cdot \left( i_\sig^* p(F_\rho)\right)}{Eu^\bT_X(x_\sig)}  
, z  \right\rangle \:.
\ee
If $p_k\in \bQ[\frac{\partial}{\partial h_\rho} \mid \rho \in \Sig(1) ]$ is the corresponding truncation of $p$ up to order $k$, then both sides of \eqref{f1133b} applied to $p_k$ converge for $k\to \infty$ locally uniformly in these $z$ to \eqref{f1133b} applied to $p$.
\ec

We also have the following.

\bp\label{toddan}
Let $X=X_\Sig$ be a complete simplicial toric variety. For any class $[\cF] \in K^\bT_0(X)$, its equivariant Todd class is an element in the analytic cohomology ring of $X$, i.e.,
$$\td_*^\bT([\cF]) \in \left( H^*_\bT(X;\bQ) \right)^{an} \subset \widehat{H}^*_\bT(X;\bQ).$$ 
\ep
\begin{proof}
For $\cF=\cO_X$, this follows from the explicit formula given in \eqref{Todd0}. More generally, this holds for sheaves $\cF=\pi_*(\cO_W\otimes \bC_{\chi^{\wti{m}}})^G$ as in the explicit formula \eqref{la22}. As in the proof of \cite{BrV2}[Cor.1.2], the equivariant Grothendieck group $K^\bT_0(X)$ is generated by classes of sheaves of the form $\pi_*(\cO_W\otimes \bC_{\chi^{\wti{m}}})^G$.
\end{proof}

Altogether, we get the following abstract Euler-Maclaurin formula based on the equivariant Hirzebruch-Riemann-Roch theorem. \index{Euler-Maclaurin formula}
\bt\label{abstrEM}
Let $X=X_P$ be the projective simplicial toric variety associated to a full-dimensional simple lattice polytope $P\subset M_\bR$. Let $\Sig:=\Sig_P$ be the inner normal fan of $P$, and $D:=D_P$ the ample Cartier divisor associated to $P$. Let $[\cF]\in K^\bT_0(X)$ be fixed, and choose a convergent power series $p(x_\rho) \in \bQ\{x_\rho \mid \rho \in \Sig(1) \}$ so that $p(F_\rho)=\td^\bT_*([\cF]) \in \left( H^*_\bT(X;\bQ) \right)^{an}$. Then 

\be\label{f1134}
\begin{split}
p\left(\frac{\partial}{\partial h}\right) \left( \int_{P(h)} e^{\langle  m, z\rangle} \ dm \right)_{\vert_{h=0}} & = \left\langle \chi^\bT(X,\cO_X(D) \otimes \cF) , z\right\rangle \\
&= \sum_{m\in M} \left( \sum_{i=0}^n (-1)^i \cdot \dim_\bC H^i(X;\cO_X(D) \otimes \cF)_{\chi^{-m}}\right) \cdot e^{\langle m, z \rangle}
\:,
\end{split}
\ee
as analytic functions in $z$ with $z$ small enough, and with $\chi^\bT(X,\cO_X(D) \otimes \cF)\in (\Lambda^{an}_\bT)_\bQ$ the cohomological equivariant Euler characteristic of $\cO_X(D) \otimes \cF$.
\et
\begin{proof}
Equation \eqref{f1133} for $y=0$ can now be calculated  as
\begin{eqnarray*}
p\left(\frac{\partial}{\partial h}\right) \left( \int_{P(h)} e^{\langle  m, z\rangle} \ dm \right)_{\vert_{h=0}} 
&=& \left\langle 
\int_X   e^{c^\bT_1(\cO_X(D))} \td^\bT_*([\cF])
, z  \right\rangle \\
&=&  \left\langle \chi^\bT(X,\cO_X(D) \otimes \cF) , z\right\rangle ,
\end{eqnarray*}
where the last equality follows from the equivariant Hirzebruch-Riemann-Roch formula \eqref{eHRR}. The second equality of \eqref{f1134} follows from \eqref{f41} which uses the eigenspace decomposition, with the minus sign of \eqref{f1134} due to the appearance of $c(m)=-s(m)$ in \eqref{f41}. Recall that we work with the identification $s: Sym_{\bb{Q}}(M)  \simeq  H_\bT^*(pt;\bb{Q})=:(\Lambda_\bT)_{\bb{Q}}$.
Finally, in the proof we first need to assume that $z$ is small enough and away from the hyperplanes $\langle i_\sig^* F_\rho, z \rangle = 0$ for each ray $\rho\in \sigma(1)$ of $\sigma\in \Sig(n)$, since the localization formula is used; however, formula \eqref{f1134} then holds for all $z$ small enough, by Remark \ref{convergence2}.
\end{proof}

Evaluating \eqref{f1134} at $z=0$, we get the following generalized \index{volume formula} {\it volume formula}.
\bc
In the notations of the previous theorem, we have
\be\label{fvol}
p\left(\frac{\partial}{\partial h}\right) \left( {vol} \ {P(h)}  \right)_{\vert_{h=0}} = \chi(X,\cO_X(D) \otimes \cF),
\ee
with ${vol} \ {P(h)} =\int_{P(h)} dm$ the volume of $P(h)$, and the Lebesgue measure normalized so that the unit cube in $M \subset M_\bR$ has volume $1$.
\ec

\bex
The classical volume formula \cite{BrV1}[Thm.2.15] corresponds to $\cF=\cO_X$ for $p$ given by the Todd operator $Todd(\frac{\partial}{\partial h})$ of \eqref{todd2} below, with $\chi(X,\cO_X(D))=\vert P \cap M \vert$. \qed
\eex

\bc\label{abstrEMc}
Let $X=X_P$ be the projective simplicial toric variety associated to a simple full-dimensional  lattice polytope $P\subset M_\bR$. Let $\Sig:=\Sig_P$ be the inner normal fan of $P$, and $D:=D_P$ the ample Cartier divisor associated to $P$. Let $[\cF]\in K^\bT_0(X)$ be fixed, and choose a formal power series $p(x_\rho) \in \bQ[[x_\rho \mid \rho \in \Sig(1)]]$ so that $p(F_\rho)=\td^\bT_*([\cF]) \in \widehat{H}^*_\bT(X;\bQ)$. Then for a polynomial  function $f$ on $M_\bR$, we have:

\be\label{f1137}
p\left(\frac{\partial}{\partial h}\right) \left( \int_{P(h)} f(m) \ dm \right)_{\vert_{h=0}} 
= \sum_{m\in M} \left( \sum_{i=0}^n (-1)^i \cdot \dim_\bC H^i(X;\cO_X(D) \otimes \cF)_{\chi^{-m}}\right) \cdot f(m)
\:.
\ee
\ec

\begin{proof}
This follows from Corollary \ref{cformal1}
by applying the operator $\left( f(\frac{\partial}{\partial z}) \right)_{\vert_{z=0}}$ to the last term of formula \eqref{f1134},
seen as a formal power series in $z$.
\end{proof}

Let us finish this subsection with a local counterpart of Theorem \ref{abstrEM}. Using Remark \ref{trans}, Corollary \ref{convergence11} and Proposition \ref{pr47}, we get:
\begin{prop}\label{pr520} Let $X=X_P$ be the projective simplicial toric variety associated to a simple full-dimensional lattice polytope $P\subset M_\bR \simeq \bR^n$. Let $\Sig:=\Sig_P$ be the inner normal fan of $P$, and $D:=D_P$ the ample Cartier divisor associated to $P$. Let $[\cF]\in K^\bT_0(X)$ be fixed, and choose a convergent power series $p(x_\rho) \in \bQ\{x_\rho \mid \rho \in \Sig(1) \}$ so that $p(F_\rho)=\td^\bT_*([\cF]) \in \left( H^*_\bT(X;\bQ) \right)^{an}$.
Fix a vertex $v \in P$ with  tangent cone $Tan(P,v)$, and let 
$\sig \in \Sig(n)$ be the corresponding cone.
Then, for $z$ small enough with $-z \in Int(\sig)$, we have the following formula:
\be\label{f1133bv}
\begin{split}
p\left(\frac{\partial}{\partial h}\right) \left( \int_{Tan(P,v)(h)} e^{\langle  m, z\rangle} \ dm \right)_{\vert_{h=0}} &=
  \langle 
e^{v} \cdot \td_*^\bT([\cF])_{x_\sig} 
, z  \rangle  \\
&=\left\langle e^{ v} \cdot \ch^\bT\left((\mathbb{S} \circ \chi^\bT_{\sig})(\cF)\right),  z  \right\rangle 
\:.
\end{split}
\ee

\end{prop}

%%%%%%%%%%%%%%%%%%%%%%%%%%%%

\subsection{Examples of Euler-Maclaurin formulae}\label{eEM}
In this subsection, we explain how various special cases of Corollary \ref{abstrEMc} and Proposition \ref{pr520} yield old and new Euler-Maclaurin type formulae.

In the global context of Corollary \ref{abstrEMc},  let $X=X_P$ be the projective simplicial toric variety associated to a full-dimensional simple lattice polytope $P\subset M_\bR$. Let $\Sig:=\Sig_P$ be the inner normal fan of $P$, and $D:=D_P$ the ample Cartier divisor associated to $P$. We fix a polynomial $g$ on $M_\bR$, and let $f(m)=g(m)\cdot e^{\langle m, z \rangle}$ be a \index{quasi-polynomial} {\it quasi-polynomial}, with $z \in N_\bC$ small enough. In the context of a tangent cone at a vertex $v$ of $P$ as in Proposition \ref{pr520}, we only use an \index{exponential function} exponential function $f(m)=e^{\langle m, z \rangle}$ with $z \in N_\bC$ small enough and $-z \in Int(\sig)$. In this local case, $p(x)$ is restricted to the variables $F_\rho$, $\rho \in \sig(1)$ (as in Remark \ref{trans}), with $\sig \in \Sig(n)$ the cone corresponding to $v$. In the concrete formulae below, this amounts to using the cone $\sig$ as a fan instead of $\Sig$, and the finite group $G_\sig$ instead of $G_\Sig$.

\bex\label{ex:classical1} The first case we consider is the classical one of the \index{Euler-Maclaurin formula for a simple polytope} {\it Euler-Maclaurin formula for the polytope $P$}. Here we choose $\cF=\cO_X$ and the infinite order differential operator
\begin{equation}\label{todd2}
Todd\left(\frac{\partial}{\partial h}\right):= \sum_{g\in G_{\Sigma}} \prod_{\rho\in \Sigma(1)} 
\frac{ \frac{\partial}{\partial h_{\rho} }}{1-a_{\rho}(g)\cdot e^{-\frac{\partial}{\partial h_{\rho}} } }
\in \bb{Q}\{\frac{\partial}{\partial h_{\rho} }\;|\; \rho \in \Sigma(1)\} \:,
\end{equation}
as dictated by formula \eqref{Todd0} for $\td_*^\bT(X)$. Then the equivariant Euler characteristic formula \eqref{am5} yields the Euler-Maclaurin formula of Brion-Vergne \cite{BrV2}[Thm.4.5], extended here to a quasi-polynomial $f$:
\begin{equation}\label{EMBV}
Todd\left(\frac{\partial}{\partial h}\right) \left( \int_{P(h)} f(m) \ dm \right)_{|_{h=0}} = \sum_{m\in P\cap M} f(m) \:.
\end{equation}
As explain before, the left-hand side of \eqref{EMBV} only depends of the class of $Todd(\frac{\partial}{\partial h})$ in the analytic Stanley-Reisner ring of $X$. For example, another such representative corresponds to the power series fitting with (the equivariant version of) Theorem \ref{24}, expressing the Todd operator in terms of suitable $L$-class versions.   
If $P$ is a Delzant polytope, one recovers formula \eqref{EM-h} of Khovanskii-Pukhlikov, fitting also with the equivariant Todd class of smooth projective toric varieties (e.g., obtained by setting $y=0$ in Corollary \ref{ceH}). 

In the local case of the tangent cone $Tan(P,v)$ of $P$ at a vertex $v$, using \eqref{rem55a} we get (see also \cite{KP}[Thm.1], \cite{KSW}[eqn.(B.1)]):
\be
\begin{split}
Todd\left(\frac{\partial}{\partial h}\right) \left( \int_{Tan(P,v)(h)} e^{\langle  m, z\rangle} \ dm \right)_{\vert_{h=0}} =e^{\langle  v, z\rangle} \cdot \sum_{m\in Tan(P,v)\cap M} e^{\langle  m, z\rangle}  \:.
\end{split}
\ee
\qed
\eex

\bex\label{ex:classical2} 
An \index{Euler-Maclaurin formula for the interior of a simple polytope} Euler-Maclaurin formula for the {\it interior} of a simple lattice polytope $P$ can be obtained similarly by using the {\it dual Todd operator} \index{dual Todd operator}
\be
Todd^\vee\left(\frac{\partial}{\partial h}\right):= \sum_{g\in G_{\Sigma}} \prod_{\rho\in \Sigma(1)} 
\frac{ a_{\rho}(g) \cdot \frac{\partial}{\partial h_{\rho}}\cdot e^{-\frac{\partial}{\partial h_{\rho}} }}{1-a_{\rho}(g)\cdot e^{-\frac{\partial}{\partial h_{\rho}} } }
\in \bb{Q}\left\{\frac{\partial}{\partial h_{\rho} }\;|\; \rho \in \Sigma(1)\right\} \:,
\ee
corresponding to the sheaf $\cF=\omega_X$ and its dual Todd class $\td_*^\bT([\omega_X]_\bT)$ of formula \eqref{dueq}. In particular, by the equivariant Euler characteristic formula \eqref{am6} one gets:
\be\label{tdu}
Todd^\vee\left(\frac{\partial}{\partial h}\right) \left( \int_{P(h)} f(m) \ dm \right)_{|_{h=0}} 
 = \sum_{m\in \Int(P)\cap M} f(m) \:.
\ee

In the local case of the tangent cone $Tan(P,v)$ of $P$ at a vertex $v$, using \eqref{eq111b} we get:
\be
\begin{split}
Todd^\vee\left(\frac{\partial}{\partial h}\right) \left( \int_{Tan(P,v)(h)} e^{\langle  m, z\rangle} \ dm \right)_{\vert_{h=0}} =e^{\langle  v, z\rangle} \cdot \sum_{m\in \Int(Tan(P,v))\cap M} e^{\langle  m, z\rangle}  \:.
\end{split}
\ee

\qed
\eex

More generally, one can prove an \index{Euler-Maclaurin formula for a polytope with some facets removed} {\it Euler-Maclaurin formula for a polytope with some facets removed}, see also \cite{KSW}[Prop.7.2] for the case of Delzant polytopes. Let $P$ be a full-dimensional polytope with $r$ facets $F_1,\ldots,F_r$. For a subset $K\subseteq \{1,\ldots,r\}$, let $P^K$ be the set obtained from $P$ by removing the facets $F_i$, $i\in K$. For example, $P^\emptyset=P$ and $P^{\{1,\ldots,r\}}=\Int(P)$. Let $\Sigma(1)=\{\rho_1,\ldots,\rho_r\}$ be the rays of the inner normal fan of $P$, and denote by $h=(h_1,\ldots,h_r)$ a vector of real numbers indexed by the facets of $P$ (i.e., $h_i:=h_{\rho_i}$). Consider the following operator in $\bQ\{\frac{\partial}{\partial h_i} \mid i=1,\ldots,r\}$:
\be
Todd^K\left(\frac{\partial}{\partial h}\right):= \sum_{g\in G_{\Sigma}} \prod_{i\in K} 
\frac{ a_{\rho_i}(g) \cdot \frac{\partial}{\partial h_i}\cdot e^{-\frac{\partial}{\partial h_i} }}{1-a_{\rho_i}(g)\cdot e^{-\frac{\partial}{\partial h_i} }} \cdot \prod_{i \notin K} 
\frac{ \frac{\partial}{\partial h_i}}{1-a_{\rho_i}(g)\cdot e^{-\frac{\partial}{\partial h_i} }},
\ee
corresponding to formula \eqref{t0comp} for the equivariant Todd class of $[\cF]=mC_0^\bT([U \hookrightarrow X])$, with $U=X\setminus D_K$ the open complement of the divisor $D_K=\bigcup_{i \in K} D_i$ and $D_i:=D_{\rho_i}$. Moreover, by \eqref{difer} and Corollary \ref{c38} for $y=0$, applied to $X$ and $D_K$, one gets a corresponding equivariant Hirzebruch-Riemann-Roch formula for $\td_*^\bT(\cO_X(D) \otimes mC_0^\bT([U \hookrightarrow X]))$, namely
$$\int_X \td_*^\bT\left(\cO_X(D) \otimes mC_0^\bT([U \hookrightarrow X])\right)=\sum_{m \in P^K \cap M} e^{s(m)}.$$ For its local counterpart, we use \eqref{sumsmusg} for $y=0$.
With the above notations, together with \eqref{eq111b}, this then gives the following.
\bt\label{EMfacesdel}
Let $P$ be a full-dimensional simple lattice polytope in $M_\bR$, and let $f$ be a polynomial function  on $M_\bR$. 
Then:
\be\label{EMgen}
Todd^K\left(\frac{\partial}{\partial h}\right)\left( \int_{P(h)} f(m)\ dm \right)_{|_{h=0}} =\sum_{m \in P^K \cap M} f(m).
\ee
In the local case of the tangent cone $Tan(P,v)$ of $P$ at a vertex $v$, we get:
\be
\begin{split}
Todd^K\left(\frac{\partial}{\partial h}\right) \left( \int_{Tan(P,v)(h)} e^{\langle  m, z\rangle} \ dm \right)_{\vert_{h=0}} =e^{\langle  v, z\rangle} \cdot \sum_{m\in Tan(P,v)^K\cap M} e^{\langle  m, z\rangle}  \:.
\end{split}
\ee
\et

\medskip

We next explain an \index{Euler-Maclaurin formula for the interior of a face} {\it Euler-Maclaurin formula for the interior of a face of a simple lattice polytope}. Let $P$ be a full-dimensional simple lattice polytope in $M_\bR$, and fix a face $E$ of $P$. Let $\sigma:=\sigma_E$ be the corresponding cone in the inner normal fan $\Sigma=\Sigma_P$ of $P$, with $V_\sigma=V_{\sigma_E}=X_E$ the closure of the orbit of $\sigma$ in $X=X_P$. Denote by $i_E=i_{\sigma}: V_\sig \hookrightarrow X$ the closed inclusion map. Then $V_\sigma$ is a simplicial toric variety whose fan is $Star(\sigma)$, as defined in \eqref{star}, which is built from cones $\tau \in \Sigma$ that have $\sigma$ as a face.  
 Recall that $\bT=T_N$ acts on $V_\sigma$ via the morphism $T_N \to T_{N(\sigma)}$. Then we get, as in Corollary \ref{c311} and with the above notations, the following formula:
 \be\label{dueqe}
\td_*^\bT([\omega_{V_\sigma}]_\bT)=i_\sig^*\left( \sum_{g\in G_{Star(\sig)}} \prod_{\rho\in  Star(\sigma)(1)} \frac{a_{\rho}(g) \cdot F_{\rho} \cdot e^{-F_{\rho}}}{1-a_{\rho}(g) \cdot e^{-F_{\rho}}} \right),
\ee
with $\rho\in Star(\sigma)(1)$ being a short notation for $\rho \in \bigcup_{\sig \preceq \nu} \nu(1) \setminus \sig(1)$.
Using \eqref{23} and the projection formula, this gives the equivariant Todd class of $[\cF]=[(i_\sig)_*\omega_{V_\sigma}]$:
\be\label{dueqee}
\td_*^\bT([(i_\sig)_*\omega_{V_\sigma}])=  \sum_{g\in G_{Star(\sig)}} 
\mult(\sig) \cdot \prod_{\rho\in\sigma(1)}F_\rho \cdot
\prod_{\rho\in  Star(\sigma)(1)} \frac{a_{\rho}(g) \cdot F_{\rho} \cdot e^{-F_{\rho}}}{1-a_{\rho}(g) \cdot e^{-F_{\rho}}} .
\ee

Using \eqref{lem2}, this yields the following result (where we use the above notations):
\bt\label{EMintface}
Let $P$ be a full-dimensional simple lattice polytope in $M_\bR$, and let $f$ be a polynomial function on $M_\bR$. Then, for a fixed face $E$ of $P$, with corresponding cone $\sigma$ in the inner normal fan $\Sig$ of $P$, one has:
\be\label{EMintd}
Todd^\vee_E\left(\frac{\partial}{\partial h}\right)\left( \int_{P(h)} f(m) \ dm \right)_{|_{h=0}} =\sum_{m \in \Relint(E) \cap M} f(m),
\ee
where
\be
Todd^\vee_E\left(\frac{\partial}{\partial h}\right):= \sum_{g\in G_{Star(\sig)}} \mult(\sig) \cdot \prod_{\rho\in\sigma(1)}\frac{\partial}{\partial h_\rho} \cdot \prod_{\rho\in  Star(\sigma)(1)} \frac{ a_{\rho}(g) \cdot \frac{\partial}{\partial h_{\rho}}\cdot e^{-\frac{\partial}{\partial h_{\rho}} }}{1-a_{\rho}(g)\cdot e^{-\frac{\partial}{\partial h_{\rho}} } }.
\ee
\et

We next discuss a similar {\it Euler-Maclaurin formula for a face of a simple lattice polytope}. With the same notations, we have by \eqref{Todd0} 
the following formula:
 \be\label{dueqet}
\td_*^\bT([\cO_{V_\sigma}]_\bT)=i_\sig^*\left( \sum_{g\in G_{Star(\sig)}} \prod_{\rho\in  Star(\sigma)(1)} \frac{F_{\rho} }{1-a_{\rho}(g)\cdot e^{-F_{\rho}}} \right).
\ee
As before, the projection formula yields the equivariant Todd class of $[\cF]=[(i_\sig)_*\cO_{V_\sigma}]$:
\be\label{dueqeet}
\td_*^\bT([(i_\sig)_*\cO_{V_\sigma}])=  \sum_{g\in G_{Star(\sig)}} 
\mult(\sig) \cdot \prod_{\rho\in\sigma(1)}F_\rho \cdot
\prod_{\rho\in  Star(\sigma)(1)} \frac{F_{\rho}}{1-a_{\rho}(g)\cdot e^{-F_{\rho}}} .
\ee

Using \eqref{lem2b}, this yields the following result:
\bt\label{EMintfaceb}
Let $P$ be a full-dimensional simple lattice polytope in $M_\bR$, and let $f$ be a polynomial function on $M_\bR$. Then, for a fixed face $E$ of $P$, with corresponding cone $\sigma$ in the inner normal fan $\Sig$ of $P$, one has:
\be\label{EMint}
Todd_E\left(\frac{\partial}{\partial h}\right)\left( \int_{P(h)} f(m) \ dm \right)_{|_{h=0}} =\sum_{m \in E \cap M} f(m),
\ee
where
\be
Todd_E\left(\frac{\partial}{\partial h}\right):= \sum_{g\in G_{Star(\sig)}} \mult(\sig) \cdot \prod_{\rho\in\sigma(1)}\frac{\partial}{\partial h_\rho} \cdot \prod_{\rho\in  Star(\sigma)(1)} \frac{  \frac{\partial}{\partial h_{\rho}} }{1-a_{\rho}(g)\cdot e^{-\frac{\partial}{\partial h_{\rho}} } }.
\ee
\et

Other interesting coherent sheaves to consider are the Zariski sheaves $\widehat{\Omega}^p_X$ of $p$-forms on the toric variety $X$. In the next section, these will be considered all at once via the formal sum $$\bigoplus_p [\widehat{\Omega}^p_X]_\bT \cdot y^p \in K^\bT_0(X)[y],$$ and similarly for suitable motivic Chern classes $mC^\bT_y$ of $\bT$-invariant constructible subsets of $X$.

We leave it to the reader to specialize the generalized  volume formula \eqref{fvol} to all situations discussed in this section. 

Another way to obtain examples of explicit Euler-Maclaurin formulae is by twisting the coherent sheaf $\cF$ by $\cO_X(D'-D)$, for $D=D_P$ the original ample divisor associated to the full-dimensional lattice polytope, and $D'$ any $\bT$-invariant Cartier divisor on $X$. By the multiplicativity of the equivariant Todd class transformation for the coherent sheaf $\cF'=\cO_X(D'-D) \otimes \cF$, we have
 $$
 \td^\bT_*([\cF'])=e^{[D'-D]_\bT} \cdot \td_*^\bT([\cF]).
 $$
So, if $p(x_\rho)\in \bQ\{x_\rho \mid \rho \in \Sig(1)\}$ is a convergent  power series with $p(F_\rho)=td_*^\bT([\cF])$, then $p'(x_\rho):=e^{\sum_{\rho \in \Sig(1)} d_\rho x_\rho} \cdot p(x_\rho) \in \bQ\{x_\rho \mid \rho \in \Sig(1)\}$ is a convergent  power series with $p'(F_\rho)=td_*^\bT([\cF'])$, where $D'-D=\sum_{\rho \in \Sig(1)} d_\rho D_\rho$ as a $\bT$-invariant Cartier divisor. 

\bex\label{ex526}
As a last concrete example of this section, assume $D'$ is a \index{globally generated} {\it globally generated} $\bT$-invariant Cartier divisor on $X$, with associated (not necessarily full-dimensional) lattice polytope $P_{D'} \subset M_\bR$.
Consider the infinite order differential operator
\begin{equation}\label{todd222}
Todd'\left(\frac{\partial}{\partial h}\right):= e^{\sum_{\rho \in \Sig(1)} d_\rho \cdot \frac{\partial}{\partial h_{\rho} }} \cdot Todd(\frac{\partial}{\partial h})
\end{equation}
with $Todd(\frac{\partial}{\partial h})$ as in \eqref{todd2} and $d_\rho$'s as above.
Then the equivariant Euler characteristic formula \eqref{am51} yields the following new Euler-Maclaurin formula for a quasi-polynomial $f$.
\begin{equation}\label{EMBV2}
Todd'\left(\frac{\partial}{\partial h}\right) \left( \int_{P(h)} f(m) \ dm \right)_{|_{h=0}} = \sum_{m\in P_{D'}\cap M} f(m) \:.
\end{equation}
Note that $P_{D'}$ is an \index{$\bN$-Minkowski summand} $\bN$-Minkowski summand (in the sense of \cite{CLS}[Def.6.2.11]) of the original polytope $P$, and any such $\bN$-Minkowski summand comes from a globally generated Cartier divisor $D'$, see \cite{CLS}[Cor.6.2.15].
\qed
\eex

%%%%%%%%%%%%%%%%%%%%%%%%%%

\section{Weighted Euler-Maclaurin formulae}\label{sec:6}
In this section we discuss several weighted (or parametrized) Euler-Maclaurin formulae, generalizing the weighted lattice point counting. For this purpose, we explain in Subsection \ref{sec6.1} the use of an additional dilation parameter $y$ fitting with the equivariant Hirzebruch classes and their renormalization. In Subsection \ref{sec:6.1}, we specialize the parametrized version of the Euler-Maclaurin formula to various situations, recovering some known weighted Euler-Maclaurin type formulae from \cite{BGM,GKO}, but also obtaining several new ones in a uniform way dictated by toric geometry. For instance, we get weighted Euler-Maclaurin type formulae for simple polytopes
in Theorem \ref{bgm1} and Corollary \ref{bgm11}, 
and for simple polytopes with some facets deleted (e.g., for the interior of a polytope) in Theorem \ref{bgm1b} and Corollary \ref{bgm11b}.
Similarly, we obtain such weighted formulae for faces of simple polytopes
 in Theorem \ref{EMintfaceh} and Corollary \ref{EMintfaceh1},
as well as for $\bN$-Minkowski summands of simple polytopes corresponding to globally generated $\bT$-invariant Cartier divisors on the corresponding toric variety in the formulae \eqref{wem01b} and \eqref{wem1b}.

\subsection{Abstract weighted Euler-Maclaurin formulae}\label{sec6.1}
Let $X=X_P$ be the projective simplicial toric variety associated to a full-dimensional simple lattice polytope $P\subset M_\bR\simeq \bR^n$. Let $\Sig:=\Sig_P$ be the inner normal fan of $P$, and $D:=D_P$ the ample Cartier divisor associated to $P$.

Let a convergent power series $p(x_\rho) \in \bQ\{x_\rho \mid \rho \in \Sig(1) \}$ be given. Additionally, we may start with a polynomial in $y$ with coefficients consisting of such power series, as we will need in the applications of this section. Define a corresponding {\it renormalized series} \index{renormalized series}
$$p_y(x_\rho):=\frac{p\left((1+y)x_\rho \right)}{(1+y)^n} \in \bQ[y,(1+y)^{-1}][[x_\rho \mid \rho \in \Sig(1)]].$$
To treat it as a convergent power series, one needs to assume that $y\in \bR\setminus \{1\}$ is fixed or it belongs to a complete subset.

Let $p_y(\frac{\partial}{\partial h})$ be the corresponding parametrized infinite order differential operator obtained  from $p_y(x_\rho)$ by substituting $x_\rho \mapsto \frac{\partial}{\partial h_\rho}$, for all $\rho \in \Sig(1)$.
Then formula \eqref{f1133} translates into the following:
\be\label{f1137w}
p_y\left(\frac{\partial}{\partial h}\right) \left( \int_{P_y(h)} e^{\langle  m, z\rangle} \ dm \right)_{\vert_{h=0}} =
 \sum_{\sig \in \Sig(n)} \frac{ e^{\langle (1+y) \cdot i_\sig^* c^\bT_1(\cO_X(D_{P})),z\rangle}}{ (1+y)^n \cdot\langle  Eu^\bT_X(x_\sig),  z \rangle} \cdot  \langle  p((1+y) \cdot i_\sig^* F_\rho), z\rangle \:.
\ee

Let $[\cF]\in K^\bT_0(X)$ be fixed, and choose a convergent power series $p(x_\rho) \in \bQ\{x_\rho \mid \rho \in \Sig(1) \}$ so that $p(F_\rho)=\td^\bT_*([\cF]) \in \left( H^*_\bT(X;\bQ) \right)^{an}$. Applying the proof of Theorem \ref{abstrEM} to formula \eqref{f1137w}, one then gets 
\be\label{f1139}
\begin{split}
p_y\left(\frac{\partial}{\partial h}\right) &\left( \int_{P_y(h)} e^{\langle  m, z\rangle} \ dm \right)_{\vert_{h=0}} \\ &= \sum_{m\in M} \left( \sum_{i=0}^n (-1)^i \cdot \dim_\bC H^i(X;\cO_X(D) \otimes \cF)_{\chi^{-m}}\right) \cdot e^{\langle (1+y)m, z \rangle}
\:.
\end{split}
\ee
Here, the use of the renormalized power series $p_y$ and polytope $P_y(h)$ correspond to multiplying each degree $2k$ equivariant cohomology class by $(1+y)^k$. (This is just the cohomological \index{Adams operation} Adams operation $\Psi^{(1+y)}$.) In particular, via the identification $s:M \simeq H^2_\bT(pt;\bZ)$, 
$m\in M$ gets multiplied by $1+y$.

For any polynomial $f$ defined on $M_\bR$, by applying the operator $\left( f(\frac{\partial}{\partial z}) \right)_{\vert_{z=0}}$ to formula \eqref{f1139}, we get the following \index{parametrized Euler-Maclaurin formula} {\it parametrized Euler-Maclaurin formula}:
 \be\label{f1140}\begin{split}
p_y\left(\frac{\partial}{\partial h}\right) &\left( \int_{P_y(h)} f(m) \ dm \right)_{\vert_{h=0}} \\&= \sum_{m\in M} \left( \sum_{i=0}^n (-1)^i \cdot \dim_\bC H^i(X;\cO_X(D) \otimes \cF)_{\chi^{-m}}\right) \cdot f\left( (1+y)m\right)\:. \end{split}
\ee
Moreover, if $f$ is homogeneous of degree $\deg(f)$, then 
 $f\left( (1+y)m\right)=(1+y)^{\deg(f)} \cdot f(m).$
 
 \begin{rem}
 In the local case of Euler-Maclaurin formulae for a tangent cone $Tan(P,v)$, we only work with the unnormalized equivariant Hirzebruch classes $T^\bT_{y*}$ and the dilation  $Tan(P,v)(h)$ of the tangent cone. In this context, a renormalization is not needed.
 \end{rem}
 
 %%%%%%%%%%%%%%%%%%%%%

\subsection{Examples of weighted Euler-Maclaurin formulae}\label{sec:6.1}
In this subsection, we apply the above renormalization not just to classes of coherent sheaves, but directly to examples of the type $$\bigoplus_p [\widehat{\Omega}^p_X]_\bT \cdot y^p \in K^\bT_0(X)[y],$$ as well as for suitable motivic Chern classes $mC^\bT_y$ of $\bT$-invariant constructible subsets of $X$.
 
The weighted lattice point counting of Theorem \ref{wco} (and Remark \ref{wcor}), coupled with the  expressions \eqref{eHirz0} and \eqref{eHirz} from Theorem \ref{eqHirz}
 for the equivariant Hirzebruch classes, suggest that a weighted Euler-Maclaurin-type formula can be computed by using the Hirzebruch (or, parametrized Todd) differential operators defined by replacing $\frac{\partial}{\partial h_{\rho}} $ for $F_\rho$:
$${T}_{y}\left(\frac{\partial}{\partial h}\right):=(1+y)^{n-r} \cdot 
 \sum_{g \in G_{\Sig}}  \prod_{\rho \in \Sig(1)} \frac{ \frac{\partial}{\partial h_{\rho} } \cdot 
\big( 1+y  \cdot a_{\rho}(g)  \cdot e^{-\frac{\partial}{\partial h_{\rho} }}\big)}{1-a_{\rho}(g) \cdot e^{-\frac{\partial}{\partial h_{\rho} }}} $$
and
$$\widehat{T}_{y}\left(\frac{\partial}{\partial h}\right):=
 \sum_{g \in G_{\Sig}}  \prod_{\rho \in \Sig(1)} \frac{ \frac{\partial}{\partial h_{\rho} } \cdot 
\big( 1+y  \cdot a_{\rho}(g)  \cdot e^{-\frac{\partial}{\partial h_{\rho} }(1+y)}\big)}{1-a_{\rho}(g) \cdot e^{-\frac{\partial}{\partial h_{\rho} }(1+y)}} .$$
The passage from ${T}_{y}(\frac{\partial}{\partial h})$ to $\widehat{T}_{y}(\frac{\partial}{\partial h})$ is just a special case of our general procedure of 
moving from $p(\frac{\partial}{\partial h})$ to $p_{y}(\frac{\partial}{\partial h})$.  This therefore fits with the renormalization of Hirzebruch classes, from ${T}^\bT_{y*}(X)$ to $\widehat{T}^\bT_{y*}(X)$.

Weighted Euler-Maclaurin formulae were considered from a combinatorial point of view in \cite{GKO}[Thm.1.1 and Rem.1.2] for ${T}_{y}(\frac{\partial}{\partial h})$, and more recently in \cite{BGM}[Thm.1.5] for $\widehat{T}^\bT_{y*}(X)$. 
We reprove here 
the following result (see \cite{GKO}[Thm.1.1 and Rem.1.2] for ${T}_{y}(\frac{\partial}{\partial h})$, and \cite{BGM}[Thm.4.6] for $\widehat{T}_{y}(\frac{\partial}{\partial h})$), from the point of view of generalized Hirzebruch-Riemann-Roch formula:
\bt\label{bgm1} \ 
For $z\in N_\bC$ small enough, one has in the above notations: 
\be\label{wem0}
{T}_{y}\left(\frac{\partial}{\partial h}\right) \left( \int_{P(h)} e^{\langle m, z \rangle} dm \right)_{|_{h=0}} =\sum_{E \preceq P} (1+y)^{\dim(E)} \sum_{m \in \Relint(E) \cap M} e^{\langle m , z\rangle}.
\ee
\be\label{wem}
\widehat{T}_{y}\left(\frac{\partial}{\partial h}\right) \left( \int_{P_y(h)} e^{\langle m, z \rangle} dm \right)_{|_{h=0}} 
=\sum_{E \preceq P} (1+y)^{\dim(E)} \sum_{m \in \Relint(E) \cap M} e^{\langle (1+y)m , z\rangle}.
\ee
In the corresponding local case of the tangent cone $Tan(P,v)$ of $P$ at a vertex $v$, for $z\in N_\bC$ small enough with $-z \in Int(\sig)$,we get
\be
\begin{split}
T_y\left(\frac{\partial}{\partial h}\right) \left( \int_{Tan(P,v)(h)} e^{\langle  m, z\rangle} \ dm \right)_{\vert_{h=0}} =\sum_{E \preceq Tan(P,v)} (1+y)^{\dim(E)} \cdot e^{\langle v , z\rangle} \cdot\sum_{m \in \Relint(E) \cap M} e^{\langle m , z\rangle} \:.
\end{split}
\ee
\et

\begin{proof}
To deduce formula \eqref{wem0}, instead of $[\cF] \in K^\bT_0(X)$ one considers  the motivic Chern class $mC^\bT_y(X) \in K^\bT_0(X)[y]$, with 
$T_{y*}(X):=\td_*^\bT(mC^\bT_y(X))$. We work with the power series $p(x_\rho):=T_y(x_\rho) \in \bQ\{x_\rho \mid \rho \in \Sig(1) \}[y]$ fitting with the differential operator ${T}_{y}(\frac{\partial}{\partial h})$. Formula \eqref{wem0} follows now from \eqref{f1134} (linearly extended in $y$), using the equivariant Hirzebruch-Riemann-Roch formula of Corollary \ref{c38}, applied to the case $P'=P$. Formula \eqref{wem} is then obtained via renormalization using \eqref{f1139}.

For the local result, we use \eqref{eq147} together with \eqref{eq111b}.
\end{proof}

Using \eqref{f1140}, one then obtains the following result (see also \cite{GKO}[Thm.1.1 and Rem.1.2] and \cite{BGM}[Thm.1.5]): \index{weighted Euler-Maclaurin formula for a simple polytope}
\bc\label{bgm11} \ 
For any polynomial function $f$  on $M_\bR$, one has: 
\be\label{wem01}
{T}_{y}\left(\frac{\partial}{\partial h}\right) \left( \int_{P(h)} f(m) \ dm \right)_{|_{h=0}} =\sum_{E \preceq P} (1+y)^{\dim(E)} \sum_{m \in \Relint(E) \cap M} f(m).
\ee
\be\label{wem1}
\widehat{T}_{y}\left(\frac{\partial}{\partial h}\right) \left( \int_{P_y(h)} f(m) \ dm \right)_{|_{h=0}} 
=\sum_{E \preceq P} (1+y)^{\dim(E)} \sum_{m \in \Relint(E) \cap M} f((1+y)m).
\ee
\ec

Note that the Brion-Verne Euler-Maclaurin formula for the simple lattice polytope $P$ is obtained from either (\ref{wem01}) or (\ref{wem1}) by specializing to $y=0$. Moreover, for $y=1$, one gets Euler-Maclaurin formulae corresponding to operators related to suitable $L$-classes, i.e., $\widehat{T}_{1*}(X)=L_*(X)$ the Thom-Milnor $L$-class of $X$ (see \cite{MS1,CMSS}). Taking the top degree in $y$ on both sides of \eqref{wem01} recovers formula \eqref{tdu}.

\medskip

Next we prove a \index{weighted Euler-Maclaurin formula for a polytope with some facets removed} {\it weighted Euler-Maclaurin formula for a polytope with some facets removed}. Here and below, for simplicity, we only indicate such a formula for a polynomial $f$, while we leave it to the reader to formulate the corresponding exponential formula. Let $P$ be a full-dimensional polytope with $r$ facets $F_1,\ldots,F_r$. For a subset $K\subseteq \{1,\ldots,r\}$, let $P^K$ be as before the set obtained from $P$ by removing the facets $F_i$, $i\in K$. So $P^\emptyset=P$ and $P^{\{1,\ldots,r\}}=\Int(P)$. Let $\Sigma(1)=\{\rho_1,\ldots,\rho_r\}$ be the rays of the inner normal fan of $P$, and denote by $h=(h_1,\ldots,h_r)$ a vector of real numbers indexed by the facets of $P$ (i.e., $h_i:=h_{\rho_i}$). Consider the following operator in $\bQ\{\frac{\partial}{\partial h_i} \mid i=1,\ldots,r\}[y]$:
\be
T_y^K\left(\frac{\partial}{\partial h}\right):= (1+y)^{n-r} \cdot\sum_{g\in G_{\Sigma}} \prod_{i\in K} 
\frac{ (1+y) \cdot a_{\rho_i}(g) \cdot \frac{\partial}{\partial h_i}\cdot e^{-\frac{\partial}{\partial h_i} }}{1-a_{\rho_i}(g)\cdot e^{-\frac{\partial}{\partial h_i} }} \cdot \prod_{i \notin K} 
\frac{ \frac{\partial}{\partial h_i} (1+y\cdot a_{\rho_i}(g)e^{-\frac{\partial}{\partial h_i} })}
{1-a_{\rho_i}(g)\cdot e^{-\frac{\partial}{\partial h_i} }},
\ee
corresponding to formula \eqref{eHirz0comp} for the equivariant Todd class of $mC_y^\bT([U \hookrightarrow X])$, with $U=X\setminus D_K$ the open complement of the divisor $D_K=\bigcup_{i \in K} D_i$ and $D_i:=D_{\rho_i}$. By additivity,
$$\td_*^\bT(mC_y^\bT([U \hookrightarrow X]))=T_{y*}([U\hookrightarrow X])=T_{y*}(X) - T_{y*}(D_K)$$
By Corollary \ref{c38}, applied to $X$ and $D_K$, one gets a corresponding equivariant Hirzebruch-Riemann-Roch formula:
\be
\int_X \td_*^\bT\left(\cO_X(D) \otimes mC_y^\bT([U \hookrightarrow X])\right)= \sum_{E \preceq P^K} (1+y)^{\dim(E)} \cdot \sum_{m \in \Relint(E) \cap M} e^{s(m)}.
\ee
With the above notations this then gives the following.
\bt\label{bgm1b}
Let $P$ be a full-dimensional simple lattice polytope in $M_\bR$, and let $f$ be a polynomial function  on $M_\bR$. 
Then:
\be\label{wEMgen}
T_y^K\left(\frac{\partial}{\partial h}\right)\left( \int_{P(h)} f(m)\ dm \right)_{|_{h=0}} =\ \sum_{E \preceq P^K} (1+y)^{\dim(E)} \cdot \sum_{m \in \Relint(E) \cap M} f(m).
\ee
\et

By renormalization, $T_y^K(\frac{\partial}{\partial h})$ changes to 
$$\widehat{T}_y^K\left(\frac{\partial}{\partial h}\right):=
\sum_{g\in G_{\Sigma}} \prod_{i\in K} 
\frac{ (1+y) \cdot a_{\rho_i}(g) \cdot \frac{\partial}{\partial h_i}\cdot e^{-(1+y) \frac{\partial}{\partial h_i} }}{1-a_{\rho_i}(g)\cdot e^{-(1+y)\frac{\partial}{\partial h_i} }} \cdot \prod_{i \notin K} 
\frac{ \frac{\partial}{\partial h_i} (1+y\cdot a_{\rho_i}(g)e^{-(1+y)\frac{\partial}{\partial h_i} })}
{1-a_{\rho_i}(g)\cdot e^{-(1+y)\frac{\partial}{\partial h_i} }}
$$
fitting with formula \eqref{eHirzcomp}. We thus get by \eqref{f1140} the following
\bc\label{bgm11b}
Let $P$ be a full-dimensional simple lattice polytope in $M_\bR$, and let $f$ be a polynomial function on $M_\bR$. 
Then:
\be\label{wEMgenn}
\widehat{T}_y^K\left(\frac{\partial}{\partial h}\right)\left( \int_{P_y(h)} f(m)\ dm \right)_{|_{h=0}} =\ \sum_{E \preceq P^K} (1+y)^{\dim(E)} \cdot \sum_{m \in \Relint(E) \cap M} f((1+y)m).
\ee
\ec

\begin{rem}
In the corresponding local case of the tangent cone $Tan(P,v)$ of $P$ at a vertex $v$, for $z\in N_\bC$ small enough with $-z \in Int(\sig)$, we get by \eqref{mclocal}, \eqref{sumsmusf} \eqref{sumsmusg}, together with \eqref{eq111b}, the following:
\be
\begin{split}
T^K_y\left(\frac{\partial}{\partial h}\right) \left( \int_{Tan(P,v)(h)} e^{\langle  m, z\rangle} \ dm \right)_{\vert_{h=0}} =\sum_{E \preceq Tan(P,v)^K} (1+y)^{\dim(E)} \cdot e^{\langle v , z\rangle} \cdot\sum_{m \in \Relint(E) \cap M} e^{\langle m , z\rangle} \:.
\end{split}
\ee
\end{rem}

\medskip

We next discuss a \index{weighted Euler-Maclaurin formula for a face} {\it weighted Euler-Maclaurin formula for a face of a simple lattice polytope}. 
Let $P$ be a full-dimensional simple lattice polytope in $M_\bR$, and fix a face $E$ of $P$. Let $\sigma:=\sigma_E$ be the corresponding cone in the inner normal fan $\Sigma=\Sigma_P$ of $P$, with $V_\sigma=V_{\sigma_E}=X_E$ the closure of the orbit of $\sigma$ in $X=X_P$. Denote by $i_E=i_{\sigma}: V_\sig \hookrightarrow X$ the closed inclusion map. Then 
$V_\sigma$ is a simplicial toric variety whose fan is $Star(\sigma)$, as defined in \eqref{star}, which is built from cones $\tau \in \Sigma$ that have $\sigma$ as a face.  
 Recall that $\bT=T_N$ acts on $V_\sigma$ via the morphism $T_N \to T_{N(\sigma)}$. We have by \eqref{eHirz0} 
the following formula:
 \be\label{dueqetw}
T_{y*}^\bT([id_{V_\sigma}])=(1+y)^{n - r} \cdot i_\sig^*\left( \sum_{g\in G_{Star(\sig)}} \prod_{\rho\in  Star(\sigma)(1)} \frac{ F_{\rho} \cdot 
\big( 1+y  \cdot a_{\rho}(g)  \cdot e^{-F_{\rho}}\big)}{1-a_{\rho}(g) \cdot e^{-F_{\rho}}} \right).
\ee
The projection formula then yields the computation of the following equivariant Hirzebruch class:
\be\label{dueqeeh}
\begin{split}
T_{y*}^\bT&([{V_\sigma} \hookrightarrow X])\\ &= (1+y)^{n - r} \cdot \sum_{g\in G_{Star(\sig)}} 
\mult(\sig) \cdot \prod_{\rho\in\sigma(1)}F_\rho \cdot
\prod_{\rho\in  Star(\sigma)(1)} \frac{ F_{\rho} \cdot 
\big( 1+y  \cdot a_{\rho}(g)  \cdot e^{-F_{\rho}}\big)}{1-a_{\rho}(g) \cdot e^{-F_{\rho}}} .
\end{split}
\ee

Using Corollary \ref{c38}, this yields the following result:
\bt\label{EMintfaceh}
Let $P$ be a full-dimensional simple lattice polytope in $M_\bR$, and let $f$ be a polynomial function on $M_\bR$. Then, for a fixed face $E$ of $P$, with corresponding cone $\sigma$ in the inner normal fan $\Sig$ of $P$, one has:
\be\label{EMinth}
T_y^E\left(\frac{\partial}{\partial h}\right)\left( \int_{P(h)} f(m) \ dm \right)_{|_{h=0}} =
\sum_{E' \preceq E} (1+y)^{\dim(E')} \cdot \sum_{m \in \Relint(E') \cap M} f(m),
\ee
where
\be\begin{split}
T_y^E& \left(\frac{\partial}{\partial h}\right)\\ &:=(1+y)^{n - r} \cdot \sum_{g\in G_{Star(\sig)}} \mult(\sig) \cdot \prod_{\rho\in\sigma(1)}\frac{\partial}{\partial h_\rho} \cdot \prod_{\rho\in  Star(\sigma)(1)} 
\frac{  \frac{\partial}{\partial h_{\rho}} \left( 1+y\cdot a_{\rho}(g)\cdot e^{-\frac{\partial}{\partial h_{\rho}} } \right)}
{1-a_{\rho}(g)\cdot e^{-\frac{\partial}{\partial h_{\rho}} } }.\end{split}
\ee
\et
For $y=0$, formula \eqref{EMinth} specializes to \eqref{EMint}.
Taking the top degree in $y$ on both sides of \eqref{EMinth} gives back formula \eqref{EMintd}.

By renormalization, $T_y^E(\frac{\partial}{\partial h})$ changes to 
$$\widehat{T}_y^E\left(\frac{\partial}{\partial h}\right):=
\sum_{g\in G_{Star(\sig)}} \mult(\sig) \cdot \prod_{\rho\in\sigma(1)}\frac{\partial}{\partial h_\rho} \cdot \prod_{\rho\in  Star(\sigma)(1)} 
\frac{  \frac{\partial}{\partial h_{\rho}} \left( 1+y\cdot a_{\rho}(g)\cdot e^{-(1+y)\frac{\partial}{\partial h_{\rho}} } \right)}
{1-a_{\rho}(g)\cdot e^{-(1+y)\frac{\partial}{\partial h_{\rho}} } }.
$$
We thus get by \eqref{f1140} the following
\bc\label{EMintfaceh1}
In the above notations, we get
\be\label{wEMgennw}
\widehat{T}_y^E\left(\frac{\partial}{\partial h}\right)\left( \int_{P_y(h)} f(m)\ dm \right)_{|_{h=0}} =\sum_{E' \preceq E} (1+y)^{\dim(E')} \cdot \sum_{m \in \Relint(E') \cap M} f((1+y)m).
\ee
\ec
For $y=1$, one gets Euler-Maclaurin formulae corresponding to operators related to suitable $L$-classes, i.e., $\widehat{T}_{1*}(X_E)=L_*(X_E)$ the Thom-Milnor $L$-class of $X_E$.

\medskip

Another way to obtain examples of explicit weighted Euler-Maclaurin formulae is by twisting the equivariant motivic Chern class $mC_y^\bT$ by sheaves of the form $\cO_X(D'-D)$, for $D=D_P$ the original ample divisor associated to the full-dimensional lattice polytope, and $D'$ any $\bT$-invariant Cartier divisor on $X$. 
In the next example, we illustrate this principle in the case of $mC_y^\bT(X)$, corresponding to Theorem \ref{bgm1} and Corollary \ref{bgm11}.

\bex\label{ex526b} Let $D'$ be a \index{globally generated} globally generated $\bT$-invariant Cartier divisor on $X$, with associated (not necessarily full-dimensional) lattice polytope $P_{D'} \subset M_\bR$. Consider $D'-D=\sum_{\rho \in \Sig(1)} d_\rho D_\rho$ as a $\bT$-invariant Cartier divisor. 
Let $X_{D'}$ be the toric variety of the lattice polytope $P_{D'}$, defined via the corresponding {generalized fan} as in \cite{CLS}[Prop.6.2.3]. Then, by \cite{CLS}[Thm.6.2.8], there is a proper toric morphism $f:X \to X_{D'}$, induced by the corresponding lattice projection $N \to  N_{D'}$ given by dividing out by the minimal cone of the generalized fan of $P_{D'}$. In particular, $f:X \to X_{D'}$ is a toric fibration. For $\sig'$ a cone in the generalized fan of $P_{D'}$, define as in \eqref{dl}
$$d_\ell(X/\sig'):=\vert \Sigma_\ell(X/\sig') \vert$$ with
$$\Sigma_\ell(X/\sig'):=\{\sig \in \Sig \mid O_\sig \subset X, \ f(O_\sig)=O_{\sig'}, \ \ell=\dim(O_\sig) - \dim(O_{\sig'})\}.$$
Consider the infinite order differential operators
\begin{equation}\label{todd222b}
T'_y\left(\frac{\partial}{\partial h}\right):= e^{\sum_{\rho \in \Sig(1)} d_\rho \cdot \frac{\partial}{\partial h_{\rho} }} \cdot T_y\left(\frac{\partial}{\partial h}\right)
\end{equation}
and 
\begin{equation}\label{todd222c}
\widehat{T}_y'\left(\frac{\partial}{\partial h}\right):= e^{(1+y) \cdot \sum_{\rho \in \Sig(1)} d_\rho \cdot \frac{\partial}{\partial h_{\rho} }} \cdot \widehat{T}_y\left(\frac{\partial}{\partial h}\right),
\end{equation}
with $T_y(\frac{\partial}{\partial h}), \widehat{T}_y(\frac{\partial}{\partial h})$ as in Theorem \ref{bgm1} and Corollary \ref{bgm11}. For any quasi-polynomial $f$ on $M_\bR$, one then has by \eqref{nefghrr} the following new weighted Euler-Maclaurin formulae:
\begin{multline}\label{wem01b}
T'_{y}\left(\frac{\partial}{\partial h}\right) \left( \int_{P(h)} f(m) \ dm \right)_{|_{h=0}} = \\ =\sum_{E \preceq P_{D'}} \left( \sum_{\ell \geq 0} 
(-1)^\ell \cdot d_\ell(X/E) \cdot (1+y)^{\ell + \dim(E)}   \right) \cdot \sum_{m \in \Relint(E) \cap M} f(m).
\end{multline}
\begin{multline}\label{wem1b}
\widehat{T}'_{y}\left(\frac{\partial}{\partial h}\right) \left( \int_{P_y(h)} f(m) \ dm \right)_{|_{h=0}} = \\
=\sum_{E \preceq P_{D'}} \left( \sum_{\ell \geq 0} 
(-1)^\ell \cdot d_\ell(X/E) \cdot (1+y)^{\ell +\dim(E)} \right)  \cdot \sum_{m \in \Relint(E) \cap M} f((1+y)m).
\end{multline}
As noted in Example \ref{ex526}, $P_{D'}$ is an $\bN$-Minkowski summand (e.g., in the sense of \cite{CLS}[Def.6.2.11]) of the original polytope $P$, see \cite{CLS}[Cor.6.2.15].
\qed
\eex

%%%%%%%%%%%%%%%%%%

\section{Euler-Maclaurin formulae via the Cappell-Shaneson algebra}\label{emcs}
In this section we provide generalizations of Cappell-Shaneson's Euler-Maclaurin formula \cite{CS2,S} and explain its connection to equivariant toric geometry. These Euler-Maclaurin formulae of Cappell-Shaneson type do not use a dilation of the lattice polytope, but a summation of integrals over the faces. 

In Subsection \ref{sec7.1} we recall the classical Euler-Maclaurin formula of Cappell-Shaneson \cite{CS2,S}, formulated in terms of the so-called Cappell-Shaneson algebra (cf. \cite{KSW}), and explain  in Remark \ref{CSrem}
the identification of this algebra with the completed equivariant cohomology ring of the associated simplicial toric variety (regarded as an algebra of the completed equivariant cohomology ring of a point space).

In Subsection \ref{sec7.2}, we prove  in Theorem \ref{abstrEM2} \and Corollary \ref{abstrEMc2}
another abstract Euler-Maclaurin formula, again for arbitrary $\bT$-equivariant  coherent sheaf coefficients (with the classical case corres\-ponding to the structure sheaf). The  main  result is Theorem \ref{them2}, which relates exponential integrals over the dilation of a face of a simple lattice polytope to toric geometry.
This is the key  ingredient in the proof of Stokes' type formulae given in Corollary \ref{EM-faces1} and, resp., Theorem \ref{thmStokes},
which bridge integrals over the dilated polytopes and  their faces. Putting these results together, we finally obtain in
Theorem \ref{abstrEM3} our new general Euler-MacLaurin formula
 of Cappell-Shaneson type, again for arbitrary $\bT$-equivariant  coherent sheaf coefficients (with the classical case corresponding to the structure sheaf; see Remark \ref{r716}  for more details on this classical case).
As before, we specialize this abstract formula  in \eqref{eq-todd-classical2} and Example \ref{finalex} to contexts dictated by various natural choices of coefficients attached to the projective simplicial  toric variety, i.e., with coefficients the canonical sheaf, the motivic Chern class (also of an open complement of a torus invariant divisor), or the canonical and, resp., structure sheaf of a $\bT$-orbit closure.

Finally, in Subsection \ref{sec7.3}, we  indicate an application of our abstract Euler-Maclaurin formula to generalized reciprocity for Dedekind sums.

%and %%%%%%%%%

\subsection{Cappell-Shaneson algebra vs.  completed equivariant cohomology ring}\label{sec7.1}
Let $P\subset M_\bR\simeq \bR^n$ be a full-dimensional simple lattice polytope. Instead of using a dilation $P(h)$ of the polytope $P$, the Euler-Maclaurin formula of Cappell-Shaneson \cite{CS2,S} 
uses a summation of integrals over the faces $E$ of $P$ (compare also with \cite{BrV3}[Sect.3.7]). More precisely, for any polynomial function $f$ on $M_\bR$ one has:
\begin{equation}\label{EM-CS}
\sum_{m\in P\cap M}\; f(m) = \sum_{E\preceq P}\:\int_E \left(p_{E}(\partial_i) f\right)(m) dm \:,
\end{equation}
with $p_{E}(\partial_i)\in \bb{Q}[[\partial_1,\dots,\partial_n]]$ suitable infinite order differential operators with constant rational coefficients
in the partial derivatives with respect to the coordinates of the vector space $M_\bR$.
Here the Lebesgue measure $dm$ on $E$ is normalized so that the unit cube in the lattice $Span(E_0)\cap M$ has volume $1$, with $E_0:=E-m_0$ a translation of $E$ by a vertex $m_0 \in E$.

The infinite order differential operators $p_{E}(\partial_i)$ are defined through some relations 
in what is called in \cite{KSW}[Sect.6] the {\it Cappell-Shaneson algebra} of $P$: \index{Cappell-Shaneson algebra}
$$\cA(P):= \bb{Q}[[\partial_1,\dots,\partial_n]][U_F \mid \text{$F$ a facet of $P$}]/\sim$$
with relations
\begin{equation}\label{rel1-CS}
U_{F_1}\cdots U_{F_k}=0 \quad \text{for distinct $F_i$ with \ $F_1\cap \dots \cap F_k=\emptyset$},
\end{equation}
and
\begin{equation}\label{rel2-CS}
\frac{\partial}{\partial m }+ \sum_F \; \langle m, n_F \rangle U_F = 0 \quad \text{for all $m\in M$ (or a basis of $M$).}
\end{equation}
Here $\frac{\partial}{\partial m }$ is the differentiation in the direction of $m$, with $n_F\in M$ the minimal lattice vector orthogonal to $F$ and pointing into $P$.
Let us rewrite this presentation of the algebra $\cA(P)$ (given above in terms of the polytope $P$) into the language of this paper using another description of the \index{completed Stanley-Reisner ring} completed Stanley-Reisner ring $$\bQ[[x_\rho \mid  \rho \in \Sig(1)]]/\sim_{SR}=:\widehat{SR_\bQ(\Sig)}$$ of the associated toric variety $X_P$ and the inner normal fan $\Sigma=\Sigma_P$. Note that the inner normal vector $n_F$ corresponds to the generator $u_\rho$ of the ray $\rho\in \Sigma(1)$  corresponding to $F$.
For simplicity, we fix a basis $m_1,\dots,m_n$ of $M\simeq \bb{Z}^n$, with $t_i$ the corresponding coordinates on $M_\bK$ (for $\bK=\bQ,\bR$) with  respect to this basis, so that $\frac{\partial}{\partial t_i }= \partial_i$ for $i=1,\dots,n$. Then
$$\bb{Q}[t_1,\dots,t_n, x_\rho \mid \text{$\rho\in \Sigma(1)$}] \simeq 
\bb{Q}[\partial_1,\dots,\partial_n, U_F \mid \text{$F$ a facet of $P$}]$$
via $t_i\mapsto \partial_i$ and $x_\rho \mapsto U_{F}$, with $F$ the facet of $P$ corresponding to $\rho \in \Sig(1)$. 
The relation (\ref{rel1-CS}) corresponds to the Stanly-Reisner relation $\sim_{SR}$, and the relation
(\ref{rel2-CS}) translates into
\begin{equation}\label{rel2-SR}
t_i + \sum_{\rho\in \Sigma(1)} \; \langle m_i, n_\rho \rangle x_\rho = 0 \quad \text{for $i=1,\dots,n$.}
\end{equation}
Denote by $\sim$ both relations together on $\bb{Q}[t_1,\dots,t_n, x_\rho \mid \text{$\rho\in \Sigma(1)$}]$. 
These are also homogeneous for the usual grading doubled (i.e., with $x_{\rho}$ and $t_i$ of degree two). Then one gets
(see also  \cite{AF}[Lemma 8.3.2]):

\begin{lem}\label{SRnew}
The inclusion $\bb{Q}[x_\rho \mid \text{$\rho\in \Sigma(1)$}] \hookrightarrow \bb{Q}[t_1,\dots,t_n, x_\rho \mid \text{$\rho\in \Sigma(1)$}]$
induces an isomorphism of graded rings
$$
SR_\bQ(\Sig)=\bb{Q}[x_\rho \mid \text{$\rho\in \Sigma(1)$}] /\sim_{SR}\: \stackrel{\sim}{\lra} 
\bb{Q}[t_1,\dots,t_n, x_\rho \mid \text{$\rho\in \Sigma(1)$}]/\sim
$$
and their formal power series completions
$$
\widehat{SR_\bQ(\Sig)}=\bb{Q}[[x_\rho \mid \text{$\rho\in \Sigma(1)$}]] /\sim_{SR}\: \stackrel{\sim}{\lra} 
\bb{Q}[[t_1,\dots,t_n, x_\rho \mid \text{$\rho\in \Sigma(1)$}]]/\sim \:,
$$
with the {\it analytic Stanley-Reisner subring} $SR^{an}_\bQ(\Sig)=\bb{Q}\{x_\rho \mid \text{$\rho\in \Sigma(1)$}\} /\sim_{SR}$ identified with the image of the projection from $\bb{Q}\{t_1,\dots,t_n, x_\rho \mid \text{$\rho\in \Sigma(1)$}\}$.
\end{lem}

\begin{proof}
Using the invertible  variable  transformation \begin{center} $x_\rho\mapsto x_\rho, \ t_i\mapsto t'_i:=t_i + \sum_{\rho\in \Sigma(1)} \; \langle m_i, n_\rho \rangle x_\rho$ \end{center}
(for $i=1,\dots,n$ and $\rho\in \Sigma(1)$),
the claim is reduced to the corresponding inclusion  $$\bb{Q}[x_\rho \mid \text{$\rho\in \Sigma(1)$}] \hookrightarrow 
\bb{Q}[t'_1,\dots,t'_n, x_\rho \mid \text{$\rho\in \Sigma(1)$}]\:,$$
 with $\sim$ given by $\sim_{SR}$ and  $t'_i=0$ for $i=1,\dots,n$.
In this case the results are obvious.
\end{proof}

\begin{cor}
Using the isomorphism
$$s: Sym_\bQ(M) \simeq  (\Lambda_\bT)_{\bb{Q}}= H_\bT^*(pt;\bb{Q}); \ \ m \mapsto -c^1_\bT(\bC_{\chi^m})$$
together with $\bQ[t_1,\dots,t_n]= Sym_\bQ(M) $, one gets isomorphisms of $(\Lambda_\bT)_{\bb{Q}}$-algebras 
$$ SR_\bQ(\Sig)  \simeq\bb{Q}[t_1,\dots,t_n, x_\rho \mid \text{$\rho\in \Sigma(1)$}] / \sim \:\simeq  
(\Lambda_\bT)_{\bb{Q}}[x_\rho \mid \text{$\rho\in \Sigma(1)$}]/\sim\: \simeq H^*_\bT(X;\bQ)$$
and
$$\widehat{SR_\bQ(\Sig)}  \simeq \bb{Q}[[t_1,\dots,t_n, x_\rho \mid \text{$\rho\in \Sigma(1)$}]]/\sim \: \simeq 
(\widehat{\Lambda}_\bT)_{\bb{Q}}[[x_\rho \mid \text{$\rho\in \Sigma(1)$}]]/\sim\: \simeq \widehat{H}^*_\bT(X;\bQ)\:,$$
with $t_i \mapsto s(m_i)=-c^1_\bT(\bC_{\chi^{m_i}})$ for $i=i,\dots,n$, and 
$x_\rho\mapsto F_\rho=[D_\rho]_\bT=c^1_\bT(\mathcal{O}(D_\rho)) $ for $\rho\in \Sigma(1)$.
\end{cor}
\begin{proof}
It suffices to note that the relation (\ref{rel2-SR}) corresponds to (\ref{alg}).
\end{proof}

\br\label{CSrem}
Let $P \subset M_{\bR}$ be a simple full-dimensional lattice polytope with associated projective toric variety $X=X_P$ and the inner normal fan 
$\Sigma=\Sigma_P$. Fix a basis $m_1,\dots,m_n$ of $M\simeq \bb{Z}^n$. There is an algebra homomorphism
from  the Cappell-Shaneson algebra $\cA(P)$ to the completed  equivariant cohomology ring 
\begin{equation}\label{main}
\begin{split}
\cA(P)&=\bb{Q}[[\partial_1,\dots,\partial_n]][U_F]/\sim \: 
\to  \widehat{SR_\bQ(\Sig)} \simeq \ \widehat{H}^*_\bT(X;\bb{Q});\: \\ 
& U_F\mapsto x_\rho \simeq  F_\rho= [D_\rho]_\bT=c^1_\bT(\mathcal{O}(D_\rho))
\end{split}\end{equation}
as a $\bb{Q}[[\partial_1,\dots,\partial_n]]\simeq \widehat{H}_\bT^*(pt;\bb{Q})\simeq (\widehat{\Lambda}_\bT)_{\bb{Q}}$-algebra by 
$\partial_i \simeq  t_i  \simeq  s(m_i)= -c^1_\bT(\chi^{m_i})$ for $i=1,\ldots,n$,
and with the facet $F$ of $P$ corresponding to the ray $\rho\in \Sigma(1)$.  
As we shall now explain, the algebra map \eqref{main} is in fact an isomorphism.
In abstract terms, the completions $\widehat{SR_\bQ(\Sig)} \simeq \ \widehat{H}^*_\bT(X;\bb{Q})$ are completions of connected integer graded commutative rings $R^*$ with respect to the maximal ideal $R^{>0}$ given by positive degree elements, whereas $\cA(P)$ uses the completion with respect to the maximal ideal $I_\bT$ of  $\bb{Q}[\partial_1,\dots,\partial_n]\simeq (\Lambda_\bT)_{\bb{Q}}$. As $(\Lambda_\bT)_{\bb{Q}}[U_F]$ is not a finitely generated $
(\Lambda_\bT)_{\bb{Q}}$-module, the completion functor with respect to $I_\bT$ is only left exact. In particular, there is an injective algebra homomorphism $\cA(P) \hookrightarrow \left(\widehat{H^*_\bT(X;\bb{Q})} \right)_{I_\bT}$, which factorizes \eqref{main}.
Using the relations defining $\cA(P)$, one gets as in \cite{CS2}[Proposition on page 888] or \cite{S}[6.1, page 621] that the image of this monomorphism is the $(\widehat{\Lambda}_\bT)_{\bb{Q}}$-submodule generated by the equivalence classes  $[V_E]\in \cA(P)$ associated to all faces $E$ of $P$. This also shows that one gets a natural algebra homomorphism $\bQ[[U_F]]  \to \cA(P)$ inducing algebra homomorphisms 
$$\widehat{H}^\bT_*(X;\bQ) \to \cA(P) \to (\widehat{\Lambda}_\bT)_{\bb{Q}}[[x_\rho]]/\sim.$$
Since this composition is an isomorphism, the first arrow is injective and the second is surjective.
Using the fact that $H^*_\bT(X;\bQ) \simeq (\Lambda_\bT)_{\bb{Q}} \otimes_\bQ H^*(X;\bQ)$, one also has that $\left(\widehat{H^*_\bT(X;\bb{Q})} \right)_{I_\bT}\simeq (\widehat{\Lambda}_\bT)_{\bb{Q}} \otimes_\bQ H^*(X;\bQ)$, which is also generated as a $(\widehat{\Lambda}_\bT)_{\bb{Q}}$-module by  the equivariant fundamental classes of the orbit closures $[V_{\sig_E}]_\bT$ associated to all faces $E$ of $P$ (since $H^*(X;\bQ)$ is generated as a $\bQ$-vector space by the fundamental classes $[V_{\sig_E}]$). This shows that $\cA(P) \simeq
 \left(\widehat{H^*_\bT(X;\bb{Q})} \right)_{I_\bT}$. Finally, the natural algebra homomorphism 
 $ \left(\widehat{H^*_\bT(X;\bb{Q})} \right)_{I_\bT} \to \widehat{H}^*_\bT(X;\bb{Q})$ is an isomorphism by \cite{Br0}[Prop.1.4]. Therefore, the homomorphism \eqref{main} is an isomorphism.
\er

%%%%%%%%%%%%%%%%

\subsection{Euler-Maclaurin formulae via the Cappell-Shaneson algebra}\label{sec7.2}
In this subsection we provide a new proof of the Euler-Maclaurin formula of Cappell-Shaneson \cite{CS2,S}, as well as several generalizations.

Let us now introduce the key functionals following  \cite{KSW}[Section 6].
\begin{df}
Let $$p(t_i,x_\rho):=\sum_{\alpha=(\alpha_i)\in \bb{N}_0^{n}}\: p_{\alpha}(x_\rho)\prod_{i=1}^n t_i^{\alpha_i}\in 
\bQ\{t_1,\dots,t_n, x_\rho \mid \text{$\rho\in \Sigma(1)$}\},$$ resp., $\bQ[[t_1,\dots,t_n, x_\rho \mid \text{$\rho\in \Sigma(1)$}]]$
be a convergent, resp., formal power series in the $t_i, x_\rho$. Then
\be\label{functional1}
p\left(\partial_i, \frac{\partial}{\partial h}\right) \left( \int_{P(h)} e^{\langle  m, z\rangle} \ dm \right)_{\vert_{h=0}}:=
\sum_{\alpha=(\alpha_i)\in \bb{N}_0^{n}}\:p_{\alpha}( \frac{\partial}{\partial h}) \left( \int_{P(h)} 
\prod_{i=1}^n \partial_i^{\alpha_i}e^{\langle  m, z\rangle} \ dm \right)_{\vert_{h=0}}
\ee
resp.,
\be\label{functional2}
p\left(\partial_i, \frac{\partial}{\partial h}\right) \left( \int_{P(h)} f(m) \ dm \right)_{\vert_{h=0}}:=
\sum_{\alpha=(\alpha_i)\in \bb{N}_0^{n}}\:p_{\alpha}( \frac{\partial}{\partial h}) \left( \int_{P(h)} 
\prod_{i=1}^n \partial_i^{\alpha_i} f(m) \ dm \right)_{\vert_{h=0}}
\ee
for $f$ a polynomial function on $M_\bR$. Note that, in the second case, in the last term only finitely many summands are non-zero due to Corollary \ref{polin}.
\end{df}

Of course in \eqref{functional1} one has to explain in which sense the series is convergent, as will be discussed in the following.
As a consequence of Theorem \ref{them1} (with $y=0$), Remark \ref{regular} and Corollary \ref{convergence1}, we get by differentiation and convergence the following.
\bc\label{convergence3}
Let $p(\partial_i,\frac{\partial}{\partial h}) \in \bQ\{\partial_1,\dots,\partial_n, \frac{\partial}{\partial h_\rho} \mid  \rho \in \Sig(1) \}$ be an infinite order differential operator with constant rational coefficients, i.e., obtained by substituting $t_i\mapsto \partial_i, x_\rho \mapsto \frac{\partial}{\partial h_\rho}$ into a convergent  power series  $p(t_i,x_\rho) \in \bQ\{t_1,\dots,t_n, x_\rho \mid \rho \in \Sig(1) \}$ with rational coefficients. Then, in the above notations, we get for $z$ small enough and away from the hyperplanes $\langle i_\sig^* F_\rho, z \rangle = 0$ for each ray $\rho\in \sigma(1)$ of $\sigma\in \Sig(n)$ the following formula:
\be\label{cs1133}
\begin{split}
&p\left(\partial_i, \frac{\partial}{\partial h}\right) \left( \int_{P(h)} e^{\langle  m, z\rangle} \ dm \right)_{\vert_{h=0}} =
\sum_{\alpha=(\alpha_i)\in \bb{N}_0^{n}}\:p_{\alpha}\left( \frac{\partial}{\partial h}\right) \left( \int_{P(h)} \prod_{i=1}^n \partial_i^{\alpha_i}
e^{\langle  m, z\rangle} \ dm \right)_{\vert_{h=0}} \\
&= \sum_{\alpha=(\alpha_i)\in \bb{N}_0^{n}}\:\prod_{i=1}^n \langle  m_i, z\rangle^{\alpha_i}  \left(
\sum_{\sig \in \Sig(n)} \frac{ e^{\langle  i_\sig^* c^\bT_1(\cO_X(D_{P})),z\rangle}}{\langle Eu^\bT_X(x_\sig),  z \rangle} \cdot  \langle
 p_{\alpha}(i_\sig^* F_\rho), z\rangle \right)\\
&= \left\langle 
\sum_{\sig \in \Sig(n)} \frac{i_\sig^* \left( e^{(c^\bT_1(\cO_X(D))}\left(  \sum_{\alpha=(\alpha_i)\in \bb{N}_0^{n}}\:\prod_{i=1}^n  s(m_i)^{\alpha_i}
 p_\alpha(F_\rho) \right)\right)}{Eu^\bT_X(x_\sig)}  
, z  \right\rangle \\
& = \left\langle 
\int_X   e^{(c^\bT_1(\cO_X(D))} p(s(m_i),F_\rho), z  \right\rangle \:.
\end{split}\ee
If $p_k\in \bQ[\partial_1,\dots,\partial_n,\frac{\partial}{\partial h_\rho} \mid \rho \in \Sig(1) ]$ is the corresponding truncation of $p$ up to order $k$, then both sides of \eqref{cs1133} applied to $p_k$ converge for $k\to \infty$ locally uniformly in these $z$ to \eqref{cs1133} applied to $p$.
\ec

\begin{proof}
The second equality  follows from $ \prod_{i=1}^n \partial_i^{\alpha_i}e^{\langle  m, z\rangle} = \prod_{i=1}^n \langle  m_i, z\rangle^{\alpha_i} e^{\langle  m, z\rangle}$ and  Corollary \ref{convergence1}.
The third equality uses  $m_i=i_\sig^*s(m_i)\in M\simeq H^2_\bT(pt;\bQ)$ for $s(m_i)\in H^2_\bT(X;\bQ)$.
The last equality follows from equation \eqref{f108n},  where  the element $p(s(m_i),F_\rho)$ corresponds to the image of $p(t_i,x_\rho)$ under the evaluation homomorphism
$$\bQ\{t_1,\dots,t_n,x_\rho \mid  \rho \in \Sig(1) \}\mapsto \left(H^{*}_\bT(X;\bQ)\right)^{an}:\: t_i\mapsto s(m_i),
x_\rho\mapsto F_\rho \:.$$
\end{proof}

\br
By \eqref{cs1133}, the operator $p(\partial_i, \frac{\partial}{\partial h}) \left( \int_{P(h)} e^{\langle  m, z\rangle} \ dm \right)_{\vert_{h=0}}$ depends only on the equivalence class of $[p(t_i,x_\rho)] \in SR^{an}_\bQ(\Sig)\simeq \bb{Q}\{t_1,\dots,t_n, x_\rho \mid \text{$\rho\in \Sigma(1)$}]\}/ \sim$ and not on the chosen convergent power series representative.
\er

\br\label{convergence4}
Assume in addition in Corollary \ref{convergence3}, that $\langle \int_X   e^{(c^\bT_1(\cO_X(D))} p(s(m_i),F_\rho), z  \rangle$
is a convergent power series in $z$ near zero. Then one gets as an application of Cauchy's integral formula (see also \cite{KSW}[p.27]),
that  both sides of \eqref{cs1133} applied to $p_n$ converge for $n\to \infty$ and $z$ small locally  uniformly  to \eqref{cs1133} applied to $p$. In particular, this limit commutes with finite order  differentiations with respect to $z$ (and $z$ small enough).
\er

\bc\label{cformal2}
Let $p(\partial_i, \frac{\partial}{\partial h}) \in \bQ[[\partial_1,\dots,\partial_n,\frac{\partial}{\partial h_\rho} \mid \rho \in \Sig(1)]]$ be an infinite order differential operator with constant rational coefficients, i.e., obtained by substituting $t_i\mapsto \partial_i, x_\rho \mapsto \frac{\partial}{\partial h_\rho}$ into a formal power series $p(t_i,x_\rho) \in \bQ[[t_1,\dots,t_n,x_\rho \mid \rho \in \Sig(1) ]]$ with rational coefficients. Then for a polynomial  function $f$ on $M_\bR$, we have the following formula:
\be\label{form2133}
p\left(\partial_i, \frac{\partial}{\partial h}\right) \left( \int_{P(h) }f(m) \ dm \right)_{\vert_{h=0}} 
=  f\left(\frac{\partial}{\partial z}\right) \left(
\left\langle \int_X   e^{c^\bT_1(\cO_X(D))} p(s(m_i),F_\rho), z  \right\rangle \right)_{\vert_{z=0}} \:,
\ee
where on the right hand side the operator $\left( f(\frac{\partial}{\partial z}) \right)_{\vert_{z=0}}$ acts on a formal power series in $z$.
\ec

\begin{proof}
The proof is similar to that of Corollary \ref{cformal1}.
\end{proof}

\br
By \eqref{form2133}, the operator $p(\partial_i, \frac{ \partial}{\partial h}) \left( \int_{P(h)} f(z) \ dm \right)_{\vert_{h=0}}$ depends only on the equivalence class of $[p(t_i, x_\rho)] \in \widehat{SR_\bQ(\Sig)}$ and not on the chosen formal power series representative.
\er

Altogether, we get the following second abstract \index{Euler-Maclaurin formula} Euler-Maclaurin formula based on the equivariant Hirzebruch-Riemann-Roch theorem.
\bt\label{abstrEM2}
Let $X=X_P$ be the projective simplicial toric variety associated to a simple full-dimensional lattice polytope $P\subset M_\bR$. Let $\Sig:=\Sig_P$ be the inner normal fan of $P$, and $D:=D_P$ the ample Cartier divisor associated to $P$. Let $[\cF]\in K^\bT_0(X)$ be fixed, and choose a convergent power series $p(t_i,x_\rho) \in \bQ\{t_1,\dots,t_n,x_\rho \mid \rho \in \Sig(1) \}$ so that $p(s(m_i),F_\rho)=\td^\bT_*([\cF]) \in \left( H^*_\bT(X;\bQ) \right)^{an}$. Then 

\be\label{cs1134}
\begin{split}
p\left(\partial_i,\frac{ \partial}{\partial h}\right) &\left( \int_{P(h)} e^{\langle  m, z\rangle} \ dm \right)_{\vert_{h=0}}  =
 \left\langle \chi^\bT(X,\cO_X(D) \otimes \cF) , z\right\rangle \\
&= \sum_{m\in M} \left( \sum_{i=0}^n (-1)^i \cdot \dim_\bC H^i(X;\cO_X(D) \otimes \cF)_{\chi^{-m}}\right) \cdot e^{\langle m, z \rangle}
\:,
\end{split}
\ee
as analytic functions in $z$ with $z$ small enough, and with $\chi^\bT(X,\cO_X(D) \otimes \cF)\in (\Lambda^{an}_\bT)_\bQ$ the cohomological equivariant Euler characteristic of $\cO_X(D) \otimes \cF$.
\et
\begin{proof}
Equation \eqref{cs1133}  can now be calculated  as
\begin{eqnarray*}
p\left(\partial_i,\frac{\partial}{\partial h}\right) \left( \int_{P(h)} e^{\langle  m, z\rangle} \ dm \right)_{\vert_{h=0}} 
&=& \left\langle 
\int_X   e^{c^\bT_1(\cO_X(D))} \td^\bT_*([\cF]), z  \right\rangle \\
&=&  \left\langle \chi^\bT(X,\cO_X(D) \otimes \cF) , z\right\rangle ,
\end{eqnarray*}
where the last equality follows from the equivariant Hirzebruch-Riemann-Roch formula \eqref{eHRR} as in the proof of Theorem \ref{abstrEM}.
Finally, in the proof we first need to assume that $z$ is small enough and away from the hyperplanes $\langle i_\sig^* F_\rho, z \rangle = 0$ for each ray $\rho\in \sigma(1)$ of $\sigma\in \Sig(n)$, since the localization formula is used; however, formula \eqref{cs1134} then holds for all $z$ small enough, by Remark \ref{convergence4}.
\end{proof}

Note that Theorem \ref{abstrEM2}  reduces to Theorem \ref{abstrEM} in the case when 
$$p(t_i,x_\rho) \in \bQ\{x_\rho \mid \rho \in \Sig(1) \} \subset 
\bQ\{t_1,\dots,t_n,x_\rho \mid \rho \in \Sig(1) \}$$ does not depend on the variables $t_i$. Similarly, the next Corollary reduces in this case to Corollary \ref{abstrEMc}.

\bc\label{abstrEMc2}
Let $X=X_P$ be the projective simplicial toric variety associated to a simple full-dimensional lattice polytope $P\subset M_\bR$. Let $\Sig:=\Sig_P$ be the inner normal fan of $P$, and $D:=D_P$ the ample Cartier divisor associated to $P$. Let $[\cF]\in K^\bT_0(X)$ be fixed, and choose a formal power series $p(t_i, x_\rho) \in \bQ[[t_1,\dots,t_n,x_\rho \mid \rho \in \Sig(1) ]]$ so that $p(s(m_i),F_\rho)=td^\bT_*([\cF]) \in \widehat{ H}^*_\bT(X;\bQ) $. Then for a polynomial  function $f$ on $M_\bR$, we have:

\be\label{cs1137}\begin{split}
p\left(\partial_i, \frac{\partial}{\partial h}\right) & \left( \int_{P(h)} f(m) \ dm \right)_{\vert_{h=0}} \\
&= \sum_{m\in M} \left( \sum_{i=0}^n (-1)^i \cdot \dim_\bC H^i(X;\cO_X(D) \otimes \cF)_{\chi^{-m}}\right) \cdot f(m)\:.\end{split}
\ee
\ec

\begin{proof}
This follows from Corollary \ref{cformal2} by applying the operator $\left( f(\frac{\partial}{\partial z}) \right)_{\vert_{z=0}}$ to the last term of formula \eqref{cs1134}, viewed as a formal power series in $z$.
\end{proof}

For the use of integration over the faces $E$ of $P$ instead of dilation of the facets of $P$ by a parameter $h_\rho$ ($\rho\in \Sigma(1)$),
we next prove an analogue of Theorem \ref{them1} for faces of a polytope.

Assume $P$ is a full-dimensional simple lattice polytope in $M_\bR$ with associated toric variety $X=X_P$ and inner normal fan $\Sigma=\Sig_P$ . 
Let $P(h)$ be the dilation of $P$ with respect to the vector $h=(h_\rho)_{\rho \in \Sigma(1)}$ with real entries indexed by the rays of $\Sigma$. So, if $P$ is defined by inequalities of the form
$$\langle m, u_\rho\rangle +c_\rho \geq 0,$$
with $u_\rho$ the ray generators and $c_\rho \in \bZ$, for each $\rho \in \Sig(1)$, then $P(h)$  is defined by inequalities 
$$\langle m, u_\rho \rangle +c_\rho +h_\rho \geq 0,$$
for each $\rho \in \Sig(1)$. Similarly, we consider the dilation $E(h)$ of a fixed face $E$ of $P$, and let $\sig_E \in \Sig$ be the cone corresponding to $E$.

\bt\label{them2} In the above notations, we have
\be\label{f113b}
\begin{split}
\int_{E(h)} e^{\langle  m, z\rangle} \ dm &= {\rm mult}(\sig_E) \cdot \sum_{\sig \in \Sig(n)} \frac{ e^{\langle i_\sig^* c^\bT_1(\cO_X(D_{P(h)})),z\rangle}}{\langle Eu^\bT_X(x_\sig),  z \rangle} \cdot  \prod_{\rho \in \sig_E(1)} \langle i_\sig^* F_\rho, z\rangle \\
&=
 {\rm mult}(\sig_E) \cdot \sum_{\sig \in \Sig(n)} \frac{ e^{\langle (i_\sig^* c^\bT_1(\cO_X(D_{P})),z\rangle}}{\langle Eu^\bT_X(x_\sig),  z \rangle} \cdot  e^{ \sum_\rho h_\rho \langle i_\sig^* F_\rho,z\rangle} \cdot  \prod_{\rho \in \sig_E(1)} \langle i_\sig^* F_\rho, z\rangle \:.
 \end{split}
\ee
\et
\begin{proof}
The proof is similar to that of Theorem \ref{them1}, with the following modifications. Instead of $\td_*^\bT(X)$, we use $\td_*^\bT([(i_E)_*\cO_{V_{\sig_E}}])$, which is given by formula \eqref{dueqeet}, together with the corresponding equivariant Riemann-Roch formula of \eqref{lem2b} which is used for the calculation of Riemann sums approximating the integral. The Lebesgue measure $dm$ on $E(h)$ is now normalized so that the unit cube in the lattice $Span(E_0)\cap M$ has volume $1$, with $E_0:=E-m_0$ a translation of $E$ by a vertex $m_0 \in E$. As before, $\langle Eu^\bT_X(x_\sig),  z \rangle$ is a homogeneous polynomial of degree $n$ in $z$, but the additional factor $\prod_{\rho \in \sig_E(1)} \langle i_\sig^* F_\rho, z\rangle$ is homogeneous of degree $\codim(E)$, fitting with the use of the multiplication factor $\frac{1}{k^{\dim(E)}}$ needed for the Riemann sums considered here. Finally, we have
\be \lim_{k\to \infty}\left( k^{\codim(E)}\cdot  \langle i^*_\sig (\td_*^\bT([(i_E)_*\cO_{V_{\sig_E}}])), \frac{1}{k} \cdot z \rangle\right)
 ={\rm mult}(\sig_E) \cdot \prod_{\rho \in \sig_E(1)} \langle i_\sig^* F_\rho, z\rangle.
\ee
This follows from formula \eqref{dueqeet} by the same calculation as in the proof of Theorem \ref{them1}.
\end{proof}

By combining the results of Theorem \ref{them1} and Theorem \ref{them2}, we obtain as an application a formula which relates integration over faces with a differentiation with respect to the corresponding $h$'s.
\bc\label{EM-faces1} Let $f$ be a polynomial on $M_\bR$. Then
\be\label{rel1}
\int_{E(h)} f(m) e^{\langle m, z \rangle}  \ dm = {\rm mult}(\sig_E) \cdot \prod_{\rho \in \sig_E(1)} \frac{\partial}{\partial h_\rho} \int_{P(h)} f(m) e^{\langle m, z \rangle}  \ dm \:,
\ee
as analytic functions in $z$ and $h$ near zero. In particular,
\begin{itemize}
\item[(a)] 
\be\label{rel1a}
\begin{split}
p\left(\frac{\partial}{\partial h}\right)& \left( \int_{E(h)} f(m)e^{\langle m, z \rangle}  \ dm \right)_{\vert_{h=0}}\\ &= {\rm mult}(\sig_E) \cdot \prod_{\rho \in \sig_E(1)} \frac{\partial}{\partial h_\rho} \cdot  p\left(\frac{\partial}{\partial h}\right)\left( \int_{P(h)} f(m)e^{\langle m, z \rangle}  \ dm \right)_{\vert_{h=0}}\:,
\end{split}
\ee
for $p(x_\rho)$ a polynomial in the variables $x_\rho$, $\rho \in \Sig(1)$  (or a convergent power series as in the context of Remark \ref{convergence2}).
\item[(b)]
\be\label{rel1b}
p\left(\frac{\partial}{\partial h}\right)\left( \int_{E(h)} f(m)  \ dm \right)_{\vert_{h=0}} = {\rm mult}(\sig_E) \cdot \prod_{\rho \in \sig_E(1)} \frac{\partial}{\partial h_\rho} \cdot  p\left(\frac{\partial}{\partial h}\right)\left( \int_{P(h)} f(m)  \ dm \right)_{\vert_{h=0}}\:,
\ee
for a formal power series $p(x_\rho)$ in the variables $x_\rho$, $\rho \in \Sig(1)$.
\end{itemize}
\ec
 \begin{proof}
 For $f=1$, formula \eqref{rel1} follows by comparing the results of Theorem \ref{them1} and Theorem \ref{them2}. In the general case, we apply the operator $f(\frac{\partial}{\partial z})$ to both sides of the formula obtained for $f=1$. In (b) we finally evaluate the result at $z=0$. By Corollary \ref{polin}, a formal power series (as opposed to a convergent one) can be used in \eqref{rel1b}.
 \end{proof}

In particular,
\begin{equation}\label{rel2}
 \frac{\partial}{\partial h_{\rho_1}}\cdots \frac{\partial}{\partial h_{\rho_k}}   \int_{P(h)} f (m)e^{\langle m, z \rangle}  \ dm\:= 0 
\end{equation}
if the corresponding facets 
for different $\rho_i$ ($i=1,\dots,k$) do not intersect.\\

As another application of Corollary \ref{EM-faces1}, we can mention the following reformulation of the classical Euler-Maclaurin formula 
\eqref{EMBV} of Brion-Vergne (using \eqref{last}), for $f$ a polynomial function on $M_\bR$ (see also \cite{KSW}[page 22]):
\be\label{Guill}
\begin{split}
& \sum_{m\in P\cap M} f(m) = \sum_{E\preceq P}  \sum_{g\in G_{\sig_E}^{\circ}}  \prod_{\rho\in \Sigma(1)} 
\frac{ \frac{\partial}{\partial h_{\rho} }}{1-a_{\rho}(g)e^{-\frac{\partial}{\partial h_{\rho}} } }\left( \int_{P(h)} f(m) \ dm \right)_{|_{h=0}}\\
&= \sum_{E\preceq P} \frac{1}{{\rm mult}(\sigma_E)}
\sum_{g\in G_{\sig_E}^{\circ}} 
\prod_{\rho\not\in \sigma_E(1)} 
\frac{ \frac{\partial}{\partial h_{\rho} }}{1-e^{-\frac{\partial}{\partial h_{\rho}} } }
\prod_{\rho\in \sigma_E(1)} 
\frac{ 1}{1-a_{\rho}(g)\cdot e^{-\frac{\partial}{\partial h_{\rho}} } }\left( \int_{E(h)} f(m) \ dm \right)_{|_{h=0}}\:.
\end{split}
\ee
The Euler-Maclaurin formula using the last operator is due to Guillemin, see \cite{G}[Thm.1.3 and Eqn.3.28].

As a further application of Theorem \ref{them2}, we give an algebro-geometric proof of the {\it Stokes' formula for polytopes} \index{Stokes' formula}
(see also \cite{KSW}[Prop.6.1]).
\bt\label{thmStokes}
Let $P\subset M_\bR$ be a full-dimensional simple lattice polytope, with corresponding inner fan $\Sig$. Let $E_\rho$ denote the facet of $P$ corresponding to the ray $\rho \in \Sig(1)$.
For a fixed $m_0 \in M$, let $\frac{\partial}{\partial m_0 }$ be the differentiation in direction of $m_0$. Then,
\be\label{stok}
\int_{P(h)} \frac{\partial}{\partial m_0 }e^{\langle m, z \rangle}  \ dm = -\sum_{\rho \in \Sig(1)} \langle m_0, u_\rho \rangle \cdot \int_{E_\rho(h)} e^{\langle m, z \rangle}  \ dm.
\ee
\et
\begin{proof}
Recall that $$s(m_0) = -c^1_\bT(\bC_{\chi^{m_0}})=-\sum_{\rho \in \Sig(1)} \langle m_0 , u_\rho \rangle \cdot F_\rho \in H^*_\bT(X;\bQ).$$
Then

\begingroup
\allowdisplaybreaks
\begin{equation*}
\begin{split}
\int_{P(h)} \frac{\partial}{\partial m_0 } e^{\langle  m_, z\rangle} \ dm &= \int_{P(h)} \langle  m_0, z\rangle \cdot e^{\langle  m, z\rangle} \ dm \\
&\overset{\eqref{f113}}{=}
\sum_{\sig \in \Sig(n)} \frac{ e^{\langle i_\sig^* c^\bT_1(\cO_X(D_{P(h)})),z\rangle}}{\langle Eu^\bT_X(x_\sig),  z \rangle} \cdot \langle  m_0, z\rangle \\
&= \sum_{\sig \in \Sig(n)} \frac{ e^{\langle i_\sig^* c^\bT_1(\cO_X(D_{P(h)})),z\rangle}}{\langle Eu^\bT_X(x_\sig),  z \rangle} \cdot \langle  i_\sig^* s(m_0), z\rangle \\
&= -\sum_{\rho \in \Sig(1)} \langle m_0 , u_\rho \rangle \cdot \sum_{\sig \in \Sig(n)} \frac{ e^{\langle i_\sig^* c^\bT_1(\cO_X(D_{P(h)})),z\rangle}}{\langle Eu^\bT_X(x_\sig),  z \rangle} \cdot \langle  i_\sig^* F_\rho, z\rangle \\
&\overset{\eqref{f113b}}{=} -\sum_{\rho \in \Sig(1)} \langle m_0, u_\rho \rangle \cdot \int_{E_\rho(h)} e^{\langle m, z \rangle}  \ dm ,
\end{split}
\end{equation*}
\endgroup
as desired.
\end{proof}

Applying the operator $f(\frac{\partial}{\partial z })$ to formula \eqref{stok} for $f$ a polynomial function on $M_\bR$, we get the following:
\be\label{stok2}
\int_{P(h)} \frac{\partial}{\partial m_0 }f(m)e^{\langle m, z \rangle}  \ dm = 
-\sum_{\rho \in \Sig(1)} \langle m_0, u_\rho \rangle \cdot \int_{E_\rho(h)} f(m)e^{\langle m, z \rangle}  \ dm
\ee
as analytic functions in $h$ and $z$ near zero. Further evaluation at $z=0$ yields the identity
\be\label{stok3}
\int_{P(h)} \frac{\partial}{\partial m_0 }f(m) \ dm = 
-\sum_{\rho \in \Sig(1)} \langle m_0, u_\rho \rangle \cdot \int_{E_\rho(h)} f(m) \ dm ,
\ee
as polynomial functions in $h$ near zero.\\

\begin{rem}\label{lastrem}
By Remark \ref{CSrem}, there is a surjection $\cA(P) \to \widehat{H}^*_\bT(X;\bQ)$, as $(\widehat{\Lambda}_\bT)_\bQ$-modules, with $\cA(P)$ generated by the equivalence classes of $[V_{E}] \in \cA(P)$, which are mapped to the equivariant fundamental classes $[V_{\sigma_E}]_\bT$ of orbit closures corresponding to all faces $E$ of $P$. Hence, $\widehat{H}^*_\bT(X;\bQ)$ is generated as a $(\widehat{\Lambda}_\bT)_\bQ$-module by the equivariant fundamental classes 
$[V_{\sigma_E}]_\bT$. 
\end{rem}

Let $[\cF]\in K^\bT_0(X)$ be fixed, and choose 
elements $$p_{E}(t_i)\in \widehat{H}^*_\bT(pt;\bQ)= (\widehat{\Lambda}_\bT)_\bQ\simeq \bQ[[t_1,\dots,t_n]]$$ with
\begin{equation}\label{main-todd}
\td_*^\bT([\cF])= \sum_{E\preceq P} \: p_{E}(t_i)[V_{\sigma_E}]_\bT \in \widehat{H}^*_\bT(X;\bQ)\:.
\end{equation}
Then $$\td_*^\bT([\cF])=p(s(m_i),F_\rho) \in \widehat{ H}^*_\bT(X;\bQ),$$ for
\be\label{main-todd2}
p(t_i,x_\rho):=\sum_{E\preceq P} \: \left( {\rm mult}(\sig_E) \cdot \prod_{\rho \in \sig_E(1)} x_\rho\right) \cdot p_{E}(t_i) \in \bQ[[t_1,\dots,t_n,x_\rho \mid \rho \in \Sig(1) ]]\:.
\ee

Altogether, with these notations,  we get our third and final   abstract Euler-Maclaurin formula based on the equivariant Hirzebruch-Riemann-Roch theorem, which provides a {\it generalization of the Cappell-Shaneson Euler-Maclaurin formula} (see Remark \ref{r716} below).
\bt\label{abstrEM3}
Let $X=X_P$ be the projective simplicial toric variety associated to a simple full-dimensional  lattice polytope $P\subset M_\bR$. Let $\Sig:=\Sig_P$ be the inner normal fan of $P$, and $D:=D_P$ the ample Cartier divisor associated to $P$. Let $[\cF]\in K^\bT_0(X)$ be fixed, and choose the formal power series $p(t_i, x_\rho) \in \bQ[[t_1,\dots,t_n,x_\rho \mid \rho \in \Sig(1) ]]$ as in \eqref{main-todd2}. Then for a polynomial  function $f$ on $M_\bR$, we have:
\be\label{cs2137}
\sum_{E\preceq P} \:  \int_{E} \left(p_{E}(\partial_i) f\right)(m) \ dm 
= \sum_{m\in M} \left( \sum_{i=0}^n (-1)^i \cdot \dim_\bC H^i(X;\cO_X(D) \otimes \cF)_{\chi^{-m}}\right) \cdot f(m)\:.
\ee
\et

\begin{proof}
The assertion follows from Corollary \ref{EM-faces1} by the following calculation:
\be\begin{split}
&p\left(\partial_i, \frac{\partial}{\partial h}\right) \left( \int_{P(h)} f(m) \ dm \right)_{\vert_{h=0}}\\ 
&= \sum_{E\preceq P} \: \left( {\rm mult}(\sig_E) \cdot \prod_{\rho \in \sig_E(1)} \frac{\partial}{\partial h_\rho}\right) 
\left(\int_{P(h)} \left(p_{E}(\partial_i) f\right)(m) \ dm \right)_{\vert_{h=0}} \\
 &= \sum_{E\preceq P} \: 
\left(\int_{E(h)} \left(p_{E}(\partial_i) f\right)(m) \ dm \right)_{\vert_{h=0}} \\
 &= \sum_{E\preceq P} \: 
\int_{E} \left(p_{E}(\partial_i) f\right)(m) \ dm \:.
\end{split}\ee
Here the first equality follows from the definition \eqref{functional2} of 
$p(\partial_i, \frac{\partial}{\partial h}) \left( \int_{P(h)} f(m) \ dm \right)_{\vert_{h=0}}$ and the choice \eqref{main-todd2} of $p(t_i,x_\rho)$.
Then the second equality follows from \eqref{rel1b}, and the final equality is just continuity in $h$.
\end{proof}

\begin{rem}\label{r716} In the classical case $\cF:=\cO_X$,
this is just Cappell-Shaneson's recipe for the definition of the differential operators $p_{E}(\partial_i)$, described here geometrically in terms of 
the equivariant Todd class $\td_*^\bT(X):=\td_*^\bT([\cO_X])\in  \widehat{H}^*_\bT(X;\bb{Q})$  (see \cite{CS2}[Theorem 2]). More precisely, the differential operator $\mathcal{E}=\sum_{E \preceq P} P_E \cdot U_E \in \bQ[[\partial_1,\dots,\partial_n]][U_F]$ used in \cite{S}[6.2, page 622] and \cite{CS2}[page 888] maps under the homomorphism \eqref{main} to a presentation of the equivariant Todd class $\td_*^\bT(X)=\sum_{E\preceq P} \: p_{E}(t_i)[V_{\sigma_E}]_\bT \in \widehat{H}^*_\bT(X;\bQ)$. Indeed, the classes $[V_{\sigma_E}]_\bT$  are the images of the  elements $U_E \in \bQ[[\partial_1,\dots,\partial_n]][U_F]$ used in loc.cit., and the operators $P_{E}\in \bQ[[\partial_1,\dots,\partial_n]]$ are deduced in \cite{CS2,S} from a representative $T(E) \in \bQ[[\partial_1,\dots,\partial_n]][U_F]$ mapping to the equivariant Todd class $\td_*^\bT(X)$ calculated via the equivariant version of Theorem \ref{24}. So we can choose $p_{E}$ to be the images of these $P_E$ in $(\widehat{\Lambda}_\bT)_\bQ$ under \eqref{main}. Hence, the presentation
$$p(t_i,x_\rho):=\sum_{E\preceq P} \: \left( {\rm mult}(\sig_E) \cdot \prod_{\rho \in \sig_E(1)} x_\rho\right) \cdot p_{E}(t_i) \in \bQ[[t_1,\dots,t_n,x_\rho \mid \rho \in \Sig(1) ]] $$
chosen here is exactly the image of $\mathcal{E}$ in $\bQ[[t_1,\dots, t_n]][[x_\rho]]$.

In this case, when $\cF:=\cO_X$, the right hand side of \eqref{cs2137} becomes 
$\sum_{m\in P\cap M} f(m)$ by Example \ref{ex:classical1}.
Specializing further to $f=1$ yields then the classical formula (see, e.g., \cite{F1}[page 112]):
\begin{equation}\label{17}
\vert P \cap M \vert=\sum_{E\preceq P}   r_E \cdot vol(E)\:,
\end{equation}
with $r_E=p_{\sig_E}(0)$ the constant coefficient of $p_{\sig_E}$ and
\begin{equation}\label{ne-todd}
\td_*(X)= \sum_{E\preceq P} \: r_E [V_{\sigma_E}]  \in H^*(X;\bQ)
\end{equation}
the non-equivariant version of
\begin{equation}\label{eq-todd-classical1}
\td_*^\bT([\cO_X])= \sum_{E\preceq P} \: p_{E}(t_i)[V_{\sigma_E}]_\bT \in \widehat{H}^*_\bT(X;\bQ)\:.
\end{equation}
\end{rem}

When $\cF:=\omega_X$, the right hand side of \eqref{cs2137} becomes 
$\sum_{m\in \Int(P)\cap M} f(m)$ by Example \ref{ex:classical2}.
Moreover, here we can choose
\begin{equation}\label{eq-todd-classical2}
\td_*^\bT([\omega_X])= \sum_{E\preceq P} \: p_{E}(-t_i)\cdot (-1)^{\codim(E)}[V_{\sigma_E}]_\bT \in \widehat{H}^*_\bT(X;\bQ)\:,
\end{equation}
by Remark \ref{rem:duality}, with $p_{E}(t_i)$ as in (\ref{eq-todd-classical1}).
Specializing further to $f=1$ yields the formula
\begin{equation}\label{17b}
\vert \Int(P) \cap M \vert=\sum_{E\preceq P}  (-1)^{\codim(E)}\cdot r_E \cdot vol(E)\:,
\end{equation}
with $r_E$ as in (\ref{ne-todd}).

\bex\label{finalex} Here we list further specializations of formula \eqref{cs2137}
for appropriate choices of $[\cF]\in K^\bT_0(X)$, resp.,  $[\cF]\in K^\bT_0(X)[y]$, with
$\td_*^\bT([\cF])= \sum_{E\preceq P} \: p_{E}(t_i)[V_{\sigma_E}]_\bT \in \widehat{H}^*_\bT(X;\bQ)$, resp., $\widehat{H}^*_\bT(X;\bQ)[y]$
for some elements $p_{E}(t_i)\in \widehat{H}^*_\bT(pt;\bQ)$, resp., $\widehat{H}^*_\bT(pt;\bQ)[y]$.
\begin{enumerate}
\item For $[\cF]=mC_y^\bT(X)\in K^\bT_0(X)[y]$,  \eqref{cs2137} becomes by equation \eqref{wem01}:
$$\sum_{E\preceq P} \:  \int_{E} \left(p_{E}(\partial_i) f\right)(m) \ dm = \sum_{E'' \preceq P} (1+y)^{\dim(E'')} \sum_{m \in \Relint(E'') \cap M} f(m)\:.$$
\item For $[\cF]=mC_0^\bT([U \hookrightarrow X])\in K^\bT_0(X)$, with $U=X\setminus D_K$ the open complement of the divisor $D_K=\bigcup_{i \in K} D_i$ and $D_i:=D_{\rho_i}$,   \eqref{cs2137} becomes by equation \eqref{EMgen}:
$$\sum_{E\preceq P} \:  \int_{E} \left(p_{E}(\partial_i) f\right)(m) \ dm = \sum_{m \in P^K \cap M} f(m)\:.$$
Here $P^K$ is the set obtained from $P$ by removing the facets $F_i$ for $i\in K$.
More generally, for $[\cF]=mC_y^\bT([U \hookrightarrow X])\in K^\bT_0(X)[y]$, \eqref{cs2137} becomes by equation \eqref{wEMgen}:
$$\sum_{E\preceq P} \:  \int_{E} \left(p_{E}(\partial_i) f\right)(m) \ dm =  \sum_{E'' \preceq P^K} (1+y)^{\dim(E'')} \cdot \sum_{m \in \Relint(E'') \cap M} f(m)\:.$$
\item For $[\cF]=[(i_\sig)_*\omega_{V_\sigma}]\in K^\bT_0(X)$, with $V_\sigma=V_{\sigma_{E'}}$ the closure of the orbit of $\sigma$ in $X=X_P$ corresponding to the face $E'$ of $P$,   \eqref{cs2137} becomes by equation \eqref{EMintd}:
$$\sum_{E\preceq P} \:  \int_{E} \left(p_{E}(\partial_i) f\right)(m) \ dm = \sum_{m \in \Relint(E') \cap M} f(m)\:.$$
\item For $[\cF]=[(i_\sig)_*\cO_{V_\sigma}]\in K^\bT_0(X)$, with $V_\sigma=V_{\sigma_{E'}}$ the closure of the orbit of $\sigma$ in $X=X_P$ corresponding to the face $E'$ of $P$,   \eqref{cs2137} becomes by equation \eqref{EMint}:
$$\sum_{E\preceq P} \:  \int_{E} \left(p_{E}(\partial_i) f\right)(m) \ dm = \sum_{m \in E' \cap M} f(m)\:.$$
More generally, for $[\cF]=(i_\sig)_*mC_y^\bT(V_\sigma) \in K^\bT_0(X)[y]$, \eqref{cs2137} becomes by equation \eqref{EMinth}:
$$\sum_{E\preceq P} \:  \int_{E} \left(p_{E}(\partial_i) f\right)(m) \ dm =
\sum_{E'' \preceq E'} (1+y)^{\dim(E'')} \cdot \sum_{m \in \Relint(E'') \cap M} f(m)\:.$$
\end{enumerate}
\qed
\eex

%%%%%%%%%%%%%%%%%%%%%%
\subsection{Generalized Reciprocity for Dedekind Sums via Euler-Maclaurin formulae  }\label{sec7.3}

We conclude this paper with the following application of formula \eqref{cs2137}, with $D$ replaced by $(1+y)D$.

\bc\label{corec} In the context of Theorem \ref{abstrEM3}, we have an equality of polynomials in $1+y$,
\be\label{273}\begin{split}
\sum_{E_y \preceq P_y} \:  \int_{E_y} \left( p_{E_y}(\partial_i) f\right) &(m) \ dm 
\\ &= \sum_{m\in M} \left( \sum_{i=0}^n (-1)^i \cdot \dim_\bC H^i(X;\cO_X((1+y)D) \otimes \cF)_{\chi^{-m}}\right) \cdot f(m)\:.\end{split}
\ee
When evaluating these polynomials at zero (i.e., for $y=-1$), one gets the following identity:
\be\label{274}
\sum_{v \in P} \left( p_{v}(\partial_i) f \right) (0) = \sum_{m\in M} \left( \sum_{i=0}^n (-1)^i \cdot \dim_\bC H^i(X;\cF)_{\chi^{-m}}\right) \cdot f(m).
\ee
where the left hand sum is over the vertices of $P$.
\ec

\begin{proof}
The right hand sides of the above expressions are calculated in Corollary \ref{strc}.
For the left hand side of \eqref{273}, note that for a given face $E$ of $P$,  if $f$ is a homogeneous polynomial of degree $d_f$, and $p(\partial_i)$ is a homogeneous differential operator of degree $d\leq d_f$, then
\be\label{275}
  \int_{E_y} \left(p(\partial_i) f\right)(m) \ dm = (1+y)^{\dim(E) + d_f - d}   \int_{E} \left(p(\partial_i) f\right)(m) \ dm.
\ee
This explains the polynomial behavior of the left hand side of  \eqref{273}, and it yields  \eqref{274} upon evaluation at zero (since by letting $y=-1$, all terms on the right hand side of \eqref{275} vanish except for $\dim(E)=0$ and $d=d_f$).
\end{proof}

\bex
If $\cF=\cO_X$ in \eqref{274}, one gets by Example \ref{rigid} the following identity:
\be
\sum_{v \in P} \left( p_{v}(\partial_i) f \right) (0) = f(0).
\ee
For instance, in the case of lattice polygones, this formula yields generalizations of \index{reciprocity law} reciprocity laws for classical \index{Dedekind sum} {\it Dedekind sums} (using, e.g., the explicit description of the operators $p_{v}(\partial_i)$ from \cite{CS2}[page 889]). Details will be explained in forthcoming work.
\qed
\eex

%%%%%%%%%%%%%%%%%%%%%%%%%%%

%\printindex

\end{document}